\documentclass[11pt,reqno]{amsart}
\usepackage{xcolor}
\usepackage{cancel}
\usepackage{color}
\usepackage{graphicx}
\usepackage[normalem]{ulem}
\usepackage[latin1]{inputenc}
\usepackage{pb-diagram}
\usepackage{mathrsfs}
\numberwithin{equation}{section}

\usepackage{amsmath,amsthm,amscd}
\usepackage[psamsfonts]{amssymb}
\usepackage[all]{xy}
\usepackage{here}

\newtheorem{Thm}{Theorem}[section]
\newtheorem{Prop}[Thm]{Proposition}
\newtheorem{Lem}[Thm]{Lemma}
\newtheorem{Cor}[Thm]{Corollary}

\newtheorem{Prob}[Thm]{Problem}

\newtheorem{Claim}[Thm]{Claim}

\newtheorem{Obs}[Thm]{Observation}
\newtheorem{mainThm}{Theorem}

\newtheorem{mainCor}[mainThm]{Corollary}

\theoremstyle{remark}
\newtheorem{Rem}[Thm]{Remark}

\theoremstyle{definition}
\newtheorem{Def}[Thm]{Definition}
\newtheorem{Assum}[Thm]{Assumption}
\newtheorem{Exa}[Thm]{Example}

\def\Z{{\mathbb Z}}
\def\R{{\mathbb R}}
\def\Q{{\mathbb Q}}

\def\L{{\mathbb L}}

\def\calB{\mathcal{B}}

\def\calE{\mathcal{E}}

\def\calI{\mathcal{I}}

\def\calM{\mathcal{M}}

\def\calP{\mathcal{P}}

\def\calV{\mathcal{V}}

\def\calX{\mathcal{X}}

\def\Diff{\mathrm{Diff}}
\def\Emb{\mathrm{Emb}}
\def\fEmb{\mathrm{Emb}^{\mathrm{fr}}}
\def\bEmb{\overline{\Emb}}
\def\CEmb{\mathrm{CEmb}}

\def\Map{\mathrm{Map}}
\def\ve{\varepsilon}

\def\fr{{\mathrm{fr}}}
\def\fig#1#2#3{
\raisebox{#1}{\includegraphics[height=#2]{#3}}
}
\def\tvec#1{\mbox{\boldmath{$\scriptstyle#1$}}}
\def\st{\iota}
\def\tbigcup{\textstyle\bigcup}
\def\colim#1#2{\underset{#1}{\mathrm{colim}}\,#2}
\def\holim#1#2{\underset{#1}{\mathrm{holim}}\,#2}
\def\hocolim#1#2{\underset{#1}{\mathrm{hocolim}}\,#2}
\def\hofib{\mathrm{hofib}}
\def\tcoprod{\,\raisebox{1pt}{$\scriptstyle\coprod$}\,}
\def\tbigwedge{\textstyle\bigwedge}
\def\Aut{\mathrm{Aut}\,}
\def\vect#1{\mbox{\boldmath{$#1$}}}
\def\tvec#1{\mbox{\boldmath{$\scriptstyle#1$}}}
\def\bvec#1{\mbox{\boldmath{$#1$}}}
\def\tbvec#1{\mbox{\scriptsize\boldmath{$#1$}}}
\def\const{\mathrm{const}}

\def\10T{\mathrm{10T}}
\def\pr{\mathrm{pr}}

\def\Lk{\mathrm{Lk}}
\def\supp{\mathrm{supp}}

\def\rh{\hookrightarrow}

\begin{document}

\title[Brunnian Links and Kontsevich Graph Complex I]{Brunnian Links and Kontsevich Graph Complex I}

\author{Boris Botvinnik}
\address{
Department of Mathematics\\
University of Oregon \\
Eugene, OR, 97405\\
USA
}
\thanks{BB was partially supported by Simons collaboration grant 708183}
\email{botvinn@uoregon.edu}

\author{Tadayuki Watanabe} 
\address{Department of Mathematics \\ 
Kyoto University \\ 
Kyoto 606-8502 \\ 
Japan 
}
\email{tadayuki.watanabe@math.kyoto-u.ac.jp}
\thanks{TW was partially supported by JSPS Grant-in-Aid for Scientific Research 21K03225, 20K03594 and by RIMS, Kyoto University.}
\date{\today}
\dedicatory{Dedicated to Professor Tomotada Ohtsuki on the occasion of his 60th birthday.}

\maketitle
\vspace*{-10mm}

\begin{abstract}
We construct a natural chain map from the Kontsevich graph complex to the rational singular chain complex of $B\mathrm{Diff}_\partial(D^{2k})$ when the dimension $2k$ is sufficiently large, generalizing Goussarov and Habiro's theories of surgery on 3-valent graphs in 3-manifolds. Our construction can be considered as a topological realization of the Kontsevich graph complex. We also give new constructions of elements in the rational homotopy groups of $B\mathrm{Diff}_\partial(D^{2k})$ which are determined by well-known cycles in the graph complex.
\end{abstract}

\setcounter{tocdepth}{2}
\parskip=1.2mm
\tableofcontents
\parskip=2mm

\renewcommand{\baselinestretch}{1}\normalsize
\section{Introduction}

\subsection{Motivation and main results}

It is known that the homology of the Kontsevich graph complex (\cite{Kon}), also called graph homology, appears naturally in various mathematical objects.  An example relevant to this paper is that the part of the graph homology consisting of 3-valent graphs plays a central role in the study of finite type invariants of homology 3-spheres (\cite{Oh96}, see also \cite{Oh02}). 
In particular, in their independent works, Goussarov and Habiro invented elegant methods of surgery on 3-valent graphs for homology 3-spheres (\cite{Gou,Hab,GGP}). Their surgery construction is known to give an isomorphism between the degree $k$ 3-valent part of the graph homology and the homogeneous degree $3k$ part of the Ohtsuki filtration of integral homology 3-spheres defining the finite type invariants (\cite{Hab,GGP}). For higher dimensional manifolds, the second author studied higher dimensional analogues of Goussarov--Habiro's surgery for 3-valent graphs and found that they determine nontrivial subspaces in the rational homotopy groups of $B\Diff_\partial(D^d)$, the classifying space of the group of diffeomorphisms of $D^d$ pointwise fixing a neighborhood of the boundary, for $d\geq 4$ (see \cite{Wa09,Wa18,Wa21,Wa22}). On the other hand, it is known that there are also many nontrivial classes in the graph homology that are not 3-valent (\cite{BNM,Br,Wi,WZ,BrWi}). 

Recently, Kupers and Randal-Williams computed $\pi_i(B\Diff_\partial(D^{2k}))\otimes\Q$ for quite general $i$, updating the computation by Farrell and Hsiang based on stable pseudoisotopy theory (\cite{FH}). 
\begin{Thm}[\cite{KRW}]\label{thm:KRW}
For $2k\geq 6$, we have
\[ \pi_i(B\Diff_\partial(D^{2k}))\otimes\Q
=\left\{\begin{array}{ll}
\Q & (i\equiv 2k-1\text{ mod }4,\text{ and }i\geq 2k-1),\\
0 & (i\not\equiv 2k-1\text{ mod }4,\text{ or }i<2k-1)
\end{array}\right. \]
outside the ``bands'' $\bigcup_{n\geq 2}[2kn-4n-1,2kn-1]$ of degrees. The term $\Q$ above is detected by Pontryagin--Weiss classes (\cite{We})\footnote{There is also a major update of the computation of \cite{FH} by Krannich and Randal-Williams (\cite{KrRW}) for odd-dimensional disks.}.
\end{Thm}
We study what happens inside the ``bands''.
Let us remind briefly what the Kontsevich graph complex $(\mathcal{GC},\partial)$ is.
Let $\mathcal{GC}$ be the vector space over $\Q$ spanned by equivalence classes of oriented graphs $(\Gamma,o)$, where
\begin{itemize}
\item $\Gamma$ is a finite connected graph with only vertices of valence at least 3,
\item $o$ is a choice of orientation of $\R^{\{\text{edges of $\Gamma$}\}}$, and 
\item $(\Gamma,-o)= -(\Gamma,o)$.
\end{itemize}
We will simply write $\Gamma$ for $(\Gamma,o)$. As in \cite{BNM}, we consider the bigrading $\mathcal{GC}=\bigoplus_{n,m}\mathcal{GC}^{(n,m)}$ by the ``degree'' $n=|E(\Gamma)|-|V(\Gamma)|$, and the ``excess'' $m=2|E(\Gamma)|-3|V(\Gamma)|$, where $E(\Gamma)=\{\text{edges of $\Gamma$}\}$ and $V(\Gamma)=\{\text{vertices of $\Gamma$}\}$. 
The boundary operator $\partial\colon \mathcal{GC}^{(n,m)}\to \mathcal{GC}^{(n,m-1)}$ is defined by
\begin{equation}\label{eq:dGamma}
 \partial\Gamma=\sum_{v\in V(\Gamma)} \text{split}(\Gamma,v), 
\end{equation}
where $\text{split}(\Gamma,v)$ is the sum of all possible graphs obtained from $\Gamma$ by replacing a vertex $v$ with an edge with two boundary vertices. The orientation of each term in $\text{split}(\Gamma,v)$ is given in a natural way so that the new edge is labelled by $1$. The formula (\ref{eq:dGamma}) is analogous to the defining relations in $L_\infty$-algebra, introduced in works of Stasheff with Schlessinger (\cite{St}). See also Example~\ref{ex:brackets}. We denote by $H_{n,m}(\mathcal{GC})$ the $(n,m)$-th term in the homology of $(\mathcal{GC},\partial)$. Kontsevich used the dual graph complex to define characteristic classes of smooth homology sphere bundles (\cite{Kon}). The version of the graph complex given here works for even dimensional manifolds. There is also another version which works for odd dimensional manifolds (see \cite{Kon}). By the definition of $\partial$, the 3-valent part $H_{n,0}(\mathcal{GC})$ is $\mathcal{GC}^{(n,0)}/{\text{IHX}}$, where IHX is the relation shown in Figure~\ref{fig:ihx}.
\begin{figure}[h]
\[ \includegraphics[height=18mm]{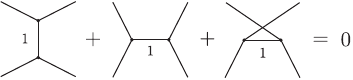} \]
\caption{The IHX relation.}\label{fig:ihx}
\end{figure}

The works of Goussarov and Habiro can be considered as a ``topological realization'' of the 3-valent part of the graph homology. Let us review this point more precisely. 
Goussarov--Habiro's graph surgery is defined by iterating surgeries on several disjoint genus 3 handlebodies embedded in a 3-manifold as in Figure~\ref{fig:theta-Y}. Surgery on a single genus 3 handlebody corresponds to a 3-valent vertex in a graph (Figure~\ref{fig:theta-Y}, right), and it was originally defined by Matveev as the replacement of the handlebody with the complement of a thickened Borromean string link (\cite{Mat}).
\begin{figure}[h]
\[ \includegraphics[height=35mm]{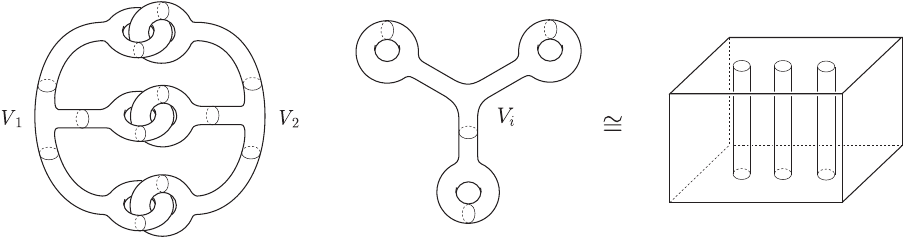} \]
\par\bigskip
\caption{From a 3-valent graph to the union of genus 3 handlebodies, each of which is identified with the complement of a thickened string link.}\label{fig:theta-Y}%
\end{figure}%

For higher dimensional manifolds, surgery on a 3-valent vertex is defined analogously by using a higher dimensional Borromean string link. In particular, an iteration of 3-valent vertex surgeries along a 1,3-valent tree $T$ with $\ell$ 1-valent vertices, which are attached to the components of a standard inclusion $I^{a_1}\cup\cdots\cup I^{a_{\ell}}\subset I^N$, yields another Brunnian string link $f_T\colon I^{a_1}\cup\cdots\cup I^{a_{\ell}}\to I^N$ when the condition $a_1+\cdots+a_\ell=(\ell-1)N-\ell$ is satisfied (see \cite[Lemma~3.7]{BW}). We will see that the $\ell$-th component of $f_T$ represents the iterated Whitehead product in $I^N-(I^{a_1}\cup\cdots\cup I^{a_{\ell-1}})\simeq S^{N-a_1-1}\vee \cdots\vee S^{N-a_{\ell-1}-1}$ of the embeddings $\gamma_1,\ldots,\gamma_{\ell-1}\colon I^{N-a_j-1}\to I^N-(I^{a_1}\cup\cdots\cup I^{a_{\ell-1}})$ that generate the rational homotopy group, where the type of the bracketting is determined by the tree $T$. For example, we have the following correspondence:
\[ \begin{split}
  T=\fig{-7mm}{15mm}{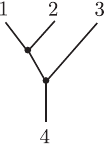} &\quad\longmapsto\quad f_T\in \Emb_\partial(I^{a_4},I^N-(I^{a_1}\cup I^{a_2}\cup I^{a_3}))\\
  &\quad\longmapsto\quad [[\gamma_1,\gamma_2],\gamma_3]\in \pi_{a_4}(I^N-(I^{a_1}\cup I^{a_2}\cup I^{a_3})).
\end{split}\]
Each 3-valent vertex produces an embedded Whitehead product (relevant works are found in \cite{Hab,GGP,Kos24a,Kos24b,KKV20,KKV24,Koy,KMV,STT}). 
From this viewpoint, the following holds\footnote{A proof based on Kirby calculus in higher dimensions was suggested by K.~Habiro to the second author many years ago.}. 

\begin{Obs}[See also Corollary~\ref{cor:jacobi-wh}]
If $N-a_i\geq 3$ for $i=1,2,3,4$, the IHX relation (Figure~\ref{fig:ihx}) for 1,3-valent trees with four 1-valent vertices corresponds to the Jacobi identity for Whitehead products in the homotopy group\footnote{When $N=3$ and $a_i=1$, the Jacobi identity still holds modulo higher commutators. The lift of the Jacobi identity to string links is given in \cite{Hab,GGP}.} $\pi_{a_4}(I^N-(I^{a_1}\cup I^{a_2}\cup I^{a_3}))$:
\[ (-1)^{q_1q_3}[[\gamma_1,\gamma_2],\gamma_3]+(-1)^{q_1q_2}[[\gamma_2,\gamma_3],\gamma_1]+(-1)^{q_2q_3}[[\gamma_3,\gamma_1],\gamma_2]=0, \]
where $q_i=N-a_i-1$.
\end{Obs}

The goal of this paper is to give a new topological realization of arbitrary connected graphs with valence at least 3 as singular chains of $B\Diff_\partial(D^d)$, similarly to the 3-valent case, to obtain homologically nontrivial cycles of $B\Diff_\partial(D^d)$. Let $B\Diff_\partial^\fr(D^d)$ denote the component of the classifying space of smooth $(D^d,\partial)$-bundles with framings along fibers that agree near $\partial D^d$ with the standard one obtained from the Euclidean coordinates and whose restriction to the base fiber is homotopic relative to the boundary to the standard one. We use certain Brunnian string links (which are natural analogues of the Borromean string link) for higher $\ell$-valent vertex (Theorem~\ref{thm:vertex}) to prove the following result.
\begin{mainThm}[Main Theorem]\label{thm:main}
Let $\mu\geq 1$ be an integer. Let $\mathcal{GC}^{(\leq \mu)}$ denote the subcomplex of $\mathcal{GC}$ obtained by replacing $\mathcal{GC}^{(n,m)}$ for $m>\mu$ with 0. Let $S_*(-;\Q)$ denote the normalized rational singular chain complex, i.e. the complex of rational singular chains modulo degenerate chains {\rm (e.g. \cite[I.4.(a)]{FHT})}. For a sufficiently large integer $k$, we have a natural chain map
\[ \overline{\phi}\colon \mathcal{GC}^{(\leq \mu)}\to S_*(B\Diff_\partial^\fr(D^{2k});\Q) \]
that extends the 3-valent graph surgery and that maps $\mathcal{GC}^{(n,m)}$ to $S_{(2k-3)n+m}(B\Diff_\partial^\fr(D^{2k});\Q)$. Hence we have the induced map
\[ \overline{\phi}_*\colon H_{n,m}(\mathcal{GC}^{(\leq \mu)})\to H_{(2k-3)n+m}(B\Diff_\partial^\fr(D^{2k});\Q). \]
\end{mainThm}

Theorem~\ref{thm:main} is separated into Theorems~\ref{thm:graph-surgery} and \ref{thm:chainmap}. The condition for the dimension $2k$ is satisfied if $2k\geq 2\mu^2+8\mu+10$. See Remark~\ref{rem:n-l-valent}. Since $m=2n-|V(\Gamma)|\leq 2n-1$, we have $H_{n,m}(\mathcal{GC}^{(\leq 2n-1)})=H_{n,m}(\mathcal{GC})$.
Our construction of the chain map $\overline{\phi}$ is a natural extension of the construction for 3-valent graphs. Namely, just as families of genus 3 handlebodies and the Borromean links are used to define surgery on 3-valent graphs, we use families of higher-genus handlebodies and families of Brunnian links satisfying $L_\infty$-like relations to define surgery on graphs. 

An explicit example of a homology class in $H_*(B\Diff_\partial^\fr(D^{2k});\Q)$ obtained from excess 2 graphs is given as follows. Let $X$ be the wheel graph with 5 spokes. Let $Y$ be another graph with 6 vertices and 10 edges of excess 2 that is nontrivial in $\mathcal{GC}$ (see Figure~\ref{fig:XY}). We give the orientations of $X$ and $Y$ by the labellings on edges in Figure~\ref{fig:XY}. 
According to Bar-Natan--McKay (\cite{BNM}) or as can be easily checked, the following linear combination is a nontrivial $\partial$-cycle in $\mathcal{GC}$. 
\[ \gamma:=\frac{2X}{|\mathrm{Aut}\,X|}-\frac{Y}{|\mathrm{Aut}\,Y|}=\frac{X}{5}-\frac{Y}{2} \]
\begin{figure}[h]
\[ X=\fig{-12mm}{25mm}{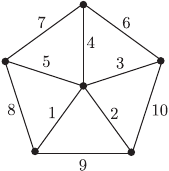},\qquad
  Y=\fig{-11mm}{23mm}{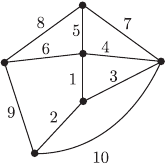}.\]
\caption{The graphs $X$ and $Y$.}\label{fig:XY}
\end{figure}
\begin{mainCor}\label{cor:gamma}
Let $k\geq 17$. Then $\overline{\phi}_{\gamma}$ is a $(8k-10)$-cycle in $S_{8k-10}(B\Diff_\partial^\fr(D^{2k});\Q)$.
\end{mainCor}%

An important feature of Theorem~\ref{thm:main} is that the cycles obtained from graphs can be evaluated by characteristic classes. In particular, Kontsevich's configuration space integrals give a map
\[ H^*(\mathcal{GC}^\vee)\otimes H_*(\mathcal{GC}^{(\leq \mu)})\stackrel{\overline{\phi}_*}{\longrightarrow}
 H^*(\mathcal{GC}^\vee)\otimes H_*(B\Diff_\partial^\fr(D^{2k});\Q)\longrightarrow \R, \]
where $\mathcal{GC}^\vee=\mathrm{Hom}_\Q(\mathcal{GC},\Q)$ is the dual complex.
In Part II (\cite{BWII}), we prove that this is far from trivial. In particular, we have the following.

\begin{mainThm}[{\cite{BWII}}]\label{thm:injective}
Let $\mu$ be an integer $\geq 1$. For a sufficiently large integer $k$, the induced map $\overline{\phi}_*\colon H_{m,n}(\mathcal{GC}^{(\leq \mu)})\to H_{(2k-3)n+m}(B\Diff_\partial^\fr(D^{2k});\Q)$ is injective up to excess $\mu-1$. Moreover, the injectivity can be detected by Kontsevich's configuration space integrals.
\end{mainThm}
Theorem~\ref{thm:injective} is the easier part of our results in the sense that the method of detecting classes in $H_{(2k-3)n+m}(B\Diff_\partial^\fr(D^{2k});\Q)$ from graphs is the same as \cite{Wa18,Wa21}.
Lemmas~\ref{lem:int-wedge-theta} and \ref{lem:leaf-form-barE} in this paper play crucial roles in the computation in \cite{BWII}. See also Remark~\ref{rem:crucial-nontriviality}. 
By \cite{Wa21}, Theorem~\ref{thm:injective} for excess $m-1=0$ (3-valent graphs) holds for $2k\geq 4$. For positive excess graphs, we have the following.

\begin{mainCor}[{\cite{BWII}}]\label{cor:nontriv}
Let $k\geq 17$. Then the $(8k-10)$-cycle $\overline{\phi}_\gamma$ represents a nontrivial class in $H_{8k-10}(B\Diff_\partial^\fr(D^{2k});\Q)$. 
\end{mainCor}

A key idea in the construction of $\overline{\phi}$ is to define surgery on a vertex which implement particular algebraic structure in the space of string links. This algebraic structure generalizes a graded Lie algebra structure on the homotopy groups of the wedge of spheres 
\begin{equation*}
  \pi_*(I^N-(I^{a_1}\cup\cdots\cup
I^{a_r}))=\pi_*(S^{N-a_1-1}\vee\cdots\vee S^{N-a_r-1}),
\end{equation*}
where surgeries on 3-valent vertices provide its geometric realization. 
What we do for vertices of valence higher than three is to realize higher relations in a $L_\infty$-algebra by families of string links with {\it Brunnian property}. We will see that such realizations exist by constucting vertex surgeries inductively on the valence of vertices. 

\begin{mainThm}[Theorem~\ref{thm:Q-Hurewicz}]\label{thm:pi}
Let $k\geq 17$. Then the $(8k-10)$-cycle $\overline{\phi}_\gamma$ represents a class in the image of the rational Hurewicz map $\pi_{8k-10}(B\Diff_\partial^\fr(D^{2k}))\otimes\Q\to H_{8k-10}(B\Diff_\partial^\fr(D^{2k});\Q)$. 
\end{mainThm}
Corollary~\ref{cor:nontriv} and Theorem~\ref{thm:pi} imply that $\pi_{8k-10}(B\Diff_\partial^\fr(D^{2k}))\otimes\Q$ is nontrivial. Since the homotopy fiber of the natural map $B\Diff^\fr_\partial(D^{2k})\to B\Diff_\partial(D^{2k})$ is given by the base point component of $\Omega^{2k}SO_{2k}$, and $\pi_i(SO_{2k})\otimes\Q=0$ for $i\geq 4k-4$ (e.g. \cite[p.220]{FHT}), we have $\pi_{8k-10}(B\Diff_\partial^\fr(D^{2k}))\otimes\Q=\pi_{8k-10}(B\Diff_\partial(D^{2k}))\otimes\Q$. It is compatible with Theorem~\ref{thm:KRW} since $8k-10$ is in the 8-th band $[8k-17,8k-2]$ of \cite[Theorem~C]{KRW}. More generally, the degrees $(2k-3)n+m=2kn-3n+m$ for $0\leq m\leq 2n-1$ are in the $2n$-th band $[2kn-4n-1,2kn-1]$ of \cite{KRW}. The set of pairs (homology degree, $2k$) for which we have the injectivity of $\overline{\phi}_*$ is $\bigcup_{n\geq 1,2k\geq 20}D(n,2k)$, where
\[ D(n,2k)=\{((2k-3)n+m,2k)\mid 0\leq m\leq \min\{\sqrt{k-1}-3,2n-1\}\}. \]
Note that the band $D(n,2k)$ is included in the $2n$-th band $[2kn-4n-1,2kn-1]$.

\subsection{Basic brackets in the space of string links}

We introduce families of links that realize vertices of higher valences. First, we do so for simple cases.

\subsubsection{String link}\label{ss:string-link}

Let $N$ be an integer such that $N\geq 3$, and let $a_1,\ldots,a_r$ be integers such that $1\leq a_1,\ldots,a_r\leq N-2$. Let $I=[0,1]$. 

\begin{Def}[String link]\label{def:string-link}
\begin{enumerate}
\item A {\it string link} of type $(a_1,\ldots,a_r;N)$ is a neat embedding
\[ I^{a_1}\cup\cdots\cup I^{a_r}\hookrightarrow I^N \]
that agrees with the \ ``standard inclusion'' \ $\st\colon I^{a_1}\cup\cdots\cup I^{a_r}\hookrightarrow I^N$ \ near the boundary $\partial(I^{a_1}\cup\cdots\cup I^{a_r})$. The standard inclusion $\st$ is defined as follows.  We fix distinct points $p_1,p_2,\ldots,p_r\in \mathrm{Int}\,I$, and consider the planes $\{p_j\}\times I^{N-1}\subset I^N$. Then we assume that $\st$ embeds $I^{a_j}$ into $\{p_j\}\times I^{N-1}$ by identifying $I^{a_j}$ with the subspace $\{p_j\}\times I^{a_j}\times \{(\textstyle\frac{1}{2},\ldots,\frac{1}{2})\}$ of $\{p_j\}\times I^{N-1}$ in the straightforward manner.

\item A {\it framing} of a string link is a trivialization of its normal bundle that agrees with the standard one near the boundary.

\item Let $\Emb_\partial(I^{a_1}\cup\cdots\cup I^{a_r},I^N)$ be the space of string links of type $(a_1,\ldots,a_r;N)$. Let $\Emb_\partial^\fr(I^{a_1}\cup\cdots\cup I^{a_r},I^N)$ be the space of framed string links of type $(a_1,\ldots,a_r;N)$. We assume that the basepoints of these spaces are $\st$ possibly with the standard normal framing.
\end{enumerate}
\end{Def}

For a subset $S$ of $\Omega=\{1,2,\ldots,r\}$, let $\calE(S)=\Emb_\partial(\tbigcup_{\lambda\in \Omega\setminus S}I^{a_{\lambda}},I^N)$, and let
\[ \pi_{S}\colon \calE(\emptyset)\to \calE(S) \]
be the map which forgets the components labelled by $S$.

\begin{Def}[Brunnian property]\label{def:brunnian}
\begin{enumerate}
\item We say that an embedding $f\in \Emb_\partial(I^{a_1}\cup\cdots\cup I^{a_r},I^N)$ has a {\it Brunnian property} if removing any one of the components in $I^{a_1}\cup\cdots\cup I^{a_r}$ makes it isotopic to the standard inclusion $\iota$. In other words, the point $\prod_{j\in \Omega}\pi_{\{j\}}(f)\in\prod_{j\in\Omega}\calE(\{j\})$ lies in the same path-component as the basepoint. 

\item In that case, we call a path in $\calE(\{j\})$ from $\pi_{\{j\}}(f)$ to the basepoint a {\it Brunnian null-isotopy of $f$ with respect to removing the $j$-th component}.
\end{enumerate}
\end{Def}

We extend Definition~\ref{def:brunnian} to families of string links. We use the standard model for the homotopy fiber of a fibration $\pi\colon E\to X$, namely, the homotopy fiber $\hofib_{x_0}(\pi)$ over the point $x_0\in X$ is the space of pairs $(\gamma,x)$, where $x\in X$ and $\gamma$ is a path $I\to X$ such that $\gamma(0)=x$ and $\gamma(1)=x_0$.

\begin{Def}[Brunnian property for family]
Suppose we are given a continuous map
\[ \sigma\colon B\to \Emb_\partial(I^{a_1}\cup\cdots\cup I^{a_r},I^N), \]
where $B$ is a finite CW-complex.
\begin{enumerate}
\item We say that $\sigma$ has a {\it Brunnian property} if it has a lift to $\hofib_*(\prod_{j\in\Omega}\pi_{\{j\}})$:
\[ \xymatrix{
  & B \ar[d]^-{\sigma} \ar@{.>}[dl] &
  \\ \hofib_*(\prod_{j}\pi_{\{j\}}) \ar[r] &
  \Emb_\partial(I^{a_1}\cup\cdots\cup I^{a_r},I^N)
  \ar[r]^-{\prod_{j}\pi_{\{j\}}} & \prod_{j\in\Omega} \calE(\{j\})
} \] In that case, we call such a lift to
$\hofib_*(\pi_{\{j\}})$ a {\it Brunnian null-isotopy of $\sigma$ with
  respect to removing the $j$-th component}.

\item Suppose that the restriction of $\sigma$ to $B_i$ in a sequence of subcomplexes $B_0\subset B_1\subset\cdots \subset B_k$ of $B$ has a Brunnian null-isotopy with respect to removing the $j$-th component for each $i$. We say that $\sigma$ has a {\it Brunnian null-isotopy relative to $\{B_i\}$ with respect to removing the $j$-th component} if 
\begin{itemize}
\item $\sigma|_{B_i}$ has a lift to $\hofib_*(\pi_{\{j\}})$ that extends the given one on $B_{i-1}$ for $i\geq 1$, and
\item $\sigma$ has a lift to $\hofib_*(\pi_{\{j\}})$ that extends the given one on $B_k$.
\end{itemize} 
\end{enumerate}
\end{Def}

\subsubsection{Basic bracket operation}

Let $T_\ell$ denote the connected tree with one $\ell$-valent vertex (``internal vertex'') and $\ell$ leaves (``external vertices'') labelled by $\{1,2,\ldots,\ell\}$ (see Figure~\ref{fig:1-val}). 
\begin{figure}[h]
\[
T_\ell=\raisebox{-12mm}{\includegraphics[height=25mm]{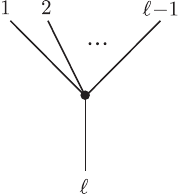}}
\]
\caption{The tree $T_\ell$.}\label{fig:1-val}
\end{figure} 
 
Let $\calP_{T_\ell}$ denote the poset of connected trees with $\ell$ labelled leaves (external vertices) such that a morphism $\sigma\to \sigma'$ or an order $\sigma \leq \sigma'$ is defined if $\sigma'$ is obtained from $\sigma$ by contracting some internal (i.e. non-leg) edges. We say that $\rho\in\calP_{T_\ell}$ is of {\it excess $k$} if there is a sequence $\rho_0\to \rho_1\to \cdots\to \rho_k=\rho$ of contractions of internal edges from a 3-valent tree $\rho_0$ with $\ell$ leaves. In particular, $T_\ell$ is of excess $\ell-3$.

\begin{Thm}[Basic bracket operation]\label{thm:vertex}
Let $\ell\geq 3$ and $n\geq 2\ell-3$. For each object $\sigma$ of $\calP_{T_\ell}$ of excess $k$, we have a $(k+\ell-3)$-chain $\beta_{T_\ell}(\sigma)=\frac{1}{m_\sigma}\omega_{T_\ell}(\sigma)$ in 
\[ S_{k+\ell-3}(B^{\ell-3}\Emb_\partial^\fr((I^{(\ell-1)n-\ell+2})^{\cup\ell},I^{\ell n-\ell+3});\Q),\]
which we call a ``bracket'', defined as a rational multiple of a map 
\[ \omega_{T_\ell}(\sigma)\colon I^k\times I^{\ell-3} \to B^{\ell-3}\Emb_\partial^\fr((I^{(\ell-1)n-\ell+2})^{\cup\ell},I^{\ell n-\ell+3}) \]
satisfying the following conditions (1)--(6).
\begin{enumerate}
\item (Compatibility) If $\sigma$ is of excess $<\ell-3$ and has internal (i.e. non-external) vertices $v_1,\ldots,v_r$ with $\deg(v_i)=\ell_i$, then the map $\omega_{T_\ell}(\sigma)$ is obtained by compositions of iterated suspensions of the maps $\omega_{T_{\ell_i}}(T_{\ell_i})$ for the internal vertex $v_i$. (See \S\ref{ss:suspension} for the definition of suspension and \S\ref{ss:iter-surg} for the composition.) 

\item (Boundary) The boundary $\partial \omega_{T_\ell}(T_\ell)$ is represented by a pointed loop in 
\[ \Omega^{2\ell-7}B^{\ell-3}\Emb_\partial^\fr((I^{(\ell-1)n-\ell+2})^{\cup\ell},I^{\ell n-\ell+3})\simeq \Omega^{\ell-4}\Emb_\partial^\fr((I^{(\ell-1)n-\ell+2})^{\cup\ell},I^{\ell n-\ell+3}). \]
For each $\sigma\in\calP_{T_\ell}$ of excess $k$, the term $\omega_{T_\ell}(\sigma)$ is represented by a map
\[ I^k\to \Omega^{\ell-3}B^{\ell-3}\Emb_\partial^\fr((I^{(\ell-1)n-\ell+2})^{\cup\ell},I^{\ell n-\ell+3})\simeq \Emb_\partial^\fr((I^{(\ell-1)n-\ell+2})^{\cup\ell},I^{\ell n-\ell+3}). \]
\item ($L_\infty$-relation) For $\ell\geq 4$, there exists a choice of the signs $\ve_{T_\ell}^n(\sigma)$ such that the relation
\[ \partial \beta_{T_\ell}(T_\ell)=\displaystyle\sum_{{{\sigma\in\calP_{T_\ell}}\atop{\text{excess}(\sigma)=\ell-4}}}^{\phantom{X}}\ve_{T_\ell}^n(\sigma)\beta_{T_\ell}(\sigma) \]
holds in $S_{2\ell-7}(B^{\ell-3}\Emb_\partial^\fr((I^{(\ell-1)n-\ell+2})^{\cup\ell},I^{\ell n-\ell+3});\Q)$ modulo degenerate chains.

\item (Brunnian property) The map $\omega_{T_\ell}(T_\ell)$ has a Brunnian property relative to those $\omega_{T_\ell}(\sigma)$ on $\partial (I^{\ell-3}\times I^{\ell-3})$ such that $\sigma\in\calP_{T_\ell}\setminus\{T_\ell\}$.

\item (Cyclic symmetry) The map $\omega_{T_\ell}(T_\ell)$ has a cyclic symmetry (see Definition~\ref{def:cyclic}).

\item (Excess 0) The map $\omega_{T_3}(T_3)\colon *\times *\to \Emb_\partial^\fr((I^{2n-1})^{\cup 3},I^{3n})$ is given by the Borromean link. 
\end{enumerate}
\end{Thm}

\begin{Rem}\label{rem:brun-stronger}
\begin{enumerate}
\item It follows from the property (Boundary) and ($L_\infty$-relation) that $\omega_{T_\ell}(T_\ell)$ can be considered as a null-homotopy of a pointed $(\ell-4)$-loop in $\Emb_\partial^\fr((I^{(\ell-1)n-\ell+2})^{\cup\ell},I^{\ell n-\ell+3})$ such that the boundary is decomposed into pieces that can be delooped.
\item Although each term $\omega_{T_\ell}(\sigma)$ is represented by a chain in $\calE=\Emb_\partial^\fr((I^{(\ell-1)n-\ell+2})^{\cup\ell},I^{\ell n-\ell+3})$ via a fixed homotopy equivalence $\calE\simeq\Omega^{\ell-3}B^{\ell-3}\calE$, we consider the $(\ell-3)$-fold delooping to realize the terms as chain summands in a single $(2\ell-7)$-loop. 
It can be shown that the space $\calE$ admits a natural action of the little $((\ell-1)n-\ell+2)$-cubes operad. Thus by the recognition principle (\cite{May}), the space $\calE$ is a $((\ell-1)n-\ell+2)$-fold loop space. Since $(\ell-1)n-\ell+2\geq \ell-3$ whenever $n\geq 2\ell-3$ and $\ell\geq 3$, the $(\ell-3)$-fold delooping makes sense.
\item In fact, we will show that $\omega_{T_\ell}$ satisfies a stronger Brunnian property, which requires Brunnian null-isotopies for any choice of a subset of the set of components (Definition~\ref{def:brun-null-iso}). Such a stronger Brunnian property will be used in several places later: in clarifying the inductive construction of $\omega_{T_\ell}$ (Lemma~\ref{lem:pres-brunnian-10T}, Assumption~\ref{assum:single-comp}), in the construction of compatible framings (Lemma~\ref{lem:framing-brunnian}), in the definition of suspension of string links (Remark~\ref{rem:suspension-generalized}, Definition~\ref{def:suspension}), in the inductive construction of Brunnian null-isotopies (Lemmas~\ref{lem:compos-brun}, \ref{lem:pres-brunnian-10T}, \ref{lem:brunnian-4}), and in giving canonical gluings between chains (Lemmas~\ref{lem:brun-cone}, \ref{lem:hcoherence}). 
\end{enumerate}
\end{Rem}

\begin{Exa}\label{ex:brackets}
\begin{itemize}
\item The bracket $\beta_{T_3}(T_3)=\omega_{T_3}(T_3)$ is a 0-chain, i.e. a point, in $\Emb_\partial^\fr((I^{2n-1})^{\cup 3},I^{3n})$. This can be represented by an embedding $I^{2n-1}\to I^{3n}-(I^{2n-1}\cup I^{2n-1})\simeq S^n\vee S^n$, which gives the Whitehead product in $\pi_{2n-1}(S^n\vee S^n)$ (see \S\ref{ss:emb-Wh}).
\item The bracket $\beta_{T_4}(T_4)=\frac{1}{4}\omega_{T_4}(T_4)$ is a 2-chain in $S_2(B\Emb_\partial^\fr((I^{3n-2})^{\cup 4},I^{4n-1});\Q)$ such that the following identity holds in $S_1(B\Emb_\partial^\fr((I^{(\ell-1)n-\ell+2})^{\cup\ell},I^{\ell n-\ell+3});\Q)$
\[ \partial \beta_{T_4}(T_4)= 
\pm\beta_{T_4}(\sigma_1)
\pm\beta_{T_4}(\sigma_2)
\pm\beta_{T_4}(\sigma_3)
\]
for the following $\sigma_i$ ($i=1,2,3$) with two 3-valent vertices and 4 labelled legs:
\[ \sigma_1=\fig{-10mm}{20mm}{sigma_1.pdf},\quad 
\sigma_2=\fig{-10mm}{20mm}{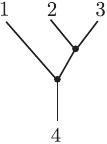},\quad
\sigma_3=\fig{-10mm}{20mm}{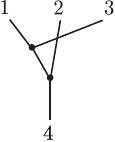}. \]
The signs $\pm$ are determined if we fix half-edge orientations of the chains as in (\ref{eq:ori-o(v)}) below for the trees.
This gives the Jacobi identity 
\[ [[a,b],c]+(-1)^n[a,[b,c]]+(-1)^n[[a,c],b]=0 \]
for iterated Whitehead products in $\pi_{3n-2}(S^n\vee S^n\vee S^n)$ realized by embeddings (see Theorem~\ref{thm:phi-jacobi}).
\item The bracket $\beta_{T_5}(T_5)$ is a 4-chain in $S_4(B^2\Emb_\partial^\fr((I^{4n-3})^{\cup 5},I^{5n-2});\Q)$ such that $\partial \beta_{T_5}(T_5)$ consists of 10 terms of trees with one 4-valent vertex, one 3-valent vertex, and 5 labelled legs (see \S\ref{s:5-valent}):
\[ \begin{split}
\partial\beta_{T_5}(T_5)=&\,\pm\beta_{T_5}(\tau_1)
\pm\beta_{T_5}(\tau_2)
\pm\beta_{T_5}(\tau_3)
\pm\beta_{T_5}(\tau_4)
\pm\beta_{T_5}(\tau_5)\\
&\pm\beta_{T_5}(\tau_6)
\pm\beta_{T_5}(\tau_7)
\pm\beta_{T_5}(\tau_8)
\pm\beta_{T_5}(\tau_9)
\pm\beta_{T_5}(\tau_{10}),\quad\text{where}
\end{split} 
\]
\[\begin{split}
&\tau_1=\fig{-8mm}{16mm}{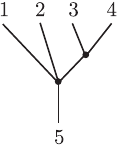},\quad 
\tau_2=\fig{-8mm}{16mm}{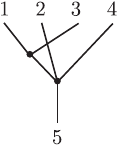},\quad
\tau_3=\fig{-8mm}{16mm}{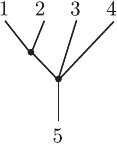},\quad
\tau_4=\fig{-8mm}{16mm}{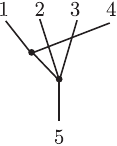},\quad
\tau_5=\fig{-8mm}{16mm}{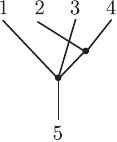},\quad\\
&\tau_6=\fig{-8mm}{16mm}{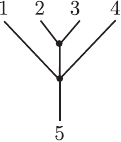},\quad
\tau_7=\fig{-8mm}{16mm}{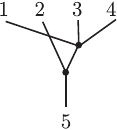},\quad
\tau_8=\fig{-8mm}{16mm}{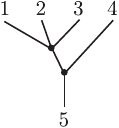},\quad
\tau_9=\fig{-8mm}{16mm}{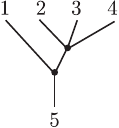},\quad
\tau_{10}=\fig{-8mm}{16mm}{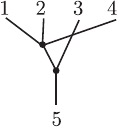}.
\end{split} \]
The signs $\pm$ are determined if we fix half-edge orientations of the chains as in (\ref{eq:ori-o(v)}) below for the trees.
This realizes the following relation in $L_\infty$-algebra:
\[\begin{split}
& [a,b,[c,d]]+(-1)^n[[a,c],b,d]+[[a,b],c,d]+[[a,d],b,c]+(-1)^n[a,c,[b,d]]\\
& +(-1)^n[a,[b,c],d]-[[a,c,d],b]-[[a,b,c],d]-[a,[b,c,d]]-[[a,b,d],c]=\partial[a,b,c,d]
\end{split} \]
for $\partial$-closed formal variables $a,b,c,d$ of degree $n$ (\cite[Example 3.133]{MSS}).
\end{itemize}
\end{Exa}

\begin{Rem}
Roughly, Theorem~\ref{thm:vertex} gives higher order generalizations of Whitehead brackets for a family of embedding spaces of string links. 
\end{Rem}

\subsection{General bracket operation}\label{ss:gen-bracket}

We will iterate certain operations, ``suspensions'' and ``deloopings'' (\S\ref{s:operations}), to generalize the construction of Theorem~\ref{thm:vertex} to more general set of dimensions. Generally, an $\ell$-valent vertex may give a $\{(N-1)\ell-N-(a_1+\cdots+a_\ell)\}$-chain in 
\[ \Emb_\partial^\fr(I^{a_1}\cup \cdots \cup I^{a_\ell},I^N), \]
where $a_1,\ldots,a_\ell$ are such that $1\leq a_i\leq N-2$ and 
\begin{equation}\label{eq:dim_p}
 a_1+\cdots+a_\ell\leq (N-2)\ell-N+3. 
\end{equation}
The equality in (\ref{eq:dim_p}) holds when there is an embedding
$I^{a_1}\to I^N-(I^{a_2}\cup\cdots\cup I^{a_\ell})\simeq S^{N-a_2-1}\vee\cdots\vee S^{N-a_\ell-1}
$ that represents an iterated Whitehead product of length $\ell-1$.

In the following, we only consider string links $(I^k)^{\cup p}\cup (I^{k-1})^{\cup q}\hookrightarrow I^{2k}$ with only $I^k$ and $I^{k-1}$ components, for simplicity. Let $T_\ell(p,q)$, $p+q=\ell$, denote $T_\ell$ equipped with directions on edges such that the legs labelled by $1,\ldots, p$ are incoming and those labelled by $p+1,\ldots,\ell$ are outgoing (see Figure~\ref{fig:t-l-pq}). 
\begin{figure}[h]
  \[ T_\ell(p,q)=\raisebox{-12mm}{\includegraphics[height=25mm]{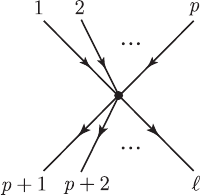}} \]
\caption{The tree $T_\ell(p,q)$, $\ell=p+q$.}\label{fig:t-l-pq}
\end{figure}

Let $\overrightarrow{\calP}_{T_\ell(p,q)}$ denote the poset of directed trees whose objects are $(\sigma,\alpha)$, where $\sigma\in \calP_{T_\ell}$ and $\alpha$ is a choice of directions on edges of $\sigma$ such that the directions on labelled legs are the same as those of $T_\ell(p,q)$, and morphisms are given by contractions of directed edges. Note that if $\sigma$ is of excess $\mu$, there are $2^{\ell-3-\mu}$  choices of $\alpha$ for each $\sigma$. We fix one to one correspondences between the incoming (resp. outgoing) legs of $T_\ell(p,q)$ and the components $I^k$ (resp. $I^{k-1}$) of $(I^k)^{\cup p}\cup (I^{k-1})^{\cup q}$.
We now assume that $(\ell-2)k-2\ell+q+3>0$, in which case the inequality in (\ref{eq:dim_p}) is strict for $a_1=\cdots=a_p=k$, $a_{p+1}=\cdots=a_\ell=k-1$. 

\begin{Thm}[General bracket operation]\label{thm:l-valent-general}
Suppose that $k$ is sufficiently large.
For each object $(\sigma,\alpha)$ of $\overrightarrow{\calP}_{T_\ell(p,q)}$, we have a chain $\overline{\beta}_{(\sigma,\alpha)}=\frac{1}{m_\sigma}\overline{\omega}_{(\sigma,\alpha)}$ in $S_{(\ell-2)k-\ell+q}(\Emb_\partial^\fr((I^k)^{\cup p}\cup (I^{k-1})^{\cup q},I^{2k});\Q)$, which we also call a ``bracket'', defined as a rational multiple of a map
\[ \overline{\omega}_{(\sigma,\alpha)}\colon \overline{B}_{(\sigma,\alpha)}\to \Emb_\partial^\fr((I^k)^{\cup p}\cup (I^{k-1})^{\cup q},I^{2k}) \]
from a finite CW-complex $\overline{B}_{(\sigma,\alpha)}$ with a relative fundamental class satisfying the following conditions.
\begin{enumerate}
\item (Compatibility) If there is a morphism $(\sigma,\alpha)\to (\sigma',\alpha')$ in $\overrightarrow{\calP}_{T_\ell(p,q)}$, then we have an inclusion $\overline{B}_{(\sigma,\alpha)}\hookrightarrow \overline{B}_{(\sigma',\alpha')}$ and $\overline{\beta}_{(\sigma,\alpha)}$ is the restriction of $\overline{\beta}_{(\sigma',\alpha')}$ to $\overline{B}_{(\sigma,\alpha)}$.
\item ($L_\infty$-relation) There is a choice $\ve_{T_\ell(p,q)}(\sigma,\alpha)$ of signs such that the relation between chains
\begin{equation}\label{eq:rel-general}
\partial \overline{\beta}_{T_\ell(p,q)}=\displaystyle\sum_{{{(\sigma,\alpha)\in\overrightarrow{\calP}_{T_\ell(p,q)}}\atop{\text{excess}(\sigma)=\ell-4}}}\ve_{T_\ell(p,q)}(\sigma,\alpha)\,\overline{\beta}_{(\sigma,\alpha)} 
\end{equation}
holds modulo degenerate chains. 
\item (Brunnian property) The map $\overline{\omega}_{T_\ell(p,q)}$ has a Brunnian property relative to those $\overline{\omega}_{T_\ell(p,q)}(\sigma,\alpha)$ on $\partial \overline{B}_{T_\ell(p,q)}$ such that $(\sigma,\alpha)\in\overrightarrow{\calP}_{T_\ell(p,q)}\setminus\{T_\ell(p,q)\}$.
\item (Excess 0) The bracket $\overline{\beta}_{T_3(p,q)}=\overline{\omega}_{T_3(p,q)}\colon \overline{B}_{T_3(p,q)}\to \Emb_\partial^\fr((I^k)^{\cup p}\cup (I^{k-1})^{\cup q},I^{2k})$ is a cycle that is homologous to a cycle from $S^{k-3+q}$ obtained by deloopings (\S\ref{ss:delooping}) of a Borromean link. If $(\sigma,\alpha)\in\overrightarrow{\calP}_{T_\ell(p,q)}$ is of excess 0, the chain $\overline{\beta}_{(\sigma,\alpha)}$ is a cycle from $\prod_{v\in V^{\mathrm{int}}(\sigma)}\overline{B}_{T(v)}$ obtained by composing $\overline{\beta}_{T(v)}=\overline{\omega}_{T(v)}$ (\S\ref{ss:iter-surg}) for the 3-valent vertices $v$ in the set $V^{\mathrm{int}}(\sigma)$ of internal vertices of $\sigma$.
\end{enumerate}
\end{Thm}

Roughly speaking, Theorem~\ref{thm:l-valent-general} gives a ``thickening'' $\overline{\beta}_{T_\ell(p,q)}$ of a multiple of a chain 
\[ \omega_{T_\ell(p,q)}\colon B_{T_\ell(p,q)}\to \Emb_\partial^\fr((I^k)^{\cup p}\cup (I^{k-1})^{\cup q},I^{2k}), \] 
where $B_{T_\ell(p,q)}=I^{\ell-3}\times S^{\tvec{a}}$ for some compact manifold $S^{\tvec{a}}$ of dimension $(\ell-2)k-2\ell+q+3$, obtained from the basic bracket operation for $T_\ell$ by suspensions and deloopings (see \S\ref{ss:suspension} and \S\ref{ss:delooping}).  
The effect of the thickening is that $\overline{\beta}_{T_\ell(p,q)}$ is compatible with $\overline{\beta}_{(\sigma,\alpha)}$ for the face graphs $(\sigma,\alpha)$ of $T_\ell(p,q)$, which is not obvious for $\beta_{T_\ell(p,q)}$ due to the suspensions and deloopings.

\subsection{Graph surgery}

Let $(\Gamma,\alpha)$ be a directed graph of excess $p\geq 1$ and $r$ vertices. For each such directed graph, we construct a framed $D^{2k}$-bundle by combining surgeries given by several general brackets of Theorem~\ref{thm:l-valent-general}.
We say that an object $\calX_{(\Gamma,\alpha)}$ determined by $(\Gamma,\alpha)$ satisfies a condition $C$ {\it relative to lower excess graphs} if it satisfies $C$ under the assumption that $\calX_{(\Gamma',\alpha')}$ does for all directed graphs $(\Gamma',\alpha')$ of excess at most $p-1$.

\begin{Thm}\label{thm:graph-surgery}
Suppose that $k$ is sufficiently large. Let $(\Gamma,\alpha)$ be a directed graph of excess $p\geq 0$ with $r$ vertices. We construct a chain $\overline{\phi}_{(\Gamma,\alpha)}=\frac{1}{m_{(\Gamma,\alpha)}}\overline{\omega}_{(\Gamma,\alpha)}$ in $S_*(B\Diff_\partial^\fr(D^{2k});\Q)$, which is a rational multiple of a map
\[ \overline{\omega}_{(\Gamma,\alpha)}\colon \overline{B}_{(\Gamma,\alpha)}\to B\Diff_\partial^\fr(D^{2k}) \]
satisfying the following conditions (1)--(4) relative to lower excess graphs.
\begin{enumerate}
\item $\overline{B}_{(\Gamma,\alpha)}=\prod_{v\in V(\Gamma)}\overline{B}_{T(v)}$, where $T(v)$ is $T_{\ell(v)}(p,q)$ for $\ell(v)=\deg(v)$ and for some $p,q$ determined by the direction structure $\alpha$. The integer $m_{(\Gamma,\alpha)}$ is $\prod_{v\in V(\Gamma)}m_{\ell(v)}$.

\item There is a disjoint union $V_1\cup\cdots\cup V_r$ of compact codimension 0 submanifolds of $\mathrm{Int}\,D^{2k}$ such that $(D^{2k},\partial)$-bundle $\overline{\pi}^{(\Gamma,\alpha)}\colon \overline{E}^{(\Gamma,\alpha)}\to \overline{B}_{(\Gamma,\alpha)}$ associated to $\overline{\omega}_{(\Gamma,\alpha)}$ has a trivialization $\tau(\Gamma,\alpha)$ over a subbundle with fiber $Q_{(\Gamma,\alpha)}:=D^{2k}-\mathrm{Int}\,(V_1\cup \cdots \cup V_r)$. It is compatible with those for lower excess graphs in $\partial(\Gamma,\alpha)$.

\item The restriction of $\overline{\pi}^{(\Gamma,\alpha)}$ to the subbundle with fiber $V_i$ is the pullback of $\overline{\rho}_{T_\ell(p,q)}$ under the projection $\overline{B}_{(\Gamma,\alpha)}\to \overline{B}_{T_\ell(p,q)}$. The pullback is compatible with those for lower excess graphs in $\partial(\Gamma,\alpha)$.

\item There is a vertical framing on $\overline{E}^{(\Gamma,\alpha)}$ that is standard near $Q_{(\Gamma,\alpha)}$, and is compatible with those for lower excess graphs in $\partial(\Gamma,\alpha)$.
\end{enumerate}
\end{Thm}

That $\overline{\phi}$ is a chain map in Theorem~\ref{thm:main} almost follows from Theorem~\ref{thm:graph-surgery} except orientations. To complete the proof of Theorem~\ref{thm:main}, we need to define suitable orientations to make the map $\overline{\omega}_{(\Gamma,\alpha)}$ into a chain. It will be done in Section~\ref{s:ind-ori}.

\subsection{Organization of the paper}

In Part I, we prove Theorems~\ref{thm:main} and \ref{thm:pi}.

In Section~\ref{s:preliminaries}, we recall some facts about homotopy groups of the spaces of string links that will be important. And also, we define a strong Brunnian property for a family of string links.

In Section~\ref{s:graph-surg}, we explain how Theorem~\ref{thm:main} follows from Theorems~\ref{thm:vertex} and \ref{thm:l-valent-general}, whose proofs are postponed to Sections~\ref{s:4-valent}--\ref{s:thicken}. The orientation convention for the chains, which is technical, is postponed to Section~\ref{s:ind-ori}. 

In Sections~\ref{s:4-valent}, \ref{s:5-valent}, and \ref{s:6-valent}, we construct the basic brackets of Theorem~\ref{thm:vertex} for vertices of 4-, 5-, and at least 6-valent, respectively. In Section~\ref{s:operations}, we provide operations for brackets that are necessary in Sections~\ref{s:5-valent} and \ref{s:6-valent}. Theorem~\ref{thm:vertex} is proved.

In Section~\ref{s:thicken}, we upgrade the basic brackets into more general ones by thickening and gluing iterated brackets for various trees. Theorem~\ref{thm:l-valent-general} is proved. 

In Section~\ref{s:homotopy}, we prove Theorem~\ref{thm:pi}. This section is independent of the proof of Theorem~\ref{thm:main}.

In Section~\ref{s:ind-ori}, we give the orientation convention for the chains, which was postponed in Section~\ref{s:graph-surg}, and complete the proof of Theorem~\ref{thm:main}.


In Section~\ref{s:concluding}, we describe concluding remarks with some open problems. In Appendix~\ref{s:leaf-form-mfd}, we give explicit leaf forms for the basic bracket for a 4-valent vertex. Appendix~\ref{s:ext-closed} gives a criterion to extend a closed form. In Appendix~\ref{s:sign-jacobi}, we check the signs in the Jacobi identity.


\subsection*{Acknowledgements}
We would like to thank Danica Kosanovi\'{c} and Robin Koytcheff for giving us information about their works and for their helpful comments. TW would like to thank the Department of Mathematics, University of Oregon for its hospitality.

\section{Preliminaries on string links}\label{s:preliminaries}

We collect some known results and necessary notations that will be important later. The reader may first skip to \S\ref{s:graph-surg} and return here when necessary.

\subsection{Some homotopy groups of the spaces of string links}

We use two theorems at the heart of the proof of Theorem~\ref{thm:vertex}. The following theorem, which is based on Haefliger's exact sequence in \cite[1.3. Th\'{e}or\`{e}me]{Hae}, gives a condition when the space of spherical links of codimension at least three has infinite number of components.

\begin{Thm}[{Crowley--Ferry--Skopenkov \cite[Corollary~1.8]{CFS}}]\label{thm:CFS}
Assume that $p_1,\ldots,p_r<m-2$. Then the group $\pi_0(\Emb(\coprod_{k=1}^r S^{p_k},S^m))$ is infinite if and only if there is a subsequence $(k_1,\ldots,k_s)\subset (1,\ldots,r)$ satisfying one of the following conditions:
\begin{itemize}
\item[(a)] $s=1$, $4|p_{k_1}+1$, and $m<3p_{k_1}/2+2$;
\item[(b)] $s=2$ and there is $(x_1,x_2)\in FCS(m-p_{k_1},m-p_{k_2})$ such that 
\begin{equation}\label{eq:CFS2}
 (m-p_{k_1}-2)x_1+(m-p_{k_2}-2)x_2=m-3;
\end{equation} 
\item[(c)] $s\geq 3$ and the equation 
\begin{equation}\label{eq:CFS3}
 (m-p_{k_1}-2)x_1+\cdots+(m-p_{k_s}-2)x_s=m-3
\end{equation}
has a solution in positive integers.
\end{itemize}
Here, $FCS(i,j)\subset \Z^2$ is a subset which depends only on the parity of $i$ and $j$. The infinite order elements in the case (a) are given by connect-summing knots in $\pi_0(\Emb(S^{p_k},S^m))$ to the $k$-th component. 
\end{Thm}

This result for spherical links can be applied to string links according to the following lemma.

\begin{Lem}\label{lem:S-I}
Suppose that $p<m-2$. Then we have an isomorphism
\[ \pi_0(\Emb((S^p)^{\cup r},S^m))\cong \pi_0(\Emb_\partial((I^{p})^{\cup r},I^m)). \]
\end{Lem}
\begin{proof}
Let $\underline{S}^p$ denote the union of $S^p\sqcup\R^p$ and a 1-handle $I\times D^p$ such that $\{0\}\times D^p$ is glued to a small disk in $S^p$ and $\{1\}\times D^p$ is glued to a small disk in $\R^p$. Let $\Emb((\underline{S}^p)^{\cup r},\R^m)$ denote the space of embeddings $(\underline{S}^p)^{\cup r}\to \R^m$ that are standard near the plane $\R^p$. 

We see that the natural map $\Emb((\underline{S}^p)^{\cup r},\R^m)\to \Emb((\R^p)^{\cup r},\R^m)$ given by removing $I\times \mathrm{Int}\,D^p$ is a homotopy equivalence, and that if $m-p>2$, the induced map 
\[ \pi_0(\Emb((\underline{S}^p)^{\cup r},\R^m))\to \pi_0(\Emb((S^p)^{\cup r},\R^m)) \]
given by restricting to the spheres $(S^p)^{\cup r}$ in $(\underline{S}^p)^{\cup r}$ is an isomorphism. It is well-known that the latter group is isomorphic to $\pi_0(\Emb((S^p)^{\cup r},S^m))$.
\end{proof}

\begin{Thm}[{Goodwillie \cite{Go}. See also \cite[3.2.2. Corollary]{GKW}}]\label{thm:goodwillie}
Let $M$ be a neat submanifold of a manifold $N$, and let $m=\dim{M}$, $n=\dim{N}$. If $m\leq n-3$, then the space $\CEmb(M,N)$ of concordance embeddings $M\times I\to N\times I$ between the inclusion $M\times\{0\}\subset N\times\{0\}$ and embeddings $M\times\{1\}\to N\times\{1\}$, which restricts to $\partial M\times I\to \partial N\times I$; $(x,t)\mapsto (x,t)$, is $(n-m-3)$-connected. (See e.g. \cite[\S{3.2}]{GKW} for the definition of $\CEmb(M,N)$.)
\end{Thm}
This can be used to estimate the connectivity of the first map (the graphing map) in the fibration sequence
\begin{equation}\label{eq:concordance-fibration}
 \Omega\Emb_\partial(M,N)\to \Emb_\partial(M\times I,N\times I)\to \CEmb(M,N)\to \Emb_\partial(M,N). 
\end{equation}
In particular, Theorem~\ref{thm:goodwillie} greatly simplifies problem since the study of higher homotopy groups of embedding spaces up to some degrees is reduced to that of $\pi_0$. More precisely, we use the following corollary. (A similar result for two component string links has been given as a part of \cite[Theorem~B]{Koy}.)
\begin{Cor}\label{cor:goodwillie}
Let $\ell\geq 3$ and $j\geq 1$. If $n\geq j+2$, the graphing map induces an isomorphism
\[ \pi_j(\Emb_\partial((I^{(\ell-1)n-\ell+2})^{\cup \ell},I^{\ell n-\ell+3}))\cong 
\pi_0(\Emb_\partial((I^{(\ell-1)n-\ell+2+j})^{\cup \ell},I^{\ell n-\ell+3+j})), \]
and both the conditions (b) and (c) of Theorem~\ref{thm:CFS} are not satisfied for the latter group.
\end{Cor}
\begin{proof}
By Theorem~\ref{thm:goodwillie}, the space $\CEmb((I^{(\ell-1)n-\ell+2})^{\cup \ell},I^{\ell n-\ell+3})$ is $(\ell n-\ell+3)-\{(\ell-1)n-\ell+2\}-3=(n-2)$-connected. Since $n-2\geq j$, we have $\pi_k(\CEmb((I^{(\ell-1)n-\ell+2})^{\cup \ell},I^{\ell n-\ell+3}))=0$ for $k=j,j-1$, and the homotopy exact sequence for the fibration (\ref{eq:concordance-fibration}) shows that 
\[ \pi_j(\Emb_\partial((I^{(\ell-1)n-\ell+2})^{\cup \ell},I^{\ell n-\ell+3}))\cong
\pi_{j-1}(\Emb_\partial((I^{(\ell-1)n-\ell+3})^{\cup \ell},I^{\ell n-\ell+4})). \]
This argument can be continued since the codimension of the embeddings do not change, and we obtain the desired isomorphism.

For the second assertion, the equation (\ref{eq:CFS2}) or (\ref{eq:CFS3}) is
$(n-1)(x_1+\cdots+x_s)=\ell(n-1)+j$,
and this does not have an integer solution if $n-1\geq j+1$.
\end{proof}

\subsection{Embeddings with Brunnian null-isotopies}\label{ss:def-brunnian}

We consider the space of string links equipped with null-isotopies of a subset of the set of components. We start with the space
$\Emb_\partial(\tbigcup_{\lambda\in \Omega}I^{a_{\lambda}},I^N)$, where $\Omega=\{1,2,\ldots,r\}$ and the dimensions $a_{\lambda}$, $\lambda\in \Omega$ are given, and $\iota\colon \tbigcup_{\lambda\in \Omega}I^{a_{\lambda}}\rh I^N$ is a standard inclusion.

For a subset $S$ of $\Omega=\{1,2,\ldots,r\}$, let $\calE(S)=\Emb_\partial(\tbigcup_{\lambda\in \Omega\setminus S}I^{a_{\lambda}},I^N)$, and let $\calB(S)$ denote the homotopy fiber over $\iota$ of the map
\[ \pi_{S}\colon \calE(\emptyset)\to \calE(S) \]
which forgets the components labelled by $S$. 
In more detail, $\calB(S)$ is the space of pairs $(f,\gamma)$, where $f\in \calE(\emptyset)$, and $\gamma$ is a path in $\calE(S)$ that gives a null-isotopy of $\pi_{S}(f)$. The assignment $S\mapsto \calB(S)$ gives a functor from the power set $\calP(\Omega)$ of $\Omega$ to $\mathrm{Top}$, where a morphism $\calB(S)\to \calB(S')$ for $S\subset S'$ is given by forgetting the components in $S'\setminus S$. Clearly, the collection of spaces $\{\calB(S)\}_{S\subset \Omega}$ and morphisms $\{\calB(S)\to \calB(S')\}_{S\subset S'\subset \Omega}$ naturally forms a cubical diagram. 

\begin{Def}[Brunnian null-isotopies for embedding]\label{def:brun-null-iso}
A {\it system of Brunnian null-isotopies} of an embedding $f\in \calE(\emptyset)$ is a choice of a lift of the constant map $\mathrm{const}_f$ at $f$ for the projection map
\[ \holim{\emptyset\neq S\subset \Omega}{\calB(S)}\to \Map(\Delta^{r-1},\calB(\Omega))=\Map(\Delta^{r-1},\calE(\emptyset)). \]
(See e.g. \cite[\S{5.3}]{MV} for the homotopy limits in cubical diagrams.)
\end{Def}
We give an example of a system of Brunnian null-isotopies for $\Omega=\{1,2,3\}$ in \S\ref{ss:pres-brun-2}.
Roughly speaking, an element of $\holim{\emptyset\neq S\subset \Omega}{\calB(S)}$ is a collection of families of null-isotopies in $\calE(S)$ parametrized by $(|S|-1)$-dimensional simplices in $\Delta^{r-1}$ which are compatible with respect to the face inclusions. The Brunnian property can be interpreted  in terms of a lifting problem. Let $P$ be the pullback of the following diagram:
\[ \xymatrix{
 \calE(\emptyset) \ar[r]^-{\mathrm{const}_\bullet} & \Map(\Delta^{\ell-1},\calE(\emptyset)) & \holim{\emptyset\neq S\subset\Omega}\calB(S) \ar[l]_-{\text{proj}},
}\]
where $\calE(\emptyset):=\Emb_\partial(\tbigcup_{\lambda\in \Omega}I^{a_{\lambda}},I^N)$ and $\mathrm{const}_\bullet(f)=\mathrm{const}_f$. Then a system of Brunnian null-isotopies of $f\in\calE(\emptyset)$ is a lift of $f$ to $P$.

We also need a concept of a relative Brunnian property for a null-isotopy of a string link. A null-isotopy of a string link is given by an element $(f,\gamma)\in \calB(\emptyset)$, where $f\in\calE(\emptyset)$, $\gamma$ is a path in $\calE(\emptyset)$ from $f$ to $\iota$. There is a natural map
$\mathrm{const}_\bullet\colon\calB(\emptyset)\to \holim{\emptyset\neq S\subset \Omega}{\calB(S)}$
which takes $(f,\gamma)$ to $\{\mathrm{const}_{(f,\gamma)}\}$.

\begin{Def}[Relative Brunnian property for null-isotopy]\label{def:rel-brunnian}
We say that an element $(f,\gamma)\in \calB(\emptyset)$ has a {\it relative Brunnian property} with respect to a system of Brunnian null-isotopies 
\[ \{\gamma_S\}_{\emptyset\neq S\subset \Omega}\in \holim{\emptyset\neq S\subset \Omega}{\calB(S)} \]
of $f$, if the point $\{\mathrm{const}_{(f,\gamma)}\}$ is homotopic to $\{\gamma_S\}_{\emptyset\neq S\subset \Omega}$ in $\mathrm{proj}^{-1}(\mathrm{const}_f)$.
\end{Def}

\begin{Rem}
The Brunnian properties defined here are stronger than those given in \S\ref{ss:string-link}. The stronger version will be important in several places in the constructions in this paper. See also Remark~\ref{rem:brun-stronger}.  
\end{Rem}%
There is a natural embedding of $\calB(\emptyset)$ into $P$. Also, there is a natural projection $P\to \calE(\emptyset)$. The following is a natural analogue of Definitions~\ref{def:brun-null-iso} and \ref{def:rel-brunnian} for families.
\begin{Def}[Brunnian null-isotopies for families]\label{def:brunnian-framily}
A {\it system of Brunnian null-isotopies of a family} $\sigma\colon B\to \calE(\emptyset)$ of embeddings is a lift of $\sigma$ to $P$. We say that a family $\tau\colon B\to \calB(\emptyset)\subset P$ has a {\it relative Brunnian property} with respect to a system of Brunnian null-isotopies $\widetilde{\sigma}\colon B\to P$ if there is a homotopy between $\tau$ and $\widetilde{\sigma}$ in $P$, which does not change the projection to $\calE(\emptyset)$.
\end{Def}

\subsection{Vertex surgery and framing}\label{ss:v-surgery}

By taking the complements, a family of framed string links gives a family of handlebodies $V$ (or a $V$-bundle) relative to the boundaries. We will define below a vertex surgery by replacing a trivial bundle with this fiber bundle. 
As in \cite[Proof of Lemma~A]{Wa22}, we consider the following diagram, which is commutative up to homotopy:
\begin{equation}\label{eq:2x3-diag}
 \xymatrix{
  \bEmb_\partial((I^k)^{\cup p}\cup (I^{k-1})^{\cup q},I^{2k})_{\iota} \ar[d] & \bEmb_\partial(N_{I^k}^{\cup p}\cup N_{I^{k-1}}^{\cup q},I^{2k})_{\widetilde{\iota}} \ar[l]_-{\simeq} \ar[r]^-{\widetilde{c}}  \ar[d]& B\Diff^\fr_\partial(V;\mathrm{st}) \ar[d]\\
  \Emb_\partial^\fr((I^k)^{\cup p}\cup (I^{k-1})^{\cup q},I^{2k})_{\iota} & \Emb_\partial(N_{I^k}^{\cup p}\cup N_{I^{k-1}}^{\cup q},I^{2k})_{\widetilde{\iota}} \ar[l]_-{\simeq} \ar[r]^-{c} & B\Diff_\partial(V)
} \end{equation}
Here, 
\begin{itemize}
\item $\bEmb_\partial(A,B)_{f_0}$ denotes the homotopy fiber of the natural map $\Emb_\partial(A,B)\to \mathrm{Imm}_\partial(A,B)$ over $f_0\in \mathrm{Imm}_\partial(A,B)$, where $\mathrm{Imm}_\partial(A,B)$ is the space of immersions $A\looparrowright B$, 
\item the leftmost vertical map is induced by the natural maps $\Omega^k(SO_{2k}/SO_k)\to \Omega^k(BSO_k)$ and $\Omega^{k-1}(SO_{2k}/SO_{k+1})\to \Omega^{k-1}(BSO_{k+1})$,
\item $B\Diff_\partial(X)$ denotes the classifying space for smooth $X$-bundles trivialized over $\partial X$, and we will call corresponding bundles $(X,\partial)$-bundles.
\item $B\Diff^\fr_\partial(A;\tau)$ denotes the classifying space of smooth $(A,\partial A)$-bundles with framings along fibers that agrees with $\tau$ near $\partial A$.
\item $N_S$ denotes an open tubular neighborhood of a neat submanifold $S$ of $I^{2k}$. 
\item $c\colon \Emb_\partial(N_{I^k}^{\cup p}\cup N_{I^{k-1}}^{\cup q},I^{2k})_{\widetilde{\iota}}\to B\Diff_\partial(V)$ is the map taking the complement, and $\widetilde{c}$ is a natural lift of the map $c$ defined in \cite[Proof of Lemma~A]{Wa22}.
\end{itemize} 
It follows from the diagram above that a family $B\to \bEmb_\partial((I^k)^{\cup p}\cup (I^{k-1})^{\cup q},I^{2k})_{\iota}$ gives a family $B\to \Emb_\partial^\fr((I^k)^{\cup p}\cup (I^{k-1})^{\cup q},I^{2k})_{\iota}$, and that in $B\Diff^\fr_\partial(V;\mathrm{st})$.

\begin{Lem}[Vertex surgery]\label{lem:vert-surg}
The chain $\overline{\omega}_{T_\ell(p,q)}\colon \overline{B}_{T_\ell(p,q)}\to \Emb_\partial^\fr((I^k)^{\cup p}\cup (I^{k-1})^{\cup q},I^{2k})$ of Theorem~\ref{thm:l-valent-general} admits a lift to the space $\bEmb_\partial(N_{I^k}^{\cup p}\cup N_{I^{k-1}}^{\cup q},I^{2k})_{\widetilde{\iota}}$, which is compatible with those for the face graphs $\sigma\in\overrightarrow{\calP}_{T_\ell(p,q)}\setminus\{T_\ell(p,q)\}$. Let $\overline{\rho}_{T_\ell(p,q)}\colon \overline{B}_{T_\ell(p,q)}\to B\Diff^\fr_\partial(V;\mathrm{st})$ denote the composition of this lift with $\widetilde{c}$.
\end{Lem}

Lemma~\ref{lem:vert-surg} will be proved in \S\ref{ss:lift-bEmb}.

\begin{Rem}\label{rem:complement-map}
  We note the following property of the complement map $c$. 
  It turns out that the fibration sequence 
\[ \Emb_\partial^\fr((I^k)^{\cup p}\cup (I^{k-1})^{\cup q},I^{2k})_{\iota} \xrightarrow{c} B\Diff_\partial(V)\to B\Diff_\partial(I^{2k})  \]
induced by the Cerf--Palais fibration $\Diff_\partial(I^{2k})\to \Emb_\partial^\fr((I^k)^{\cup p}\cup (I^{k-1})^{\cup q},I^{2k})_\iota$ splits, and we have a homotopy equivalence
\[ B\Diff_\partial(V)\simeq B\Diff_\partial(I^{2k})\times \Emb_\partial^\fr((I^k)^{\cup p}\cup (I^{k-1})^{\cup q},I^{2k})_\iota. \]
Hence the map $c$ has the homotopy left inverse taking the quotient by $B\Diff_\partial(I^{2k})$. This kind of relation will be used to move back and forth between families of embeddings and fiber bundles.

We give a canonical choice of the homotopy left inverse to $c$. Let $D^{k+1}$ and $D^k$ be spanning disks of the inclusions $I^k\subset I^{2k}$ and $I^{k-1}\subset I^{2k}$, respectively. Let $\partial_\sqcup D^{k+1}$ and $\partial_\sqcup D^k$ denote $D^{k+1}\cap \partial I^{2k}$ and $D^k\cap\partial I^{2k}$, respectively, and let $\Emb_{\partial_\sqcup}((D^{k+1})^{\cup p}\cup (D^k)^{\cup q},I^{2k})$ denote the space of embeddings $(D^{k+1})^{\cup p}\cup (D^k)^{\cup q}\to I^{2k}$ fixed on $(\partial_\sqcup D^{k+1})^{\cup p}\cup (\partial_\sqcup D^k)^{\cup q}$. Then we have the fiber bundle
\[  \Emb_{\partial_\sqcup}((D^{k+1})^{\cup p}\cup (D^k)^{\cup q},I^{2k})
\to \Emb_\partial^\fr((I^k)^{\cup p}\cup (I^{k-1})^{\cup q},I^{2k})_\iota \]
with fiber $\Emb_\partial((D^{k+1})^{\cup p}\cup (D^k)^{\cup q},V)$. Since $\Emb_{\partial_\sqcup}((D^{k+1})^{\cup p}\cup (D^k)^{\cup q},I^{2k})$ is contractible, the map $\Omega\Emb_\partial^\fr((I^k)^{\cup p}\cup (I^{k-1})^{\cup q},I^{2k})_\iota\to\Emb_\partial((D^{k+1})^{\cup p}\cup (D^k)^{\cup q},V)$ is a homotopy equivalence. The homotopy inverse of this map can be given explicitly as follows. We consider an embedding $D^{k+1}\to V$ (resp. $D^k\to V$) as a path of framed embeddings $I^k\to I^{2k}$ (resp. $I^{k-1}\to I^{2k}$) from the standard inclusion $I^k\subset I^{2k}$ (resp. $I^{k-1}\subset I^{2k}$) to $I^k\cong \partial_\sqcup D^{k+1}\subset I^{2k}$ (resp. $I^{k-1}\cong \partial_\sqcup D^k\subset I^{2k}$). By connecting this path with the path given by the standard inclusion $D^{k+1}\subset I^{2k}$ (resp. $D^k\subset I^{2k}$), and by modifying near $\partial_\sqcup D^{k+1}$ (resp. $\partial_\sqcup D^k$), we get a loop of framed embeddings $I^k\to I^{2k}$ (resp. $I^{k-1}\to I^{2k}$), thus obtaining a map
\[  \Emb_\partial((D^{k+1})^{\cup p}\cup (D^k)^{\cup q},V)
\to \Omega \Emb_\partial^\fr((I^k)^{\cup p}\cup (I^{k-1})^{\cup q},I^{2k})_\iota. \]
By precomposing with the action $\Diff_\partial(V)\to \Emb_\partial((D^{k+1})^{\cup p}\cup (D^k)^{\cup q},V)$, we get a map
\[ \lambda\colon B\Diff_\partial(V)\to \Emb_\partial^\fr((I^k)^{\cup p}\cup (I^{k-1})^{\cup q},I^{2k})_\iota, \]
which is a homotopy left inverse of $c$.
\end{Rem}

We say that two framings of a string link are homotopic if they are homotopic relative to the boundary as sections of the normal frame bundle.
\begin{Lem}\label{lem:framing-brunnian}
If a string link $f\in \calE(\emptyset)$ has a system of Brunnian null-isotopies, then it defines a unique homotopy class of a framing on $f$, which is compatible with the framings defined for $f$ with some components removed.
\end{Lem}
\begin{proof}
There is a natural map $\calB(S)\to \bEmb_\partial(\bigcup_{\lambda\in \Omega\setminus S}I^{a_\lambda},I^N)$ which maps null-isotopy to a null regular homotopy. Thus for each $\lambda\in\Omega$, we get a homotopy class of framing of the $\lambda$-th component by the composite map
$\calB(\Omega\setminus\{\lambda\})\to \bEmb_\partial(I^{a_\lambda},I^N)_\iota\to \Emb_\partial^\fr(I^{a_\lambda},I^N)_\iota$.
The compatibility with other choices of $S\neq \emptyset$ follows by the compatibility condition in the definition of homotopy limit.
\end{proof}

\section{Graph surgery}\label{s:graph-surg}

In this section, we prove Theorem~\ref{thm:graph-surgery} assuming Theorem~\ref{thm:l-valent-general}.

\subsection{Directed graph cycle}\label{ss:directed-graph}

We take lifts of cycles in $\mathcal{GC}$ to those for directed graphs.
For a labelled graph $\Gamma$, we consider directed graphs $(\Gamma,\alpha)$, where $\alpha$ is a direction assignment on the edges of $\Gamma$. We consider a label on $\Gamma\in \mathcal{GC}$ represents an orientation of the real vector space $\R^{E(\Gamma)}$. 

\begin{Def}[Directed graph complex]
Let $\overrightarrow{\mathcal{GC}}$ denote the chain complex freely spanned over $\Q$ by directed graphs, oriented by labels on edges, with the differential defined by the half of the sum of the results of all possible ways to expand one vertex of degree $\geq 4$ with directions.
\end{Def}
There is a natural map $\eta\colon \mathcal{GC}\to \overrightarrow{\mathcal{GC}}$, which takes a labelled graph $\Gamma$ to
\begin{equation}\label{eq:eta}
 \frac{1}{2^{|E(\Gamma)|}}\sum_\alpha (\Gamma,\alpha), 
\end{equation}
where the sum is over all direction assignments on the labelled graph $\Gamma$. Or in other words, it takes each edge to the average of the two directions:
\begin{equation}\label{eq:doubling}
 \fig{-0mm}{1mm}{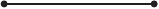}\quad\mapsto\quad 
\frac{1}{2}\Bigl(\quad\fig{-0mm}{1.5mm}{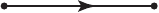}\quad+\quad\fig{-0mm}{1.5mm}{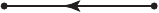}\quad\Bigr).
\end{equation}
If we forget labels, the element $\eta(\Gamma)$ can be rewritten as  
\begin{equation}\label{eq:sum-graphs}
 \frac{1}{2^{|E(\Gamma)|}}\sum_\alpha \frac{|\mathrm{Aut}\,\Gamma|}{|\mathrm{Aut}\,(\Gamma,\alpha)|}(\Gamma,\alpha), 
\end{equation}
where the sum is over all classes of direction assignments on $\Gamma$. For example,
\[ \frac{1}{10}\,\,\fig{-10mm}{20mm}{graph_X.pdf}
\stackrel{\eta}{\longmapsto}
\frac{1}{2^{10}}\left(\frac{1}{5}\,\,\fig{-10mm}{20mm}{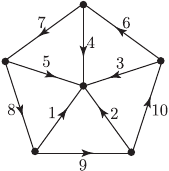}+\cdots\Bigl({{\text{all possible ways}}\atop{\text{to assign directions}}}\Bigr)\right),  \]
where the labels on edges represent a graph orientation, which is a structure independent of directions. Note that for each direction assignment $\alpha_X$ on $X$, $|\Aut(X,\alpha_X)|=5$ or $1$. The following lemma is evident from the definitions of the differentials.
\begin{Lem}
$\eta\colon \mathcal{GC}\to \overrightarrow{\mathcal{GC}}$ is a chain map, which is a splitting of the forgetful map $\overrightarrow{\mathcal{GC}}\to \mathcal{GC}$. Hence $\eta_*$ identifies $H_*(\mathcal{GC})$ with a direct summand of $H_*(\overrightarrow{\mathcal{GC}})$.
\end{Lem}

Now we consider the linear combination
\begin{equation}\label{eq:gamma}
 \vec{\gamma}=\eta(2^{10}\gamma)=\sum_{\alpha_X}\frac{2}{|\mathrm{Aut}\,(X,\alpha_X)|}(X,\alpha_X)
-\sum_{\alpha_Y}\frac{1}{|\mathrm{Aut}\,(Y,\alpha_Y)|}(Y,\alpha_Y), 
\end{equation}
where $|\mathrm{Aut}\,(X,\alpha_X)|=1$ or $5$, and $|\mathrm{Aut}\,(Y,\alpha_Y)|=1$ for all $\alpha_Y$. 
\begin{Lem}\label{lem:gamma-cycle}
The combinatorial chain $\vec{\gamma}$ is a $\partial$-cycle in $\overrightarrow{\mathcal{GC}}$, i.e., $\partial\vec{\gamma}=0$. Moreover, the homology class of $\vec{\gamma}$ is nontrivial in $H_*(\overrightarrow{\mathcal{GC}})$. 
\end{Lem}

\begin{proof}
That $\vec{\gamma}=\eta(2^{10}\gamma)$ is a $\partial$-cycle follows since $\gamma$ is a $\partial$-cycle and $\eta$ is a chain map. There is a well-defined bilinear pairing between the graph cohomology and the directed graph homology, which is induced from the pairing $\mathcal{GC}^\vee\otimes \overrightarrow{\mathcal{GC}}\to \Q$;
\[ \Gamma\otimes (\Gamma',\alpha)\mapsto \left\{\begin{array}{ll}
\pm|\Aut\Gamma| & (\text{if $\Gamma'\cong\pm \Gamma$}),\\
0 & (\text{otherwise}),
\end{array}\right. \]
where we identify the dual complex $\mathcal{GC}^\vee$ with that of \cite{Kon}.
This pairing is indeed a cocycle in $\mathrm{Hom}(\mathcal{GC}^\vee\otimes\overrightarrow{\mathcal{GC}},\Q)$.
According to Bar-Natan--McKay (\cite{BNM}), $X$ represents a nontrivial graph cohomology class. We have
\[ \langle \vec{\gamma},X\rangle=2^{10}\langle \gamma,X\rangle=\frac{2^{10}\cdot 2\,|\mathrm{Aut}\,X|}{|\mathrm{Aut}\,X|}\neq 0. \]
Hence $\vec{\gamma}$ gives a nontrivial homology class. (More generally, $\mathrm{Aut}\,\Gamma$ acts on the set of all $2^{|E(\Gamma)|}$ ways of assignments of directions on a labelled $\Gamma$. It follows from the orbit decomposition formula that
\[ 2^{|E(\Gamma)|}=\sum_{i}\frac{|\mathrm{Aut}\,\Gamma|}{|\mathrm{Aut}\,(\Gamma,\alpha_i)|}, \]
where the sum is taken over different $\mathrm{Aut}\,\Gamma$-orbits.)
\end{proof}

\subsection{From graphs to chains in $B\Diff_\partial(D^d)$}

Here, we assume $d=2k$ and use Hopf links of $S^{k-1}\cup S^k$ to decompose a graph into $\Psi_j$-graphs (Definition~\ref{def:psi-graph}). 
We embed a graph $\Gamma$ in $\mathrm{Int}\,D^{2k}$ and decompose each edge by a Hopf link $S^k\cup S^{k-1}\to \mathrm{Int}\,D^{2k}$, as follows. Namely, suppose that each embedded directed edge $e$ is the union of a pair of half-edges $e_\pm$ with $e_-\cap e_+$ being one point $p_e$, and that the direction of $e$ gives the order $e_-<e_+$. We cut $e$ into shortened half-edges $e_\pm'$ by removing a small open interval neighborhood of $p_e$ in $e$. 
We replace each $e$ by the union of a Hopf link with shortened half-edges attached to each component at the non-vertex one of the two endpoints of each of $e_\pm'$. As a result, we obtain a disjoint union of objects, each of which we call {\it $\Psi$-graphs}. We call the components of the Hopf links in $\Psi$-graphs {\it leaves} (see Definition~\ref{def:psi-graph}).
Figure~\ref{fig:graph_X} below shows schematic pictures of the examples for the following directed graphs.
\begin{equation}\label{eq:graph_X_Y}
 (X,\alpha_X)=\,\,\fig{-12mm}{25mm}{graph_X2.pdf},\qquad
(Y,\alpha_Y)=\,\,\fig{-12mm}{24mm}{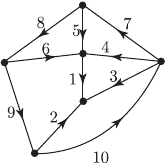}. 
\end{equation}
\begin{figure}[h]
\[  \fig{-20mm}{40mm}{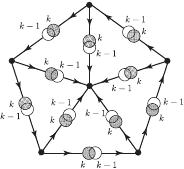}\qquad\qquad
\fig{-19mm}{37mm}{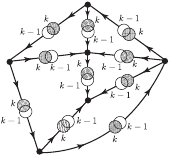}
 \]
\caption{The union of $\Psi$-graphs obtained from an embedded graph.}\label{fig:graph_X}
\end{figure}

We use direction of $e$ to determine the choice of the dimensions $k$ or $k-1$ of the terminal spheres in such a way that $S^k$ is attached to $e_-$, and $S^{k-1}$ is attached to $e_+$. 
We also choose orientations of the components of the Hopf links in such a way that their linking numbers (for two component links) are all $+1$. Note that the choice of orientations of the components of a Hopf link is not essential because two choices can be related to each other by isotopy (the oriented components $S^k$ and $S^{k-1}$ in a Hopf link can be turned inside out in $\R^{2k}$ simultaneously by isotopy).

\begin{proof}[Proof of Theorem~\ref{thm:graph-surgery}]
An embedded directed graph $(\Gamma,\alpha)$ gives a chain
\[ \overline{\omega}_{(\Gamma,\alpha)}\colon \overline{B}_{(\Gamma,\alpha)}\to B\Diff_\partial^\fr(D^{2k}), \]
where $\overline{B}_{(\Gamma,\alpha)}=\prod_{v\in V(\Gamma)}\overline{B}_{T(v)}$, $T(v)$ is $T_{\ell_v}(p_v,q_v)$ for $\ell_v=\deg(v)$ and for some $p_v,q_v$ determined by the direction structure $\alpha$. This is the chain given by the composition
\[ \xymatrix{
  \prod_{v\in V(\Gamma)}\overline{B}_{T(v)} \ar[rr]^-{\prod_v \overline{\rho}_{T(v)} } 
&& \prod_v B\Diff_\partial^\fr(V_v) \ar[r]^-{\simeq} & B\Diff_{Q_{(\Gamma,\alpha)}}^\fr(D^{2k}) \ar[r] 
& B\Diff_\partial^\fr(D^{2k}), 
} \]
where $\overline{\rho}_{T_(v)}\colon \overline{B}_{T(v)}\to B\Diff^\fr_\partial(V_v)$ is the map given in Lemma~\ref{lem:vert-surg}, and $Q_{(\Gamma,\alpha)}=D^{2k}-\mathrm{Int}(V_1\cup\cdots\cup V_r)$. Now the conditions (1)--(4) are obvious from the construction. In particular, the compatibility with lower excess graphs follows by Theorem~\ref{thm:l-valent-general}.
\end{proof}

In the definition of the complement map $c$, we need to fix identifications of the complements of the string links with the handlebodies $V_v$ in $D^{2k}$. This is done by using graph orientation given in \S\ref{ss:ori-chain}.

\section{4-valent vertex}\label{s:4-valent}

In this section, we prove Theorem~\ref{thm:vertex} for $\ell=4$. We let $X=I^{4n-1}-(I^{3n-2})^{\cup 3}$ and shall construct an embedding $\varphi_{\mathrm{IHX}}\in \Emb_\partial^\fr(D^{3n-2},X)$, which is a topological realization of the Jacobi relation:
\[ [[a,b],c]+[[b,c],a]+[[c,a],b]. \]
Then we construct a path $\{\varphi_t\}$ of embeddings 
\[ \varphi_t\in\Emb_\partial^\fr(D^{3n-2},X)\quad (t\in [0,1]) \]
with the following properties (Theorem~\ref{thm:phi-jacobi}).
\begin{enumerate}
\item[(i)] (Null-isotopy) $\varphi_0$ is the standard inclusion, and $\varphi_1$ is $\varphi_{\mathrm{IHX}}$.
\item[(ii)] (Brunnian) $\{\varphi_t\}$ has a Brunnian property in the sense of Definition~\ref{def:brun-null-iso}.
\end{enumerate}
The property (i) suggests that the path $\{\varphi_t\}$ is a topological realization of $[a,b,c]$ in the identity:
\[ [[a,b],c]+[[b,c],a]+[[c,a],b]=\partial[a,b,c]. \]

To find $\{\varphi_t\}$, we give an explicit null-isotopy $\varphi_t$ of $\varphi_{\mathrm{IHX}}$ geometrically in the attaching $(3n-1)$-sphere of a $3n$-handle in a handle decomposition of $S^n\times S^n\times S^n$. Then the Brunnian property can be proved by using this explicit model.

\subsection{Embedded Whitehead product}

\subsubsection{Embedded Whitehead product}\label{ss:emb-Wh}

Let $X=I^N-(I^{N-m-1}\cup I^{N-k-1})$ and let $f\colon D^m\to X$, $g\colon D^k\to X$ be neat embeddings that agree on the boundaries with fixed small spheres $o^{m-1}$, $o^{k-1}$, respectively. Suppose that $o^{m-1}$ and $o^{k-1}$ are in the same face $F=\{0\}\times \mathrm{Int}\,I^{N-1}\subset I^N=I\times \overline{F}$. Let $M,K$ denote the images of $f,g$, respectively. We equip $M$ (resp. $K$) with a normal $k$-framing (resp. $m$-framing) that has a fixed behavior near $\partial M$ (resp. $\partial K$). Roughly, the embedded Whitehead product $[f,g]$ is defined by connect-summing the boundary sphere of the $(m+k)$-skeleton $h^0\cup h^m\cup h^k$ of a handle decomposition $h^0\cup h^m\cup h^k\cup h^{m+k}$ of $S^m\times S^k$ with a fixed $(m+k-1)$-plane, where the $m$-handle is embedded parallel to $f$ and the $k$-handle is embedded parallel to $g$.

To be precise, suppose $m+k<N-1$. We fix a standard framed oriented $(m+k-1)$-disk $L^{m+k-1}_{\mathrm{st}}$ in $I^N$ whose boundary is a small framed sphere $o^{m+k-2}\subset F$.  We denote its image also by $L^{m+k-1}_{\mathrm{st}}$. Let $N_{L_{\mathrm{st}}^{m+k-1}}$ be a small tubular neighborhood of $L^{m+k-1}_{\mathrm{st}}$ in $I^N$. We fix a handle decomposition $h^0\cup h^m\cup h^k\cup h^{m+k}$ of $S^m\times S^k$ with a 0-handle $h^0$ at $*\times *$, 
an $m$-handle $h^m$ embedded along the core $S^m\times *$,
a $k$-handle $h^k$ embedded along the core $*\times S^k$, and
an $(m+k)$-handle $h^{m+k}$.
Let $H=h^0\cup h^m\cup h^k$ and equip this with an orientation induced from the orientation $o(S^m)\wedge o(S^k)$ of $S^m\times S^k$. 
Let $D_\ve^j=\{x\in\R^j\mid |x|\leq \ve\}$. We fix an embedding $\gamma_0\colon H\to \mathrm{Int}\,D^N_\ve$ by the restriction of a standard embedding $S^m\times S^k\to D^{m+1}_{\ve'}\times D^{k+1}_{\ve''}\subset \mathrm{Int}\,D^N_\ve$. Then we construct an embedding $\gamma\colon H\to X$ as follows. (See Figure~\ref{fig:wh_f_g}.)

\begin{enumerate}
\item We slightly thicken $M$ and $K$ to the images $N_M$ and $N_K$ of embeddings from $M\times D^k_\ve$ and $K\times D^m_\ve$, respectively, where $M\times\{0\}$ and $K\times\{0\}$ are mapped to $M$ and $K$, respectively. 
The intersection $N_M\cap \partial X$ (resp. $N_K\cap \partial X$) is $o^{m-1}\times D^k_\ve$ (resp. $o^{k-1}\times D^k_\ve$). 

\item We consider the decomposition $X=(I^{N-1}-(I^{N-m-2}\cup I^{N-k-2}))\times I=X'\cup X''$, where $X'=(I^{N-1}-(I^{N-m-2}\cup I^{N-k-2}))\times [0,\frac{1}{2}]$ and $X''=(I^{N-1}-(I^{N-m-2}\cup I^{N-k-2}))\times [\frac{1}{2},1]$. Let $M'$, $K''$ be obtained from $M,K$ by affine linear rescalings along the last factor $I$ so that $M'\subset X'$ and $K''\subset X''$. Let $N_{M'}, N_{K''}$ be the analogues of $N_{M}, N_{K}$ in $X'\cup X''$, respectively. 

\item We decompose $\overline{F}=\overline{G}\times I$. Let $\varphi\colon I^N\to [0,\frac{1}{2}]\times\overline{F}$ and $\psi\colon X'\cup X''\to [\frac{1}{2},1]\times\overline{F}$ be the affine linear diffeomorphisms such that 
  \begin{itemize}
  \item $\psi(o^{m-1})\subset \{\frac{1}{2}\}\times \overline{G}\times [0,\frac{1}{2}]$, $\psi(o^{k-1})\subset \{\frac{1}{2}\}\times \overline{G}\times [\frac{1}{2},1]$, and
\item $X=\varphi(I^N)\cup \psi(X')\cup \psi(X'')$.
  \end{itemize}
\item We replace $\gamma_0(h^m)$ and $\gamma_0(h^k)$ with their connected sums (of thickened disks) with $N_{M'}$ and $N_{K''}$, respectively, along thin arcs. The resulting embedding $\gamma\colon H\to X$ is such that $\gamma(H)\cap X'=\gamma(h^m)\cap X'=N_{M'}$ and $\gamma(H)\cap X''=\gamma(h^k)\cap X''=N_{K''}$. 
\end{enumerate}
\begin{figure}[h]
\[ \includegraphics[height=35mm]{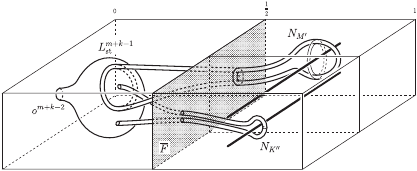} \]
\caption{Embedded Whitehead product $[f,g]=\gamma|_{\partial H}\# L^{m+k-1}_{\mathrm{st}}$.}\label{fig:wh_f_g}
\end{figure}

\begin{Def}[Embedded Whitehead product]\label{def:emb-wh}
Let $[f,g]\colon D^{m+k-1}\to X$ denote the framed embedding obtained by connect summing $\gamma|_{\partial H}$ and $L_{\mathrm{st}}^{m+k-1}$, where we orient $\partial H$ by the ``inward-normal-first'' convention\footnote{This is the same as the ``outward-normal-first'' convention from the handle $h^{m+k}$, and also as its attaching map.} from $o(H)|_{\partial H}=o(S^m)\wedge o(S^k)|_{\partial H}$.
\end{Def}

This is essentially the same as the construction in \cite[Theorem~1.9]{KKV24} by ``graphing'' from a configuration space, and ``the embedded version of the Whitehead product'' in \cite{Kos24b}.

The following lemma is evident from Definition~\ref{def:emb-wh}.

\begin{Lem}\label{lem:emb-conti-Wh}
Let $X=I^N-(I^{N-m-1}\cup I^{N-k-1})$ and let $p\colon X\to S^m\vee S^k$ be a homotopy equivalence which maps $o^{m+k-2}$ to the basepoint. Let $f\colon D^m\to X$ and $g\colon D^k\to X$ be neat embeddings obtained from meridian spheres of $I^{N-m-1}$ and $I^{N-k-1}$ by connect-summing small disks bounded by $o^{m-1}$ and $o^{k-1}$, respectively. Then $p\circ [f,g]$ represents the Whitehead product $[p\circ f,p\circ g]$ in $\pi_{m+k-1}(S^m\vee S^k)=[(D^{m+k-1},\partial D^{m+k-1}),(S^m\vee S^k,*)]$. 
\end{Lem}
\begin{Lem}
Let $X,f,g$ be as in Lemma~\ref{lem:emb-conti-Wh}. Let $\iota\colon I^{N-m-1}\cup I^{N-k-1}\to I^N$ be the standard inclusion. 
The framed link $[f,g]\cup \iota\colon I^{m+k-1}\cup (I^{N-m-1}\cup I^{N-k-1})\to I^N$ is relatively isotopic to that obtained from the Borromean link of type $(m+k-1,N-m-1,N-k-1;N)$ with the canonical framing by component-wise connect-summing to the standard inclusion.
\end{Lem}
\begin{proof}
This can be proved by deforming the well-known coordinate presentation of the Borromean link (see e.g. \cite[\S{4}]{Ma}). Namely, one of the three components, e.g. the first one, in the Borromean link bounds a framed submanifold diffeomorphic to $S^m\times S^k-\mathrm{Int}\,D^{m+k}$ in the complement of other two components in $I^N$. The spanning submanifold gives the plumbing as given in the definition of $[f,g]$. Since two components in the Borromean link form an unlink, it gives the embedded Whitehead product in the complement of the standard inclusion $\iota$.
\end{proof}

\subsubsection{Iterated embedded Whitehead product and the Jacobi relation}

This procedure can be iterated. Let $X=I^N-(I^{N-m-1}\cup I^{N-k-1}\cup I^{N-j-1})$ and $f\colon D^m\to X$, $g\colon D^k\to X$, $h\colon D^j\to X$ be neat embeddings that agree on the boundaries with fixed small spheres $o^{m-1}$, $o^{k-1}$, $o^{j-1}$, respectively. Then $[[f,g],h]$, $[[g,h],f]$, $[[h,f],g]$ are neat embeddings $D^{m+k+j-2}\to X$ bounded by a fixed small sphere $o^{m+k+j-3}\subset F$. 

The sum $a\,\#\, a'$ of two neat embeddings $a,a'\colon D^{m+k+j-2}\to X$ bounded by $o^{m+k+j-3}$ is defined as follows. We consider the decomposition
\[ X=\varphi(I^N)\cup \psi(X')\cup \psi(X'') \]
as in the definition of the embedded Whitehead product. 
We fix a standard disk $L_{\mathrm{st}}^{m+k+j-2}$ in $\varphi(I^N)$ bounded by $o^{m+k+j-3}\subset F$. Let $b_1$ and $b_2$ be small disjoint $(m+k+j-2)$-disks embedded in $\mathrm{Int}\,L_{\mathrm{st}}^{m+k+j-2}$.
Then we may connect-sum $L_{\mathrm{st}}^{m+k+j-2}$ with $a(D^{m+k+j-2})$ and $a'(D^{m+k+j-2})$ by thin cylinders $S^{m+k+j-3}\times I$ at $b_1$ and $b_2$, respectively, where $S^{m+k+j-3}\times \{0\}$ is glued to $\partial b_i$, and $S^{m+k+j-3}\times \{1\}$ is glued to $o^{m+k+j-3}$ in $\psi(X')$ or $\psi(X'')$. The sum $a_1\,\#\, \cdots\,\#\, a_r$ of three or more neat embeddings $a_i\colon  D^{m+k+j-2}\to X$ bounded by $o^{m+k+j-3}$ is defined similarly, with small disjoint small disjoint $(m+k+j-2)$-disks $b_1,\ldots,b_r$ embedded in $\mathrm{Int}\,L_{\mathrm{st}}^{m+k+j-2}$, and a decomposition $X=X^{(1)}\cup \cdots\cup X^{(r)}$. See Figure~\ref{fig:connect_sum}.
\begin{figure}[h]
\[ \includegraphics[height=35mm]{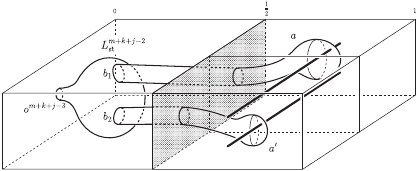} \]
\caption{Connected sum $a\,\#\, a'$.}\label{fig:connect_sum}
\end{figure}

\begin{Def}[Jacobi relation]\label{def:jacobi-rel}
Let $X=I^N-(I^{N-m-1}\cup I^{N-k-1}\cup I^{N-j-1})$. We say that a neat embedding $\varphi\colon D^{m+k+j-2}\to X$ bounded by $o^{m+k+j-3}$ {\it represents the Jacobi relation for $f,g,h$} if it has a decomposition $\varphi=a_1\,\#\, a_2\,\#\, a_3$, where $a_1,a_2,a_3\colon D^{m+k+j-2}\to X$ are neat embeddings bounded by $o^{m+k+j-3}$ such that 
$a_1\,\#\, a_2\,\#\, a_3$ is relatively isotopic to 
\[ [[f,g],h]\,\#\, (-1)^{mk+mj}[[g,h],f]\,\#\, (-1)^{mj+kj}[[h,f],g] \]
relative to $L_{\mathrm{st}}^{m+k+j-2}$ (See Appendix~\ref{s:sign-jacobi}).
\end{Def}

\subsection{Prescribed Brunnian null-isotopy of embedded Whitehead product}\label{ss:pres-brun-2}

Now we see that an embedded Whitehead product has a natural Brunnian null-isotopy (Definition~\ref{def:brun-null-iso}). A system of Brunnian null-isotopy of an embedding $f\in \calE(\emptyset)=\Emb_\partial((I^{2n-1})^{\cup 3},I^{3n})$ is by definition a lift of the constant map $\mathrm{const}_f$ at $f$ in $\Map(\Delta^2,\calE(\emptyset))$ to $\holim{\emptyset\neq S\subset \Omega}{\calB(S)}$ for the following diagram:
\begin{equation}\label{eq:cubical3}
\vcenter{\small
\xymatrix@C=0em{
& 
& & \calB(\{1\}) \ar[dd] \ar[dl]
\\
\calB(\{3\}) \ar[rr]\ar[dd]
& & \calB(\{1,3\}) \ar[dd]
\\
& \calB(\{2\}) \ar'[r][rr] \ar[dl]
& & \calB(\{1,2\}) \ar[dl]
\\
\calB(\{2,3\}) \ar[rr]
& & \calB(\{1,2,3\}) 
}
}
\end{equation}
Recall that the homotopy limit $\holim{\emptyset\neq S\subset \Omega}{\calB(S)}$ in this case is the space of natural transformations from the diagram 
\[ \xymatrix{\small
& 
& & {}_\bullet \ar[dd] \ar[dl]
\\
{}_\bullet \ar[rr]\ar[dd]
& & I \ar[dd]
\\
& {}_\bullet \ar'[r][rr] \ar[dl]
& & I \ar[dl]
\\
I \ar[rr]
& & \Delta^2
}
\]
in $\mathrm{Top}$, where the morphisms are given by the faces of the simplices in $\Delta^2$, to the diagram (\ref{eq:cubical3}) (see e.g. \cite[\S{5.3}]{MV}). A system of Brunnian null-isotopy of $f$ is one such mapping that maps $\Delta^2$ to the point $f$. Thus it consists of paths $\gamma_1,\gamma_2,\gamma_3$ in $\calE(\{1\})$, $\calE(\{2\})$, $\calE(\{3\})$ and homotopies $\gamma_{12},\gamma_{23},\gamma_{13}$ in $\calE(\{1,2\})$, $\calE(\{2,3\})$, $\calE(\{1,3\})$, respectively, such that 
\begin{itemize}
\item $\gamma_i$ ($i=1,2,3$) is a path from $\pi_{\{i\}}(f)$ to the standard inclusion $\iota$,
\item $\gamma_{ij}$ ($i\neq j$) is a homotopy in $\calE(\{i,j\})$ between the paths $\overline{\gamma}_i$ and $\overline{\gamma}_j$ induced from $\gamma_i$ and $\gamma_j$, respectively. 
\end{itemize}

\begin{Lem}\label{lem:brun-borromean}
The Borromean string link $(I^{2n-1})^{\cup 3}\to I^{3n}$ admits a system of Brunnian null-isotopies, which is cyclically symmetric.
\end{Lem}
\begin{proof}
To give a sequence of paths for an embedded Whitehead product $f$ as above, we consider lifts $(f,\rho_1),(f,\rho_2),(f,\rho_3)$ of $f\in\calE(\emptyset)=\calB(\{1,2,3\})$ in $\calB(\{2,3\})$, $\calB(\{1,3\})$, $\calB(\{1,2\})$, respectively, given as follows. A well-known coordinate presentation of the Borromean link (\cite[\S{4}]{Ma}) provides a null-isotopy $\rho_i$ of the $i$-th component, which is along a spanning disk and may intersect other component. We take the null-isotopies $\gamma_1,\gamma_2,\gamma_3$ of the pairs of components labelled by $\{2,3\},\{1,3\},\{1,2\}$, respectively, so that $\gamma_1$ is mapped to $(f,\rho_2)$ in $\calB(\{1,3\})$, $\gamma_2$ is mapped to $(f,\rho_3)$ in $\calB(\{1,2\})$, and $\gamma_3$ is mapped to $(f,\rho_1)$ in $\calB(\{2,3\})$.
We deform the null-isotopy $\rho_3$ in $\calE(\{1,2\})$ into another path $\overline{\gamma}_1$ that can be lifted to $\gamma_1$ in $\calE(\{1\})$ by sliding along the spanning disk of the second component (Figure~\ref{fig:borromean}). 

Let $\gamma_{12}$ be this homotopy in $\calE(\{1,2\})$ from $\rho_3$ to a path from $\calE(\{1\})$. By symmetry, the homotopies $\gamma_{23}$, $\gamma_{13}$ and paths $\gamma_2$, $\gamma_3$ are defined similarly from $\rho_1$, $\rho_2$, respectively. Hence we obtain a sequence $(\gamma_1,\gamma_2,\gamma_3,\gamma_{12},\gamma_{23},\gamma_{13})$ giving a lift of $\mathrm{const}_f\in\Map(\Delta^2,\calE(\emptyset))$ to $\holim{\emptyset\neq S\subset \Omega}{\calB(S)}$. The cyclic symmetry is obvious from the construction.
\end{proof}
\begin{figure}[h]
\[ \xymatrix{
\fig{-17mm}{35mm}{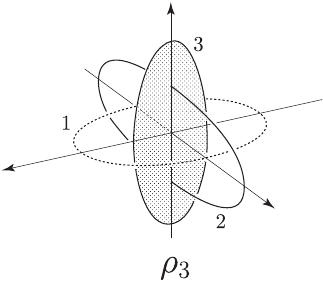}\quad\ar[r]^-{\displaystyle\gamma_{12}}&
\quad\fig{-17mm}{35mm}{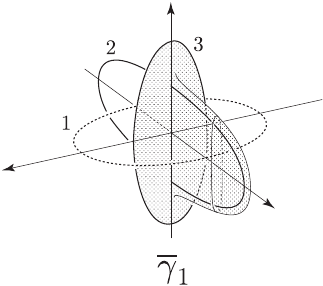} 
}\]
\caption{A homotopy in $\calE(\{1,2\})$.}\label{fig:borromean}
\end{figure}

\begin{Rem}\label{rem:pushing}
The deformation $\gamma_{12}$ of $\rho_3$ into $\overline{\gamma}_1$ can be interpreted in the plumbing model of the embedded Whitehead product, as in Figure~\ref{fig:plumb-disk1}, left (for $n=1$).
This deformation between the spanning disks is given by the disk bounded by the union of the spanning disks, as in Figure~\ref{fig:plumb-disk1}, right, which has been essentially given in \cite[Definition~5.8]{Kos24b}.
\begin{figure}[H]
\[ \fig{-15mm}{30mm}{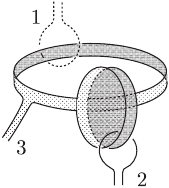}\quad\stackrel{\gamma_{12}}{\longrightarrow}
\quad\fig{-15mm}{30mm}{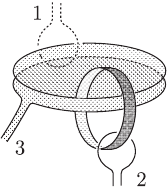}\,\,,\qquad  
\fig{-9mm}{18mm}{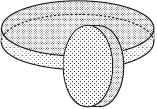}\quad\cong
\quad\fig{-12mm}{24mm}{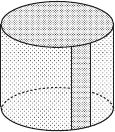} \]
\caption{Left: The path $\gamma_{12}$ in the plumbing model. Right: The deformation between two spanning disks.}\label{fig:plumb-disk1}
\end{figure}
\end{Rem}

\begin{Def}
We define $\beta_{T_3}(T_3)=\omega_{T_3}(T_3)\colon *\to \Emb_\partial^\fr((I^{2n-1})^{\cup 3},I^{3n})$ by the Borromean string link. We define $m_3=1$.
\end{Def}

\begin{proof}[Proof of Theorem~\ref{thm:vertex} for $\ell=3$]
The relevant conditions for $\ell=3$ are (4) Brunnian property, (5) cyclic symmetry, and (6) excess 0, which follow from Lemma~\ref{lem:brun-borromean}.
\end{proof}

\subsection{Explicit null-isotopy of the Jacobi relation}\label{ss:explict-null}

Let $X=I^{4n-1}-(I^{3n-2}\cup I^{3n-2}\cup I^{3n-2})$. Let $c_i$ ($i=1,2,3$) be a small meridian $n$-sphere of the $i$-th component in $I^{3n-2}\cup I^{3n-2}\cup I^{3n-2}$ removed from $I^{4n-1}$. Let 
$f,g,h\colon D^n\to X$ be neat embeddings obtained from the connected sum of $L_{\mathrm{st}}^n$ with $c_1,c_2,c_3$, respectively.

\begin{Thm}\label{thm:phi-jacobi}
Let $n\geq 2$. For the neat embeddings $f,g,h\colon D^n\to X$, we construct 
an explicit neat embedding
\[ \varphi_{\mathrm{IHX}}\colon D^{3n-2}\to I^{4n-1}-(I^{3n-2}\cup I^{3n-2}\cup I^{3n-2}) \]
 bounded by $o^{3n-3}$ that represents the Jacobi relation for $f,g,h$.
Moreover, we construct a path $[f,g,h]$ from $\iota$ to $\varphi_{\mathrm{IHX}}$ that has the Brunnian property relative to the prescribed system of Brunnian null-isotopies of $\varphi_{\mathrm{IHX}}$ induced from those given in \S\ref{ss:pres-brun-2}. In particular, $\partial[f,g,h]$ represents the Jacobi relation for $f,g,h$:
\[ \partial[f,g,h]=\bigl([[f,g],h]\,\#\, [[g,h],f]\,\#\, [[h,f],g]\bigr)-\iota. \]
\end{Thm}

The RHS of the identity of Theorem~\ref{thm:phi-jacobi} is considered as an element of $S_0(\calE;\Q)$, in which the operation `$-$' makes sense. If it is considered in $S_0'(\calE;\Q)$, $\#$ is identified with $+$, and summing $\iota$ does not change the homology class. 

\subsubsection{A decomposition of the $(3n-1)$-sphere $S$}

We consider the $(3n-1)$-sphere $S:=\partial(D^n\times D^n\times D^n)$ and the natural continuous map $D^n\times D^n\times D^n\to S^n\times S^n\times S^n$ induced by the map $D^n\to S^n$ collapsing the boundary to $\infty$. The natural cell decomposition of $S^n\times S^n\times S^n$ lifts to a decomposition of $D^n\times D^n\times D^n$. In this, $S=\partial(D^n\times D^n\times D^n)$ is decomposed as the sum of the following faces:
\begin{itemize}
\item Dimension $3n-1$: $D^n\times D^n\times \partial D^n$, $D^n\times \partial D^n\times D^n$, $\partial D^n\times D^n\times D^n$.

\item Dimension $3n-2$: $\partial D^n\times \partial D^n\times D^n$, $\partial D^n\times D^n\times \partial D^n$, $D^n\times \partial D^n\times \partial D^n$.

\item Dimension $3n-3$: $\partial D^n\times \partial D^n\times \partial D^n$.
\end{itemize}

\subsubsection{The $(3n-2)$-sphere $\Sigma_{\mathrm{IHX}}$ in $S$}\label{ss:Sigma_IHX}

Now we assume $n\geq 2$. We consider the following contractible subsets of $S=\partial(D^n\times D^n\times D^n)$:
\[\begin{split}
K_1&=(D^n\times D^n\times \partial D^n)\cup(*\times *\times D^n),\\
K_2&=(D^n\times \partial D^n\times D^n)\cup(*\times D^n\times *),\\
K_3&=(\partial D^n\times D^n\times D^n)\cup(D^n\times *\times *),
\end{split}\]
where $*\in \partial D^n$ is a basepoint.
We shrink the codimension 1 faces $D^n\times D^n\times \partial D^n$ etc. in $K_i$ slightly by replacing the two $D^n$ factors with slightly smaller one $D^n_{1-\ve}$ of radius $1-\ve$. Further, we enlarge the $n$-cell $*\times *\times D^n$ etc. slightly to a larger one $D^n_{1+\ve}\subset S$ so that its boundary is attached to a shrinked codimension 1 face $D^n_{1-\ve}\times D^n_{1-\ve}\times \partial D^n$ etc. and the collar $D^n_{1+\ve}-\mathrm{Int}\,D^n\cong [1,1+\ve]\times \partial D^n$ agrees with the locus of (the basepoint of $\partial (D^n_{s}\times D^n_s)$)$\times \partial D^n$ for $s\in[1,1+\ve]$. Let $K_1',K_2',K_3'$ denote the resulting objects obtained from $K_1,K_2,K_3$, respectively (see Figure~\ref{fig:K1}).
\begin{figure}[H]
\[ \includegraphics[height=30mm]{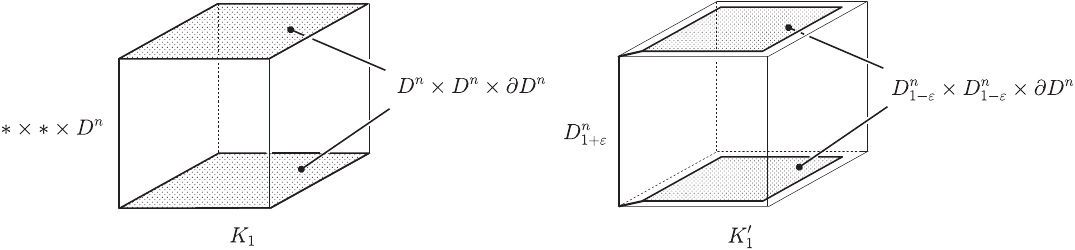} \]
\caption{$K_1$ and $K_1'$.}\label{fig:K1}
\end{figure}

If $n\geq 2$, the three objects $K_1',K_2',K_3'$ in $S$ can be assumed to be mutually disjoint (by perturbing the $n$-cells $*\times *\times D^n_{1+\ve}$ etc. slightly). Note that for $n\geq 2$, a pair of $n$-submanifolds in a small $(3n-1)$-disk can be disjuncted in general position. 
It can be seen that $K_i'$ is contractible. 
We take small neighborhoods $N_1,N_2,N_3$ of $K_1',K_2',K_3'$ in $S$ which are mutually disjoint $(3n-1)$-disks. 
One can see that $\partial N_1,\partial N_2,\partial N_3$ represent the three terms $[[\alpha,\beta],\gamma]$, $[[\beta,\gamma],\alpha]$, $[[\gamma,\alpha],\beta]$ in the Jacobi relation in the $n$-skeleton of $S^n\times S^n\times S^n$ (Figure~\ref{fig:modified-IHX-link}).
\begin{figure}[h]
\[\includegraphics[height=32mm]{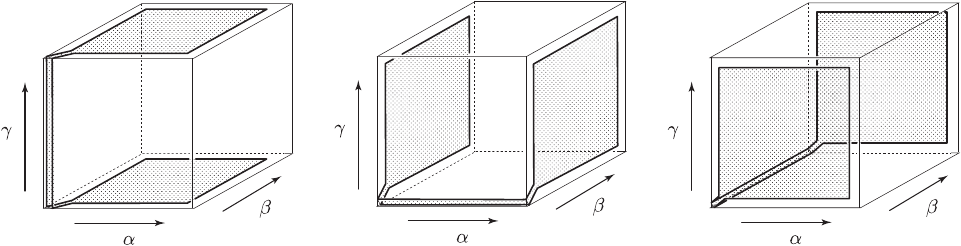} \]
\caption{Iterated Whitehead products $[[\alpha,\beta],\gamma]$, $[[\beta,\gamma],\alpha]$, $[[\gamma,\alpha],\beta]$ with prescribed spanning disks $N_1,N_2,N_3$, respectively.}\label{fig:modified-IHX-link}
\end{figure}

\begin{Lem}\label{lem:null-iso}
Let $n\geq 2$. Let 
$N_{\mathrm{IHX}}=N_1\natural N_2\natural N_3$ and $\Sigma_{\mathrm{IHX}}=\partial(N_1\natural N_2\natural N_3)$,
where the boundary connected sum $\natural$ can be performed within $S$ when $n\geq 2$.
Then the $(3n-2)$-sphere $\Sigma_{\mathrm{IHX}}$ bounds the disk $W_{\mathrm{IHX}}=S-\mathrm{Int}\,N_{\mathrm{IHX}}$, which is disjoint from the union 
\[ O:=(D_{1-\ve}^n\times D_{1-\ve}^n\times \partial D^n)\cup (D_{1-\ve}^n\times \partial D^n\times D_{1-\ve}^n)\cup (\partial D^n\times D_{1-\ve}^n\times D_{1-\ve}^n)\]
of the codimension 1 faces. 
\end{Lem}
\begin{proof}
The lemma holds since a smooth orientation preserving embedding $D^{3n-1}\to S^{3n-1}$ is unique up to isotopy. Hence the complement of the disk $\mathrm{Int}\,N_{\mathrm{IHX}}$ in $S$ is diffeomorphic to a disk (Figure~\ref{fig:codim0-disk}). The disk $W_{\mathrm{IHX}}$ is disjoint from $O$ since $O\subset K_1'\cup K_2'\cup K_3'\subset \mathrm{Int}\,N_{\mathrm{IHX}}$. 
\end{proof}
\begin{figure}[h]
\[ \includegraphics[height=35mm]{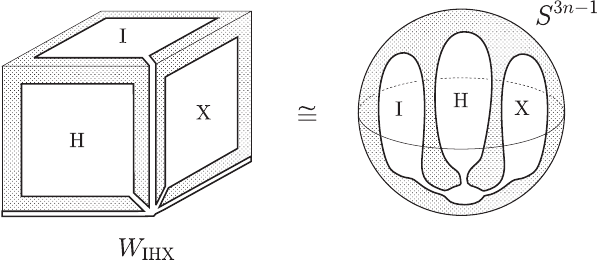} \]
\caption{The $(3n-1)$-disk $W_{\mathrm{IHX}}$ in $S^{3n-1}$ bounded by $\Sigma_{\mathrm{IHX}}\cong S^{3n-2}$.}\label{fig:codim0-disk}
\end{figure}

\begin{Rem}\label{rem:disjoint-2n-cells}
That the disk $W_{\mathrm{IHX}}$ is disjoint from $O$ implies that its image in $S$ can be assumed to be included in the preimage of the $n$-skeleton of a handle decomposition of $S^n\times S^n\times S^n$.
\end{Rem}

\begin{Lem}\label{lem:B-dP}
Let $P$ be the $n$-skeleton of a handle decomposition of $S^n\times S^n\times S^n$ with one 0-handle, three $n$-handles, three $2n$-handles, and one $3n$-handle, which is obtained from the sum $\pr_1^*h_1+\pr_2^*h_2+\pr_3^*h_3$ of the pullbacks of Morse functions $h_i\colon S^n\to \R$ ($i=1,2,3$) with two critical points by the projection $\pr_i\colon S^n\times S^n\times S^n\to S^n$ to the $i$-th factor. The embedding $P\to X$ induced from those of the $n$-handles by $f,g,h$ induces a sequence of embeddings $W_{\mathrm{IHX}}\stackrel{\subset}{\to} S-O\to X$.
\end{Lem}
\begin{proof}
$S-O$ is embedded naturally in $\partial P$ as a part of the attaching map of the $3n$-handle. 
\end{proof}

\begin{Def}[Triple bracket]
Let $n\geq 2$. We define $\varphi_{\mathrm{IHX}}\colon D^{3n-2}\to X$ by a connected sum of the sphere $\Sigma_{\mathrm{IHX}}$ embedded as in Lemma~\ref{lem:B-dP} with the fixed disk $L_{\mathrm{st}}^{3n-2}$. We define the bracket $[f,g,h]$ to be a path in $\Emb_\partial(D^{3n-2},I^{4n-1}-(I^{3n-2})^{\cup 3})$ from the basepoint $\iota$ to $\varphi_{\mathrm{IHX}}$ which gives an isotopy along the spanning disk $W_{\mathrm{IHX}}\to X$.
\end{Def}

\begin{proof}[Proof of Theorem~\ref{thm:phi-jacobi}]
The decomposition $\Sigma_{\mathrm{IHX}}=\partial N_1\# \partial N_2\# \partial N_3$ gives rise to a decomposition $a_1\# a_2\# a_3$ of $\varphi_{\mathrm{IHX}}$ for the Jacobi relation. The extension of the embedding in Lemma~\ref{lem:B-dP} to the disk $W_{\mathrm{IHX}}$ gives the desired null-isotopy of $\varphi_{\mathrm{IHX}}$. The Brunnian property is proved in Lemma~\ref{lem:brunnian-3} below.
\end{proof}

The construction of $W_{\mathrm{IHX}}$ above works for $D^p\times D^q\times D^r$ ($p,q,r\geq 2$), and gives a path of embeddings in $\Emb_\partial(D^{p+q+r-3},I^N-(I^{N-p-1}\cup I^{N-q-1}\cup I^{N-r-1}))$. As a byproduct, we obtain an elementary geometric proof of the following. (See also Lemma~\ref{lem:emb-conti-Wh}.)
\begin{Cor}[Jacobi identity for Whitehead products {(\cite{Hil,NT,UM} etc.)}]\label{cor:jacobi-wh}
Let $M$ be a pointed topological space. If $p,q,r\geq 2$ and $\alpha\in\pi_p(M)$, $\beta\in\pi_q(M)$, $\gamma\in\pi_r(M)$, then the following identity holds in $\pi_{p+q+r-3}(M)$.
\[ [[\alpha,\beta],\gamma]+(-1)^{pq+pr}[[\beta,\gamma],\alpha]+(-1)^{pr+qr}[[\gamma,\alpha],\beta]=0 \]
\end{Cor}
The detail about the sign is discussed in Lemma~\ref{lem:sign-jacobi}. It should be mentioned that there are also geometric proofs of the Jacobi identity for Whitehead products in different settings: \cite{GO,Hab,HW,CST} for 3-manifolds, \cite{CST} for Whitney towers in 4-manifolds.

\subsection{Brunnian property}\label{ss:brunnian-3}

We check the Brunnian property of the path $[f,g,h]$ in
\begin{equation}\label{eq:brunn}\Emb_\partial(D^{3n-2},I^{4n-1}-(I^{3n-2})^{\cup 3}).\end{equation} Namely, we extend the Brunnian property for $\Emb_\partial((I^{3n-2})^{\cup 4},I^{4n-1})_\iota$ to that of (\ref{eq:brunn})
by considering $D^{3n-2}$ as the fourth component, and by restricting the system of Brunnian null-isotopies to those fixing the first three components. Thus, we only need to check the relative Brunnian property of null-isotopies as in Definition~\ref{def:rel-brunnian} for the cubical diagram $\{\calB(S)\}_{S\subset \Omega'=\{1,2,3\}}$, where $\calB(S)$ is the homotopy fiber of $\calE(\emptyset)\to \calE(S)=\Emb_\partial(D^{3n-2},I^{4n-1}-\bigcup_{\lambda\notin S}I^{3n-2}_\lambda)_\iota$. Such a system of Brunnian null-isotopies induces that for the cubical diagram $\{\calB(S)\}_{S\subset\Omega=\{1,2,3,4\}}$ by taking the constant family of the standard inclusion for $S$ with $4\in S$.

\begin{Lem}[Prescribed Brunnian null-isotopies]\label{lem:pres-brunnian-3}
Let $n\geq 2$. The embedding $\varphi_{\mathrm{IHX}}$ has a system of Brunnian null-isotopies induced from those of Lemma~\ref{lem:brun-borromean}.
\end{Lem}
\begin{proof}
We observe the following property of the embedded Whitehead product: Suppose that an embedded Whitehead product in the manifold $X$ is represented by the boundary of the handlebody $U$ obtained by attaching a $p$-handle $\alpha$ and a $q$-handle $\beta$ to a small $(p+q)$-disk in $X$ (Figure~\ref{fig:plumb-disk}, left) and suppose that the isotopy class of $\beta$ in $X$ relative to the attaching sphere is null. Then there is a surgery which turns the $(p+q)$-dimensional handlebody $U$ into another $(p+q)$-manifold $U'$ obtained by removing a $p$-handle $\alpha'$ from and attaching $\beta$ to a $(p+q)$-disk in $I^{4n-1}$, which is not in $X$ since it intersects a component of the link (Figure~\ref{fig:plumb-disk}, left arrow). 
Now $U'$ is diffeomorphic to a $(p+q)$-disk in $I^{4n-1}$, which is not in $X$, and gives an extension of the embedded Whitehead product to an embedding from a disk. There is another disk $U''$ obtained by attaching $\alpha$ to and removing a $q$-handle $\beta'$ from a $(p+1)$-disk in $X$. The disks $U'$ and $U''$ are related by ``pushing'' over a $(p+q+1)$-disk in $I^{4n-1}$ bounded by $U'\cup (-U'')$ as in Remark~\ref{rem:pushing} (Figure~\ref{fig:plumb-disk}, right arrow).
\begin{figure}[h]
\[ \xymatrix{
  \fig{-15mm}{30mm}{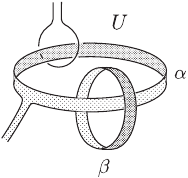} \quad \ar[r]^-{\text{surgery}} &
  \quad \fig{-12mm}{25mm}{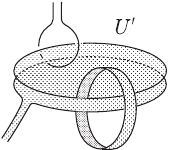} \quad \ar[r]^-{\text{push}} &
  \quad \fig{-12mm}{25mm}{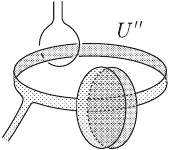} \ar[l]
}\]
\caption{Spanning submanifolds of the embedded Whitehead product.}\label{fig:plumb-disk}
\end{figure}

We use this trick to construct a prescribed system of Brunnian null-isotopies. Recall that each term in $\varphi_{\mathrm{IHX}}$ is given by an embedding from a subset of the $n$-skeleton of a handle decomposition of $S^n\times S^n\times S^n$. Suppose one of the $S^n$, say the last one, is null in $X$. We obtain three $(3n-2)$-spheres in a small neighborhood of $S^n\vee S^n\vee *$, each of which extends to a $(3n-1)$-disk (see Figure~\ref{fig:collapse-iterated-Wh}). 
Two of the extensions are induced from the prescribed Brunnian null-isotopies with respect to removing the last component, and the other one is obtained from the prescribed Brunnian null-isotopy by the deformation $U'\to U''$ of the previous paragraph for $\beta$ being the last $S^n$ factor of $S^n\vee S^n\vee S^n$. The boundary connected-sum of the three disks gives a null-isotopy path in $\calE(\{3\})$. The case of removing the first and second components are similar to this case. The homotopies in $\calE(\{1,2\})$, $\calE(\{2,3\})$, and $\calE(\{1,3\})$ are given by the deformations of the spanning disks as in Figure~\ref{fig:plumb-disk}, right.
\begin{figure}[h]
\[ \includegraphics[height=18mm]{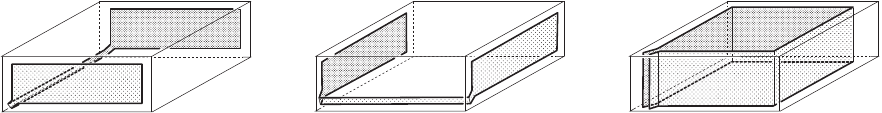} \]
\caption{The null-isotopies of $[[\gamma,\alpha],\beta]$, $[[\beta,\gamma],\alpha]$, and $[[\alpha,\beta],\gamma]$. The first two are given by $N_3$ and $N_2$, respectively.}\label{fig:collapse-iterated-Wh}
\end{figure}
\end{proof}

\begin{Lem}[Relative Brunnian property]\label{lem:brunnian-3}
Let $n\geq 2$. The null-isotopy $[f,g,h]$ has a relative Brunnian property relative to the prescribed system of Brunnian null-isotopies of $\varphi_{\mathrm{IHX}}$ of Lemma~\ref{lem:pres-brunnian-3}.
\end{Lem}
\begin{proof}
Suppose one of the first three components, say the last one, of the family $[f,g,h]$ of string links is removed.
The union of the three $(3n-1)$-disks of the prescribed Brunnian null-isotopies of $\varphi_{\mathrm{IHX}}$ of Lemma~\ref{lem:pres-brunnian-3} with respect to removing the last component and the null-isotopy $[f,g,h]$ of $\varphi_{\mathrm{IHX}}$ gives the following two null-isotopies of $\varphi_{\mathrm{IHX}}$:
\begin{enumerate}
\item The null-isotopy given by $W_{\mathrm{IHX}}$.
\item The null-isotopy given by the boundary connected sum of the extensions to the $(3n-1)$-disks in the previous paragraph.
\end{enumerate}
These null-isotopies are arranged in a small neighborhood of $\partial(D^n\times D^n\times D^n)$. The two null-isotopies, taken along two $(3n-1)$-disk glued along their boundaries, can be extended inside the $3n$-disk (Figure~\ref{fig:two-null-isotopies}). Namely, the boundary of the $3n$-disk is obtained by filling two of the three disks $N_1,N_2,N_3$ in Lemma~\ref{lem:null-iso} into $W_{\mathrm{IHX}}$, but the other one pushed inside the box by using the trick of the first paragarph in the proof of Lemma~\ref{lem:pres-brunnian-3}. (For the example considered in the proof of Lemma~\ref{lem:pres-brunnian-3}, the two disks $N_2$ and $N_3$ (Figure~\ref{fig:collapse-iterated-Wh}, left two items) are filled, and $N_1$ is pushed inside the box.) The $3n$-disk gives a homotopy between $[f,g,h]$ and the prescribed Brunnian null-isotopy of $\varphi_{\mathrm{IHX}}$ in $\calB(\{3\})$. 

By doing similarly, we see that a homotopy between $[f,g,h]$ and the prescribed Brunnian null-isotopy of $\varphi_{\mathrm{IHX}}$ in $\calB(S)$ for $|S|\geq 1$ is given by an embedded $3n$-disk in $D^n\times D^n\times D^n$. Since the space of smooth $3n$-disks in $D^n\times D^n\times D^n$ is contractible, we can find a system of homotopies between $[f,g,h]$ and $\gamma_S$ in $\calB(S)$ for each $S\subset \{1,2,3\}$. This completes the proof.
\end{proof}
\begin{figure}[h]
\[ \includegraphics[height=35mm]{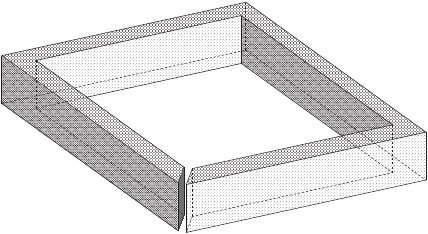} \]
\caption{The union of two null-isotopies, which bounds a $3n$-disk. The outer boundary is the union of $W_{\mathrm{IHX}}$ and the prescribed null-isotopies of $[[\gamma,\alpha],\beta]$ and $[[\beta,\gamma],\alpha]$. The inner boundary is the null-isotopy of $[[\alpha,\beta],\gamma]$ given in (2), which is obtained from the prescribed one by pushing inside the cube.}\label{fig:two-null-isotopies}
\end{figure}

\begin{Rem}
The above proof of Lemma~\ref{lem:brunnian-3} is an analogue of the following homotopy version. Let $f,g,h\colon S^n\to X$ be pointed continuous maps. Recall that if $f$ is null-homotopic, then one has the following homotopy commutative diagram:
\[\xymatrix{
  S^{2n-1} \ar@/_{40pt}/[dd]_{[f,g]} \ar[d] \ar[r]^-{\mathrm{incl.}} & D^{2n} \ar[d] \\
  S^n\vee S^n \ar[d]_-{f\vee g} \ar[r] & {*}\vee S^n \ar[ld]^-{{*}\vee g} \\
  X & 
}\]
In this case, we have the following homotopy commutative diagram:
\[ \xymatrix{
  {*}\vee S^n\vee S^n  \ar@/_{60pt}/[dd]_-{{*}\vee g\vee h} & D^{2n}\vee S^n \ar[l] & D^{3n-1} \ar[l] \\
S^n\vee S^n\vee S^n \ar[u] \ar[d]_-{f\vee g\vee h} & S^{2n-1}\vee S^n \ar[l] \ar[ld]^-{[f,g]\vee h} \ar[u] & S^{3n-2} \ar[l] \ar[u] \ar@/^{20pt}/[lld]^{[[f,g],h]}\\
X & & 
} \]
This shows that $[[f,g],h]\simeq 0$ if $f$ is null-homotopic.
\end{Rem}

\subsection{Symmetry property.}

We fix distinct points $p_0,p_1,\ldots,p_{r-1}$ of $\mathrm{Int}\,I^2$, and consider that the standard pair $(I^N,I^a\cup\cdots\cup I^a)$ for string links of type $(a,\ldots,a;N)$ is given by the direct product
$(I^2\times I^{N-2}, \{p_0,p_1,\ldots,p_{r-1}\}\times I^a\times \{0\})$.
We take a relative diffeomorphism $\rho_0\colon (I^2,\partial I^2)\to (I^2,\partial I^2)$ that induces a cyclic permutation of $\{p_0,p_1,\ldots,p_{r-1}\}$, and extend it to the relative diffeomorphism 
\[ \rho:=\rho_0\times\mathrm{id}_{I^{N-2}}\colon(I^2\times I^{N-2},(\partial I^2)\times I^{N-2})\to (I^2\times I^{N-2},(\partial I^2)\times I^{N-2}). \]
Let $\xi\colon \Emb_\partial(I^a\cup\cdots\cup I^a,I^N)\to \Emb_\partial(I^a\cup\cdots\cup I^a,I^N)$ be defined for $f\in \Emb_\partial(I^a\cup\cdots\cup I^a,I^N)$ by $\xi(f)=(-1)^{(r-1)(N-a-1)}\rho\circ f\circ (\rho|_{I^a\cup\cdots\cup I^a})^{-1}$, where $(-1)\colon \Emb_\partial(I^a\cup\cdots\cup I^a,I^N)\to \Emb_\partial(I^a\cup\cdots\cup I^a,I^N)$ reverses the first coordinate of $I^{N-2}$. The reason for the sign $(-1)^{(r-1)(N-a-1)}$ is that the leaf forms for the symmetrized basic bracket satisfies
\[ \theta_r\wedge\theta_1\wedge\cdots\wedge\theta_{r-1}=(-1)^{(r-1)(N-a-1)}\theta_1\wedge\theta_2\wedge\ldots\wedge\theta_r. \]
(See Definition~\ref{def:ori}. We should replace $k-1$ and $k$ with $N-a-1$.)
\begin{Def}[Cyclic symmetry]\label{def:cyclic}
We say that a string link $f\colon I^{a}\cup\cdots\cup I^{a}\to I^{N}$ has a {\it cyclic symmetry} if $f$ is relatively isotopic to $\xi(f)$. We may also define a cyclic symmetry property of a path $\varphi_t$ in $\Emb_\partial(I^a\cup\cdots\cup I^a,I^N)$ from the stantard inclusion by requiring that for any isotopy from $\varphi_1$ to $\xi(\varphi_1)$, there are isotopies from $\varphi_t$ to $\xi(\varphi_t)$ for all $t$ that extends that of $t=0,1$. 
\end{Def}

The null-isotopy $[f,g,h]$ in $\Emb_\partial(D^{3n-2},I^{4n-1}-(I^{3n-2})^{\cup 3})$ can be modified into a path in
\[ \Emb_\partial((I^{3n-2})^{\cup 4},I^{4n-1}) \]
by attaching the trivial family of a punctured $(3n-2)$-plane in $I^{4n-1}$ to the embeddings of $D^{3n-2}$. 
We assume that the nonstandard component in this family of embeddings of $(I^{3n-2})^{\cup 4}$ is labelled by $0$, and the remaining three standard components are labelled by $1,2,3$, respectively. Taking this into account, we denote the above path of embeddings by $\omega_0^{(4)}$. 

Let $\omega_1^{(4)},\omega_2^{(4)},\omega_3^{(4)}$ denote the paths of embeddings in $\Emb_\partial((I^{3n-2})^{\cup 4},I^{4n-1})$ given by $\xi(\omega_0^{(4)})$, $\xi^2(\omega_0^{(4)})$, $\xi^3(\omega_0^{(4)})$, respectively.
Let
\[ \omega_{T_4}=\omega_0^{(4)}\#\omega_1^{(4)}\#\omega_2^{(4)}\#\omega_3^{(4)}, \]
where the sum $\#$ is formed by component-wise concatenation along $I^{3n-2}$. The following lemma is evident from the definition.

\begin{Lem}[Cyclic symmetry]\label{lem:cyclic-4}
The path $\omega_{T_4}$ has a cyclic symmetry property.
\end{Lem}

\begin{Def}[Triple bracket as a chain]\label{def:triple-chain}
Let $m_4=4$. We denote by $\beta_{T_4}=\,\beta_{T_4}(T_4)$ the chain $\frac{1}{m_4}\omega_{T_4}$ in $\Emb_\partial((I^{3n-2})^{\cup 4},I^{4n-1})$. For a face graph $\sigma$ of $T_4$, let $\beta_{T_4}(\sigma)$ denote the corresponding summand of $\partial\beta_{T_4}$ given by the iterated bracket for $\sigma$. 
\end{Def}

\begin{proof}[Proof of Theorem~\ref{thm:vertex} for $\ell=4$]
Each term $\omega_i^{(4)}$ of $\omega_{T_4}$ satisfies the Jacobi relation (Theorem~\ref{thm:phi-jacobi}), the Brunian property (Lemma~\ref{lem:brunnian-3}), and $\beta_{T_4}(T_4)$ satisfies the cyclic symmetry (Lemma~\ref{lem:cyclic-4}). The property (Boundary) is clear from the definition. We need only to check that the symmetrization does not break the identity of Theorem~\ref{thm:phi-jacobi} and the decomposition structure of its RHS. The boundary of each term $\omega_i^{(4)}$ represents the Jacobi relation and thus has a decomposition into three terms (See Figure~\ref{fig:Y2-cyclic}):
\[ \begin{array}{ll}
\partial\omega_0^{(4)}=\theta_1\# \theta_2\# \theta_3, &
\partial\omega_1^{(4)}=\xi(\theta_1)\# \xi(\theta_2)\# \xi(\theta_3),\\
\partial\omega_2^{(4)}=\xi^2(\theta_1)\# \xi^2(\theta_2)\# \xi^2(\theta_3), &
\partial\omega_3^{(4)}=\xi^3(\theta_1)\# \xi^3(\theta_2)\# \xi^3(\theta_3).
\end{array} \]
\begin{figure}[h]
\[ \includegraphics[height=30mm]{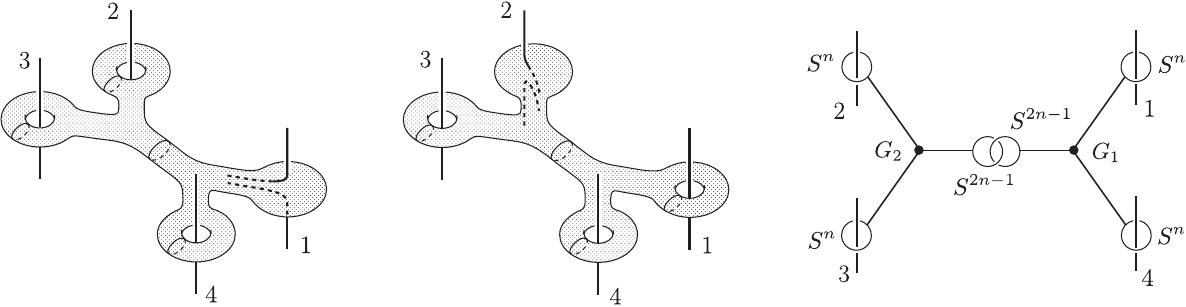} \]
\caption{Left: $\theta_1$. Middle: $\xi(\theta_2)$. Right: $\Psi$-graphs $G_1\cup G_2$ for $\theta_1$. One of the four components crawls inside a handlebody (shadowed), which is identified with the complement of a tubular neighborhood of $(I^{3n-2})^{\cup 3}$ in $I^{4n-1}$.}\label{fig:Y2-cyclic}
\end{figure}

If $\theta_1,\theta_2,\theta_3$ correspond to the three terms in the LHS of Figure~\ref{fig:ihx}, then by the cyclic symmetry of the Borromean string link (Lemma~\ref{lem:brun-borromean}), we have the following relative isotopy relations:
$\theta_1\sim \xi(\theta_2)\sim \xi^2(\theta_1)\sim \xi^3(\theta_2)$, $\theta_2\sim \xi(\theta_1)\sim \xi^2(\theta_2)\sim \xi^3(\theta_1)$, $\theta_3\sim \xi(\theta_3)\sim \xi^2(\theta_3)\sim \xi^3(\theta_3)$. More precisely, 
\begin{itemize}
\item the string link $\theta_1$ is obtained by the composition of two $\Psi_3$-graphs $G_1$ and $G_2$ of type $(n,n,2n-1)$ (Definitions~\ref{def:psi-graph} and \ref{def:iter-surg}), where the leaves of $G_1$ are linked to the first and fourth $I^{3n-2}$ components and to a leaf of $G_2$, and two leaves of $G_2$ are linked to the second and third $I^{3n-2}$ components (see Figure~\ref{fig:Y2-cyclic}, right). 
\item The string link $\xi(\theta_2)$ is obtained by the composition of two $\Psi_3$-graphs $G_1'$ and $G_2'$, where the leaves of $G_1'$ are linked to the second and third $I^{3n-2}$ components and to a leaf of $G_2'$, and two leaves of $G_2'$ are linked to the first and fourth $I^{3n-2}$ components. 
\end{itemize}
The iterated surgery on $G_1\cup G_2$ (Definition~\ref{def:iter-surg}), which gives $\theta_1$, replaces the first component with an embedded Whitehead product in a small regular neighborhood of $G_1$, and then modifies the part parallel to the distinguished $S^{2n-1}$-leaf of $G_1$ linked to a leaf of $G_2$ by surgery on $G_2$. The symmetry property for the $\Psi_3$-surgery shows that the surgeries on $G_1$ and $G_2$ can be deformed into those on $G_2'$ and $G_1'$, respectively. Thus we have $\theta_1\sim \xi(\theta_2)$. The relative isotopy equivalences for other terms are similar to this case. It follows that $\partial\omega_{T_4}$ is relatively isotopic to the connected sum of four copies of $\varphi_\mathrm{IHX}$.
This completes the proof.
\end{proof}

Since we also have cyclic symmetry of the system of Brunnian null-isotopies of the Borromean string link by Lemma~\ref{lem:brun-borromean}, the argument in the proof of Theorem~\ref{thm:vertex} for $\ell=4$ also proves the following.
\begin{Lem}\label{lem:cyclic-brun-4}
The system of Brunnian null-isotopies of $\omega_{T_4}$ from Lemma~\ref{lem:brunnian-3} has a cyclic symmetry property. 
\end{Lem}

\section{Operations for vertex surgery}\label{s:operations}

We shall introduce some operations for vertex surgery to obtain surgeries for more general set of dimensions of links and for more general trees. Since the definitions are technical, the reader may skip this section for the first reading and return when necessary, keeping in mind that there is a natural procedure to turn a basic bracket into the brackets for general dimensions, described in Theorem~\ref{thm:l-valent-general}, and that there is a natural iteration of vertex surgeries for graphs, such as in Figure~\ref{fig:U-4-valent}, also equipped with an induced system of Brunnian null-isotopies.

\begin{Def}[$\Psi_{r+1}$-graph]\label{def:psi-graph}
We define the {\it $\Psi_{r+1}$-graph} of type $(b_1,\ldots,b_{r+1})$ to be the space obtained from the disjoint union of 
\begin{itemize}
\item the spheres $S^{b_1},\ldots, S^{b_{r+1}}$, which we call {\it leaves}, and 
\item one 0-cell $e^0$, which we call the {\it vertex},
\end{itemize}
by attaching a 1-cell $e_i^1$, which we call an {\it edge}, between $S^{a_i}$ and $e^0$ for each $i$:
\[ (S^{b_1}\tcoprod \cdots\tcoprod S^{b_{r+1}})\cup e^0\cup (e_1^1\cup\cdots\cup e_{r+1}^1).\]
A {\it $\Psi$-graph} is a $\Psi_{r+1}$-graph for some $r$ (See Figure~\ref{fig:psi-graph}). 
\end{Def}
\begin{figure}[h]
\includegraphics[height=35mm]{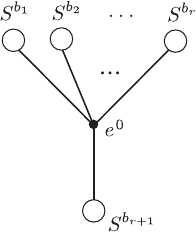}
\caption{$\Psi_{r+1}$-graph of type $(b_1,\ldots,b_{r+1})$.}\label{fig:psi-graph}
\end{figure}

We will define in Definition~\ref{def:Psi-surgery} surgery along a $\Psi_{r+1}$-graph in a $d$-manifold by taking a family in $\Emb_\partial(I^{a_1}\cup \cdots \cup I^{a_{r+1}}, I^d)$, where $a_i=d-b_i-1$. For example, the embedded Whitehead product gives an embedding in $\Emb_\partial(I^{2n-1}\cup I^{2n-1}\cup I^{2n-1}, I^{3n})$, and surgery along a $\Psi_3$-graph, in which the vertex $e^0$ is 3-valent, of type $(n,n,n)$ can be defined by using it. Also, we have constructed a path in $\Emb_\partial(I^{a_1}\cup \cdots \cup I^{a_4}, I^d)$, where $a_1=a_2=a_3=a_4=3n-2$ and $d=4n-1\geq 7$, and surgery along a $\Psi_4$-graph of type $(n,n,n,n$) could be defined by using such a path. Similarly, we will construct a path in the space $\Omega^{\ell-4}\Emb_\partial((I^{(\ell-1)n-(\ell-2)})^{\cup \ell},I^{\ell n-\ell+3})$ to define surgery along a $\Psi_{r+1}$-graph of type $(n,n,\ldots,n)$ in \S\ref{s:5-valent} and \S\ref{s:6-valent}. 
Here we define below the operations, ``suspension'', ``delooping'', and ``iteration'', to generalize $\Psi_{r+1}$-surgery to more general types $(b_1,\ldots,b_{r+1})$ and trees. 

\subsection{Suspension}\label{ss:suspension}

\begin{Def}[Presuspension]\label{def:presuspension}
We assume that the first $j$ components in the string link $f\colon I^{a_1}\cup \cdots \cup I^{a_r}\to I^d$ are standard inclusions. The {\it presuspension} of $f$ with respect to the first $j$ components is another embedding
\[ \Sigma'_{1,2,\ldots,j} f\colon (I^{a_1}\cup \cdots \cup I^{a_{j}})\times I\cup I^{a_{j+1}}\cup \cdots\cup I^{a_r}\to I^d\times I=I^{d+1} \]
defined by 
\[ \begin{split}
&(\Sigma_{1,2,\ldots,j}' f)(x_p,t)=(f(x_p),t)\quad \text{for $p=1,\ldots,j$},\\
&(\Sigma_{1,2,\ldots,j}' f)(x_p)=(f(x_p),\textstyle\frac{1}{2})\quad \text{for $p=j+1,\ldots,r$}.
\end{split} \]
The labels of the standard components may be exchanged, and $\Sigma'_{i_1,\ldots,i_j}f$ may be defined similarly for any $\{i_1,\ldots,i_j\}\subset\{1,\ldots,r\}$ under the assumption that the components labelled by $\{i_1,\ldots,i_j\}$ are standard inclusions.
\end{Def}
The change of the type under a presuspension $\Sigma'_{1,2,\ldots,j}$ is
\[ (a_1,\ldots,a_{j},a_{j+1},\ldots,a_r;d)
\to (a_1+1,\ldots,a_{j}+1,a_{j+1},\ldots,a_r;d+1). \]

\begin{Rem}\label{rem:suspension-generalized}
The assumption that the first $j$ components are standard inclusions in Definition~\ref{def:presuspension} can be weakened. It can be replaced by the condition that it is given a Brunnian null-isotopy $\gamma$ of $f$ with respect to removing the last $r-j$ components. In this case, the presuspension can be defined by first putting $f$ in $I^d\times\{\frac{1}{2}\}$, and then extending it to the two sides $I^d\times [\frac{1}{2},1]$ and $I^d\times[0,\frac{1}{2}]$ by the tracks of the isotopies $\gamma$ (of the first $j$ components) and its reverse. Since the last $r-j$ components are included only in $I^d\times\{\frac{1}{2}\}$, the result is another string link. If we have a system of Brunnian null-isotopies for a string link, all possible presuspensions can be obtained simultaneously in a canonical manner as prescribed in the system of Brunnian null-isotopies.
\end{Rem}

We modify the definition of presuspension to define suspension. We want that all the components in the suspension of a string link have the last coordinate so that it can be perturbed into the graph of a 1-parameter family of embeddings. The following definition is an analogue of \cite[Definition~5.2]{Wa21}.
\begin{Def}[Suspension of a string link]\label{def:suspension}
Let $L=L_1\cup \cdots\cup L_r\colon I^{a_1}\cup \cdots\cup I^{a_r}\to I^d$ (where $0<a_i<d$, $i=1,\dots r$) be a string link in $\fEmb_\partial(I^{a_1}\cup\cdots\cup I^{a_r},I^d)$ equipped with a framed isotopy $H_{1,t}\cup \cdots \cup H_{j,t}\colon I^{a_1}\cup\cdots\cup I^{a_j}\to I^d$ ($t\in [0,1]$) of the first $j$ components fixing a neighborhood of the boundary $\partial I^{a_1}\cup\cdots\cup \partial I^{a_j}$, such that $H_{1,0}\cup \cdots \cup H_{j,0}$ is the standard inclusions of the first $j$ components and $H_{1,1}\cup\cdots\cup H_{j,1}=L_1\cup\cdots\cup L_j$. Suppose that for each $k\geq j+1$, $L_k$ agrees with the standard inclusion $I^{a_k}\to I^d$ outside a $d$-ball about the barycenter $a=(\frac{1}{2},\ldots,\frac{1}{2})\in I^{a_k}\subset I^d$ with small radius $R\ll\frac{1}{2}$. Then the {\it suspension} $L'=L_1'\cup\cdots\cup L_r'\colon I^{a_1+1}\cup\cdots\cup I^{a_j+1}\cup I^{a_{j+1}}\cup\cdots\cup I^{a_r}\to I^{d+1}$ of $L$ is defined by
\[ \begin{split}
  &L_k'(u_k,w)=(H_{k,\chi(w)}(u_k),w),\quad (\text{for $1\leq k\leq j$}),\\
  &L_k'(u_k)=\left\{\begin{array}{ll}
  (L_k(u_k),\frac{1}{2}) & (|u_k-a|\leq R),\\
  (p_k,\mu_d^{-1}\circ\rho_{a_k}\circ\mu_{a_k}(u_k)) & (|u_k-a|\geq R),
  \end{array}\right.\quad (\text{for $j+1\leq k\leq r$}),
\end{split} \]
where $u_k\in I^{a_k}$, $w\in I$, $\chi\colon I\to [0,1]$ is a smooth function supported on a small neighborhood of $\frac{1}{2}$ such that $\chi(\frac{1}{2})=1$, $\mu_n\colon [0,1]^n\to [-1,1]^d$ is the embedding defined by $\mu_n(t_1,\ldots,t_n)=(2t_1-1,\ldots,2t_n-1,0,\ldots,0)$, and $\rho_m\colon [-1,1]^d\to [-1,1]^d$ is the diffeomorphism defined by 
\begin{equation}\label{eq:rho_3}
\begin{split}
&\rho_r(x_1,\ldots,x_d)
=(x_1,\ldots,x_{m-1},x_m',x_{m+1},\ldots,x_{d-1},x_d'),\mbox{ where}\\
&x_m'=x_m\cos\psi(|\bvec{x}|)-x_d\sin\psi(|\bvec{x}|),\quad
x_d'=x_m\sin\psi(|\bvec{x}|)+x_d\cos\psi(|\bvec{x}|)
\end{split}
\end{equation}
(Here $|\bvec{x}|=\sqrt{x_1^2+\cdots+x_d^2}$), and $\psi\colon [0,\sqrt{2}]\to [0,\frac{\pi}{2}]$ is  a smooth function with $\frac{d}{dt}\psi(t)\geq 0$, which takes the value 0 on $[0,2R]$ and the value $\frac{\pi}{2}$ on $[R',\sqrt{2}]$ for some $R'$ with $2R<R'<\frac{\sqrt{2}}{2}$. The diffeomorphism $\rho_m$ rotates the sphere $S^{d-1}_{|\tbvec{x}|}$ of radius $|\bvec{x}|$ by angle $\psi(|\bvec{x}|)$ along the $x_mx_d$-plane. The rotation $\rho_m|_{S^{d-1}_{R'}}$  exchanges the $x_m$-axis and the $x_d$-axis.) The resulting embedding $L'$ has a canonical normal framing induced from the original one since the embedding $\rho_m\circ \mu_m$ can be extended to the diffeomorphism $\rho_m$. 
By permuting the coordinates so that the components agree with $\st$ near $\partial I^{d+1}$, $L'$ with the induced framing can be considered giving an element $\Sigma_{1,2,\ldots,j}L$ of $\fEmb_\partial(I^{a_1+1}\cup\cdots\cup I^{a_j+1}\cup I^{a_{j+1}}\cup I^{a_r},I^{d+1})$ (see Figure~\ref{fig:suspension}, left). Suspensions for other choices of components are defined similarly by symmetry. 
\end{Def}
\begin{Rem}
\begin{enumerate}
\item The rotation $\rho_m$ is needed since the components $L_1',\ldots,L_j'$ have the coordinate $w$, which will correspond to the parameter for the delooping, and we would also like to let $L_{j+1}',\ldots,L_r'$ have the coordinate $w$ near the boundary, too. The permutation of the coordinates can be given by moving the $d$-th factor before the first factor. 
\item Remark~\ref{rem:suspension-generalized} also applies to suspension of string link.
\end{enumerate}
\end{Rem}
\begin{figure}[h]
\[\includegraphics[height=35mm]{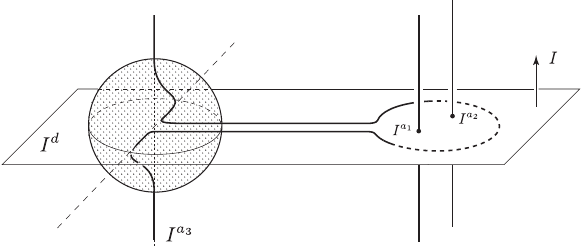}\qquad\qquad 
\includegraphics[height=40mm]{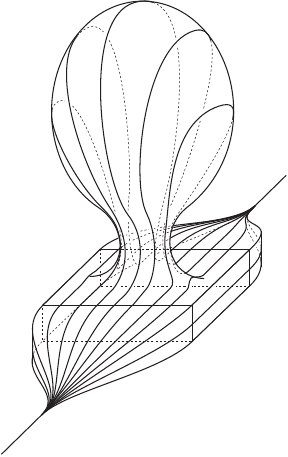}\]
\caption{Suspension with respect to the first two components, and delooping.}\label{fig:suspension}
\end{figure}

\subsection{Delooping}\label{ss:delooping}

We may associate to an embedding $I^{a}\to I^d$ a 1-parameter family of embeddings $I^{a-1}\to I^{d}$ by considering $I^a$ as a family over $I$ of $I^{a-1}$. By modifying (i.e. ``closing'') the family near the boundary $\partial (I^{a-1}\times I)=\partial I^{a-1}\times I\cup I^{a-1}\times \partial I$ as in \cite[\S{5.3}]{Wa21}, we obtain a loop of embeddings $I^{a-1}\to I^{d}$. We call this process a {\it delooping} of an embedding $I^{a}\to I^d$ (see Figure~\ref{fig:suspension}, right).

\begin{Def}[Delooping]\label{def:delooping}
We assume that all but the $j$-th component of a string link $f\colon I^{a_1}\cup \cdots \cup I^{a_r}\to I^d$ are standard inclusions. 
A {\it delooping} of $f$ with respect to the $j$-th component is a loop of embeddings
\[ I^{a_1}\cup \cdots \cup I^{a_j-1}\cup\cdots\cup I^{a_r}\to I^{d}, \]
obtained by applying the above delooping for the $j$-th component. This defines a map 
\[ \Emb_\partial(I^{a_1}\cup\cdots\cup I^{a_r},I^d)\to \Omega\Emb_\partial(I^{a_1}\cup\cdots\cup I^{a_j-1}\cup\cdots\cup I^{a_r},I^d). \]
Iterating the deloopings $p_j$ times for the $j$-th component, we will get a map
\[ \Emb_\partial(I^{a_1}\cup\cdots\cup I^{a_r},I^d)\to \mathrm{Map}_*(S^{p_1}\times\cdots\times S^{p_r},\Emb_\partial(I^{a_1-p_1}\cup\cdots\cup I^{a_r-p_r},I^d)). \]
\end{Def}
\begin{Rem}\label{rem:delooping}
\begin{enumerate}
\item The assumption that all but the $j$-th component of $f$ are standard inclusions in Definition~\ref{def:delooping} can be weakened. It can be replaced by the condition that it is given a Brunnian null-isotopy $\gamma$ of $f$ with respect to removing the $j$-th component. 

\item The iterated delooping does not depend on the choice of the order of components since the closing of a component to deloop is disjoint from other components, and the order of two delooping can be exchanged without changing the result.

\item  A similar construction to turn a disk into a family of embeddings of disks was recently used in \cite{Kos24a,Kos24b,BG}.
\end{enumerate}
\end{Rem}
We symbolize this transformation as
\begin{equation}\label{eq:symbolize-suspension}
 (a_1,\ldots,a_r;d)
\to \Omega^{\tvec{p}}(a_1-p_1,\ldots,a_r-p_r;d), 
\end{equation}
where $\vect{p}=(p_1,\ldots,p_r)$.

\begin{Lem}\label{lem:suspension-delooping}
Let $f\colon I^{a_1}\cup\cdots\cup I^{a_r}\to I^d$ be a string link such that $a_j\geq 1$ for all $j$, equipped with a null-isotopy $H$ of the restriction of $f$ to the first $r-1$ components $I^{a_1}\cup\cdots\cup I^{a_{r-1}}$. The following string links of type $(a_1+1,\ldots,a_{r-1}+1,a_r;d+1)$ are relatively isotopic.
\begin{enumerate}
\item[(a)] The suspension $\Sigma_{1,\ldots,r-1}f$ with respect to $H$. 
\item[(b)] The graph of the delooping 
\[ (a_1,\ldots,a_{r-1},a_r;d)\to \Omega^{\tvec{a}}(a_1,\ldots,a_{r-1},a_r-1;d)\]
of $f$ for $\vect{a}=(0,\ldots,0,1)$. 
\end{enumerate} 
\end{Lem}
Lemma~\ref{lem:suspension-delooping} can be proved in the same way as \cite[Proof of Lemma~5.3 (2)]{Wa21}.

\begin{Exa}\label{ex:3-valent-family}
Let $d=2k$, and we consider families of string links associated to the $Y$-graph (or $\Psi_3$-graph) in a $2k$-disk, as depicted in Figure~\ref{fig:3-valent-suspend}, left.
\begin{figure}[h]
\[ \includegraphics[height=35mm]{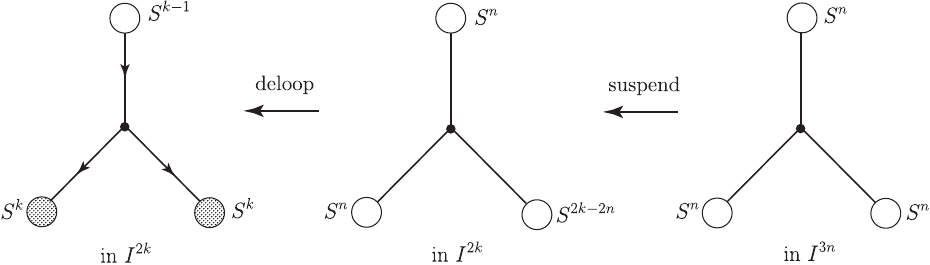} \]
\caption{}\label{fig:3-valent-suspend}
\end{figure}

Let $\Delta=2k-3n$, $a=2n-k-1$, where $n$ is an integer such that $\frac{k+1}{2}\leq n\leq\frac{2k}{3}$ (such an $n$ exists when $k=3$ or $k\geq 5$). We iterate suspensions and deloopings as follows (Figure~\ref{fig:3-valent-suspend}): 
\[ \begin{split}
&(2n-1,2n-1,2n-1;3n)\to (2n-1+\Delta,2n-1+\Delta,2n-1;3n+\Delta)\\
&\to \Omega^{\tvec{a}}(2n-1+\Delta-(a+\Delta),2n-1+\Delta-(a+\Delta+1),2n-1-(a+1);3n+\Delta)\\
&= \, \Omega^{\tvec{a}}(2n-1-a,2n-2-a,2n-2-a;3n+\Delta)\\
&= \, \Omega^{\tvec{a}}(k,k-1,k-1;2k),
\end{split} \]
where $\vect{a}=(a+\Delta,a+\Delta+1,a+1)=(k-n-1,k-n,2n-k)$.
We get a family over $S^{\tvec{a}}:=S^{k-n-1}\times S^{k-n}\times S^{2n-k}$ ($\dim S^{\tvec{a}}=k-1$) of framed embeddings 
\[ I^k\cup I^{k-1}\cup I^{k-1}\to I^{2k}. \]

More generally, if a Y-graph has $j$ outgoing edge(s) and $3-j$ incoming edge(s), then we iterate suspensions and deloopings as follows:
\[ \begin{split}
&(2n-1,2n-1,2n-1;3n)\to (2n-1+\Delta,2n-1+\Delta,2n-1;3n+\Delta)\\
&\to \, \left\{\begin{array}{ll}
\Omega^{\tvec{a}}(\underbrace{2n-1-a,\ldots,2n-1-a}_{3-j},\underbrace{2n-2-a,\ldots,2n-2-a}_{j};3n+\Delta)& (j\geq 1),\\
\Omega^{\tvec{a}}(2n-1-a,2n-1-a,2n-1-a;3n+\Delta)& (j=0),
\end{array}\right.\\
&= \, \Omega^{\tvec{a}}(\underbrace{k,\ldots,k}_{3-j},\underbrace{k-1,\ldots,k-1}_{j};2k),\text{ where}
\end{split} \]
\[
\vect{a}=\left\{\begin{array}{ll}
(\underbrace{a+\Delta,\ldots,a+\Delta}_{3-j},\underbrace{a+\Delta+1,\ldots,a+\Delta+1}_{j-1},a+1) & (j\geq 1),\\
(a+\Delta,a+\Delta,a) & (j=0).
\end{array}\right. \]
We get a family over $S^{\tvec{a}}$ ($\dim S^{\tvec{a}}=(k-3)+j$) of framed embeddings 
\[ (I^k)^{\cup 3-j}\cup (I^{k-1})^{\cup j}\to I^{2k}. \]
\end{Exa}

\begin{Exa}\label{ex:4-valent-family}
Let $d=2k$, and we consider a family of string links associated to the $\Psi_4$-graph in a $2k$-disk, as depicted in Figure~\ref{fig:4-valent-suspend}, left.
\begin{figure}[H]
\[ \includegraphics[height=40mm]{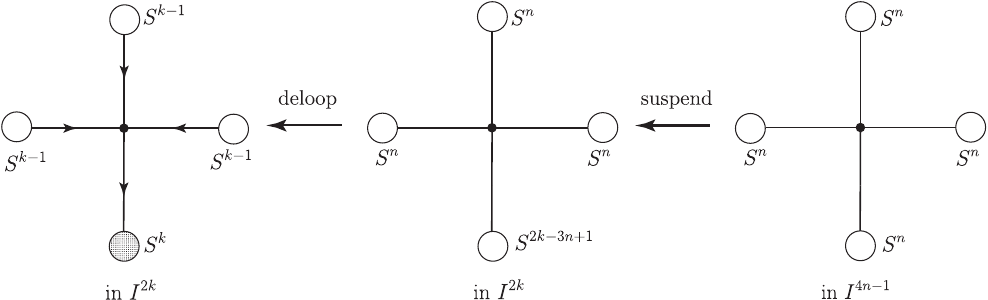} \]
\caption{}\label{fig:4-valent-suspend}
\end{figure}

We start with the null-isotopy path of the IHX link $\varphi_{\mathrm{IHX}}$ in $\Emb_\partial^\fr((I^{3n-2})^{\cup 4},I^{4n-1})$ we have discussed in \S\ref{s:4-valent}. Let $\Delta=2k-4n+1$, $a=3n-k-2$, where $n$ is an integer such that $\frac{k+2}{3}\leq n\leq\frac{2k+1}{4}$ (such an $n$ exists when $k=4$ or $k\geq 6$). 
We iterate suspensions and deloopings as follows (Figure~\ref{fig:4-valent-suspend}): 
\[ \begin{split}
&(3n-2,3n-2,3n-2,3n-2;4n-1)\\
&\to\,(3n-2+\Delta,3n-2+\Delta,3n-2+\Delta,3n-2;4n-1+\Delta)\\
&\to\, \Omega^{\tvec{a}}(3n-2-a,3n-2-a,3n-2-a,3n-2-(a+1);4n-1+\Delta)\\
&= \,\Omega^{\tvec{a}}(k,k,k,k-1;2k),
\end{split} \]
where $\vect{a}=(\Delta+a,\Delta+a,\Delta+a,a+1)=(k-n-1,k-n-1,k-n-1,3n-k-1)$.
This sequence of operations can be done simultaneously over the path, and we get a path of families over $S^{\tvec{a}}:=S^{k-n-1}\times S^{k-n-1}\times S^{k-n-1}\times S^{3n-k-1}$ ($\dim S^{\tvec{a}}=2k-4$) of framed embeddings 
\[ I^k\cup I^k\cup I^k\cup I^{k-1}\to I^{2k}. \]

More generally, if an $\Psi_4$-graph has $j$ outgoing edge(s) and $4-j$ incoming edge(s), then we iterate suspensions and deloopings as follows:
\[ \begin{split}
&(3n-2,3n-2,3n-2,3n-2;4n-1)\\
&\to\, (3n-2+\Delta,3n-2+\Delta,3n-2+\Delta,3n-2;4n-1+\Delta)\\
&\to\, \left\{\begin{array}{ll}
\Omega^{\tvec{a}}(\underbrace{3n-2-a,\ldots,3n-2-a}_{4-j},\underbrace{3n-3-a,\ldots,3n-3-a}_{j};4n-1+\Delta)& (j\geq 1),\\
\Omega^{\tvec{a}}(\underbrace{3n-2-a,\ldots,3n-2-a}_{4};4n-1+\Delta)& (j=0),
\end{array}\right.\\
&= \, \Omega^{\tvec{a}}(\underbrace{k,\ldots,k}_{4-j},\underbrace{k-1,\ldots,k-1}_{j};2k),\text{ where}
\end{split} \]
\[ \vect{a}=\left\{\begin{array}{ll}
(\underbrace{a+\Delta,\ldots,a+\Delta}_{4-j},
\underbrace{a+\Delta+1,\ldots,a+\Delta+1}_{j-1},
a+1) & (j\geq 1),\\
(a+\Delta,a+\Delta,a+\Delta,a) & (j=0).
\end{array}\right. \]
We get a path of families over $S^{\tvec{a}}$ ($\dim S^{\tvec{a}}=(2k-5)+j$) of framed embeddings 
\[ (I^k)^{\cup 4-j}\cup (I^{k-1})^{\cup j}\to I^{2k}. \]
\end{Exa}

\subsection{Iterated surgery}\label{ss:iter-surg}

\subsubsection{Surgery on a $\Psi_{r+1}$-graph}\label{ss:surg-psi}

Suppose we are given a pair $(G,\sigma)$ of a $\Psi_{r+1}$-graph $G$ of type $(b_1,\ldots,b_{r+1})$ in a $N$-manifold $X$, and a chain $\sigma\colon B_\sigma\to \Emb_\partial^\fr(I^{a_{r+1}},I^N-(I^{N-b_1-1}\cup\cdots\cup I^{N-b_r-1}))$, where $a_{r+1}+b_{r+1}=N-1$, from a compact manifold $B_\sigma$. 
Let $V$ be a small thickening of $G$ in $X$.  
Let $c_p$ ($p=1,\ldots,r+1$) be the $p$-th leaf, and let $\delta_{r+1}$ be a spanning disk of $c_{r+1}$. 
Let $U$ be a small closed regular neighborhood of $V\cup \delta_{r+1}$ in $X$ that deformation retracts onto $V\cup \delta_{r+1}$ (Figure~\ref{fig:U-4-valent}, left). Suppose further that an $a_{r+1}$-disk $L$ intersects $\mathrm{Int}\,\delta_{r+1}-V$ transversally in one point and that $L_{U}:=(\mathrm{Int}\,L)\cap U$ is a closed $a_{r+1}$-disk. 
\begin{figure}[h]
\[ \includegraphics[height=35mm]{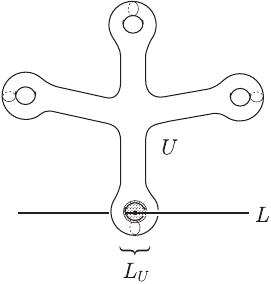} \hspace{20mm}
\includegraphics[height=50mm]{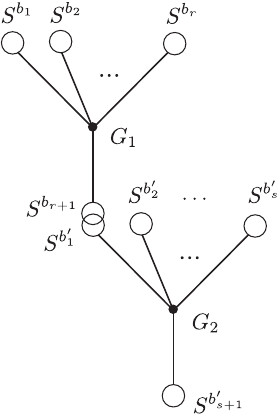}
\]
\caption{Left: The neighborhood $U$ of $V\cup \delta_{r+1}$ in $X$. Right: A pair of spheres $S^{b_{r+1}}$ and $S^{b_1'}$ from $G_1$ and $G_2$, respectively, forms a Hopf link in $X$.}\label{fig:U-4-valent}
\end{figure}

\begin{Def}[$\Psi_{r+1}$-surgery]\label{def:Psi-surgery}
Let $\varphi^{(G,\sigma)}_{U,t}\colon L_{U}\to U$ ($t\in B_\sigma$) be the family of embeddings obtained from the family of embeddings $\sigma\colon B_\sigma\to \Emb_\partial^\fr(I^{a_{r+1}},I^N-(I^{N-b_1-1}\cup\cdots\cup I^{N-b_r-1}))$
by identifying the component $I^{a_{r+1}}$ with $L_U$. 
Let $\varphi^{(G,\sigma)}=\{\varphi^{(G,\sigma)}_t\}\colon B_\sigma\to \Emb_\partial^\fr(L,X)$ be the family of framed embeddings obtained from the inclusion $L\to X$ by replacing the inclusion of $L_{U}$ with $\varphi^{G}_{U,t}$. We say that $\varphi^{(G,\sigma)}$ is obtained from the inclusion $L\to X$ by {\it $\Psi_{r+1}$-surgery on $(G,\sigma)$}.
\end{Def}

\subsubsection{Composition of $\Psi_{r+1}$-graphs}\label{ss:compos-psi}

Let $a_i=N-b_i-1$ ($1\leq i\leq r$), $a_j'=N-b_j'-1$ ($2\leq j\leq s+1$), and 
\[ X=I^N-(I^{a_1}\cup\cdots\cup I^{a_{r}}\cup I^{a_2'}\cup\cdots\cup I^{a_{s+1}'}). \]
We consider the disjoint union $G_1\tcoprod G_2\subset X$ of a $\Psi_{r+1}$-graph $G_1$ of type $(b_1,\ldots,b_{r+1})$ and a $\Psi_{s+1}$-graph $G_2$ of type $(b_1',\ldots,b_{s+1}')$, where 
\begin{itemize}
\item $r+s=\ell$,
\item $b_{r+1}+b_1'=N-1$,
\item the $(r+1)$-th sphere $S^{b_{r+1}}$ in $G_1$ and the first sphere $S^{b_1'}$ in $G_2$ forms a Hopf link in a small ball,
\item $S^{b_i}$ ($i\neq r+1$) links with $I^{a_i}$ by a small meridian, $S^{b_j'}$ ($j\neq 1$) links with $I^{b_j'}$ by a small meridian.
\end{itemize}
(See Figure~\ref{fig:U-4-valent}, right.) 
We equip $G_i$ with a framing, namely, choices of normal framings of the leaves, edges and the vertex. The framings of $G_i$ are used to trivialize (parametrize) the thickenings. Let $V_1\tcoprod V_2\subset X$ be small disjoint thickenings of $G_1\tcoprod G_2$. 

\subsubsection{Composition of surgeries on $G_1$ and $G_2$}

Let $G_1\tcoprod G_2$ be two composable graphs as above. Let $(G_1,\sigma_1)$, $(G_2,\sigma_2)$ be pairs as in \S\ref{ss:compos-psi}. Let $U_1,U_2$ be the $U$ for $G=G_1,G_2$, respectively. Suppose that a disk $L'=I^{a_{s+1}'}$ intersects the spanning disk $\delta_{s+1}'$ of the $(s+1)$-th leaf $c_{s+1}'$ of $G_2$ transversally in one point. We consider a $a_{r+1}$-disk $L$ is a part of the first leaf $S^{a_1'}$ of $G_2$ so that the family $\varphi^{(G_1,\sigma_1)}_t\colon L_{U_1}\to U_1$ is extended to a family of embeddings $U_2\to X$. We also denote this extension by $\varphi^{(G_1,\sigma_1)}_t$, abusing the notation.

\begin{Def}[Iterated surgery]\label{def:iter-surg}
We define the family $\varphi_{(t,u)}^{G_1\circ_1 G_2}\colon L'\to X$ ($t\in B_{\sigma_1}$, $u\in B_{\sigma_2}$) of embeddings by 
\[ \varphi_{(t,u)}^{G_1\circ_1 G_2}:=\left\{
\begin{array}{ll}
  \varphi^{(G_1,\sigma_1)}_t\circ \varphi^{(G_2,\sigma_2)}_{U_2,u} & \text{on $L_{U_2}':=L'\cap U_2$},\\
  \mathrm{incl} & \text{on $L'-U_2$}.
\end{array}\right.\]
The family $\{\varphi_{(t,u)}^{G_1\circ_1 G_2}\}$ defines a chain $B_{\sigma_1}\times B_{\sigma_2}\to\Emb_\partial^\fr(I^{a_{s+1}'},X)$. We denote this chain by 
\[ (L')^{(G_1,\sigma_1)\circ_1(G_2,\sigma_2)}\colon B_{\sigma_1}\times B_{\sigma_2}\to\Emb_\partial^\fr(L',X). \]
When $L'$ can be considered trivial, we write this chain as $\sigma_1\circ\sigma_2$ for short.
\end{Def}

\subsubsection{$\Psi_{r+1}$-surgery via fiber bundles}

By postcomposing the map $c\colon \Emb_\partial^\fr(I^{a_1}\cup\cdots\cup I_{a_{r+1}},I^N)\to B\Diff_\partial(V)$ to the family $\sigma\colon B_\sigma\to \Emb_\partial(I^{a_{r+1}},I^N-(I^{a_1}\cup\cdots\cup I^{a_r}))$; $t\mapsto \varphi_t^{(G,\sigma)}$, we get a map $B_\sigma\to B\Diff_\partial(V)$, which gives a $(V,\partial)$-bundle over $B_\sigma$. We also say that this bundle is obtained by {\it $\Psi_{r+1}$-surgery on $(G,\sigma)$}. One can recover the family of string links $\varphi^{(G,\sigma)}$ by postcomposing the map $\lambda\colon B\Diff_\partial(V)\to \Emb_\partial^\fr(I^{a_1}\cup\cdots\cup I_{a_{r+1}},I^N)$ as in Remark~\ref{rem:complement-map}.

Then an alternative definition of the iterated surgery can be given as follows. The chains $\sigma_1\colon B_{\sigma_1}\to B\Diff_\partial(V_1)$ and $\sigma_2\colon B_{\sigma_2}\to B\Diff_\partial(V_2)$ give maps
\begin{equation}\label{eq:BB-BDiff(V)}
 B_{\sigma_1}\times B_{\sigma_2}\longrightarrow B\Diff_\partial(V_1\tcoprod V_2)\longrightarrow B\Diff_\partial(V_{12}),
\end{equation}
where $V_{12}$ is a handlebody obtained from $I^N$ by removing tubular neighborhoods of the string link components $I^{a_1}\cup \cdots\cup I^{a_r}\cup I^{a_2'}\cup\cdots\cup I^{a_{s+1}'}$. By postcomposing the map $\lambda\colon B\Diff_\partial(V_{12})\to \Emb_\partial^\fr(I^{a_1}\cup \cdots\cup I^{a_r}\cup I^{a_2'}\cup\cdots\cup I^{a_{s+1}'},I^N)$ as in Remark~\ref{rem:complement-map} to (\ref{eq:BB-BDiff(V)}), we obtain a map
\[ B_{\sigma_1}\times B_{\sigma_2}\to \Emb_\partial^\fr(I^{a_1}\cup \cdots\cup I^{a_r}\cup I^{a_2'}\cup\cdots\cup I^{a_{s+1}'},I^N). \]

\subsection{Composition of the systems of Brunnian null-isotopies}\label{ss:compos-brunnian}

Let $G_1$ be a $\Psi_{r+1}$-graph, and $G_2$ be a $\Psi_{s+1}$-graph that are composable as in \S\ref{ss:compos-psi}. 
We see that the surgery on $G_1\tcoprod G_2$ induces canonical composition of the systems of Brunnian null-isotopies. The composition of the surgeries on $G_1$ and $G_2$ gives a family of embeddings in $\Emb_\partial^\fr(I^{a_{s+1}'},I^N-\tbigcup_\lambda I^{a_\lambda})$, where $\lambda$ runs over $\{1,\ldots,r\}\cup \{r+2,\ldots,r+s\}$. We suppose the following:
\begin{itemize}
\item The chains $\omega_{T_{\ell}}$ of Theorem~\ref{thm:vertex} for $\ell\leq \max\{r+1,s+1\}$ have been constructed.
\item $G_1$ and $G_2$ are of types that can be obtained from the basic cases of $T_{r+1}$ and $T_{s+1}$ (with only $S^n$-leaves) by iterated suspensions at the $(r+1)$-th leaf of $T_{r+1}$ and the first leaf of $T_{s+1}$. Let $\omega_{G_1}$ and $\omega_{G_2}$ be the corresponding families of string links obtained in this way.
\item We have systems of Brunnian null-isotopies $\gamma_{G_1}=\{\gamma_{G_1}(S_1)\}_{\emptyset\neq S_1\subset \{1,\ldots,r\}}$ and $\gamma_{G_2}=$\\ $\{\gamma_{G_2}(S_2)\}_{\emptyset\neq S_2\subset \{1,\ldots,s\}}$ for $\omega_{G_1}$ and $\omega_{G_2}$ in the cubical diagrams $\{\calE_{G_1}(S_1)\}_{S_1}$ and $\{\calE_{G_2}(S_2)\}_{S_2}$, respectively, where
\[ \begin{split}
&\calE_{G_1}(S_1)=\Emb_\partial(I^{b_1'},I^N-\tbigcup_{\lambda\notin S_1}I^{a_\lambda}),\\
&\calE_{G_2}(S_2)=\Emb_\partial(I^{a_{s+1}'},I^N-\tbigcup_{\lambda\notin S_2}I^{a_\lambda'}),
\end{split} \]
for $b_1'=a_{r+1}$ and $a_\mu'=a_{r+\mu}$ for $\mu\in\{2,\ldots,s\}$.
\end{itemize}
For $S\subset \{1,\ldots,r\}\cup \{r+2,\ldots,r+s\}$, we define
\[ \calE_{G_1\circ_1 G_2}(S)=\Emb_\partial(I^{a_{s+1}'},I^N-\tbigcup_{\lambda\notin S}I^{a_\lambda}). \]

\begin{Def}[Composition of Brunnian null-isotopies]\label{def:compos-paths}
We define the composition 
\[ \begin{split}
&\gamma_{G_1}\circ \gamma_{G_2}=\{(\gamma_{G_1}\circ \gamma_{G_2})(S)\}_S,\\
&S=S'\tcoprod S'',\quad S'\subset \{1,\ldots,r\},\quad S''\subset \{r+2,\ldots,r+s\},
\end{split}\]
where $(\gamma_{G_1}\circ \gamma_{G_2})(S)$ is a family in $\calE_{G_1\circ_1 G_2}(S)$, as follows.
\begin{enumerate}
\item When $S'=\emptyset$, $S''\neq \emptyset$, we define
\[ (\gamma_{G_1}\circ \gamma_{G_2})(S)=\omega_{G_1}\circ \gamma_{G_2}(S''_{-r}) \]
in $\calE_{G_1\circ_1 G_2}(S)$, where $S''_{-r}$ is obtained from $S''\subset \{r+2,\ldots,r+s\}$ by replacing each element $x$ with $x-r$. 
We consider $\gamma_{G_1}(S')=\omega_{G_1}$ as giving a family in $\calE_{G_1}(S')=\calE_{G_1}(\emptyset)$, and $\gamma_{G_2}(S''_{-r})$ as giving a family in $\calE_{G_2}(S''_{-r})$. 

\item When $S'\neq \emptyset$, $S''= \emptyset$, we define
\[ (\gamma_{G_1}\circ \gamma_{G_2})(S)=(\gamma_{G_1}(S')\circ \omega_{G_2})*\gamma_{G_2}(\{1\}), \]
where $(-)*\gamma_{G_2}(\{1\})$ extends each path in $\gamma_{G_1}(S')\circ \omega_{G_2}$ by the Brunnian null-isotopies in $\gamma_{G_2}(\{1\})$ so that it gives a family of null-isotopies of $\omega_{G_1}\circ\omega_{G_2}$.

\item When $S'\neq \emptyset$, $S''\neq\emptyset$, we define
\[ (\gamma_{G_1}\circ \gamma_{G_2})(S)=(\gamma_{G_1}(S')\circ \gamma_{G_2}(S''_{-r}))*\gamma_{G_2}(\{1\}\cup S''_{-r}), \]
where $(-)*\gamma_{G_2}(\{1\}\cup S''_{-r})$ attaches a homotopy in $\gamma_{G_2}(\{1\}\cup S''_{-r})$ from a path in $\iota\circ \gamma_{G_2}(S''_{-r})$ to a path in $\gamma_{G_2}(\{1\})$ fixing the endpoints (Figure~\ref{fig:compos-paths}, gray disk).
We impose a structure of a family of paths in the family $(\gamma_{G_1}\circ \gamma_{G_2})(S)$ as follows. For each pair $\gamma_1,\gamma_2$ of paths in $\gamma_{G_1}(S'),\gamma_{G_2}(S''_{-r})$, respectively, the composition $\gamma_1(t_1)\circ \gamma_2(t_2)$ gives a 2-parameter family of string knots in $\Emb_\partial^\fr(I^{a_{s+1}'},I^N-\tbigcup_{\lambda\notin S}I^{a_\lambda})$ parametrized by $(t_1,t_2)\in I^2$, on which $(1,1)$ corresponds to $\omega_{G_1}\circ\omega_{G_2}$, $(t_1,0)$ corresponds to $\iota$. We give $I^2$ a structure of a 1-parameter family of gradient flow-lines of the Morse function $t_1^2(2-t_1)^2+t_2^2(2-t_2)^2$ (See \cite[Example~3.3]{Hut}). The side $\{0\}\times I$ is connected in $\calE_{G_2}(\{1\}\cup S''_{-r})$ to homotopies to paths in $\gamma_{G_2}(\{1\})$. (See Figure~\ref{fig:compos-paths}).
\end{enumerate}
\end{Def}
\begin{figure}[h]
\[ \includegraphics[height=40mm]{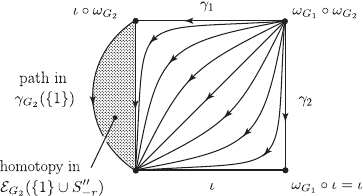} \]
\caption{A structure of a 1-parameter family of null-isotopies in $\calE_{G_1\circ_1 G_2}(S)$ from $\omega_{G_1}\circ\omega_{G_2}$ (when $S'\neq\emptyset$, $S''\neq\emptyset$).}\label{fig:compos-paths}
\end{figure}

\begin{Lem}\label{lem:compos-brun}
The composition $\gamma_{G_1}\circ\gamma_{G_2}$ defines a system of Brunnian null-isotopies for the family 
\[ \omega_{G_1}\circ \omega_{G_2}\colon B_{T_{r+1}}\times B_{T_{s+1}}\to \calE_{G_1\circ_1 G_2}(\emptyset)\]
relative to those for lower excess faces.
\end{Lem}
\begin{proof}
Let $(f_{u,v},\gamma_{u,v})$, $(u,v)\in B_{T_{r+1}}\times B_{T_{s+1}}$, be the pair of the embedding $f_{u,v}=(\omega_{G_1}\circ\omega_{G_2})(u,v)$ and the restriction $\gamma_{u,v}=\{\gamma_{u,v}(S)\}_S$ of $\gamma_{G_1}\circ\gamma_{G_2}$ over $f_{u,v}$. By Definition~\ref{def:brun-null-iso}, it suffices to prove that $(f_{u,v},\gamma_{u,v})$ satisfies the following conditions (a) and (b) at each $(u,v)$.
\begin{enumerate}
\item[(a)] $\gamma_{u,v}(S)$ defines a map $\Delta^{|S|-1}\to \calB_{G_1\circ_1 G_2}(S)$, where $\calB_{G_1\circ_1 G_2}(S)$ is the homotopy fiber over $\iota$ of the map $\pi_S\colon \calE_{G_1\circ_1 G_2}(\emptyset)\to \calE_{G_1\circ_1 G_2}(S)$.
\item[(b)] $\{\gamma_{u,v}(S)\}_{S\neq \emptyset}$ is compatible with respect to morphisms $S\stackrel{\subset}{\to} S\cup \{i\}$. Namely, the following diagram is commutative.
\[ \xymatrix{
  \Delta^{|S|-1} \ar[rr]^-{\gamma_{u,v}(S)} \ar[d] && \calB_{G_1\circ_1 G_2}(S) \ar[d] \\
  \Delta^{|S|} \ar[rr]^-{\gamma_{u,v}(S\cup \{i\})} && \calB_{G_1\circ_1 G_2}(S\cup\{i\})
}\]
\end{enumerate}

When $S'=\emptyset$, $S''\neq\emptyset$, we have $\gamma_{u,v}(S)=\omega_{G_1}(u)\circ \gamma_{G_2}(S''_{-r})(v)$, which is a family over $\Delta^{|S''|-1}$ in $\calB_{G_1\circ_1 G_2}(S)$ since $\omega_{G_1}(u)$ is a single embedding and $\gamma_{G_2}(S''_{-r})(v)$ is a family over $\Delta^{|S''|-1}$ in $\calB_{G_2}(S''_{-r})$. This proves (a) for $S'=\emptyset$. The compatibility of (b) for $S'=\emptyset$ follows from that of $\{\gamma_{G_2}(S''_{-r})\}_{S''}$.

When $S'\neq\emptyset$, $S''=\emptyset$, we have $\gamma_{u,v}(S)=(\gamma_{G_1}(S')(u)\circ \omega_{G_2}(v))*\gamma_{G_2}(\{1\})(v)$, which is a family over $\Delta^{|S'|-1}$ in $\calB_{G_1\circ G_2}(S)$ since $\gamma_{G_1}(S')(u)$ is a family over $\Delta^{|S'|-1}$ in $\calB_{G_1}(S')$, $\omega_{G_2}(v)$ is a single embedding, and $\gamma_{G_2}(\{1\})(v)$ is a single path. This proves (a) for $S''=\emptyset$. The compatibility of (b) for $S''=\emptyset$ follows from that of $\{\gamma_{G_1}(S')\}_{S'}$.

When $S'\neq\emptyset$, $S''\neq\emptyset$, we have $\gamma_{u,v}(S)=(\gamma_{G_1}(S')(u)\circ \gamma_{G_2}(S''_{-r})(v))*\gamma_{G_2}(\{1\}\cup S''_{-r})(v)$, where
\begin{itemize}
\item $\gamma_{G_1}(S')(u)$ is a family over $\Delta^{|S'|-1}$ in $\calB_{G_1}(S')$,
\item $\gamma_{G_2}(S''_{-r})(v)$ is a family over $\Delta^{|S''|-1}$ in $\calB_{G_2}(S''_{-r})$.
\end{itemize}
Thus $\gamma_{G_1}(S')(u)\circ \gamma_{G_2}(S''_{-r})(v)$ is parametrized by $\Delta^{|S'|-1}\times\Delta^{|S''|-1}\times I$, where $I$ parametrizes paths in the 2-disk model as in Definition~\ref{def:compos-paths} (3) and Figure~\ref{fig:compos-paths}. We see that after attaching $\gamma_{G_2}(\{1\}\cup S''_{-r})(v)(t)$, it can be parametrized by the join $\Delta^{|S'|-1}*\Delta^{|S''|-1}=\Delta^{|S'|-1}\times\Delta^{|S''|-1}\times I/{\sim}=\Delta^{|S'|+|S''|-1}$, where $\Delta^{|S'|-1}\times\Delta^{|S''|-1}\times \{0\}$ is collapsed to $\Delta^{|S'|-1}$, $\Delta^{|S'|-1}\times\Delta^{|S''|-1}\times \{1\}$ is collapsed to $\Delta^{|S''|-1}$. Indeed, for each point $(z_1,z_2)\in \Delta^{|S'|-1}\times\Delta^{|S''|-1}$, the 1-parameter family $\gamma_{u,v}(S)(z_1,z_2)$ in $\calB_{G_1\circ_1 G_2}(S)$ can be given by a map from the 2-disk model as in Figure~\ref{fig:compos-paths} to $\calE_{G_1\circ_1 G_2}(S)$, as explained in Definition~\ref{def:compos-paths} (3). Note that the path from $\gamma_{G_2}(\{1\})(v)$ is unique and does not depend on $(z_1,z_2)$. 
At the endpoints of the 1-parameter family $\gamma_{u,v}(S)(z_1,z_2)$, which correspond to the two boundary paths from the upper right corner to the lower left corner in the 2-disk of Figure~\ref{fig:compos-paths}, the parametrization by $(z_1,z_2)$ of paths in $\gamma_{G_1}(S')(u)$ and $\gamma_{G_2}(S''_{-r})(v)$ degenerates to families parametrized by $\Delta^{|S'|-1}$ and $\Delta^{|S''|-1}$, respectively. This proves (a) for $S'\neq\emptyset$, $S''\neq\emptyset$. 

The compatibility of (b) for $S'\neq\emptyset$, $S''\neq\emptyset$ is induced from those of $\{\gamma_{G_1}(S')\}_{S'}$ and $\{\gamma_{G_2}(S''_{-r})\}_{S''}$. Note that the compatibility of (b) with respect to $S'\tcoprod S''\to (S'\tcoprod S'')\cup \{i\}$ when $S'$ or $S''$ is empty but nonempty after adding $\{i\}$ is apparent from the 2-disk model: the boundary path passing through the upper left corner is the image from the $S''=\emptyset$ case, and that through the lower right corner is the image from the $S'=\emptyset$ case.
\end{proof}
\begin{Lem}
If $G_1$ and $G_2$ are both $\Psi_3$-graphs, then the composition $\gamma_{G_1}\circ \gamma_{G_2}$ of Definition~\ref{def:compos-paths} of the systems of Brunnian null-isotopies $\gamma_{G_i}$ agrees with that of Lemma~\ref{lem:pres-brunnian-3}.
\end{Lem}
\begin{proof}
We consider the composition $\gamma_{G_1}\circ \gamma_{G_2}$ of the Brunnian null-isotopies with respect to removing one of the components of $\{1,2\}\cup \{4\}$. The null-isotopies $(\gamma_{G_1}\circ \gamma_{G_2})(\{1\})$ and $(\gamma_{G_1}\circ \gamma_{G_2})(\{2\})$ are defined as the case of $S'\neq \emptyset$ and $S''=\emptyset$, in which case $(\gamma_{G_1}\circ \gamma_{G_2})(S)=(\gamma_{G_1}(S')\circ \omega_{G_2})*\gamma_{G_2}(\{1\})$. The null-isotopy $\gamma_{G_1}(S')\circ \omega_{G_2}$ is induced by that of a spanning disk of the core of the $(2n-1)$-handle for the first leaf of $G_2$, which is taken along the disk of the form $U''$ as in the proof of Lemma~\ref{lem:pres-brunnian-3}, where the thin $n$-handle of $U''$ goes along the fourth leaf. The null-isotopy $(\gamma_{G_1}\circ \gamma_{G_2})(\{4\})$ is defined as the case of $S'=\emptyset$ and $S''\neq \emptyset$, in which case $(\gamma_{G_1}\circ \gamma_{G_2})(S)=\omega_{G_1}\circ \gamma_{G_2}(S''_{-2})$. This is again given by the disk of the form $U''$ but the thin $(2n-1)$-handle goes along the embedded Whitehead product of the first and second leaves. The resulting spanning disk is that obtained from the previous cases, which corresponds to $U'$, by the deformation $U'\leftrightarrow U''$ as in the proof of Lemma~\ref{lem:pres-brunnian-3}. 
\end{proof}

Now we consider what happens if the roles of $G_1$ and $G_2$ are switched. By the symmetries of $\omega_{T_{r+1}}$ and $\omega_{T_{s+1}}$, which follow by the assumption that Theorem~\ref{thm:vertex} hold for $\ell\leq\max\{r+1,s+1\}$, we know that the chain $\omega_{G_1}\circ\omega_{G_2}$ can be deformed in $\Emb_\partial(I^{a_{s+1}'}\cup\bigcup_\lambda I^{a_\lambda},I^N)$ into the reversed composition $\omega_{G_2}\circ \omega_{G_1}$ by using the Brunnian null-isotopy with respect to removing the $(s+1)$-st component in $\omega_{G_2}$. 
\begin{Lem}\label{lem:switch-G1-G2} 
This deformation can be extended to that between the systems of Brunnian null-isotopies for the families $\omega_{G_1}\circ\omega_{G_2}$ and $\omega_{G_2}\circ\omega_{G_1}$.
\end{Lem}
\begin{proof}
We will also prove cyclic symmetry properties for the system of Brunnian null-isotopies (see Lemma~\ref{lem:cyclic-brun-5}). Then by using it, the lemma follows, as in the proof of Theorem~\ref{thm:vertex} for $\ell=4$. Note that during the deformation, the gray disk in Figure~\ref{fig:compos-paths} is shrinked into the constant path at $\iota$, and the bottom path at $\iota$ is expanded into a homotopy in $\calE_{G_1}(S'\cup\{r+1\})$.
\end{proof}

\subsection{Suspension of a system of Brunnian null-isotopies}\label{ss:suspend-brunnian}
The suspension of a string link in Definition~\ref{def:suspension} induces that of a system of Brunnian null-isotopies of the string link. 
\begin{Def}[Suspension of Brunnian null-isotopies]\label{def:suspend-brunnian}
Let $L\colon I^{a_1}\cup \cdots\cup I^{a_r}\to I^d$ ($0<a_i<d$) be a string link in $\Emb_\partial^\fr(I^{a_1}\cup\cdots\cup I^{a_r},I^d)$ equipped with a system $\{\gamma_L(S)\}_S$ of Brunnian null-isotopies. Let $L'$ be the suspension of $L$ with respect to a subset $K\subset \Omega=\{1,\ldots,r\}$ of components of $L$ and the Brunnian null-isotopy $\gamma_L(\Omega\setminus K)$.
We define the induced system of Brunnian null-isotopies $\{\gamma_{L'}(S)\}_S$ of $L'$ as follows. 
\begin{enumerate}
\item For $S\subset \Omega$, let $J=\Omega\setminus S$. We deform the null-isotopy $\gamma_L(S)$ into another one that can be extended to a null-isotopy of the components in $K\cup J$. This is possible since $\gamma_L(S)\in \calB(S)$ can be deformed to one in the image of $\calB(\Omega\setminus(K\cup J))\to \calB(\Omega\setminus J)=\calB(S)$. The resulting null-isotopy extends to the null-isotopy $\gamma_L(\Omega\setminus(K\cup J))\in \calB(\Omega\setminus(K\cup J))$ and it induces a null-isotopy of the components in $K$ by the map $\calB(\Omega\setminus(K\cup J))\to \calB(\Omega\setminus J)$. 
\item We use the null-isotopy of the components in the previous item to suspend $L$. Let $L''$ be the resulting string link. Then the null-isotopy of the components in $K\cap J$ can be extended to that of $L''$. Thus we have a Brunnian null-isotopy of $L''$ for the components in $J$.
\item By construction, $L'$ and $L''$ are relatively isotopic as 1-parameter families of string links. Thus the isotopy extension theorem gives a Brunnian null-isotopy $\gamma_{L'}(S)$ of $L'$ for the components in $J$. 
\end{enumerate}
\end{Def}

\section{5-valent vertex}\label{s:5-valent}

In this section, we prove Theorem~\ref{thm:vertex} for $\ell=5$. The construction in this paper for the case of the $\ell$-valent vertex for $\ell\geq 5$ is less explicit than the 4-valent vertex case in \S\ref{s:4-valent}. Since the general $\ell$-valent vertex case will be quite complicated, we separately explain about the case $\ell=5$ first to see what is needed to construct higher brackets in detail.
The relation in $L_\infty$-algebra $(L,\partial)$ involving brackets with four $\partial$-closed inputs is formally given by
\[ \begin{split}
&[[c,d],a,b]+[[a,b],c,d]+[[a,d],b,c]+[[b,c],a,d]+(-1)^n[[a,c],b,d]+(-1)^n[[b,d],a,c]\\
&-[[a,c,d],b]-[[a,b,c],d]-[[b,c,d],a]-[[a,b,d],c]=\partial[a,b,c,d]
\end{split}\]
for $\partial$-closed formal variables $a,b,c,d$ of degree $n$. We call this relation the {\it 10 term relation} (or the {\it 10T relation} for short) in $L_\infty$-algebra\footnote{The signs of the terms in the LHS of the 10T relation are uniquely determined up to overall reversal and symmetry in triple brackets. This can be proved by applying Lemmas~\ref{lem:sign-jacobi} and \ref{lem:sym-relation}.}. We shall construct a string link model (``10T-link'') for the LHS of this identity (Definition~\ref{def:10T-link}), and a path $[f,g,h,k]$ in the space $\Omega\Emb_\partial^\fr(D^{4n-3},X)$ (Definitions~\ref{def:quad-bracket} and \ref{def:quadruple}). We will not explicitly determine the orientations of the terms at this stage. It is enough to see that the construction works for {\it some} orientations. We will suspend and deloop it several times, and it is enough for our purpose to define explicit orientations to the result as in Definition~\ref{def:ori}. 

\subsection{10T-link}\label{ss:10T-link}

We need to define a suitable notion of the ``sum'' of terms in the 10T relation to obtain a string link (10T-link) that represents the LHS of the above identity. The resulting string link, which is an embedding in $\Emb_\partial(I^{4n-2},I^{5n-1}-(I^{4n-2})^{\cup 4})$, will be the graph of a pointed loop in $\Emb_\partial(I^{4n-3},I^{5n-2}-(I^{4n-3})^{\cup 4})$. 

We first give each term in the LHS of the 10T relation as a concordance string link obtained from an embedding $I^{4n-2}\to I^{5n-1}-(I^{4n-2})^{\cup 4}$.
The boundary of each term in the 10T relation, which is the sum of three iterated Whitehead brackets of type $[[[*,*],*],*]$, gives a string link $(I^{4n-3})^{\cup 5}\to I^{5n-2}$. We assume that the first four components of this string link are standard inclusions. 
\begin{itemize}
\item Iterated brackets $[[[*,*],*],*]$: Let $\alpha_i\colon I^{4n-3}\to I^{5n-2}-(I^{4n-3})^{\cup 4}$ ($i=1,2,\ldots,10$) denote the one distinguished component in the 5 component string link of the boundary of the $i$-th term in the 10T relation. It is given by the connected sum of three links each of which is given by the iterated composition of three $\Psi_3$-graphs of types $(n,n,3n-2)$, $(2n-1,n,2n-1)$, $(3n-2,n,n)$, which are obtained from $(n,n,n;3n)$ by suspensions (Definition~\ref{def:suspension}), in the sense of Definition~\ref{def:iter-surg}. (See Figure~\ref{fig:alpha_i}, left.)
\item Let $m_4\alpha_i=4\alpha_i\colon I^{4n-3}\to I^{5n-2}-(I^{4n-3})^{\cup 4}$ ($i=1,2,\ldots,10$) denote the embedding obtained by taking the sum of 4 copies of $\alpha_i$.
\item Iterated brackets $4[[*,*,*],*]$: Let $4\alpha_i'\colon I^{4n-2}\to I^{5n-1}-(I^{4n-2})^{\cup 4}$ $(i=1,2,\ldots,10$) be the embedding obtained from $4\alpha_i\colon I^{4n-3}\to I^{5n-2}-(I^{4n-3})^{\cup 4}$ equipped with the null-isotopies induced from $\omega_{T_4}$ for the 4-valent vertices, by taking the track of the null-isotopy of Definition~\ref{def:iter-surg}. It is given by the composition $\omega_{T_4}'\circ \omega_{T_3}'$ of one $\Psi_4$-graph surgery $\omega_{T_4}'$ of type $(n,n,n,2n-1)$ and one $\Psi_3$-graph surgery $\omega_{T_3}'$ of type $(3n-2,n,n)$ in the sense of Definition~\ref{def:iter-surg}, which are obtained from $\omega_{T_4}$ and $\omega_{T_3}$ by suspensions (Definition~\ref{def:suspension}). (See Figure~\ref{fig:alpha_i}, right.)
\end{itemize}
\begin{figure}[h]
\[ \includegraphics[height=40mm]{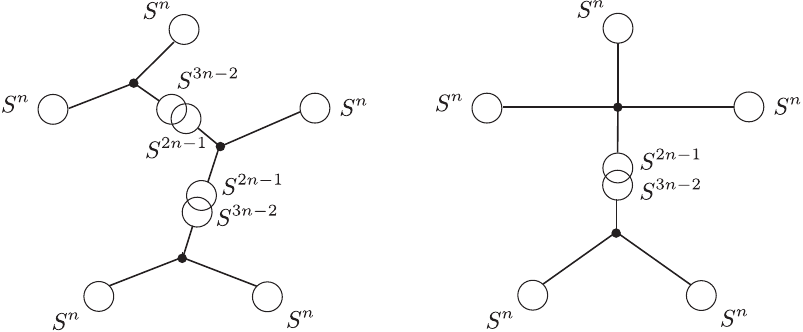} \]
\caption{The compositions of the $\Psi$-graphs in $I^{5n-2}-(I^{4n-3})^{\cup 5}$ for a term in $4\alpha_i$ and for $4\alpha_i'$.}\label{fig:alpha_i}
\end{figure}

\begin{Lem}\label{lem:Sigma_10T}
There is a decomposition of $I^{4n-2}$:
\[ I^{4n-2}=A'\cup \bigcup_{i=1}^{10}C_i' \]
into compact submanifolds $A',C_i'$ of $I^{4n-2}$ with corners, and an embedding 
\[ \Sigma_{\mathrm{10T}}\colon I^{4n-2}\to I^{5n-1}-(I^{4n-2})^{\cup 4} \]
such that
\begin{enumerate}
\item $C_i'$ is diffeomorphic to $I^{4n-2}$ as a smooth manifold with boundary,
\item the intersection $C_i'\cap C_j'$ is either empty or a common codimension 1 face which is diffeomorphic to $I^{4n-3}$,
\item $A'$ is the closure of $I^{4n-2}-\bigcup_{i=1}^{10}C_i'$, 
\item the restriction of $\Sigma_{\mathrm{10T}}$ to $\partial C_i'$ is given by $4\alpha_i$, and
\item the restriction of $\Sigma_{\mathrm{10T}}$ to $C_i'$ is given by $4\alpha_i'$, 
\item the restriction of $\Sigma_{\mathrm{10T}}$ to $A'$ agrees with the restriction of the standard inclusion.
\end{enumerate}
\end{Lem}

From now on, we construct $\Sigma_{\mathrm{10T}}$ of Lemma~\ref{lem:Sigma_10T}.
We use the ``Lie-hedron'' $L_4$ (consisting of 10 vertices and 15 edges, Figure~\ref{fig:Lie-hedron}) to inductively extend the 10 string links $4\alpha_i\colon I^{4n-3}\to I^{5n-2}-(I^{4n-3})^{\cup 4}$, each of which correspond to the union of the three incident edges of at a vertex in $L_4$, to a string link $I^{4n-2}\to I^{5n-1}-(I^{4n-2})^{\cup 4}$ (10T-link) by pasting together the extensions from $4\alpha_i$ to $4\alpha_i'$. We will see that the resulting string link will be isotopically trivial (Lemma~\ref{lem:null-iso-10T}), but the decomposition structure into 10 terms will be important to glue between families for different graphs. 
We will need to identify two boundary terms from different terms in the 10T relation, and the decomposition structure of the family of embeddings to those of the three 3-valent graphs should not be broken. From now on, we construct the 10T-link in 6 steps.
\begin{figure}[h]
\includegraphics[height=60mm]{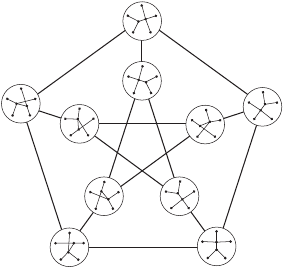}
\caption{Lie-hedron $L_4$ (\cite{MS,MSS}): The Lie analogue of the associahedron $K_4$.}\label{fig:Lie-hedron}
\end{figure}

\subsubsection{Step 1: Embed the Lie-hedron in $I^{4n-2}$}\label{ss:5-step1}

We take a thickened ribbon graph in $I^3$ as in Figure \ref{fig:lie-hedron2}, which consists of 3-cubes $A_1,A_2,\ldots,A_{10}\cong I^3$ and rods $B_1,B_2,\ldots,B_{15}\cong D^2\times I$.
\begin{figure}[h]
\[ \includegraphics[height=55mm]{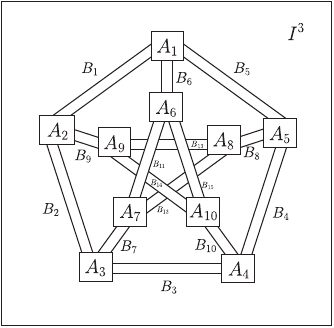}\qquad 
\includegraphics[height=45mm]{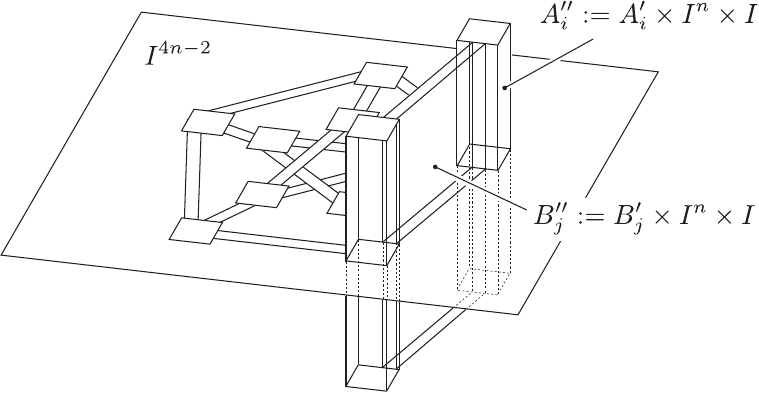} \]
\caption{Step 1 (left) and Step 2 (right).}\label{fig:lie-hedron2}
\end{figure}

We set $a=\frac{1}{2}$ and take a small positive real number $\ve$ such that $[a-2\ve,a+2\ve]\subset (0,1)$. Let $A_i'=A_i\times [a-2\ve,a+2\ve]^{4n-5}$ and $B_j'=B_j\times[a-\ve,a+\ve]^{4n-5}$, which are $(4n-2)$-disks, and consider these as subspaces of $I^3\times I^{4n-5}=I^{4n-2}$. The union 
\begin{equation}\label{eq:NL4}
 N_{L_4}:=\bigcup_i A_i'\cup\bigcup_j B_j' 
\end{equation}
is a thickened ribbon graph in $I^{4n-2}$.

\subsubsection{Step 2: Arrange $N_{L_4}$ in $I^{5n-1}=I^{4n-2}\times I^n\times I$}

Let $A_i''=A_i'\times I^n\times I$ and $B_j''=B_j'\times I^n\times I$, which are $(5n-1)$-disks, and consider these as subspaces of $I^{5n-1}$. We fix distinct values $b_1,b_2,b_3,b_4\in (0,1)\setminus \{a\}$ and let 
\[ \sigma_j:=I^{4n-2}\times (\underbrace{a,\ldots,a}_n,b_j). \]
We consider the four component standard string link $\sigma_1\cup \sigma_2\cup \sigma_3\cup \sigma_4$ in $I^{5n-1}$.
Each box $A_i''$ intersects the disk $\sigma_j$ by $A_i'\times (\underbrace{a,\ldots,a}_n,b_j)\cong I^{4n-2}$.

\subsubsection{Step 3: Decompose $N_{L_4}$ into 10 boxes}

We decompose $B_j\cong D^2\times I$ into two parts $D^2\times [0,\frac{1}{2}]$ and $D^2\times [\frac{1}{2},1]$. Then decompositions of $B_j'$ and $B_j''$ are induced, and those decompositions induce a decomposition of $N_{L_4}$ into 10 boxes, each of which is an enlargement of $A_i'$ by attaching three thickened half-edges. Let $\widetilde{A}'$, $\widetilde{A}_i''$ denote the enlarged box including $A_i'$, $A_i''$, respectively. Let $\widetilde{C}_j$ denote the middle face $D^2\times\{\frac{1}{2}\}\times [a-\ve,a+\ve]^{4n-5}\times I^n\times I$ in $B_j''$.
\begin{figure}[h]
\[ \includegraphics[height=50mm]{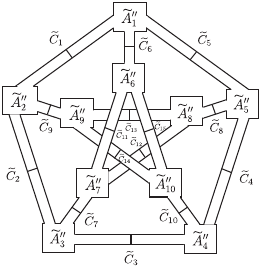} \]
\caption{Step 3.}
\end{figure}

\subsubsection{Step 4: Taking string links $(I^{4n-3})^{\cup 5}\to \widetilde{C}_j$}

In this and the next step, we replace the $N_{L_4}$ in $I^{5n-1}$ with a nonstandard link. The face $\widetilde{C}_j$ is a $(5n-2)$-disk. We choose a 5-component string link inside each $\widetilde{C}_j$, in which the first four components are $(\sigma_1\cup\sigma_2\cup\sigma_3\cup\sigma_4)\cap \widetilde{C}_j$. The fifth component is given by 4 times the iterated Whitehead products for the tree with three 3-valent vertices corresponding to the $j$-th edge of $L_4$. This is one of the three terms in $4\alpha_i$ for a vertex $i$ of $L_4$. The result is a string link $(I^{4n-3})^{\cup 5}\to \widetilde{C}_j$. 

\subsubsection{Step 5: Taking concordance string links $(I^{4n-2})^{\cup 5}\to \widetilde{A}_i''$ extending those in $\widetilde{C}_j$}

Suppose that $\widetilde{C}_{j_1},\widetilde{C}_{j_2},\widetilde{C}_{j_3}$ are the three faces incident to $\widetilde{A}_i''$. The string links in $\widetilde{C}_{j_1},\widetilde{C}_{j_2},\widetilde{C}_{j_3}$ can be extended to that in $B_{j_1}''\tcoprod B_{j_2}''\tcoprod B_{j_3}''$ by taking the direct product with $I$ in the product structure $B_{j_i}''=\widetilde{C}_{j_i}\times I$. The restriction of the resulting family of string links to $\widetilde{C}_{j_i}\times \partial I$ can be extended over the whole of $\partial A_i'\times I^n\times I$ by the restriction of the standard inclusion $(I^{4n-2})^{\cup 5}\to I^{5n-1}$. This gives a link in $\partial A_i'\times I^n\times I$ isotopic to 4 times the IHX-link, which is obtained by suspending $\varphi_{\mathrm{IHX}}$ of Theorem~\ref{thm:phi-jacobi} as the definition of $\alpha_i$ in the beginning of \S\ref{ss:10T-link}, and can be extended to a concordance string link $(I^{4n-2})^{\cup 5}\to \widetilde{A}_i''$ by the iterated surgery path $4\alpha_i'$ (from $\varphi_t^{G_1\circ_1 G_2}$ etc. in Definition~\ref{def:iter-surg}) for the $i$-th vertex of $L_4$.

\subsubsection{Step 6: Extension to whole of $I^{4n-2}$}

We extend the resulting embedding $N_{L_4}\to I^{5n-1}-(\sigma_1\cup \sigma_2\cup \sigma_3\cup \sigma_4)$, where $N_{L_4}$ was defined in (\ref{eq:NL4}), to an embedding from $I^{4n-2}$ by the standard inclusion $I^{4n-2}\to I^{5n-1}-(\sigma_1\cup \sigma_2\cup \sigma_3\cup \sigma_4)$ on $I^{4n-2}-\mathrm{Int}\,N_{L_4}$. The following definition completes the proof of Lemma~\ref{lem:Sigma_10T}.

\begin{Def}[10T-link]\label{def:10T-link}
Let 
\[ \varphi_{\mathrm{10T}}=(\sigma_1\cup \sigma_2\cup \sigma_3\cup \sigma_4)\cup\Sigma_{\mathrm{10T}}\colon (I^{4n-2})^{\cup 5}\to I^{5n-1} \]
be the resulting 5-component string link. Here, the first four components are standard, we also denote by $\sigma_i$ the inclusion of $\sigma_i$, and the first component is standard near the boundary. We call this string link $\varphi_{\mathrm{10T}}$ the {\it 10T-link}.
\end{Def}
\begin{proof}[Proof of Lemma~\ref{lem:Sigma_10T}]
Let $C_i'=\iota(I^{4n-2})\cap \widetilde{C}_i$, where $\iota\colon I^{4n-2}\to I^{5n-1}$ is the standard inclusion (of the fifth component), and $\widetilde{C}_i$ is the $(5n-2)$-disk of Step 3 above. Let $A'$ be the closure of $I^{4n-2}-\bigcup_{i=1}^{10}C_i'$. Then the properties (1)--(6) follows from the construction of $\Sigma_{\mathrm{10T}}$.  
\end{proof}

\begin{Rem}\label{rem:deform-within-box}
By Lemma~\ref{lem:pres-brunnian-10T} about the Brunnian property of the 10T-link, the 10T-link is isotopic to a string link such that all the components but the $p$-th one are standard inclusions for any $p\in\{1,2,3,4\}$. 
Since the restriction of the 10T-link to each box $\widetilde{A}_i''\cong I^{5n-1}$ is a string link $(I^{4n-2})^{\cup 5}\to \widetilde{A}_i''$, and the Brunnian property of Lemma~\ref{lem:pres-brunnian-10T} is constructed box-wise, we may assume that the deformation preserves the decomposition into the boxes $\widetilde{A}_i''$. This property will be used later in the proof of the compatibility of the cyclic symmetry property of $\Psi_5$-surgery with respect to the box decomposition in the proof of Theorem~\ref{thm:vertex} for $\ell=5$.
\end{Rem}

\subsection{10T-link as a loop and existence of null-isotopy of the 10T-link}\label{ss:null-iso-10T}

The following lemma shows that the 10T-link can be deformed to a loop of embeddings in $\Emb_\partial((I^{4n-3})^{\cup 5},I^{5n-2})$.
\begin{Lem}\label{lem:deform-to-loop}
Let $n\geq 2$. The string link $\varphi_{\mathrm{10T}}$ can be chosen to be the graph of a loop of embeddings
\[ (I^{4n-3})^{\cup 5}\to I^{5n-2}. \]
\end{Lem}
\begin{proof}
The desired string link can be constructed by considering $\Sigma_{\mathrm{10T}}\colon I^{4n-2}\to I^{5n-1}$ as the track of a loop of embeddings $I^{4n-3}\to I^{5n-2}$, as follows. First, we observe that the string link $(I^{4n-3})^{\cup 5}\to \widetilde{C}_j$ in Step 4 of the previous subsection can be represented as the graph of a loop of embeddings $(I^{4n-4})^{\cup 5}\to I^{5n-3}$. This is because the Borromean string links used at the 3-valent vertices of excess 0 trees are suspensions of those of lower dimensions. Indeed, the string link $(I^{4n-3})^{\cup 5}\to \widetilde{C}_j\cong I^{5n-2}$ in Step 4 can be obtained by suspensions of the following type:
\[ (3n-2,3n-2,3n-2,3n-2,4n-3;4n-1)\to (4n-3,4n-3,4n-3,4n-3,4n-3;5n-2).\]
(The $\Psi_4$-graph of type $(n,n,n,n)$ gives a string link of type $(3n-2,3n-2,3n-2,3n-2;4n-1)$ and the $\Psi_3$-graph of type $(n,n,n)$ gives a string link of type $(2n-1,2n-1,2n-1;3n)$. The $(n-1)$-fold suspension of $(2n-1,2n-1,2n-1;3n)$ for the first two components gives $(3n-2,3n-2,2n-1;4n-1)$. The composition of the $\Psi_4$-graph and the $(n-1)$-fold suspension of the $\Psi_3$-graph gives the left hand side $(3n-2,3n-2,3n-2,3n-2,4n-3;4n-1)$.)
The fifth component of dimension $4n-3$ in the result of the suspensions can be deformed into the graph of $(n-1)$-fold delooping, where $n-1\geq 1$.
By Lemma~\ref{lem:suspension-delooping}, a suspension can be exchanged with the graph of delooping, and we obtain the graph of a loop of string links of type $(4n-4,4n-4,4n-4,4n-4,4n-4;5n-3)$.

Next, we arrange the handlebody $N_{L_4}$ embedded in $I^{4n-2}$ to a general position. Namely, if we let $q\colon I^{4n-2}=I^{4n-3}\times I\to I$ be the projection onto the last factor, there is a sequence of values $0=c_0<c_1<c_2<\cdots<c_r=1$ of $q|_{L_4}$ such that 
\begin{itemize}
\item $q^{-1}(c_i)$ intersects edges of $L_4$ transversally,
\item the singularity included in $q^{-1}([c_i,c_{i+1}])$ is either 
\begin{enumerate}
\item[(a)] a single critical point of $q$ in the interior of an edge of $L_4$, or
\item[(b)] a single vertex of $L_4$. 
\end{enumerate}
\end{itemize}
(See Figure~\ref{fig:L4-critical}.) We may further assume that the restriction of $\Sigma_{\mathrm{10T}}$ to $q^{-1}(c_i)$ is the graph of a loop of embeddings $I^{4n-4}\to I^{5n-3}$ supported on $q^{-1}(c_i)\cap N_{L_4}$.
\begin{figure}[h]
\[ \includegraphics[height=20mm]{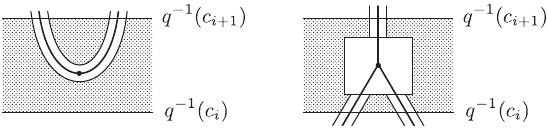} \]
\caption{Examples of singularities of type (a) (left) and (b) (right).}\label{fig:L4-critical}
\end{figure}

Now we have a sequence $\{\rho_i\}$ of pointed loops in $\Emb_\partial((I^{4n-4})^{\cup 5},I^{5n-3})$ for $\varphi_{\mathrm{10T}}|_{q^{-1}(c_i)}$, $i=0,1,\ldots,r$. We then see that the restriction of $\varphi_{\mathrm{10T}}$ to $q^{-1}([c_i,c_{i+1}])$ is given by a homotopy between the loops $\rho_i$ and $\rho_{i+1}$. The cobordism of type (a) is just an explicit null-homotopy of a composed loop of the form $\gamma*\gamma^{-1}$ in $\Emb_\partial((I^{4n-4})^{\cup 5},I^{5n-3})$, which is a path of loops. Namely, the loops $\rho_i$ and $\rho_{i+1}$ can be considered as a map $I\times \{0,1\}\to \Emb_\partial((I^{4n-4})^{\cup 5},I^{5n-3})$ that can be extended to a map from $I\times I$. On the other hand, each excess 0 tree $\tau$ gives a loop $\gamma_\tau\colon I\to \Emb_\partial((I^{4n-4})^{\cup 5},I^{5n-3})$, and the embedding over the corresponding 1-handle of $N_{L_4}$ is given by the pullback of $\gamma_\tau$ by a smooth map $h_\tau\colon I\times I\to I$, for which the 1-handle is thickened from the preimage of a regular value in $I$ under $h_\tau$. The explicit null-homotopy of $\gamma*\gamma^{-1}$ can be given by the pullback $h_\tau^*\gamma_\tau\colon I\times I\to \Emb_\partial((I^{4n-4})^{\cup 5},I^{5n-3})$. The cobordism of type (b) can be chosen as a path of loops because the null-isotopy in $\Emb_\partial((I^{4n-3})^{\cup 5},I^{5n-2})$ for a tree with a 4-valent vertex was constructed by suspensions and can be lifted to that in $\Omega\Emb_\partial((I^{4n-4})^{\cup 5},I^{5n-3})$ (see Lemma~\ref{lem:suspension-delooping}). 
\end{proof}

\begin{Lem}\label{lem:null-iso-10T}
Let $n\geq 3$. The pointed loop of embeddings in $\Emb_\partial((I^{4n-3})^{\cup 5},I^{5n-2})$ associated to the string link $\varphi_{\mathrm{10T}}$ is null-homotopic in $\Omega\Emb_\partial((I^{4n-3})^{\cup 5},I^{5n-2})$ up to taking connected sums of finite copies. 
\end{Lem}
\begin{proof}
The associated loop of embeddings gives an element of $\pi_1(\Emb_\partial((I^{4n-3})^{\cup 5},I^{5n-2}))$.
By Corollary~\ref{cor:goodwillie}, the natural map
\[ \pi_1(\Emb_\partial((I^{4n-3})^{\cup 5},I^{5n-2}))
\to\pi_0(\Emb_\partial((I^{4n-2})^{\cup 5},I^{5n-1})) \]
is an isomorphism for $n\geq 3$. According to Theorem~\ref{thm:CFS} and Lemma~\ref{lem:S-I}, the latter group is finite. Namely, the three conditions of Theorem~\ref{thm:CFS} for $p_i=4n-2$ and $m=5n-1$ are not satisfied. The first condition fails since $p_i=4n-2$. The second and third condition fail by Corollary~\ref{cor:goodwillie}.
\end{proof}

\begin{Rem}
That the Jacobi relation is null-isotopic (Theorem~\ref{thm:phi-jacobi}) cannot be proved by a similar argument as Lemma~\ref{lem:null-iso-10T} because $\pi_0(\Emb_\partial((I^{3n-2})^{\cup 4},I^{4n-1}))$ is infinite. Nevertheless, there is another proof of Theorem~\ref{thm:phi-jacobi} based on the rational classification of links and assuming the Jacobi identity for Whitehead products. With the proof given in \S\ref{s:4-valent}, the proofs of the properties of $\omega_{T_4}$ and the motivation for the technical arguments in higher valence cases are more transparent.
\end{Rem}

From this, we get a family over $I^2$ of embeddings in $\Emb_\partial((I^{4n-3})^{\cup 5},I^{5n-2})$.

\begin{Def}[Quadruple bracket]\label{def:quad-bracket}
We fix a positive integer $q_5$ such that the component-wise connect-sum of $q_5$ copies of $\varphi_{\mathrm{10T}}$ is null-homotopic in $\Omega\Emb_\partial((I^{4n-3})^{\cup 5},I^{5n-2})$ as in Lemma~\ref{lem:null-iso-10T}. Let $\varphi^{(5)}=\{\varphi^{(5)}_s\}$, $\varphi^{(5)}_s\in \Omega\Emb_\partial((I^{4n-3})^{\cup 5},I^{5n-2})$ denote a null-homotopy of the loop $(\varphi_{\mathrm{10T}})^{\# q_5}$ such that $\varphi^{(5)}_0=\mathrm{const}_\iota$ and $\varphi^{(5)}_1=(\varphi_{\mathrm{10T}})^{\# q_5}$.
\end{Def}

We will fix a relative homotopy class of $\varphi^{(5)}$ in Assumption~\ref{assum:single-comp}.

\subsection{10T-link as a sum of chains}\label{ss:10T-chain}

To realize the LHS of the 10T relation as a sum of chains, we need to work with 3-chains instead of 1-chains for the loop since we need 3 dimensions to embed $L_4$. Namely, we want that a sum of ten 3-chains in the delooping $B^2\Emb_\partial(I^{4n-3},I^{5n-2}-(I^{4n-3})^{\cup 4})$ gives rise to a pointed 3-loop. Here, instead of doing so, we consider the graphing map
\[ \mathrm{graph}\colon \Omega^3\Emb_\partial(I^{4n-5},I^{5n-4}-(I^{4n-5})^{\cup 4}))
\to \Omega\Emb_\partial(I^{4n-3},I^{5n-2}-(I^{4n-3})^{\cup 4}))\]
and see that the 10T-link has an explicit lift in $\Omega^3\Emb_\partial(I^{4n-5},I^{5n-4}-(I^{4n-5})^{\cup 4}))$. This is related to the deloopings of a component of $(\varphi_{\mathrm{10T}})^{\# q_5}$ by the following homotopy commutative diagram (see Lemma~\ref{lem:suspension-delooping}):
\[ \xymatrix{
  \Omega^3\Emb_\partial(I^{4n-5},I^{5n-4}-(I^{4n-5})^{\cup 4})) \ar[d]_-{\mathrm{graph}} \ar[rrd]^-{\Sigma_{1,2,3,4}} & & \\
  \Omega\Emb_\partial(I^{4n-3},I^{5n-2}-(I^{4n-3})^{\cup 4})) \ar[rr]_-{\mathrm{deloopings}} & & 
  \Omega^3\Emb_\partial(I^{4n-5},I^{5n-2}-(I^{4n-3})^{\cup 4}))
}\]
Since we will deloop components to construct general brackets, 3-chains in $\Emb_\partial(I^{4n-5},I^{5n-2}-(I^{4n-3})^{\cup 4}))$ is enough for our purpose. The following is a consequence of the proof of Lemma~\ref{lem:deform-to-loop}.
\begin{Lem}\label{lem:Sigma-triple-loop}
Let $n\geq 3$. The string link $(\varphi_{\mathrm{10T}})^{\# q_5}$, considered as an element of $\Omega\Emb_\partial(I^{4n-3},I^{5n-2}-(I^{4n-3})^{\cup 4}))$, admits a lift to $\Omega^3\Emb_\partial(I^{4n-5},I^{5n-4}-(I^{4n-5})^{\cup 4}))$. 
\end{Lem}
\begin{proof}
We saw in the proof of Lemma~\ref{lem:deform-to-loop} that the embedding $\Sigma_{\mathrm{10T}}\colon I^{4n-2}\to I^{5n-1}-(I^{4n-2})^{\cup 4}$ can be arranged into a double loop $I\times I\to \Emb_\partial(I^{4n-4},I^{5n-3}-(I^{4n-4})^{\cup 4})$. For a similar reason regarding the suspension structure of the iterated binary brackets, it can be further arranged into a triple loop $I^3\to \Emb_\partial(I^{4n-5},I^{5n-4}-(I^{4n-5})^{\cup 4})$. Here, we need the fact that the null-isotopy in $\Emb_\partial(I^{4n-3},I^{5n-2}-(I^{4n-3})^{\cup 4})$ for a tree with a 4-valent vertex was constructed by at least $(4n-3)-(3n-2)=n-1$ suspensions, which is at least 2 if $n\geq 3$. In particular, the cobordism of type (b) in the proof of Lemma~\ref{lem:deform-to-loop} is a homotopy in $\Omega^2\Emb_\partial(I^{4n-5},I^{5n-4}-(I^{4n-5})^{\cup 4})$.
\end{proof}

\begin{Lem}\label{lem:5-cube-decomp}
Let $n\geq 3$. The deloopings of $(\varphi_{\mathrm{10T}})^{\# q_5}$ of the following type
\[ \begin{split}
&\Omega(4n-3,4n-3,4n-3,4n-3,4n-3;5n-2)\\
&\to \Omega^3(4n-3,4n-3,4n-3,4n-3,4n-5;5n-2) 
\end{split}\]
can be parametrized as follows. There is a decomposition
\[ I^3=A\cup \bigcup_{i=1}^{10}K_i \]
into closed subsets $A,K_i$ of $I^3$ satisfying the following.
\begin{enumerate}
\item $K_i$ is a closed 3-cell.
\item $\bigcup_{i=1}^{10}K_i$ is a cell decomposition of $\bigcup_{i=1}^{10}A_i\cup\bigcup_{j=1}^{15}B_j$ in \S\ref{ss:5-step1} obtained by cutting the 1-handles $B_j\cong D^2\times I$ along its cocore $D^2\times\{\frac{1}{2}\}$.
\item $A$ is the closure of $I^3-\bigcup_{i=1}^{10}K_i$.
\item The restriction of $\varphi^{(5)}$ to $A$ is the constant family at the standard inclusion
\[ (I^{4n-3})^{\cup 4}\cup I^{4n-5}\hookrightarrow I^{5n-2}. \]
\item The restriction of $\varphi^{(5)}$ to $K_i$ gives the $i$-th term in the 10T-link.
\end{enumerate}
\end{Lem}
\begin{proof}
The lemma is almost straightforward from the construction of $\Sigma_{\mathrm{10T}}$. We consider $I^{4n-2}$ as $I^3\times I^{4n-5}$. The thickened Lie-hedron $N_{L_4}$ in $I^{4n-2}$ can be isotoped so that it is in a small neighborhood of $I^3\times \{(\frac{1}{2},\ldots,\frac{1}{2})\}$. Then we do the 6 steps in \S\ref{ss:10T-link} to obtain the desired family.
\end{proof}
\begin{Rem}
We will apply suspensions and deloopings to $\varphi^{(5)}$ several times to define sugery on $\Psi_5$-graphs with $k$ or $k-1$ dimensional leaves in $I^{2k}$ in Example~\ref{ex:5-valent-family}. 
If $2k$ is sufficiently high, the total number $3k-7+j$ of deloopings needed is at least 2, and Lemma~\ref{lem:5-cube-decomp} can be applied.
\end{Rem}

If we replace the ``desuspension'' process in the proof of Lemma~\ref{lem:Sigma-triple-loop} to turn $\Sigma_{\mathrm{10T}}$ into a triple loop in $\Emb_\partial(I^{4n-5},I^{5n-4}-(I^{4n-5})^{\cup 4})$ with the natural map 
\[ \Omega\Emb_\partial(I^{4n-3},I^{5n-2}-(I^{4n-3})^{\cup 4})\to \Omega^3B^2\Emb_\partial(I^{4n-3},I^{5n-2}-(I^{4n-3})^{\cup 4}), \]
we obtain the following. 
\begin{Lem}\label{lem:Sigma-triple-loop2}
Let $n\geq 3$. The string link $(\varphi_{\mathrm{10T}})^{\# q_5}$, considered as an element of 
\[\Omega^3B^2\Emb_\partial(I^{4n-3},I^{5n-2}-(I^{4n-3})^{\cup 4}),\]
can be parametrized as follows.
There is a decomposition
\[ I^3=A\cup \bigcup_{i=1}^{10}K_i \]
into closed subsets $A,K_i$ of $I^3$ satisfying the following.
\begin{enumerate}
\item $K_i$ is a closed 3-cell.
\item $\bigcup_{i=1}^{10}K_i$ is a cell decomposition of $\bigcup_{i=1}^{10}A_i\cup\bigcup_{j=1}^{15}B_j$ in \S\ref{ss:5-step1} obtained by cutting the 1-handles $B_j\cong D^2\times I$ along its cocore $D^2\times\{\frac{1}{2}\}$.
\item $A$ is the closure of $I^3-\bigcup_{i=1}^{10}K_i$.
\item The restriction of $(\varphi_{\mathrm{10T}})^{\# q_5}$ to $A$ is the constant loop at the base point.
\item The restriction of $(\varphi_{\mathrm{10T}})^{\# q_5}$ to $K_i$ gives the $i$-th term in the 10T-link.
\end{enumerate}\end{Lem}

\subsection{Prescribing the class of the null-isotopy of the 10T-link.}\label{ss:prescribe-null-iso-10T}

Now we discuss about the rational uniqueness of the null-isotopy of Lemma~\ref{lem:null-iso-10T}. 

\begin{Lem}\label{lem:non-unique-null-iso-psi4}
For $n\geq 4$, the null-isotopy path $\varphi^{(5)}$ in $\Omega\Emb_\partial((I^{4n-3})^{\cup 5},I^{5n-2})$ is unique modulo torsion and connected sum of a knot class from the rank 1 group $\pi_0(\Emb_\partial(I^{4n-1},I^{5n}))$ inside a small ball. 
\end{Lem}
\begin{proof}
Studying $\pi_0$ of the space of paths in $\Omega\Emb_\partial((I^{4n-3})^{\cup 5},I^{5n-2})$ from the basepoint to the 10T-link is equivalent to studying $\pi_2(\Emb((I^{4n-3})^{\cup 5},I^{5n-2}))$.
It follows from Corollary~\ref{cor:goodwillie} that 
\[ \pi_2(\Emb((I^{4n-3})^{\cup 5},I^{5n-2}))\to \pi_1(\Emb_\partial((I^{4n-2})^{\cup 5},I^{5n-1}))\to 
\pi_0(\Emb_\partial((I^{4n-1})^{\cup 5},I^{5n}))
 \]
are isomorphisms for $n\geq 4$. 
By Theorem~\ref{thm:CFS}, the group $\pi_0(\Emb_\partial((I^{4n-1})^{\cup 5},I^{5n}))$ is infinite, but by Corollary~\ref{cor:goodwillie} the free part is generated by elements given by the rank 1 group $\pi_0(\Emb_\partial(I^{4n-1},I^{5n}))$ added at each component.
\end{proof}

The ambiguity in Lemma~\ref{lem:non-unique-null-iso-psi4} may obstruct the relative Brunnian property of the $\Psi_5$-surgery. 
Below, we remove the ambiguity of the choice of the null-isotopy by making one more assumption (Assumption~\ref{assum:single-comp} below). To fix the relative homotopy class of the null-homotopy $\varphi^{(5)}$, we take a Brunnian null-isotopy of (the family) $\Sigma_{\mathrm{10T}}$ in the cubical diagram $\{\calB(S)\}_{S\subset \{1,2,3,4\}}$, where $\calB(S)$ is the homotopy fiber of $\calE(\emptyset)\to \calE(S)=\Emb_\partial^\fr(I^{4n-3},I^{5n-2}-\bigcup_{\lambda\notin S}\sigma_\lambda)$. Suppose we have a prescribed Brunnian null-isotopy of an iterated embedded Whitehead product ($\Psi_3$-surgery): We saw in the proof of Lemma~\ref{lem:brunnian-3} that each term in the Jacobi relation has a prescribed Brunnian null-isotopy. 

\begin{Lem}[Prescribed Brunnian null-isotopies]\label{lem:pres-brunnian-10T}
The 10T-link $\varphi_{\mathrm{10T}}$, considered as an element of $\Omega\Emb_\partial((I^{4n-3})^{\cup 5},I^{5n-2})$, has a system of Brunnian null-isotopies that extends that of the IHX-links, which are obtained by suspending $\varphi_{\mathrm{IHX}}$ of Theorem~\ref{thm:phi-jacobi}. In particular, the component $\Sigma_{\mathrm{10T}}$ is unknotted as a loop of embeddings in $\Omega\calE(\{1,2,3,4\})=\Omega\Emb_\partial(I^{4n-3}, I^{5n-2})$. 
\end{Lem}
\begin{proof}
We construct a system of Brunnian null-isotopies $\{\gamma_{\mathrm{10T}}(S)\}_S$ of $\Sigma_{\mathrm{10T}}$ extending the previous ones as follows.

A Brunnian null-isotopy of $\Psi_4$-surgery: As in the proof of Lemma~\ref{lem:brunnian-3}, the prescribed Brunnian null-isotopies of the terms of the Jacobi relation and the $\Psi_4$-surgery give a ``lens'', which extends to a family over $(D^n)^{\times 3}$. The extension gives a prescribed Brunnian null-isotopy of the $\Psi_4$-surgery relative to that of the Jacobi relation in the previous item.

A Brunnian null-isotopy of a composition of a $\Psi_4$-surgery and a $\Psi_3$-surgery: Lemma~\ref{lem:compos-brun} gives Brunnian null-isotopies for the terms like $[[[a,b],c],d]$ or $[[a,b,c],d]$. This gives rise to a Brunnian null-isotopy of each term of the 10T relation. The compatibility of the $\Psi_4$-surgery with lower excess case implies that the Brunnian null-isotopies for the 10 terms are glued together to form a Brunnian null-isotopy of $\Sigma_{\mathrm{10T}}$.
\end{proof}

\begin{Assum}\label{assum:single-comp}
We assume that the null-homotopy $\varphi^{(5)}$ of the 10T-link as a path in the loop space $\Omega\Emb_\partial(I^{4n-3},I^{5n-2})$ is relatively homotopic to that of the prescribed one given in Lemma~\ref{lem:pres-brunnian-10T}.
\end{Assum}

\subsection{Brunnian property for $\varphi^{(5)}$}\label{ss:brunnian-4}

As in \S\ref{ss:brunnian-3}, we need only to check the relative Brunnian property of null-isotopies for the cubical diagram $\{\calB(S)\}_{S\subset \Omega'=\{1,2,3,4\}}$ under the assumption that the first four components of string links are the standard inclusions.

\begin{Lem}[Relative Brunnian property]\label{lem:brunnian-4}
Suppose $n\geq 7$. There is a positive integer $r_5$ such that the $r$-fold connected sum $(\varphi^{(5)})^{\# r_5}$ of the null-isotopy $\varphi^{(5)}$ of $(\varphi_{\mathrm{10T}})^{\# q_5}$ has a Brunnian property relative to the prescribed system of Brunnian null-isotopies for $(\varphi_{\mathrm{10T}})^{\# q_5r_5}$. 
\end{Lem}
\begin{proof}
We only consider the Brunnian null-isotopy of $\Sigma_{\mathrm{10T}}$ in the 10T-link $(\varphi_{\mathrm{10T}})^{\# q_5}$, under the assumption that $\sigma_1\cup\sigma_2\cup\sigma_3\cup \sigma_4$ is standard. Other cases are similar to this case. According to Definition~\ref{def:rel-brunnian}, it suffices to prove that the following data are homotopic in $\Omega\holim{\emptyset\neq S\subset\Omega'}{\calB(S)}\simeq \holim{\emptyset\neq S\subset\Omega'}{\Omega\calB(S)}$ ($\Omega'=\{1,2,3,4\}$) up to taking connected-sums of finite copies:
\begin{enumerate}
\item The point of $\holim{\emptyset\neq S\subset\Omega'}{\Omega\calB(S)}$ induced from $\varphi^{(5)}\in\Omega\calB(\emptyset)$ by the natural map $\Omega\calB(\emptyset)\to \holim{\emptyset\neq S\subset\Omega'}{\Omega\calB(S)}$. (Recall that $\calB(S)$ is the space of pairs $(f,\gamma)$ of $f\in\calE(\emptyset)$ and a null-isotopy (path) $\gamma$ in $\calE(S)=\Emb_\partial(I^{4n-3},I^{5n-2}-\bigcup_{\lambda\not\in S}I^{4n-3}_\lambda)$ from $f$.)
\item The prescribed system $\{\gamma_{\mathrm{10T}}(S)\colon \Delta^{|S|-1}\to \Omega\calB(S)\}_{\emptyset\neq S\subset \Omega'}$ of Brunnian null-isotopies for the 10T-link given in Lemma~\ref{lem:pres-brunnian-10T}. This gives an element of $\holim{\emptyset\neq S\subset\Omega'}{\Omega\calB(S)}$.
\end{enumerate}
We construct a homotopy $\{h(S)\}_{\emptyset\neq S\subset \Omega'}$ between (1) and (2) inductively on $|S|$. 

Suppose $S=\{i\}$, $1\leq i\leq 4$. We construct a homotopy $h(\{i\})$ in $\Omega\calB(\{i\})$ between $\varphi^{(5)}$ and the prescribed null-isotopy $\gamma_{\mathrm{10T}}(\{i\})$. 
Note that the homotopy classes of paths in $\Omega\calE(S)$ relative to fixed endpoints are in one-to-one correspondence with $\pi_1(\Omega\calE(S))=\pi_2(\calE(S))$, which is isomorphic to $\pi_0(\Emb_\partial(I^{4n-1},I^{5n}-\bigcup_{\lambda\notin S}I^{4n-1}_\lambda))$ by Corollary~\ref{cor:goodwillie}. By Theorem~\ref{thm:CFS} and the proof of Lemma~\ref{lem:non-unique-null-iso-psi4}, two classes in the group $\pi_0(\Emb_\partial(I^{4n-1},I^{5n}-\bigcup_{\lambda\notin S}I^{4n-1}_\lambda))$, which is isomorphic to the subgroup of $\pi_0(\Emb_\partial((I^{4n-1})^{\cup 5},I^{5n}))$ of classes of embeddings with the first four components standard, are the same if their images in $\pi_0(\Emb_\partial(I^{4n-1},I^{5n}))$ for the distinguished component are the same. According to Assumption~\ref{assum:single-comp}, the latter condition is satisfied. Thus we can find a homotopy $h(\{i\})$ between the two classes.

Suppose $S=\{i,j\}$, $1\leq i< j\leq 4$. We construct a homotopy in $\Omega\calB(\{i,j\})$ between the constant map $\const_{\varphi^{(5)}}$ at the loop $\varphi^{(5)}$ (of null-isotopies) and the prescribed 1-parameter family $\gamma_{\10T}(\{i,j\})$ in $\Omega\calB(\{i,j\})$. The choices of homotopies in $\Omega\calB(\{i\})$ and $\Omega\calB(\{j\})$, and $\gamma_{\10T}(\{i,j\})$ give a triangular loop $\partial\Delta^2\to \Omega\calB(\{i,j\})$, which we denote by $\lambda_{ij}$ (Figure~\ref{fig:fill-cone}, left). This is a family over $\partial\Delta^2$ of paths in $\Omega\calE(\{i,j\})$ from the 10T-link to the standard inclusion. To arrange that it is null-homotopic, we consider the homotopy classes of paths in $\Omega(\Omega\calE(\{i,j\}))$ relative to the fixed endpoints that are in one-to-one correspondence with $\pi_1(\Omega^2\calE(\{i,j\}))=\pi_3(\calE(\{i,j\}))$, which is isomorphic to $\pi_0(\Emb_\partial(I^{4n},I^{5n+1}-\bigcup_{\lambda\notin\{i,j\}}I^{4n}_\lambda))$ by Corollary~\ref{cor:goodwillie}. By Theorem~\ref{thm:CFS}, two classes in the group $\pi_0(\Emb_\partial(I^{4n},I^{5n+1}-\bigcup_{\lambda\notin \{i,j\}}I^{4n}_\lambda))$, which is isomorphic to the subgroup of $\pi_0(\Emb_\partial((I^{4n})^{\cup 5},I^{5n+1}))$ of classes of embeddings with the first four components standard, are rationally the same if their images in $\pi_0(\Emb_\partial(I^{4n},I^{5n+1}))\otimes\Q$ obtained by extracting the distinguished component are the same. Hence the triangular loop $\lambda_{ij}\colon\partial\Delta^2\to \Omega\calB(\{i,j\})$ is rationally null-homotopic if and only if the corresponding element of $\pi_0(\Emb_\partial(I^{4n},I^{5n+1}))$ is rationally trivial.
This condition can be satisfied for $(i,j)$ if we alter $h(\{i\})$ and/or $h(\{j\})$ by connect-summing loops in small balls if necessary. Then we can find a homotopy $h(\{i,j\})$ after possibly taking connected-sums of finite copies. 
\begin{figure}[h]
\[ \includegraphics[height=27mm]{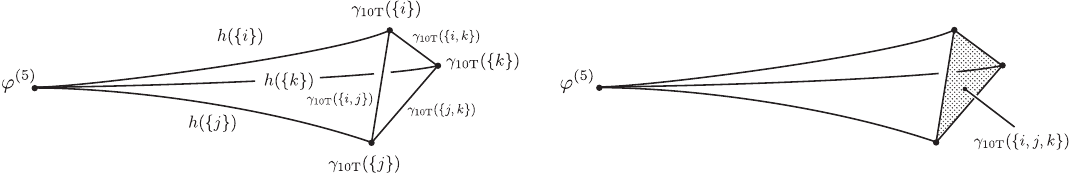} \]
\caption{Left: The homotopies among paths with fixed endpoints, where $\gamma_{\mathrm{10T}}(S)$ is in $\Omega\calB(S)$. Right: The induced homotopies in $\Omega P\Emb_\partial(I^{4n-3},I^{5n-2})$. The faces of the cone are filled if the homotopies $h(\{i\})$, $h(\{j\})$, $h(\{k\})$ are altered suitably.}\label{fig:fill-cone}
\end{figure}

We need to show that the modifications of $h(\{k\})$ can be done compatibly for all the triangles $\lambda_{ij}\colon \partial\Delta^2\to \Omega\calB(\{i,j\})$. Let $P\Emb_\partial(I^{4n-3},I^{5n-2})$ denote the space of null-isotopy paths in $\Emb_\partial(I^{4n-3},I^{5n-2})$, which end at the standard inclusion. Let $\iota_{ij}\colon \calB(\{i,j\})\to P\Emb_\partial(I^{4n-3},I^{5n-2})$ be the map induced by the inclusion $I^{5n-2}-\bigcup_{\lambda\notin\{i,j\}}I^{4n-3}_\lambda\to I^{5n-2}$. It suffices to see that the triangles $\overline{\lambda}_{ij}:=\iota_{ij}\circ\lambda_{ij}\colon \partial\Delta^2\to \Omega P\Emb_\partial(I^{4n-3},I^{5n-2})$ can be made null-homotopic for all $(i,j)$ simultaneously. Since $\gamma_{\mathrm{10T}}$ gives an element of $\holim{\emptyset\neq S\subset \Omega'}{\Omega\calB(S)}$, the homotopies $\gamma_{\mathrm{10T}}(\{i,j\})\colon \Delta^1\to \Omega\calB(\{i,j\})$ for $1\leq i<j\leq 4$ induce those in $\Omega\calB(\Omega')=\Omega P\Emb_\partial(I^{4n-3},I^{5n-2})$, which can be extended to the family $\gamma_{\mathrm{10T}}(\Omega')\colon \Delta^3\to \Omega P\Emb_\partial(I^{4n-3},I^{5n-2})$. We know by Lemma~\ref{lem:null-iso-10T} that $\gamma_{\mathrm{10T}}(\Omega')$ can be further extended to the unreduced cone $C\Delta^3=\Delta^4$ whose cone point is mapped to $\varphi^{(5)}$ (Figure~\ref{fig:fill-cone}, right). Indeed, by Lemma~\ref{lem:null-iso-10T}, there is a path in $\Omega P\Emb_\partial(I^{4n-3},I^{5n-2})$ from $\varphi^{(5)}$ to the barycenter of the 3-simplex. The union of this path and the 3-simplex gives the extension to a family over $\Delta^4$.
We may alter the relative homotopy classes of $h(\{k\})$ so that the induced triangles $\overline{\lambda}_{ij}$ are homotopic to the restrictions of the corresponding triangles in the family over $\Delta^4$. Then $\overline{\lambda}_{ij}$ are null-homotopic, and the homotopies $h(\{i,j\})$ are constructed.

Next cases for $|S|=3,4$ are almost the same as the previous case. We can alter the homotopies chosen in the previous step so that we can find a homotopy in the current step. Note that to apply Corollary~\ref{cor:goodwillie}, we need that $|S|+1\leq n-2$, which is satisfied if $n\geq 7$. 
\end{proof}

\subsection{Symmetry property for $\varphi^{(5)}$}\label{ss:cyclic-psi5}

The path $\varphi^{(5)}$ in $\Omega\Emb_\partial(I^{4n-3},I^{5n-2}-(I^{4n-3})^{\cup 4})$ can be modified into a path in
\[ \Omega\Emb_\partial((I^{4n-3})^{\cup 5},I^{5n-2}) \] 
by turning the four removed components into the trivial family of the standard inclusions.
We assume that the nonstandard component, which is relevant to $\Sigma_{\mathrm{10T}}$, in this family of embeddings of $(I^{4n-3})^{\cup 5}$ is labelled by $0$, and the remaining four standard components are labelled by $1,2,3,4$, respectively. Taking this into account, we denote the above path of embeddings by $\omega_0^{(5)}$. 


Let $\omega_1^{(5)},\omega_2^{(5)},\omega_3^{(5)},\omega_4^{(5)}$ denote the paths in $\Omega\Emb_\partial((I^{4n-3})^{\cup 5},I^{5n-2})$ given by $\xi(\omega_0^{(5)})$, $\xi^2(\omega_0^{(5)})$, $\xi^3(\omega_0^{(5)})$, $\xi^4(\omega_0^{(5)})$, respectively. Let
\[ \omega_{T_5}=\omega_0^{(5)}\#\omega_1^{(5)}\#\omega_2^{(5)}\#\omega_3^{(5)}\#\omega_4^{(5)}, \]
where the sum $\#$ is formed by component-wise concatenation along $I^{4n-3}$. The following lemma is evident from the definition.

\begin{Lem}[Cyclic symmetry]\label{lem:cyclic-5}
The path $\omega_{T_5}$ has a cyclic symmetry property (Definition~\ref{def:cyclic}).
\end{Lem}

\begin{Def}[Quadruple bracket as a chain]\label{def:quadruple}
Let $m_5=5\cdot q_5\cdot r_5\cdot m_4$ (see Definition~\ref{def:quad-bracket} for $q_5$, and Lemma~\ref{lem:brunnian-4} for $r_5$). We denote by $\beta_{T_5}=\,\beta_{T_5}(T_5)$ the chain $\frac{1}{m_5}\omega_{T_5}$ in $\Omega\Emb_\partial((I^{4n-3})^{\cup 5},I^{5n-2})$. For a face graph $\sigma$ of $T_5$, let $\beta_{T_5}(\sigma)$ denote the corresponding summand of $\partial\beta_{T_5}$ given by the iterated bracket for $\sigma$, which is of the form $\frac{1}{m_5}\cdot\frac{m_5}{m_4m_3}\omega_{T_4}'\circ \omega_{T_3}'=\frac{1}{m_4m_3}\omega_{T_4}'\circ \omega_{T_3}'$. 
\end{Def}

\begin{Rem}
The multiplicity of $\omega_{T_5}$ is that of each excess 0 tree in the boundary of $\omega_{T_5}$, which is $m_5=5\cdot q_5\cdot r_5\cdot m_4$.
\end{Rem}

\begin{proof}[Proof of Theorem~\ref{thm:vertex} for $\ell=5$]
It follows from Lemma~\ref{lem:null-iso-10T} and the definition of the 10T-link that the conditions (1) (Compatibility), (2) (Boundary), and (3) ($L_\infty$-relation) are satisfied by $\varphi^{(5)}$. The conditions (4) (Brunnian) and (5) (Cyclic symmetry) follow from Lemmas~\ref{lem:brunnian-4} and \ref{lem:cyclic-5}.  We need only to check that the symmetrization does not break the 10T relation and the decomposition structure of its LHS. The proof is analogous to that of the case $\ell=4$. Namely, The boundary of each term $\omega_i^{(5)}$ represents the 10T relation:
\[ \partial \omega_i^{(5)}=\xi^i(\theta_1)+\cdots+\xi^i(\theta_{10}). \]
In this case, by the symmetry of the $\omega_{T_4}$, there is a term-wise relative isotopy of $\partial \omega_i^{(5)}$ which induce a graph isomorphism from the 10T graph for $\partial \omega_i^{(5)}$ to that for $\partial \omega_0^{(5)}$ (see Remark~\ref{rem:deform-within-box}).

Hence $\partial \omega_{T_5}$ is relatively isotopic to the sum of 5 copies of $\varphi_{\mathrm{10T}}$. Thus the decomposition structure into 10 terms is preserved. This completes the proof.
\end{proof}

The following lemma will be used in the next cases $\ell\geq 6$, especially in the proof of cyclic symmetry property.
\begin{Lem}\label{lem:cyclic-brun-5}
The system of Brunnian null-isotopies of $\omega_{T_5}$ from Lemma~\ref{lem:brunnian-4} has a cyclic symmetry property. 
\end{Lem}
\begin{proof}
This follows from the cyclic symmetry of the system of Brunnian null-isotopies of $\omega_{T_4}$ by Lemma~\ref{lem:brun-borromean}, and from Lemma~\ref{lem:switch-G1-G2}. We apply them along the argument in the proof of Theorem~\ref{thm:vertex} for $\ell=5$.
\end{proof}

\subsection{Suspensions and deloopings of $\omega_{T_5}$}

\begin{Exa}\label{ex:5-valent-family}
Let $d=2k$, and we consider a family of string links associated to the $\Psi_5$-graph in a $2k$-disk, as depicted in Figure~\ref{fig:5-valent-suspend}, left.
\begin{figure}[H]
\[ \includegraphics[height=40mm]{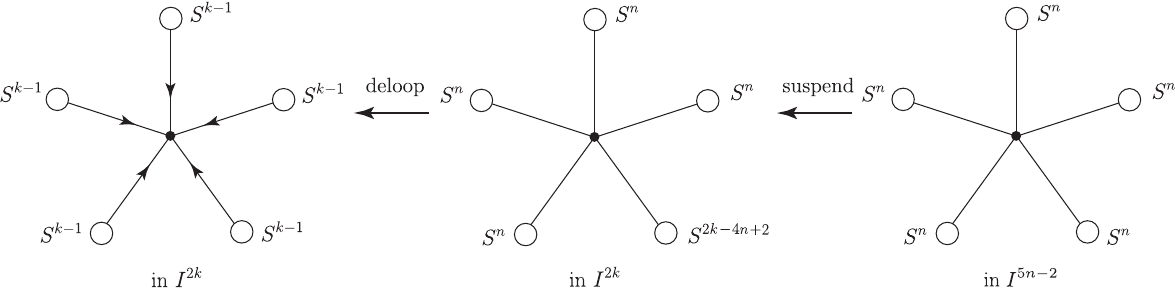} \]
\caption{}\label{fig:5-valent-suspend}
\end{figure}

We start with the null-homotopy of the loop in $\Emb_\partial^\fr((I^{4n-3})^{\cup 5},I^{5n-2})$ we have constructed in \S\ref{ss:10T-link}. Let $\Delta=2k-5n+2$, $a=4n-k-3$, where $n$ is an integer such that $\frac{k+3}{4}\leq n\leq\frac{2k+2}{5}$ (such an $n$ exists when $k=4,5$ or $k\geq 7$). 
We iterate suspensions and deloopings as follows (Figure~\ref{fig:5-valent-suspend}):
\[ \begin{split}
&\Omega(4n-3,4n-3,4n-3,4n-3,4n-3;5n-2)\\
&\to \,\Omega(4n-3+\Delta,4n-3+\Delta,4n-3+\Delta,4n-3+\Delta,4n-3;5n-2+\Delta)\\
&\to \,\Omega\,\Omega^{\tvec{a}}(4n-3-a,4n-3-a,4n-3-a,4n-3-a,4n-3-a;5n-2+\Delta)\\
&= \,\Omega\,\Omega^{\tvec{a}}(k,k,k,k,k;2k),\\
\end{split} \]
where $\vect{a}=(a+\Delta,a+\Delta,a+\Delta,a+\Delta,a)=(k-n-1,k-n-1,k-n-1,k-n-1,4n-k-3)$. 
This sequence of operations can be done simultaneously over the path, and we get a path of families over $S^1\times S^{\tvec{a}}:=S^1\times S^{k-n-1}\times S^{k-n-1}\times S^{k-n-1}\times S^{k-n-1}\times S^{4n-k-3}$ ($\dim S^1\times S^{\tvec{a}}=3k-6$) of framed embeddings 
\[ I^k\cup I^k\cup I^k\cup I^k\cup I^k\to I^{2k}. \]

More generally, if an $\Psi_5$-graph has $j$ outgoing edge(s) and $5-j$ incoming edge(s), then we iterate suspensions and deloopings as follows ($n, \Delta, a$ are the same as above):
\[ \begin{split}
&\Omega(4n-3,4n-3,4n-3,4n-3,4n-3;5n-2)\\
&\to\, \Omega(4n-3+\Delta,4n-3+\Delta,4n-3+\Delta,4n-3+\Delta,4n-3;5n-2+\Delta)\\
&\to \, \left\{\begin{array}{ll}
\Omega\,\Omega^{\tvec{a}}(\underbrace{4n-3-a,\ldots,4n-3-a}_{5-j},\underbrace{4n-4-a,\ldots,4n-4-a}_{j};5n-2+\Delta)& (j\geq 1),\\
\Omega\,\Omega^{\tvec{a}}(\underbrace{4n-3-a,\ldots,4n-3-a}_{5};5n-2+\Delta)& (j=0),
\end{array}\right.\\
&= \, \Omega\,\Omega^{\tvec{a}}(\underbrace{k,\ldots,k}_{5-j},\underbrace{k-1,\ldots,k-1}_{j};2k),\text{ where}
\end{split} \]
\[ \vect{a}=\left\{\begin{array}{ll}
(\underbrace{a+\Delta,\ldots,a+\Delta}_{5-j},\underbrace{a+\Delta+1,\ldots,a+\Delta+1}_{j-1},a+1) & (j\geq 1),\\
(\underbrace{a+\Delta,\ldots,a+\Delta}_{4},a) & (j=0).
\end{array}\right.
\]
We get a path of families over $S^1\times S^{\tvec{a}}$ ($\dim S^1\times S^{\tvec{a}}=3k-6+j$) of framed embeddings 
\[ (I^k)^{\cup 5-j}\cup (I^{k-1})^{\cup j}\to I^{2k}. \]
\end{Exa}

\begin{Rem}\label{rem:n-l-valent}
The $\Psi_\ell$-surgery for $\ell\geq 6$ can be obtained similarly. Namely, we will construct the basic bracket for an $\ell$-valent vertex in \S\ref{s:6-valent} which gives a path in $\Omega^{\ell-4}\Emb_\partial((I^{(\ell-1)n-(\ell-2)})^{\cup \ell},I^{\ell n-\ell+3})$. Starting from this, we iterate suspensions and deloopings as follows. Let $\Delta=2k-\ell n+\ell-3$, and $a=(\ell-1)n-k-(\ell-2)$, where $n$ is an integer such that 
\begin{equation}\label{eq:n-l-valent}
 \frac{k+(\ell-2)}{\ell-1}\leq n\leq \frac{2k+(\ell-3)}{\ell}. 
\end{equation}
Such an $n$ exists if $2k$ is sufficiently large. If the number of outgoing edges $j=0$, 
\[ \begin{split}
&\Omega^{\ell-4}((\ell-1)n-(\ell-2),\ldots,(\ell-1)n-(\ell-2);\ell n-\ell+3)\\
&\to \,\Omega^{\ell-4}((\ell-1)n-(\ell-2)+\Delta,\ldots,(\ell-1)n-(\ell-2)+\Delta,(\ell-1)n-(\ell-2);\ell n-\ell+3+\Delta)\\
&\to \,\Omega^{\ell-4}\,\Omega^{\tvec{a}}((\ell-1)n-(\ell-2)-a,\ldots,(\ell-1)n-(\ell-2)-a;\ell n-\ell+3+\Delta)\\
&= \,\Omega^{\ell-4}\,\Omega^{\tvec{a}}(k,\ldots,k;2k),\\
\end{split} \]
where $\bvec{a}=(a+\Delta,\ldots,a+\Delta,a)$, 
and similarly for other $j$ in $1\leq j\leq \ell$, for which
\[ \vect{a}=(\underbrace{a+\Delta,\ldots,a+\Delta}_{\ell-j},\underbrace{a+\Delta+1,\ldots,a+\Delta+1}_{j-1},a+1), \]
and $\dim{S^{\ell-4}\times S^{\tvec{a}}}=(\ell-2)k-\ell-1+j$. 
The inequality (\ref{eq:n-l-valent}) has an integer solution $n$ if $\frac{2k+(\ell-3)}{\ell}-\frac{k+(\ell-2)}{\ell-1}\geq 1$, or equivalently, $2k\geq \frac{2\ell^2+2\ell-6}{\ell-2}$. On the other hand, we need to assume $n\geq 2\ell-3$ to obtain Theorem~\ref{thm:vertex}, in which we apply Theorem~\ref{thm:CFS} and Corollary~\ref{cor:goodwillie}. It gives the restriction $2\ell-3\leq \frac{2k+(\ell-3)}{\ell}$, which is equivalent to 
\[ 2k\geq 2\ell^2-4\ell+3. \]
This implies the condition $2k\geq \frac{2\ell^2+2\ell-6}{\ell-2}$ above if $\ell\geq 4$. Since an excess $m$ graph may have $(m+3)$-valent vertices, the above constructions work if $2k\geq 2m^2+8m+9$.
\end{Rem}

\section{6-valent or higher vertex}\label{s:6-valent}

We prove Theorem~\ref{thm:vertex} for $\ell\geq 6$. The constructions of the basic brackets for vertices of valence 6 or higher are almost parallel to the case of 5-valent vertex. The existence of a system of Brunnian null-isotopies clarifies the necessary underlying structure of the inductive construction. In this section, we will only describe what is different for higher vertices from the previous cases.

\subsection{Lie-hedron}\label{ss:lie-hedron}

In \S\ref{s:5-valent}, we used the Lie-hedron $L_4$, which is a graph with 10 vertices. We use similar complexes naturally defined for the poset of trees with more leaves. 
\begin{Def}[Lie-hedron\footnote{This is the ``link of the origin'' of the ``tree space'' studied by Billera, Holmes, Vogtmann in \cite{BHV}.}]
For $T\in \calP_{T_\ell}\setminus\{T_\ell\}$, let $E^{\mathrm{int}}(T)$ be the set of internal edges of $T$, which are disjoint from the leaves, and we define
\[ \calM_T=\{(\mu_e)_e\in [0,\infty)^{E^{\mathrm{int}}(T)}\mid \textstyle\sum_e\mu_e=1\}\cong \Delta^{|E^{\mathrm{int}}(T)|-1}. \]
This is the (closure of the) moduli space of metric trees with a normalization condition. The codimension one faces of $\calM_T$ are canonically diffeomorphic to $\calM_{T/e}$ for face trees $T/e$. 
For $\ell\geq 4$, we define the $(\ell-4)$-dimensional geometric simplicial complex $L_{\ell-1}$ by
\[ L_{\ell-1}=\coprod_{T\in\calP_{T_\ell}\setminus\{T_\ell\}}\calM_T\Big/{\sim}, \]
where we identify the face of $\calM_T$ for the face tree $T'$ of codimension $\geq 1$ with $\calM_{T'}$. In other words, $L_{\ell-1}$ is the geometric realization of the full subcategory $\calP_{T_\ell}\setminus\{T_\ell\}$ of $\calP_{T_\ell}$.
\end{Def}
For example, $L_3$ is the disjoint union of three points, and $L_4$ is the graph shown in Figure~\ref{fig:Lie-hedron}.
\subsection{Surgery on 6-valent vertex}\label{ss:surg-6-valent}

We construct a string link model for the 25 term relation, and a path $\omega_{T_6}$ in the space $\Omega^2\Emb_\partial(D^{5n-4},I^{6n-3}-(I^{5n-4})^{\cup 5})$. 

\subsubsection{Step 1: Construct the 25T-link (analogue of Lemma~\ref{lem:Sigma_10T})}\label{ss:25T}

The Lie-hedron $L_5$ for a 6-valent vertex is a 2-dimensional simplicial complex consisting of 
\begin{itemize}
\item 25 vertices corresponding to the excess 2 trees (with one 5-valent and one 3-valent, or with two 4-valent vertices),
\item 105 edges corresponding to the excess 1 trees (with one 4-valent and two 3-valent vertices), and 
\item 105 triangle faces corresponding to the excess 0 trees (with four 3-valent vertices).
\end{itemize}
We take a $(5n-2)$-dimensional handlebody $N_{L_5}$ embedded in $I^{5n-2}$ obtained by thickening the Lie-hedron $L_5$. (Figure~\ref{fig:L5}.)
\begin{figure}[h]
\[ \includegraphics[height=45mm]{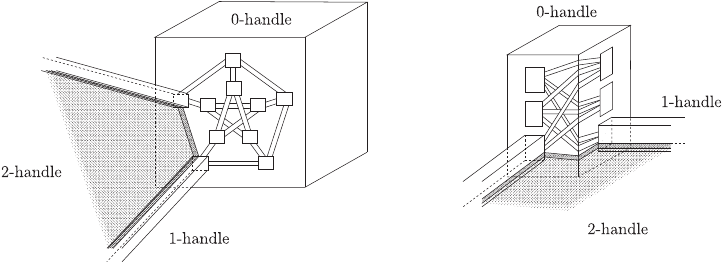} \]
\caption{Parts of the handlebody structure of $N_{L_5}$ near 0-handles. Left: vertex for a tree with one 5-valent and one 3-valent vertex. Right: vertex for a tree with two 4-valent vertices.}\label{fig:L5}
\end{figure}

We construct a ``25T-link'', which is a 2-parameter family of string links $(I^{5n-4})^{\cup 6}\to I^{6n-3}$, by inductively extending from lower-excess families as follows (see Figure~\ref{fig:cocores}). 
\begin{enumerate}
\item We first do some preliminary calculations related to multiplicity. We observe that terms in the 25T-link are of the forms $\omega_{T_5}'\circ\omega_{T_3}'$ or $\omega_{T_4}'\circ\omega_{T_4}'$, where $\omega_{T_i}'$ is obtained from $\omega_{T_i}$ by suspensions. Their multiplicities (of excess 0 faces) are $m_5m_3$ or $m_4^2$, respectively. We define
\[ \mu_6^\partial=\mathrm{lcm}(m_5m_3,m_4^2). \]
This will be the multiplicity of excess 0 faces in the 25T-link.
\item We take an embedding of each 2-handle $D^2\times D^{5n-4}$ of $N_{L_5}$ into $I^{6n-1}-(I^{5n-2})^{\cup 5}$, which is a trivial family over $D^2$ of embeddings $D^{5n-4}\to I^{6n-1}-(I^{5n-2})^{\cup 5}$ representing $\mu_6^\partial$ times an embedded iterated Whitehead product of type $[[[[*,*],*],*],*]$ or $[[[*,*],*],[*,*]]$ at the cocore $(5n-4)$-disk of the 2-handle. 
\item We extend it over the cocore $(5n-3)$-disk transverse to each 1-handle, which is given by $\frac{\mu_6^\partial}{m_4}$ times an excess 1 tree (with one 4-valent vertex and two 3-valent vertices).
\item We extend it over the $(5n-2)$-disk at each 0-handle, which is given by an excess 2 tree ($\omega_{T_5}'\circ\omega_{T_3}'$ or $\omega_{T_4}'\circ\omega_{T_4}'$) multiplied by $\frac{\mu_6^\partial}{m_5m_3}$ or $\frac{\mu_6^\partial}{m_4m_4}$.
\item We extend it further over the rest of $I^{5n-2}$ by the standard inclusions to obtain an embedding $\Sigma_{\mathrm{25T}}\colon I^{5n-2}\to I^{6n-1}-(I^{5n-2})^{\cup 5}$.
\end{enumerate}
\begin{figure}[h]
\[ \includegraphics[height=45mm]{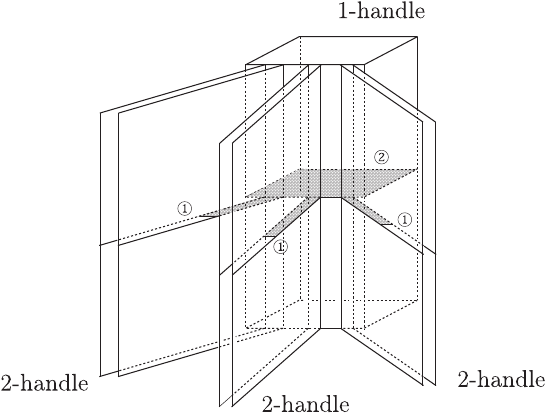} \]
\caption{Cocores of the 2-handles and then those of the 1-handles.}\label{fig:cocores}
\end{figure}

As in the previous case (Lemma~\ref{lem:deform-to-loop}), the 25T-link
\[\varphi_{\mathrm{25T}}:=(\sigma_1\cup\sigma_2\cup\sigma_3\cup\sigma_4\cup\sigma_5)\cup \Sigma_{\mathrm{25T}}\colon(I^{5n-2})^{\cup 6}\to I^{6n-1}\] can be chosen to be the graph of a 2-parameter family of embeddings $(I^{5n-4})^{\cup 6}\to I^{6n-3}$. Namely, the Lie-hedron $L_5$ can be embedded in $I^{5n-2}$ so that for the projection $q\colon I^{5n-2}=I^{5n-3}\times I\to I$, there is a sequence of values $0=c_0<c_1<\cdots<c_r=1$ of $q|_{L_5}$ such that
\begin{itemize}
\item $q^{-1}(c_i)$ intersects edges and faces of $L_5$ transversally,
\item the singularity included in $q^{-1}([c_i,c_{i+1}])$ is either
\begin{enumerate}
\item[(a)] a single vertex of $L_5$, or
\item[(b)] a single critical point of $q$ in the interior of an edge of $L_5$, or
\item[(c)] a single critical point of $q$ in the interior of a face of $L_5$.
\end{enumerate}
\end{itemize}
The intersection of $q^{-1}(c_i)$ and $L_5$ is a graph embedded in $I^{5n-3}$. By the same argument as in the proof of Lemma~\ref{lem:null-iso-10T}, the restriction of $\Sigma_{\mathrm{25T}}$ to $q^{-1}(c_i)$ gives a loop of embeddings of $(I^{5n-4})^{\cup 6}\to I^{6n-3}$. Then the cobordisms for the cases (b) and (c) are given by homotopies in the loop space $\Omega\Emb_\partial((I^{5n-4})^{\cup 6},I^{6n-3})$. Namely, we have a sequence $\{\rho_i\}$ of pointed 2-loops $I^2\to \Emb_\partial((I^{5n-5})^{\cup 6},I^{6n-4})$ for $\varphi_{\mathrm{25T}}|_{q^{-1}(c_i)}$, $i=0,1,\ldots,r$. We then see that the restriction of $\varphi_{\mathrm{25T}}$ to $q^{-1}([c_i,c_{i+1}])$ is given by a homotopy between the 2-loops $\rho_i$ and $\rho_{i+1}$. The 2-loops $\rho_i$ and $\rho_{i+1}$ can be considered as a map $I^2\times \{0,1\}\to \Emb_\partial((I^{5n-5})^{\cup 6},I^{6n-4})$ that can be extended to a map from $I^2\times I$. On the other hand, each excess 0 tree $\tau$ gives a loop $\gamma_\tau\colon I\to \Emb_\partial((I^{5n-5})^{\cup 6},I^{6n-4})$, each excess 1 tree $\tau'$ gives a 2-loop $\gamma_{\tau'}\colon I^2\to \Emb_\partial((I^{5n-5})^{\cup 6},I^{6n-4})$, and the embedding over the corresponding handle of $N_{L_5}$ is given by the pullback of $\gamma_\tau$ or $\gamma_{\tau'}$ by a smooth map $h_\tau\colon I^2\times I\to I$ or $h_{\tau'}\colon I^2\times I\to I^2$, for which the handle is thickened from the preimage of a regular value in $I$ or $I^2$ under $h_\tau$ or $h_{\tau'}$, respectively. The explicit homotopy between $\rho_i$ and $\rho_{i+1}$  can be given by the pullback $h_\tau^*\gamma_\tau$ or $h_{\tau'}^*\gamma_{\tau'}$. The cobordism of type (a) can be chosen as a path of 2-loops because the null-isotopy in $\Emb_\partial((I^{5n-2})^{\cup 6}, I^{6n-1})$ for a tree of excess 2 was constructed by suspensions and can be lifted to that in $\Omega^2\Emb_\partial((I^{5n-4})^{\cup 6},I^{6n-3})$ (see Lemma~\ref{lem:suspension-delooping}). 

To get analogues of Lemmas~\ref{lem:Sigma-triple-loop} and \ref{lem:5-cube-decomp}, we need $5$-fold delooping of the 25T-link. Namely, by the suspension structure of the iterated brackets, the 25T-link, considered as an element of $\Omega^2\Emb_\partial(I^{5n-4},I^{6n-3}-(I^{5n-4})^{\cup 5})$, admits a lift to $\Omega^5\Emb_\partial(I^{5n-7},I^{6n-6}-(I^{5n-7})^{\cup 5})$. Since the total number $4k-9+j$ of deloopings needed ($\dim{S^2\times S^{\tvec{a}}}=4k-7+j$. See Remark~\ref{rem:n-l-valent}) is at least $5-2=3$ when $k\geq 3$, which results in at least 5-fold delooping, we may apply the analogues of Lemmas~\ref{lem:Sigma-triple-loop} and \ref{lem:5-cube-decomp}, and we can embed $L_5$ into the parameter space of the delooped family.

If we replace the ``desuspension'' process in the analogue of Lemma~\ref{lem:Sigma-triple-loop} to turn $\Sigma_{\mathrm{25T}}$ into a 5-fold loop in $\Emb_\partial((I^{5n-7})^{\cup 6},I^{6n-6})$ with the natural map 
\[ \Omega^2\Emb_\partial(I^{5n-4},I^{6n-3}-(I^{5n-4})^{\cup 5})\to \Omega^5B^3\Emb_\partial(I^{5n-4},I^{6n-3}-(I^{5n-4})^{\cup 5}), \]
we obtain the 5-loop in $B^3\Emb_\partial(I^{5n-4},I^{6n-3}-(I^{5n-4})^{\cup 5})$ that is the sum of 25 5-chains. 

\subsubsection{Step 2: Find a null-isotopy $\varphi^{(6)}$ of the 25T-link (analogue of Lemma~\ref{lem:null-iso-10T})}

We use Corollary~\ref{cor:goodwillie} again to see that the natural map
\[ \pi_2(\Emb_\partial((I^{5n-4})^{\cup 6},I^{6n-3}))\to \pi_0(\Emb_\partial((I^{5n-2})^{\cup 6},I^{6n-1})) \]
is an isomorphism for $n\geq 4$. According to Theorem~\ref{thm:CFS} and Corollary~\ref{cor:goodwillie}, the latter group is finite if $n\geq 4$ and $5n-2\not\equiv 3$ (mod 4). Moreover, even if $5n-2\equiv 3$ (mod 4), each single component of the 25T-link has a prescribed null-isotopy, and it is trivial as an embedding $I^{5n-2}\to I^{6n-1}$. Thus the 25T-link is rationally null-isotopic. Let $\varphi^{(6)}$ be a null-isotopy of a multiple of the 25T-link $(\varphi_{\mathrm{25T}})^{\#  q_6}$.

\subsubsection{Step 3: Brunnian and cyclic symmetry property (analogues of Lemmas~\ref{lem:brunnian-4} and \ref{lem:cyclic-5})}\label{ss:brunnian-l=6}

As in the previous case (Lemmas~\ref{lem:brunnian-4} and \ref{lem:cyclic-5}), we may assume the rational Brunnian property and the rational cyclic symmetry for the null-isotopy $\varphi^{(6)}$, if $n$ is sufficiently large and if we make an assumption analogous to Assumption~\ref{assum:single-comp} that the null-isotopy is relatively isotopic to the prescribed null-isotopy for each component. Thus there is a positive integer $r_6$ such that $(\varphi^{(6)})^{\# r_6}$ has a Brunnian property relative to the prescribed system of Brunnian null-isotopies for $(\varphi_{\mathrm{25T}})^{\# q_6}$, and its symmetrization satisfies the cyclic symmetry property. Moreover, we have an analogue of Lemma~\ref{lem:cyclic-brun-5} (cyclic symmetry of the system of Brunnian null-isotopies) proved by using Lemma~\ref{lem:compos-brun}. So we have the symmetrization $\omega_{T_6}=\#_{i=1}^6 \omega_i^{(6)}$ of the null-isotopy path $(\varphi^{(6)})^{\# r_6}$.

Here, the condition for $n$ for the relative Brunnian property is given by $|S|+2\leq n-2$ as in the last part of the proof of Lemma~\ref{lem:brunnian-4}, and $|S|\leq 5$ gives the condition $n\geq \max\{|S|+4\}=9$. 

\begin{Def}
Let $m_6=6\cdot q_6\cdot r_6\cdot \mu_6^\partial$ ($\mu_6^\partial$ from Step 1, $q_6$ from Step 2, and $r_6$ from Step 3 above). We denote by $\beta_{T_6}=\beta_{T_6}(T_6)$ the chain $\frac{1}{m_6}\omega_{T_6}$ in the loop space $\Omega^2\Emb_\partial((I^{5n-4})^{\cup 6},I^{6n-3})$. For a face graph $\sigma$ of $T_6$, let $\beta_{T_6}(\sigma)$ denote the corresponding summand of $\partial\beta_{T_6}$ given by the iterated bracket for $\sigma$, which is of the form $\frac{1}{m_6}\cdot\frac{m_6}{m_{\ell_1}m_{\ell_2}}\omega_{T_{\ell_1}}'\circ \omega_{T_{\ell_2}}'=\frac{1}{m_{\ell_1}m_{\ell_2}}\omega_{T_{\ell_1}}'\circ \omega_{T_{\ell_2}}'$. We put $\omega_\sigma=\omega_{T_{\ell_1}}'\circ \omega_{T_{\ell_2}}'$ and $m_\sigma=m_{\ell_1}m_{\ell_2}$ in this term.
\end{Def}

\begin{proof}[Proof of Theorem~\ref{thm:vertex} for $\ell=6$]
It follows from the definition of the 25T-link that the conditions (1) (Compatibility), (2) (Boundary), and (3) ($L_\infty$-relation) are satisfied by $\varphi^{(6)}$. The conditions (4) (Brunnian) and (5) (Cyclic symmetry) follow from Step 3 above. This completes the proof.
\end{proof}

\par\bigskip

\subsection{Surgery on $\ell$-valent vertex}

Surgery for an $\ell$-valent vertex for $\ell\geq 6$ can be summarized as follows, which is parallel to the case $\ell=6$. We construct a string link model for the $(2^{\ell-1}-\ell-1)$T-relation, and a path $\omega_{T_\ell}$ in the space $\Omega^{\ell-4}\Emb_\partial(D^{(\ell-1)n-(\ell-2)},I^{\ell n-\ell+3}-(I^{(\ell-1)n-(\ell-2)})^{\cup (\ell-1)})$ as follows.

\subsubsection{Step 1: Construct the $(2^{\ell-1}-\ell-1)$T-link (analogue of Lemma~\ref{lem:Sigma_10T})}

The Lie-hedron $L_{\ell-1}$ for an $\ell$-valent vertex is a $(\ell-4)$-dimensional simplicial complex consisting of $2^{\ell-1}-\ell-1$ vertices and cells of dimensions $\leq \ell-4$ parametrized by connected trees with $\ell$ labelled legs. It follows from the Whitney embedding theorem that a $k$-dimensional simplicial complex can be piecewise smoothly embedded in $\R^N$ if $N\geq 2k+1$. Since $(\ell-1)n-2\geq 2(\ell-4)+1=2\ell-7$, we may take a $((\ell-1)n-2)$-dimensional handlebody $N_{L_{\ell-1}}$ embedded in $I^{(\ell-1)n-2}$ obtained by thickening the Lie-hedron $L_{\ell-1}$.  

We construct a ``$(2^{\ell-1}-\ell-1)$T-link'', which is an $(\ell-4)$-parameter family of string links $(I^{(\ell-1)n-(\ell-2)})^{\cup \ell}\to I^{\ell n-\ell+3}$, by inductively extending from lower-excess families.

\begin{enumerate}
\item We first do some preliminary calculations related to multiplicity. We observe that terms in the $(2^{\ell-1}-\ell-1)$T-link are of the forms $\omega_{T_{p+1}}'\circ\omega_{T_{q+1}}'$ ($p+q=\ell$, $p\geq 2$, $q\geq 2$), where $\omega_{T_{p+1}}'$ and $\omega_{T_{q+1}}'$ are obtained by suspensions. Their multiplicities (of excess 0 faces) are $m_{p+1}m_{q+1}$. We define
\[ \mu_\ell^\partial=\mathrm{lcm}(\{m_{p+1}m_{q+1}\mid p+q=\ell, p\geq 2, q\geq 2\}). \]
This will be the multiplicity of excess 0 faces in the $(2^{\ell-1}-\ell-1)$T-link. Clearly, we have $m_{p+1}m_{q+1}\mid \mu_\ell^\partial$ for $p+q=\ell$, $p,q\geq 2$, and moreover, we will define $m_\ell$ as an integer multiple of $\mu_\ell^\partial$, so we will have $m_{p+1}m_{q+1}\mid m_\ell$.
\item The families of string links for the excess 0 trees, which are compositions of several suspended $\omega_{T_3}$s, are arranged along $(\ell-4)$-handle of $N_{L_{\ell-1}}$ with coefficient $\mu_\ell^\partial$. We then extend inductively the families along the $(\ell-5)$-, $(\ell-6)$-,$\cdots$, 0-handles of $N_{L_{\ell-1}}$. The coefficient of the compositions of $\omega_{T_{\ell_1}}',\omega_{T_{\ell_2}}',\ldots,\omega_{T_{\ell_i}}'$ (excess $\ell-2-i$), which is arranged over an $(i-2)$-handle, is the integer
\[ \frac{\mu_\ell^\partial}{m_{\ell_1}m_{\ell_2}\cdots m_{\ell_i}}. \]
\item We extend it further over the rest of $I^{(\ell-1)n-2}$ by the standard inclusions to obtain an embedding $\Sigma_{\mathrm{(2^{\ell-1}-\ell-1)T}}\colon I^{(\ell-1)n-2}\to I^{\ell n-1}-(I^{(\ell-1)n-2})^{\cup (\ell-1)}$.
\end{enumerate}

As in the previous cases, the $(2^{\ell-1}-\ell-1)$T-link can be sliced into elementary paths of $(\ell-4)$-loops in $(I^{(\ell-1)n-(\ell-1)})^{\cup \ell}\to I^{\ell n-\ell+2}$.

To get analogues of Lemmas~\ref{lem:Sigma-triple-loop} and \ref{lem:5-cube-decomp}, we need $2(\ell-4)+1=(2\ell-7)$-fold delooping of the $(2^{\ell-1}-\ell-1)$T-link. Namely, by the suspension structure of the iterated brackets, the $(2^{\ell-1}-\ell-1)$T-link, considered as an element of $\Omega^{\ell-4}\Emb_\partial(I^{(\ell-1)n-(\ell-2)},I^{\ell n-\ell+3}-(I^{(\ell-1)n-(\ell-2)})^{\cup (\ell-1)})$, admits a lift to $\Omega^{2\ell-7}\Emb_\partial(I^{(\ell-1)n-2\ell+5},I^{\ell n-2\ell+6}-(I^{(\ell-1)n-2\ell+5})^{\cup (\ell-1)})$.  Since the total number $(\ell-2)k-\ell-1+j-(\ell-4)$ of deloopings needed ($\dim{S^{\ell-4}\times S^{\tvec{a}}}=(\ell-2)k-\ell-1+j$. See Remark~\ref{rem:n-l-valent}) is at least $2\ell-7-(\ell-4)=\ell-3$ when $k\geq 3$, which results in at least $(2\ell-7)$-fold delooping, we may apply the analogues of Lemmas~\ref{lem:Sigma-triple-loop} and \ref{lem:5-cube-decomp}, and we can embed $L_{\ell-1}$ into the parameter space of the delooped family.

If we replace the ``desuspension'' process in the analogue of Lemma~\ref{lem:Sigma-triple-loop} to turn $\Sigma_{(2^{\ell-1}-\ell-1)\mathrm{T}}$ into a $(2\ell-7)$-fold loop in $\Emb_\partial(I^{(\ell-1)n-2\ell+5},I^{\ell n-2\ell+6}-(I^{(\ell-1)n-2\ell+5})^{\cup (\ell-1)})$ with the natural map 
\[ \begin{split}
&\Omega^{\ell-4}\Emb_\partial(I^{(\ell-1)n-(\ell-2)},I^{\ell n-\ell+3}-(I^{(\ell-1)n-(\ell-2)})^{\cup (\ell-1)})\\
&\to \Omega^{2\ell-7}B^{\ell-3}\Emb_\partial(I^{(\ell-1)n-(\ell-2)},I^{\ell n-\ell+3}-(I^{(\ell-1)n-(\ell-2)})^{\cup (\ell-1)}), 
\end{split}\]
we obtain the $(2\ell-7)$-loop in $B^{\ell-3}\Emb_\partial(I^{(\ell-1)n-(\ell-2)},I^{\ell n-\ell+3}-(I^{(\ell-1)n-(\ell-2)})^{\cup (\ell-1)})$ that is the sum of $2^{\ell-1}-\ell-1$ $(2\ell-7)$-chains.

\subsubsection{Step 2: Find a null-isotopy $\varphi^{(\ell)}$ of the $(2^{\ell-1}-\ell-1)$T-link (analogue of Lemma~\ref{lem:null-iso-10T})}

We use Corollary~\ref{cor:goodwillie} again to see that the natural map
\[ \pi_{\ell-4}(\Emb_\partial((I^{(\ell-1)n-(\ell-2)})^{\cup \ell},I^{\ell n-\ell+3}))\to \pi_0(\Emb_\partial((I^{(\ell-1)n-2})^{\cup \ell},I^{\ell n-1})) \]
is an isomorphism for $n\geq \ell-2$. According to Theorem~\ref{thm:CFS} and Corollary~\ref{cor:goodwillie}, the latter group is finite if $n\geq \ell-2$ and $(\ell-1)n-2\not\equiv 3$ (mod 4). Moreover, even if $(\ell-1)n-2\equiv 3$ (mod 4), each single component of the $(2^{\ell-1}-\ell-1)$T-link has a prescribed null-isotopy, and it is trivial as an embedding $I^{(\ell-1)n-2}\to I^{\ell n-1}$. Thus the $(2^{\ell-1}-\ell-1)$T-link is rationally null-isotopic. Let $\varphi^{(\ell)}$ be a null-isotopy of a multiple $(\varphi_{\mathrm{(2^{\ell-1}-\ell-1)T}})^{\#  q_{\ell}}$ of the $(2^{\ell-1}-\ell-1)$T-link.

\subsubsection{Step 3: Brunnian and cyclic symmetry property (analogues of Lemmas~\ref{lem:brunnian-4} and \ref{lem:cyclic-5})}

As in the previous cases, we may assume the rational Brunnian property and the rational cyclic symmetry for the null-isotopy $\varphi^{(\ell)}$, if $n$ is sufficiently large and if we make an assumption analogous to Assumption~\ref{assum:single-comp} that the null-isotopy is relatively isotopic to the prescribed null-isotopy for each component. Thus there is a positive integer $r_\ell$ such that $(\varphi^{(\ell)})^{\# r_\ell}$ has a Brunnian property relative to the prescribed system of Brunnian null-isotopies for $(\varphi_{\mathrm{(2^{\ell-1}-\ell-1)T}})^{\# q_\ell}$, and its symmetrization satisfies the cyclic symmetry property. Moreover, we have an analogue of Lemma~\ref{lem:cyclic-brun-5} (cyclic symmetry of the system of Brunnian null-isotopies) proved by using Lemma~\ref{lem:compos-brun}. So we have the symmetrization $\omega_{T_\ell}=\#_{i=1}^\ell\omega_i^{(\ell)}$ of the null-isotopy path $(\varphi^{(\ell)})^{\# r_\ell}$.

Here, the condition for $n$ for the relative Brunnian property is given by $|S|+\ell-4\leq n-2$ as in the last part of the proof of Lemma~\ref{lem:brunnian-4}, and $|S|\leq \ell-1$ gives the condition $n\geq \max\{|S|+\ell-2\}=2\ell-3$. 

\begin{Def}\label{def:beta_ell}
Let $m_\ell=\ell\cdot q_\ell\cdot r_\ell\cdot \mu_\ell^\partial$  ($\mu_\ell^\partial$ from Step 1, $q_\ell$ from Step 2, and $r_\ell$ from Step 3 above). We denote by $\beta_{T_\ell}=\beta_{T_\ell}(T_\ell)$ the chain $\frac{1}{m_\ell}\omega_{T_\ell}$ in the loop space $\Omega^{\ell-4}\Emb_\partial((I^{(\ell-1)n-(\ell-2)})^{\cup \ell},I^{\ell n-\ell+3})$. For a face graph $\sigma$ of $T_\ell$, let $\beta_{T_\ell}(\sigma)$ denote the corresponding summand of $\partial\beta_{T_\ell}$ given by the iterated bracket for $\sigma$, which is of the form $\frac{1}{m_\ell}\cdot\frac{m_\ell}{m_{\ell_1}m_{\ell_2}}\omega_{T_{\ell_1}}'\circ \omega_{T_{\ell_2}}'=\frac{1}{m_{\ell_1}m_{\ell_2}}\omega_{T_{\ell_1}}'\circ \omega_{T_{\ell_2}}'$. We put $\omega_\sigma=\omega_{T_{\ell_1}}'\circ \omega_{T_{\ell_2}}'$ and $m_\sigma=m_{\ell_1}m_{\ell_2}$ in this term.
\end{Def}

\begin{proof}[Proof of Theorem~\ref{thm:vertex} for $\ell\geq 6$]
It follows from the definition of the $(2^{\ell-1}-\ell-1)$T-link that the conditions (1) (Compatibility), (2) (Boundary), and (3) ($L_\infty$-relation) are satisfied by $\varphi^{(\ell)}$. The conditions (4) (Brunnian) and (5) (Cyclic symmetry) follow from Step 3 above. This completes the proof.
\end{proof}

\section{Thickening of a vertex surgery: Proof of Theorem~\ref{thm:l-valent-general}}\label{s:thicken}

\subsection{Outline of the thickening}

We have a chain 
\[ \omega_{T_\ell(p,q)}\colon B_{T_\ell(p,q)}\to \Emb_\partial^\fr((I^k)^{\cup p}\cup (I^{k-1})^{\cup q},I^{2k}), \] 
where $B_{T_\ell(p,q)}=I^{\ell-3}\times S^{\tvec{a}}$ for some compact manifold $S^{\tvec{a}}$ of dimension $(\ell-2)k-2\ell+q+3$, obtained from the basic bracket $\omega_{T_\ell}$ for $T_\ell$ by suspensions and deloopings (see \S\ref{ss:suspension}, \S\ref{ss:delooping}, Examples~\ref{ex:3-valent-family}, \ref{ex:4-valent-family}, and \ref{ex:5-valent-family}).
We define a map $\overline{\omega}_{T_\ell(p,q)}\colon \overline{B}_{T_\ell(p,q)}\to \Emb_\partial^\fr((I^k)^{\cup p}\cup (I^{k-1})^{\cup q},I^{2k})$ of Theorem~\ref{thm:l-valent-general} by gradually updating $\omega_{T_\ell(p,q)}$ as follows. Then we will define $\overline{\beta}_{T_\ell(p,q)}=\frac{1}{m_\ell}\overline{\omega}_{T_\ell(p,q)}$.
\subsubsection{Simplification of the domain of $\omega_{T_\ell(p,q)}$}
First we simplify the domain of $\omega_{T_\ell(p,q)}$. We define
\[ 
   \widehat{B}_{T_\ell(p,q)}=I^{\ell-3}\times S^{|\tvec{a}|},\\
\] 
where $|\vect{a}|=(\ell-2)k-2\ell+q+3$. The definition of $\widehat{B}_{T_\ell(p,q)}$ given later in Definition~\ref{def:hat_B} is slightly different, although it is homotopy equivalent to that given here. We will see in Lemma~\ref{lem:brun-cone} that there is a canonical reduction 
\[ \widehat{\omega}_{T_\ell(p,q)}\colon \widehat{B}_{T_\ell(p,q)}\to \Emb_\partial^\fr((I^k)^{\cup p}\cup (I^{k-1})^{\cup q},I^{2k}) \]
of $\omega_{T_\ell(p,q)}$. 
The boundary $\partial I^{\ell-3}\times S^{|\tvec{a}|}$ of $\widehat{B}_{T_\ell(p,q)}$ has a decomposition
\[ \begin{split}
&\partial I^{\ell-3}\times S^{|\tvec{a}|}=A_{T_\ell(p,q)}\cup\bigcup_{\sigma\in\calP_{T_\ell}}K(\sigma)
\end{split}\]
induced from that for the boundary of $\omega_{T_\ell}$ (Lemma~\ref{lem:5-cube-decomp}, \S\ref{ss:25T}).
On the other hand, the iterated brackets for the directed tree $(\sigma,\alpha)\in\overrightarrow{\calP}_{T_\ell(p,q)}$ of excess $\ell-4$, for which $K(\sigma)\cong I^{\ell-4+|\tvec{a}|}$ is codimension 1 in $\widehat{B}_{T_\ell(p,q)}$, is parametrized by the space like $(I^{\ell_1-3}\times S^{|\tvec{a}_1|})\times (I^{\ell_2-3}\times S^{|\tvec{a}_2|})$ etc., which may not be homeomorphic to $K(\sigma)/{A_{T_\ell(p,q)}}$. Thus the faces of $\widehat{\omega}_{T_\ell(p,q)}$ may not agree with the chains obtained from lower degree vertices. To adjust this defect, we will further thicken $\widehat{B}_{T_\ell(p,q)}$ and extend $\widehat{\omega}_{T_\ell(p,q)}$ to the thickening so that the boundary is given by the sum of terms obtained by iterating bracket operations of lower degrees. 

\subsubsection{Multiple chains}
Also, to take multiplicities of terms in chains into account, we define for a compact manifold with corners $W$ its multiplex
\[ nW=\colim{}{\Bigl(\textstyle\coprod^n W\longleftarrow \textstyle\coprod^n\partial W\longrightarrow \partial W\Bigr)}. \]
This can also be defined when the space $W$ may not be a manifold but its subspace $\partial W$ is specified.
If $(\sigma,\alpha)\in \overrightarrow{\calP}_{T_\ell(p,q)}$ is of excess $\lambda$, we put $n_\sigma=2^\lambda$. We define
\[ 
  n_\sigma\widehat{B}_{T_\ell(p,q)}(\sigma,\alpha)=\left\{\begin{array}{ll}
  \prod_{v\in V^{\mathrm{int}}(\sigma)}n_{T(v)}\widehat{B}_{T(v)} & ((\sigma,\alpha)\neq T_\ell(p,q)),\\
  n_{T_\ell}\widehat{B}_{T_\ell(p,q)}/{A_{T_\ell(p,q)}} & ((\sigma,\alpha)=T_\ell(p,q)),
  \end{array}\right.
 \]
where $V^{\mathrm{int}}(\sigma)$ is the set of the internal vertices of $\sigma$. This is to make (\ref{eq:eta}) or the RHS of (\ref{eq:doubling}) an integral chain. The space $n_\sigma\widehat{B}_{T_\ell(p,q)}(\sigma,\alpha)$ parametrizes a chain
\begin{equation}\label{eq:beta(sigma)}
 n_\sigma\widehat{\omega}_{T_\ell(p,q)}(\sigma,\alpha)\colon n_\sigma\widehat{B}_{T_\ell(p,q)}(\sigma,\alpha)\to 
\Emb_\partial^\fr((I^k)^{\cup p}\cup (I^{k-1})^{\cup q},I^{2k}) 
\end{equation}
obtained by the product of $n_{T(v)}\widehat{\omega}_{T(v)}$ for internal vertices $v$ in $\sigma$ with lower degrees than $T_\ell$.

\subsubsection{Thickening of the domain of $n_\sigma\widehat{\omega}_{T_\ell(p,q)}(\sigma,\alpha)$}
We glue the chains $n_\sigma\widehat{\omega}_{T_\ell(p,q)}(\sigma,\alpha)$ together to thicken them as follows. 
The assignment $(\sigma,\alpha)\mapsto n_\sigma\widehat{B}_{T_\ell(p,q)}(\sigma,\alpha)$ defines a functor $\overrightarrow{\calP}_{T_\ell(p,q)}\to \mathrm{Top}$, where the map
$\widehat{B}_{T_\ell(p,q)}(f)\colon n_\sigma\widehat{B}_{T_\ell(p,q)}(\sigma,\alpha)\to n_{\sigma'}\widehat{B}_{T_\ell(p,q)}(\sigma',\alpha')$ ($n_{\sigma'}=2n_\sigma$) for a morphism $f\colon (\sigma,\alpha)\to (\sigma',\alpha')$ in $\overrightarrow{\calP}_{T_\ell(p,q)}$ is induced by collapsing lower skeleton. Note that $\widehat{B}_{T_\ell(p,q)}(f)$ maps all the seats in $n_\sigma\widehat{B}_{T_\ell(p,q)}(\sigma,\alpha)$ onto a single face of the multiplex $n_{\sigma'}\widehat{B}_{T_\ell(p,q)}(\sigma',\alpha')$. There are also maps $n_\sigma\widehat{B}_{T_\ell(p,q)}(\sigma,\alpha)\to K(\sigma)/{A_{T_\ell(p,q)}}\subset \widehat{B}_{T_\ell(p,q)}/{A_{T_\ell(p,q)}}$, which all together define a natural transformation $n_\sigma\widehat{B}_{T_\ell(p,q)}(-)\to K(-)/{A_{T_\ell(p,q)}}$. This gives a commutative diagram:
\[ \xymatrix{
  n_\sigma\widehat{B}_{T_\ell(p,q)}(\sigma,\alpha) \ar[r] \ar[d] & n_{\sigma'}\widehat{B}_{T_\ell(p,q)}(\sigma',\alpha') \ar[d]&\\
  K(\sigma)/{A_{T_\ell(p,q)}} \ar[r]^-{\subset} & K(\sigma')/{A_{T_\ell(p,q)}} \ar[r]^-{\subset}& \widehat{B}_{T_\ell(p,q)}/{A_{T_\ell(p,q)}}
} \]
Also, this has a canonical lift $n_\sigma\widehat{B}_{T_\ell(p,q)}(\sigma,\alpha)\to K^{(n_{T_\ell})}(\sigma)/{A_{T_\ell(p,q)}}$, where 
\[ K^{(n_{T_\ell})}(\sigma)=\left\{\begin{array}{ll}
K(\sigma) & ((\sigma,\alpha)\neq T_\ell(p,q)),\\
n_{T_\ell}\widehat{B}_{T_\ell(p,q)} & ((\sigma,\alpha)=T_\ell(p,q)).
\end{array}\right.\]
\begin{Def}\label{def:barB}
We define $\overline{B}_{T_\ell(p,q)}:=\underset{(\sigma,\alpha)\in\overrightarrow{\calP}_{T_\ell(p,q)}}{\mathrm{hocolim}}\, n_\sigma\widehat{B}_{T_\ell(p,q)}(\sigma,\alpha)$.
Then we obtain a natural map
\[
\overline{B}_{T_\ell(p,q)}\longrightarrow \,\underset{(\sigma,\alpha)\in\overrightarrow{\calP}_{T_\ell(p,q)}}{\mathrm{colim}}\, K^{(n_{T_\ell})}(\sigma)/{A_{T_\ell(p,q)}}
= n_{T_\ell}\widehat{B}_{T_\ell(p,q)}/{A_{T_\ell(p,q)}}.
\]
(See \cite[Definition~4.5]{Du} for a definition of the homotopy colimit.)
\end{Def}

\begin{Rem} Since the space $\overline{B}_{T_\ell(p,q)}$ is constructed by gluing and collapsing several smooth manifolds with corners along boundaries, we may consider $\overline{B}_{T_\ell(p,q)}$ having a structure of a finite CW-complex with a well-defined fundamental class, where the gluing maps are cellular.
\end{Rem}

\subsubsection{Thickening of the chain $n_\sigma\widehat{\omega}_{T_\ell(p,q)}(\sigma,\alpha)$}
We will see in \S\ref{ss:merge}, \ref{ss:compress-codim1}, and \ref{ss:compatible-higher} that the collection $\{n_\sigma\widehat{\omega}_{T_\ell(p,q)}(\sigma,\alpha)\}$ of chains can be extended and glued together to a map
\[ \overline{\omega}_{T_\ell(p,q)}\colon \overline{B}_{T_\ell(p,q)}\to 
\Emb_\partial^\fr((I^k)^{\cup p}\cup (I^{k-1})^{\cup q},I^{2k})\]
such that the diagram
\[ \xymatrix{
  \overline{B}_{T_\ell(p,q)} \ar[rrd]^-{\overline{\omega}_{T_\ell(p,q)}} \ar[d] & &\\
  n_{T_\ell}\widehat{B}_{T_\ell(p,q)}/{A_{T_\ell(p,q)}} \ar[rr]^-{n_{T_\ell}\widehat{\omega}_{T_\ell(p,q)}} & &\Emb_\partial^\fr((I^k)^{\cup p}\cup (I^{k-1})^{\cup q},I^{2k}) & 
} 
\]
is commutative up to a compressible bordisms (see below) near the boundary. Moreover, the compatibility and the Brunnian property for $\overline\omega_{T_\ell(p,q)}$ analogous to Theorem~\ref{thm:vertex} hold. Here, 
\begin{itemize}
\item $\overline{\beta}_{(\sigma,\alpha)}$ is the restriction of $\overline{\beta}_{T_\ell(p,q)}$ to the homotopy colimit taken for the full subcategory of $\overrightarrow{\calP}_{T_\ell(p,q)}$ of directed trees that can be contracted to $(\sigma,\alpha)$,
\item we say that the two chains $\overline{\omega}_{T_\ell(p,q)}$ and $n_{T_\ell}\widehat{\omega}_{T_\ell(p,q)}$ are related by a {\it compressible bordism} if there is a continuous map $\sigma\colon \Sigma\to \Emb_\partial^\fr((I^k)^{\cup p}\cup (I^{k-1})^{\cup q},I^{2k})$ from a finite CW-complex $\Sigma$ with oriented cells such that $\partial\sigma=\overline{\omega}_{T_\ell(p,q)}-n_{T_\ell}\widehat{\omega}_{T_\ell(p,q)}$ as a chain, and the following diagram is commutative.
\[ \xymatrix{
  \Sigma \ar[rd]^-{\sigma} \ar[d] & \\
  \mathrm{Cyl}(\overline{B}_{T_\ell(p,q)}\to n_{T_\ell}\widehat{B}_{T_\ell(p,q)}/A_{T_\ell(p,q)}) \ar[r] & \Emb_\partial^\fr((I^k)^{\cup p}\cup (I^{k-1})^{\cup q},I^{2k})
}\] 
where $\mathrm{Cyl}(\overline{B}_{T_\ell(p,q)}\to n_{T_\ell}\widehat{B}_{T_\ell(p,q)}/{A_{T_\ell(p,q)}})$ is the mapping cylinder of the natural map $\overline{B}_{T_\ell(p,q)}\to n_{T_\ell}\widehat{B}_{T_\ell(p,q)}/{A_{T_\ell(p,q)}}$. 
\end{itemize}

\subsection{Merging the spheres in $B_{T_\ell(p,q)}$ into a single sphere.}\label{ss:merge}

Recall that $B_{T_\ell(p,q)}$ is of the form $I^{\ell-3}\times S^{\tvec{a}}$, where $S^{\tvec{a}}=\prod_{j=1}^\ell S^{\lambda_j}$ for $\vect{a}=(\lambda_1,\ldots,\lambda_\ell)$.
For a subset $J$ of $\Omega=\{1,\ldots,\ell\}$, we define the subset $S^{\tvec{a}}[J]=\prod_{j=1}^\ell L_j$ of $S^{\tvec{a}}$, where 
\[ L_j=\left\{\begin{array}{ll}
S^{\lambda_j} & (j\in J),\\
\{t_0^j\} & (j\notin J).
\end{array}\right. 
\]
The \emph{fat wedge} $R^{\tvec{a}}$ of $S^{\tvec{a}}$ is the subset $\bigcup_{i=1}^\ell S^{\tvec{a}}[\Omega\setminus\{i\}]$.
For a subset $Q$ of $S^{\tvec{a}}$, let $CQ$ denote the cone over $Q$. Then there are homotopy equivalences 
\[ S^{\tvec{a}}\cup CR^{\tvec{a}}\simeq S^{\tvec{a}}/R^{\tvec{a}}\simeq S^{|\tvec{a}|}, \]
where $|\vect{a}|=\lambda_1+\cdots+\lambda_\ell$.
Now we consider a slightly thickened alternative of the cone defined by a homotopy colimit. To define it, we consider the cone $CR^{\tvec{a}}$ as the colimit of the sub cones $C(S^{\tvec{a}}[J])$ on $S^{\tvec{a}}[J]$:
\[ CR^{\tvec{a}}=\colim{{J\subset \Omega}\atop{J\neq \Omega}}{C(S^{\tvec{a}}[J])}. \] 
We then define the ``thickened cone'' of $R^{\tvec{a}}$ by
\[ \widehat{C}R^{\tvec{a}}=\hocolim{{J\subset \Omega}\atop{J\neq \Omega}}{C(S^{\tvec{a}}[J])}. \]
We glue $\widehat{C}R^{\tvec{a}}$ to $S^{\tvec{a}}$ by gluing the subset $\hocolim{J}{S^{\tvec{a}}[J]}$ to $R^{\tvec{a}}=\colim{J}{S^{\tvec{a}}[J]}$ by the natural map from $\mathrm{hocolim}$ to $\mathrm{colim}$, and denote the result by
$S^{\tvec{a}}\cup \widehat{C}R^{\tvec{a}}$. The natural map gives a homotopy equivalence $S^{\tvec{a}}\cup \widehat{C}R^{\tvec{a}}\simeq S^{\tvec{a}}\cup CR^{\tvec{a}}$, since the diagram $\{C(S^{\tvec{a}}[J])\}_J$ is cofibrant.

\begin{Lem}\label{lem:brun-cone}
Let $\ell\geq 3$. Suppose that $k$ is sufficiently large. The system of Brunnian null-isotopies for the $\ell$-valent vertex surgery gives a canonical extension 
\[ \widehat{\omega}_{T_\ell(p,q)}\colon I^{\ell-3}\times (S^{\tvec{a}}\cup \widehat{C}R^{\tvec{a}})\to \Emb_\partial^\fr((I^k)^{\cup p}\cup (I^{k-1})^{\cup q},I^{2k}) \]
of the family $\omega_{T_\ell(p,q)}\colon B_{T_\ell(p,q)}=I^{\ell-3}\times S^{\tvec{a}}\to \Emb_\partial^\fr((I^k)^{\cup p}\cup (I^{k-1})^{\cup q},I^{2k})$. Thus we have the following homotopy commutative diagram:
\[ \xymatrix{
  I^{\ell-3}\times S^{\tvec{a}} \ar[rrd]^-{\omega_{T_\ell(p,q)}} \ar[d]_-{\simeq} & & \\
  I^{\ell-3}\times S^{|\tvec{a}|} \ar[rr] & & \Emb_\partial^\fr((I^k)^{\cup p}\cup (I^{k-1})^{\cup q},I^{2k})
} \]
\end{Lem}
\begin{proof}
For each subset $J\subset \Omega$ such that $J\neq \Omega$, an extension of $\omega_{T_\ell(p,q)}$ over $C(S^{\tvec{a}}[J])$ is constructed by the Brunnian null-isotopy with respect to removing the components labelled by $\Omega\setminus J$. The extensions over the various cones $C(S^{\tvec{a}}[J])$ can be further extended to $\hocolim{J}{C(S^{\tvec{a}}[J])}$ by the definition of the homotopy colimit. 
\end{proof}
\subsection{Compressibility at codimension one faces}\label{ss:compress-codim1}

To prove Theorem~\ref{thm:l-valent-general}, we check the homotopy coherence of the following diagram:
\[ \xymatrix{
  (-)\widehat{B}_{T_\ell(p,q)}(-) \ar[rrrd]^-{(-)\widehat{\omega}_{T_\ell(p,q)}(-)} \ar[d] & & & \\
   K(-)/{A_{T_\ell(p,q)}} \ar[rrr]_-{\widehat{\omega}_{T_\ell(p,q)}|_{K(-)/{A_{T_\ell(p,q)}}}} & & & \Emb_\partial^\fr((I^k)^{\cup p}\cup (I^{k-1})^{\cup q},I^{2k})\\
 } \]
which is parametrized by $\overrightarrow{\calP}_{T_\ell(p,q)}$. For simplicity, we check the homotopy coherence of the corresponding diagram without multiplicity, which is sufficient since it is a single seat in the multiplex. 
\begin{Def}\label{def:hat_B}
We define $\widehat{B}_{T_\ell(p,q)}$ by the pushout of the diagram 
\[ I^{\ell-3}\times (S^{\tvec{a}}\cup \widehat{C}R^{\tvec{a}})\longleftarrow \partial I^{\ell-3}\times \widehat{C}R^{\tvec{a}}\longrightarrow \partial I^{\ell-3}\times *, \]
where the left map is the inclusion and the right map is the projection. The boundary $\partial \widehat{B}_{T_\ell(p,q)}$ is defined by the subspace $\partial I^{\ell-3}\times (S^{\tvec{a}}/\widehat{C}R^{\tvec{a}})$ of $\widehat{B}_{T_\ell(p,q)}$.
\end{Def}
 This is the space obtained by collapsing the subspace $\partial I^{\ell-3}\times \widehat{C}R^{\tvec{a}}$ of $I^{\ell-3}\times (S^{\tvec{a}}\cup \widehat{C}R^{\tvec{a}})$ onto $\partial I^{\ell-3}\times *$. Since the collapsing of $\partial I^{\ell-3}\times \widehat{C}R^{\tvec{a}}$ onto $\partial I^{\ell-3}\times *$ does not affect the homotopy type, we have that $\widehat{B}_{T_\ell(p,q)}$ has the homotopy type of $I^{\ell-3}\times S^{|\tvec{a}|}$. 
For $(\sigma,\alpha)\in \overrightarrow{\calP}_{T_\ell(p,q)}$, we define
\[ \begin{split}
  & \widehat{B}_{T_\ell(p,q)}(\sigma,\alpha)=\left\{\begin{array}{ll}
  \prod_{v\in V^{\mathrm{int}}(\sigma)}\widehat{B}_{T(v)} & ((\sigma,\alpha)\neq T_\ell(p,q)),\\
  \widehat{B}_{T_\ell(p,q)}/{A_{T_\ell(p,q)}} & ((\sigma,\alpha)=T_\ell(p,q)).
  \end{array}\right.
\end{split} \]
as in \S\ref{ss:gen-bracket}. We also denote by $\widehat{\omega}_{T_\ell(p,q)}\colon \widehat{B}_{T_\ell(p,q)}\to \Emb_\partial^\fr((I^k)^{\cup p}\cup (I^{k-1})^{\cup q},I^{2k})$ the map induced by the map $\widehat{\omega}_{T_\ell(p,q)}$ of Lemma~\ref{lem:brun-cone}. Let $\widehat{\omega}_{T_\ell(p,q)}(\sigma,\alpha)\colon \widehat{B}_{T_\ell(p,q)}(\sigma,\alpha)\to \Emb_\partial^\fr((I^k)^{\cup p}\cup (I^{k-1})^{\cup q},I^{2k})$ be the map obtained by the product of $\widehat{\omega}_{T(v)}$ for internal vertices $v$ in $\sigma$ with lower dgrees than $T_\ell$.
\begin{Lem}\label{lem:codim1-coherence}
Let $\ell\geq 3$. Suppose that $k$ is sufficiently large. Let $(\sigma,\alpha)\in\overrightarrow{\calP}_{T_\ell(p,q)}$ be a face of $T_\ell(p,q)$ of excess $\ell-4$, i.e. a face of codimension 1. Then there is a natural map $\widehat{B}_{T_\ell(p,q)}(\sigma,\alpha)\to K(\sigma)/{A_{T_\ell(p,q)}}$ which makes the following diagram is homotopy commutative.
\begin{equation}\label{eq:codim1-coherence}
\vcenter{
\xymatrix{
  \widehat{B}_{T_\ell(p,q)}(\sigma,\alpha) \ar[rrrd]^-{\widehat{\omega}_{T_\ell(p,q)}(\sigma,\alpha)} \ar[d] & & & \\
  K(\sigma)/{A_{T_\ell(p,q)}} \ar[rrr]_-{\widehat{\omega}_{T_\ell(p,q)}|_{K(\sigma)/{A_{T_\ell(p,q)}}}} & & & \Emb_\partial^\fr((I^k)^{\cup p}\cup (I^{k-1})^{\cup q},I^{2k})\\
 } 
}
\end{equation}
\end{Lem}
\begin{proof}
Lemma~\ref{lem:codim1-coherence} can be proved by extending $\widehat{\omega}_{T_\ell(p,q)}(\sigma,\alpha)$ and $\widehat{\omega}_{T_\ell(p,q)}|_{K(\sigma)/{A_{T_\ell(p,q)}}}$ to the mapping cylinder of the projection $\widehat{B}_{T_\ell(p,q)}(\sigma,\alpha)\to K(\sigma)/{A_{T_\ell(p,q)}}$. Namely, the face tree $(\sigma,\alpha)$ has two internal vertices $v_1,v_2$ connected by a middle edge. 
The left vertical map in (\ref{eq:codim1-coherence}) can be given as follows. We have $B_{T_\ell(p,q)}(\sigma,\alpha)=B_{T(v_1)}\times B_{T(v_2)}=(I^{\ell_1-3}\times S^{\tvec{a}_1})\times (I^{\ell_2-3}\times S^{\tvec{a}_2})$ and 
$\widehat{B}_{T_\ell(p,q)}(\sigma,\alpha)=\widehat{B}_{T(v_1)}\times \widehat{B}_{T(v_2)}\simeq (I^{\ell_1-3}\times S^{|\tvec{a}_1|})\times (I^{\ell_2-3}\times S^{|\tvec{a}_2|})$, and $K(\sigma)/{A_{T_\ell(p,q)}}$ is homeomorphic to a quotient of $I^{\ell_1-3}\times I^{\ell_2-3}\times S^{|\tvec{a}_1|+|\tvec{a}_2|}$. The projection $\widehat{B}_{T_\ell(p,q)}(\sigma,\alpha)\to K(\sigma)/{A_{T_\ell(p,q)}}$ is induced by the quotient $(I^{\ell_1-3}\times S^{|\tvec{a}_1|})\times (I^{\ell_2-3}\times S^{|\tvec{a}_2|})\to I^{\ell_1-3}\times I^{\ell_2-3}\times S^{|\tvec{a}_1|+|\tvec{a}_2|}$. We need to prove that the restriction of $\widehat{\omega}_{T_\ell(p,q)}(\sigma,\alpha)$ to $(I^{\ell_1-3}\times S^{|\tvec{a}_1|})\vee (I^{\ell_2-3}\times S^{|\tvec{a}_2|})$ is null-homotopic. Recall that the iterated surgery of two $\Psi$-graphs in Definition~\ref{def:iter-surg} is defined by choosing one component corresponding to an external vertex. By Lemma~\ref{lem:slide-para}, we may change the choices of the leaves of $G_1$ and $G_2$ in Definition~\ref{def:iter-surg} in such a way that for each $i$ one of the 1-valent vertex of $G_i$ that is disjoint from the internal edge is chosen. Then the restriction of $\widehat{\omega}_{T_\ell(p,q)}(\sigma,\alpha)$ to $I^{\ell_i-3}\times S^{|\tvec{a}_i|}$ is null-homotopic by the Brunnian property with respect to removing the component corresponding to the leaf at the middle of the internal edge of $\sigma$. Hence $\widehat{\omega}_{T_\ell(p,q)}(\sigma,\alpha)$ factors up to homotopy through $K(\sigma)/A_{T_\ell(p,q)}$.
\end{proof}

\subsection{Compatibility at higher codimensional strata}\label{ss:compatible-higher}

We require that the faces of the codimension 1 strata are consistent along the codimension 2 strata, and so on for higher codimension strata. For example, over $B_{\partial_j\partial_i T_5(p,q)}$, where $\partial_j\partial_i T_5(p,q)$ is a term (of a 3-valent graph) in the IHX relation from the 4-valent vertex in $\partial_i T_5(p,q)$. For the consistency along the codimension 2 strata, we glue $B_{\partial_j\partial_i T_5(p,q)}\times \Delta^2$ to each codimension 2 stratum. 

We need to prove that the chains $\widehat{\omega}_{T_\ell(p,q)}(\sigma,\alpha)$ are compatibly glued together after thickenings. According to \cite[Remark~9.6 and Proposition~9.7]{Du}, giving a map $\hocolim{}{\calX(-)}\to Z$ is the same as giving a homotopy coherent diagram $\calX(-)\to Z$, the following lemma is sufficient.

\begin{Lem}\label{lem:hcoherence}
Let $\ell\geq 3$. Suppose that $k$ is sufficiently large.
The diagrams $\widehat{\omega}_{T_\ell(p,q)}(-)\colon \widehat{B}_{T_\ell(p,q)}(-) \to \Emb_\partial^\fr((I^k)^{\cup p}\cup (I^{k-1})^{\cup q},I^{2k})$ and $\widehat{\omega}_{T_\ell(p,q)}|_{K(-)/{A_{T_\ell(p,q)}}}\colon K(-)/{A_{T_\ell(p,q)}} \to \Emb_\partial^\fr((I^k)^{\cup p}\cup (I^{k-1})^{\cup q},I^{2k})$, both parametrized by the poset $\overrightarrow{\calP}_{T_\ell(p,q)}$, are compatible with respect to the natural transformation $\widehat{B}_{T_\ell(p,q)}(-)\to  K(-)/{A_{T_\ell(p,q)}}$. Namely, the following diagram is homotopy coherent.
\[ \xymatrix{
  \widehat{B}_{T_\ell(p,q)}(-) \ar[rrrd]^-{\widehat{\omega}_{T_\ell(p,q)}(-)} \ar[d] & & &\\
   K(-)/{A_{T_\ell(p,q)}} \ar[rrr]_-{\widehat{\omega}_{T_\ell(p,q)}|_{ K(-)/{A_{T_\ell(p,q)}}}} & & & \Emb_\partial^\fr((I^k)^{\cup p}\cup (I^{k-1})^{\cup q},I^{2k})\\
 } \]
\end{Lem}
\begin{proof}
We put $T=T_{\ell}(p,q)$ and denote $ K(\sigma)/{A_{T_\ell(p,q)}}$ by $ K_{T}(\sigma)$ for simplicity. 

Let us first check the homotopy coherence at codimension 2 strata. 
Let $(\sigma,\alpha)\in\overrightarrow{\calP}_{T}$ be a face of $T$ of excess $\ell-4$, and let $(\sigma',\alpha')\in\overrightarrow{\calP}_{T}$ be a face of $(\sigma,\alpha)$ of excess $\ell-5$. In other words, the sequence $T\leftarrow (\sigma,\alpha)\leftarrow (\sigma',\alpha')$ is obtained by contracting two edges.
Suppose that $(\sigma,\alpha)$ is the union of $T_{\ell_1}(p_1,q_1)$ and $T_{\ell_2}(p_2,q_2)$ meeting at an internal edge. We have $\widehat{B}_{T}(\sigma,\alpha)=\widehat{B}_{T_{\ell_1}(p_1,q_1)}\times \widehat{B}_{T_{\ell_2}(p_2,q_2)}$. 
Suppose that $(\sigma',\alpha')$ is induced from a face of $T_{\ell_1}(p_1,q_1)$, and is the union of $T_{\ell_3}(p_3,q_3)$, $T_{\ell_4}(p_4,q_4)$, and $T_{\ell_2}(p_2,q_2)$. Then we have $\widehat{B}_{T}(\sigma',\alpha')=\widehat{B}_{T_{\ell_3}(p_3,q_3)}\times \widehat{B}_{T_{\ell_4}(p_4,q_4)}\times \widehat{B}_{T_{\ell_2}(p_2,q_2)}$. Let $ K_{\sigma}(\sigma'):= K_{T_{\ell_1}(p_1,q_1)}(\sigma')\times \widehat{B}_{T_{\ell_2}(p_2,q_2)}$.
We consider the following commutative diagram:
\[ 
\xymatrix@C=1em{
  \widehat{B}_{T}(\sigma',\alpha') \ar[d] \ar[drr] & & & &\\
   K_{\sigma}(\sigma') \ar[rr]_-{\subset} \ar[d] 
  & & \widehat{B}_{T}(\sigma,\alpha) \ar[d] \ar[drr] & &\\
   K_{T}(\sigma') \ar[rr]_-{\subset} 
  & &  K_{T}(\sigma) \ar[rr]_-{\subset} & & \widehat{B}_{T}
}
 \]
where each vertical arrow corresponds to the projection map of Lemma~\ref{lem:codim1-coherence}. The homotopy coherence of the diagram from this triangle diagram to $\Emb_\partial^\fr((I^k)^{\cup p}\cup (I^{k-1})^{\cup q},I^{2k})$ follows from that restricted to the lower left square, which holds since the homotopy between $\widehat{\omega}_{T}(\sigma,\alpha)|_{ K_{\sigma}(\sigma')}$ and $\widehat{\omega}_{T}|_{ K_{T}(\sigma')}$ is the restriction of that between $\widehat{\omega}_{T}(\sigma,\alpha)$ and $\widehat{\omega}_{T}|_{ K_{T}(\sigma)}$. 

Next, we consider another sequence $T\leftarrow (\rho,\gamma)\leftarrow (\rho',\gamma')$ of contractions of two edges such that $(\sigma',\alpha')=(\rho',\gamma')$. In this case, we consider the following commutative diagram:
\begin{equation}\label{eq:hcoherent-diag}
\vcenter{
\xymatrix{
  & \widehat{B}_T(\rho',\gamma') \ar[d] \ar@{=}[r]& \widehat{B}_T(\sigma',\alpha') \ar[d]  \\
 & \llap{$\widehat{B}_T(\rho,\gamma)\supset{}$} K_\rho(\rho')\ar[d] & K_\sigma(\sigma')\rlap{${}\subset \widehat{B}_T(\sigma,\alpha)$} \ar[d] \\
  & \llap{$\widehat{B}_T\supset{}$} K_T(\rho') \ar@{=}[r]& K_T(\sigma') \rlap{${}\subset \widehat{B}_T$}
}
}\end{equation}
We check the homotopy coherence of the diagram from (\ref{eq:hcoherent-diag}) to $\Emb_\partial^\fr((I^k)^{\cup p}\cup (I^{k-1})^{\cup q},I^{2k})$.
Let $X=\widehat{B}_T(\sigma',\alpha')$, $Y=\Emb_\partial^\fr((I^k)^{\cup p}\cup (I^{k-1})^{\cup q},I^{2k})$, and let $Q$ be the fat wedge of the spheres (see \S\ref{ss:merge}) collapsed by $\widehat{B}_T(\sigma',\alpha')\to  K_T(\sigma')$. Then the homotopies for the sequences $\widehat{B}_{T}(\sigma',\alpha')\to  K_{\sigma}(\sigma')\to  K_T(\sigma')$ and $\widehat{B}_{T}(\rho',\gamma')\to  K_{\rho}(\rho')\to  K_T(\sigma')$ give two extensions $X\cup CQ\to Y$ of $\widehat{\omega}_T(\sigma',\alpha')$. The extension to the homotopy colimit $X\cup \widehat{C}Q\to Y$ gives a unified extension of the two extensions.

The case of higher codimension stratum is similar. For example, we consider a sequence $T\leftarrow (\sigma,\alpha)\leftarrow (\sigma',\alpha')\leftarrow (\sigma'',\alpha'')$ of taking boundary faces. 
Since the sequence $T\leftarrow (\sigma,\alpha)\leftarrow (\sigma',\alpha')\leftarrow (\sigma'',\alpha'')$ corresponds to collapsing three independent factors in $\widehat{B}_T$, we have the following commutative 3-cube diagram.
\[ \small\xymatrix@C=0em{
& \widehat{B}_T(\sigma'',\alpha'') \ar[rr] \ar[dl] \ar'[d][dd] & & \widehat{B}_T(\rho',\gamma') \ar[dd] \ar[dl]\\
\widehat{B}_T(\theta',\delta') \ar[rr]\ar[dd] & & \widehat{B}_T(\theta,\delta) \ar[dd]\\
& \widehat{B}_T(\sigma',\alpha') \ar'[r][rr] \ar[dl] & & \widehat{B}_T(\rho,\gamma) \ar[dl]\\
\widehat{B}_T(\sigma,\alpha) \ar[rr] & & \widehat{B}_T
}
\]
The homotopy coherence at the faces of this diagram follows from the previous case (of codimension 2). The homotopy coherence for the 3-cube follows by taking an extension to the homotopy colimit $X\cup \widehat{C}Q\to Y$ as in the previous case.
\end{proof}

\begin{Cor}\label{cor:hocolim-hcomm}
The diagram of induced maps
\[ \xymatrix{
  \overline{B}_{T_\ell(p,q)}=\underset{(\sigma,\alpha)}{\mathrm{hocolim}}\,n_\sigma\widehat{B}_{T_\ell(p,q)}(\sigma,\alpha) \ar[d] \ar[rrd]^-{\overline{\omega}_{T_\ell(p,q)}} & & \\
  n_{T_\ell}\widehat{B}_{T_\ell(p,q)}/{A_{T_\ell(p,q)}}= \underset{\sigma}{\mathrm{colim}}\,K_{T_\ell(p,q)}^{(n_{T_\ell})}(\sigma) \ar[rr]_-{n_{T_\ell}\widehat{\omega}_{T_\ell(p,q)}} & & \Emb_\partial^\fr((I^k)^{\cup p}\cup (I^{k-1})^{\cup q},I^{2k})
}\]
is homotopy commutative.
\end{Cor}

To prove Theorem~\ref{thm:l-valent-general}, we recall the interpretation of the Brunnian property in Definition~\ref{def:brunnian-framily} in terms of a lifting problem. With that interpretation, we have the following.
\begin{Lem}\label{lem:omega-hat-brunnian}
The chain $\widehat{\omega}_{T_\ell(p,q)}\colon I^{\ell-3}\times (S^{\tvec{a}}\cup \widehat{C}R^{\tvec{a}})\to \calE(\emptyset)$ has a natural system of Brunnian null-isotopies that is compatible with that of $\omega_{T_\ell(p,q)}\colon I^{\ell-3}\times S^{\tvec{a}}\to \calE(\emptyset)$.
\end{Lem}
\begin{proof}
The chain $\widehat{\omega}_{T_3(p,q)}\colon I^{\ell-3}\times (S^{\tvec{a}}\cup \widehat{C}R^{\tvec{a}})\to \calE(\emptyset)$ naturally induces a lift $I^{\ell-3}\times (S^{\tvec{a}}\cup \widehat{C}R^{\tvec{a}})\to P$ as follows. Each null-isotopy in $\calB(S)$ gives a path $\gamma(S)\colon I\to \calE(\emptyset)$, which is also a restriction of $\widehat{\omega}_{T_3(p,q)}$ in the thickened cone $I^{\ell-3}\times (R^{\tvec{a}}\cup \widehat{C}R^{\tvec{a}})\to \calE(\emptyset)$. The path $\gamma(S)=\{\gamma(S)(u)\}_{u\in I}$ is naturally equipped with the null-isotopy $\{\gamma(S)(u+(1-u)t)\}_{u\in I}$, which gives the lift of $\gamma(S)$ to $\calB(S)$. Doing this lift over all the paths from $I^{\ell-3}\times R^{\tvec{a}}$ in $I^{\ell-3}\times (R^{\tvec{a}}\cup \widehat{C}R^{\tvec{a}})$, we obtain a lift $I^{\ell-3}\times (S^{\tvec{a}}\cup \widehat{C}R^{\tvec{a}})\to P$ of $\widehat{\omega}_{T_3(p,q)}$.
\end{proof}

\begin{proof}[Proof of Theorem~\ref{thm:l-valent-general}]
The compatibility with respect to lower excess trees and the condition for excess 0 trees are obvious from the construction. The Brunnian property follows from Lemma~\ref{lem:omega-hat-brunnian}.

It remains to prove the $L_\infty$-relation (\ref{eq:rel-general}). We already have the boundary structures of $\omega_{T_\ell(p,q)}$ and of $\widehat{\omega}_{T_\ell(p,q)}$ induced from Theorem~\ref{thm:vertex}(2). It is the sum of the chains from the boundary faces $K_{T_\ell(p,q)}(\sigma)/{A_{T_\ell(p,q)}}$ of $\widehat{B}_{T_\ell(p,q)}/{A_{T_\ell(p,q)}}$.
By the thickening process above, the boundary face $K_{T_\ell(p,q)}(\sigma)/{A_{T_\ell(p,q)}}$ of $\widehat{B}_{T_\ell(p,q)}/{A_{T_\ell(p,q)}}$ is replaced with a thickening of $\widehat{B}_{T_\ell(p,q)}(\sigma,\alpha)$, which gives a face of $\partial \overline{\omega}_{T_\ell(p,q)}$ corresponding to $(\sigma,\alpha)$. This completes the proof.
\end{proof}

\begin{Lem}\label{lem:slide-para}
Let $\alpha\colon I^{\ell'-3}\to \Emb_\partial^\fr(I^{p_1}\cup I^{p_2}\cup I^{p_3-\lambda}\cup\cdots\cup I^{p_{\ell'}-\lambda},I^{N-\lambda})_\iota$ be a family of string links of type $(p_1,p_2,p_3-\lambda,\ldots,p_{\ell'}-\lambda;N-\lambda)$ obtained by suspending a $\ell'$-valent basic bracket operation of Theorem~\ref{thm:vertex}. Let $\alpha_1,\alpha_2\colon I^{\ell'-3}\to \Omega^\lambda\Emb_\partial^\fr(I^{p_1}\cup I^{p_2}\cup I^{p_3}\cup\cdots\cup I^{p_{\ell'}},I^{N})_\iota$ be the families of embeddings obtained from $\alpha$ as follows (see Figure~\ref{fig:move-deloop}).
\begin{enumerate}
\item The family $\alpha_1$ is obtained by suspending $\alpha$ at the $2,3,4,\ldots,\ell'$-th components along the Brunnian null-isotopy with respect to removing the first component (Remark~\ref{rem:suspension-generalized}), and then by delooping with respect to the second component.
\[ \begin{split}
  (p_1,p_2,p_3-\lambda,\ldots,p_{\ell'}-\lambda;N-\lambda)
&\to (p_1,p_2+\lambda,p_3,\ldots,p_{\ell'};N)\\
&\to \Omega^\lambda(p_1,p_2,p_3,\ldots,p_{\ell'};N)
\end{split} \]

\item The family $\alpha_2$ is obtained by suspending $\alpha$ at the $1,3,4,\ldots,\ell'$-th components along the Brunnian null-isotopy with respect to removing the second component, and then by delooping with respect to the first component.
\[ \begin{split}
  (p_1,p_2,p_3-\lambda,\ldots,p_{\ell'}-\lambda;N-\lambda)
&\to (p_1+\lambda,p_2,p_3,\ldots,p_{\ell'};N)\\
&\to \Omega^\lambda(p_1,p_2,p_3,\ldots,p_{\ell'};N)
\end{split} \]
\end{enumerate}
Then there is a homotopy between $\alpha_1$ and $\alpha_2$ that is compatible with those of the lower excess graphs. The same holds also for other choice of a pair of components. 
\end{Lem} 
\begin{figure}[h]
\[ \includegraphics[height=45mm]{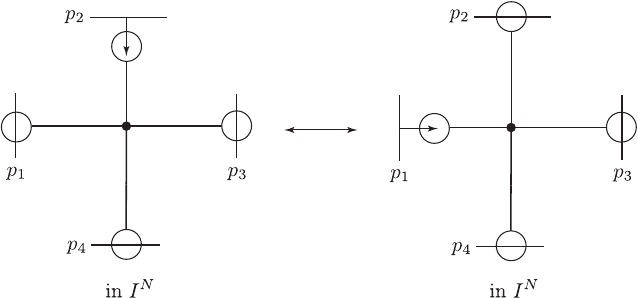} \]
\caption{Exchanging the components for the delooping.}\label{fig:move-deloop}
\end{figure}
\begin{proof}
We only prove the lemma for $\ell'=4$ and $\lambda=1$ for simplicity since the proofs for other cases are the same. In this case, $\alpha$ is a path of string links of type $(p_1,p_2,p_3-1,p_4-1;N-1)$. We find a homotopy between $\alpha_1(t)$ and $\alpha_2(t)$ for each $t\in [0,1]$ in such a way that it depends smoothly on $t$.

By Lemma~\ref{lem:brunnian-3}, $\alpha$ has a Brunnian null-isotopy with respect to removing the first component. The suspension $\Sigma_{2,3,4}\alpha(t)$ of $\alpha(t)$ with respect to the last three components is defined by using this Brunnian null-isotopy and yields a string link of type $(p_1,p_2+1,p_3,p_4;N)$. Then the path $\alpha_1(t)$ is the delooping of the second component in this string link. Its graph is a string link $M$ of type $(p_1+1,p_2+1,p_3+1,p_4+1;N+1)$.

Here, one can see that the graph of delooping of the second component and the suspension with respect to the $1,3,4$-th components can be mutually transformed by a small isotopy (Lemma~\ref{lem:suspension-delooping}). Similarly, the suspension $\Sigma_{2,3,4}\alpha(t)$ can be transformed into the graph of the delooping of the first component by a small isotopy. 
Thus, $M$ can be considered as the graph of the 2-parameter family of type $\Omega^{(1,1,0,0)}(p_1-1,p_2-1,p_3-1,p_4-1;N-1)$, and the difference between $\alpha_1(t)$ and $\alpha_2(t)$ is the choice of the $I$-direction of the parameter space $I^2$ to take the partial graph. More explicitly, $\alpha_1(t)$ (for each fixed $t$) can be deformed to the sequence
\[ \begin{split}
  (p_1,p_2,p_3-1,p_4-1;N-1) &\to\Omega^{(1,0,0,0)}(p_1-1,p_2,p_3-1,p_4-1;N-1)\\
&\to  \Omega^{(1,1,0,0)}(p_1-1,p_2-1,p_3-1,p_4-1;N-1)
\end{split}\]
of taking graphs, where $\Omega^{(1,0,0,0)}$ etc. are the notations of (\ref{eq:symbolize-suspension}), and $\alpha_2(t)$ (for each fixed $t$) can be deformed to the sequence
\[ \begin{split}
  (p_1,p_2,p_3-1,p_4-1;N-1) &\to\Omega^{(0,1,0,0)}(p_1,p_2-1,p_3-1,p_4-1;N-1)\\
&\to  \Omega^{(1,1,0,0)}(p_1-1,p_2-1,p_3-1,p_4-1;N-1)
\end{split}\]
of taking graphs. After deforming to families over $I^2$, any $I$-direction in $I^2$ gives a graph, which is a string link of type $(p_1,p_2,p_3,p_4;N)$. 
The two choices of the $I$-direction for the first graphing can be mutually deformed (see Figure~\ref{fig:rotate-axis}).
\begin{figure}[h]
\[\kern-10mm\includegraphics[height=45mm]{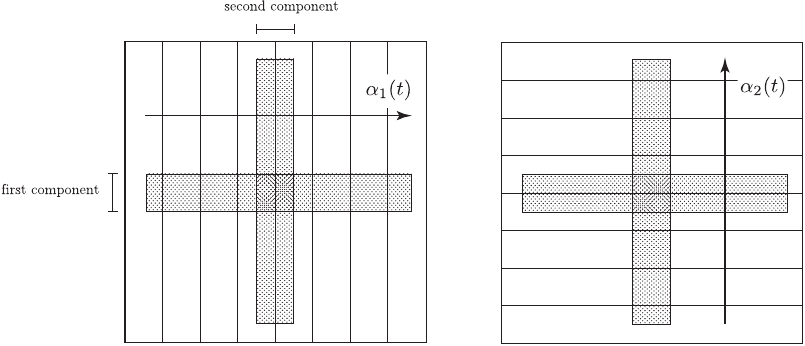}\]
\caption{The 1-parameter families $\alpha_1(t)$ and $\alpha_2(t)$ for each fixed $t$. The vertical (resp. horizontal) gray rectangle includes the parameters on which the second (resp. first) component may not be standard.}\label{fig:rotate-axis}
\end{figure}
\end{proof}

\subsection{Lift of $\overline{\omega}_{T_\ell(p,q)}$ to $\bEmb$}\label{ss:lift-bEmb}

Now we prove Lemma~\ref{lem:vert-surg} postponed in \S\ref{ss:v-surgery}.

\begin{Lem}[Lemma~\ref{lem:vert-surg}]\label{lem:vert-surg-2}
For $\ell\geq 3$, the chain $\overline{\omega}_{T_\ell(p,q)}\colon \overline{B}_{T_\ell(p,q)}\to \Emb_\partial^\fr((I^k)^{\cup p}\cup (I^{k-1})^{\cup q},I^{2k})$ of Theorem~\ref{thm:l-valent-general} admits a lift to the space $\bEmb_\partial(N_{I^k}^{\cup p}\cup N_{I^{k-1}}^{\cup q},I^{2k})_{\widetilde{\iota}}$, which is compatible with those for the lower excess trees $\sigma\in\overrightarrow{\calP}_{T_\ell(p,q)}\setminus\{T_\ell(p,q)\}$. 
\end{Lem}
By (\ref{eq:2x3-diag}), it suffices to give a lift of $\overline{\omega}_{T_\ell(p,q)}$ to $\bEmb_\partial((I^k)^{\cup p}\cup (I^{k-1})^{\cup q},I^{2k})_{\iota}$. Roughly speaking, the system of Brunnian null-isotopies for $\overline{\omega}_{T_\ell(p,q)}$ induces a lift to $\bEmb_\partial((I^k)^{\cup p}\cup (I^{k-1})^{\cup q},I^{2k})_{\iota}$. More precisely, let $\calE(\emptyset)=\Emb_\partial(I^{a_1}\cup\cdots\cup I^{a_\ell},I^N)$ and $\overline{\calE}(\emptyset)=\bEmb_\partial(I^{a_1}\cup\cdots\cup I^{a_\ell},I^N)_\iota$. We consider the following commutative diagram:
\[ \xymatrix{
 & & Q \ar[r] \ar[d] \ar[lld] & \prod_{1\leq i\leq \ell}\calB(\Omega\setminus\{i\}) \ar[d]\\
  \overline{\calE}(\emptyset) \ar[rr] & & \calE(\emptyset) \ar[r]^-{\Delta} & \calE(\emptyset)^{\times \ell}
} \]
where the right square is the pullback square, the rightmost vertical map is the product of the projections $\calB(\Omega\setminus\{i\})\to \calE(\emptyset)$, and the left skew arrow is induced by considering null-isotopies as null regular homotopies. This gives the following lemma.
\begin{Lem}\label{lem:Q-lift}
If a family $\sigma\colon B\to \calE(\emptyset)$ admits a lift to $Q$, then $\sigma$ has a lift to $\overline{\calE}(\emptyset)$.
\end{Lem}
A lift to $Q$ is induced from a lift to $P$ of Definition~\ref{def:brunnian-framily}, as follows. The restrictions to the barycenter of $\Delta^{\ell-2}$ gives the restriction map
\[ \prod_{1\leq i\leq \ell}\Map(\Delta^{\ell-2},\calB(\Omega\setminus\{i\}))\to \prod_{1\leq i\leq \ell}\calB(\Omega\setminus\{i\}). \]
The postcompositions of this map on the factor of $\prod_{1\leq i\leq \ell}\Map(\Delta^{\ell-2},\calB(\Omega\setminus\{i\}))$ gives the following commutative diagram:
\[ \xymatrix{
&&& P \ar[rr] \ar[dlll] \ar'[d][dd] & & \holim{\emptyset\neq S\subset\Omega}\calB(S) \ar[dd] \ar[dl]\\
\calE(\emptyset) \ar[rrrr] \ar@{=}[dd] & & & & \Map(\Delta^{\ell-1},\calE(\emptyset)) \ar[dd]\\
&&& Q \ar'[r][rr] \ar[dlll] & & \prod_{1\leq i\leq \ell}\calB(\Omega\setminus\{i\}) \ar[dl]\\
\calE(\emptyset) \ar[rrrr] & & & & \calE(\emptyset)^{\times \ell}
}\]  
where the vertical map $\Map(\Delta^{\ell-1},\calE(\emptyset))\to \calE(\emptyset)^{\times \ell}$ is the restriction to the barycenters of the codimension 1 faces of $\Delta^{\ell-1}$.
Thus we have the following.
\begin{Lem}\label{lem:P-Q-lift}
If a family $\sigma\colon B\to \calE(\emptyset)$ admits a lift to $P$, then $\sigma$ has a lift to $Q$.
\end{Lem}

\begin{proof}[Proof of Lemma~\ref{lem:vert-surg-2}]
For $\ell=3$, we take the lift of $\omega_{T_3}\in\Emb_\partial((I^{2n-1})^{\cup 3},I^{3n})$ to $\holim{\emptyset\neq S\subset\Omega}\calB(S)$ as in the proof of Lemma~\ref{lem:brun-borromean}. Then suspensions and deloopings naturally induce those for the null regular homotopies as in \S\ref{ss:suspend-brunnian}. Thus we obtain a lift $S^{\tvec{a}}\to P$ of $\omega_{T_3(p,q)}\colon S^{\tvec{a}}\to \calE(\emptyset):=\Emb_\partial((I^k)^{\cup p}\cup (I^{k-1})^{\cup q},I^{2k})$. By Lemmas~\ref{lem:P-Q-lift} and \ref{lem:Q-lift}, we obtain a lift $S^{\tvec{a}}\to \overline{\calE}(\emptyset)$.
Also, the extension $\widehat{\omega}_{T_3(p,q)}\colon S^{\tvec{a}}\cup \widehat{C}R^{\tvec{a}}\to \calE(\emptyset)$ naturally induces a lift $S^{\tvec{a}}\cup \widehat{C}R^{\tvec{a}}\to P$ by Lemma~\ref{lem:omega-hat-brunnian}. Then by Lemmas~\ref{lem:P-Q-lift} and \ref{lem:Q-lift} again, we obtain a lift $S^{\tvec{a}}\cup \widehat{C}R^{\tvec{a}}\to \overline{\calE}(\emptyset)$ of $\widehat{\omega}_{T_3(p,q)}$.

In the rest of the proof, we assume $\ell\geq 4$. First we see that $\omega_{T_\ell}\colon I\to \Omega^{\ell-4}\Emb_\partial^\fr((I^{(\ell-1)n-(\ell-2)})^{\cup \ell},I^{\ell n-\ell+3})_{\iota}$ admits a lift to $\Omega^{\ell-4}\bEmb_\partial((I^{(\ell-1)n-(\ell-2)})^{\cup \ell},I^{\ell n-\ell+3})_{\iota}$. We put $a=(\ell-1)n-(\ell-2)$ and $N=\ell n-\ell+3$ for simplicity. We redefine $\calE(\emptyset):=\Emb_\partial((I^a)^{\cup \ell},I^N)_{\iota}$ and $\overline{\calE}(\emptyset):=\bEmb_\partial((I^a)^{\cup \ell},I^N)_{\iota}$.
To give a lift of $\omega_{T_\ell}\colon I\to \Omega^{\ell-4}\calE(\emptyset)$ to $\Omega^{\ell-4}Q$, it suffices to find a path
\[ I\to \prod_{1\leq i\leq \ell}\Omega^{\ell-4}\calB(\Omega\setminus\{i\}) \]
to the point given by $(\omega_{T_\ell})^{\times \ell}$, considered as an element of $\prod_{i=1}^{\ell}\Omega^{\ell-4}\calB(\Omega\setminus\{i\})$, lifting $(\omega_{T_\ell})^{\times \ell}\colon I\to (\Omega^{\ell-4}\calE(\emptyset))^{\times \ell}=\Omega^{\ell-4}(\calE(\emptyset))^{\times \ell}$.
Note that $\omega_{T_\ell}$ can be considered as both a path in $\Omega^{\ell-4}\calE(\emptyset)$ and an element of $\Omega^{\ell-4}\calB(\Omega\setminus\{i\})$.
Moreover, we require that it is compatible with those choices for the face trees of $T_\ell$. 
Such a path in $\Omega^{\ell-4}Q$ is induced from the lift $I\to \Omega^{\ell-4}P$ of $\omega_{T_\ell}\colon I\to \Omega^{\ell-4}\calE(\emptyset)$ given by the Brunnian property of $\omega_{T_\ell}$ (Definition~\ref{def:brunnian-framily} and Lemma~\ref{lem:P-Q-lift}), and then Lemma~\ref{lem:Q-lift} implies that it induces a lift $I\to \Omega^{\ell-4}\overline{\calE}(\emptyset)$.

Finally, we remark that the suspensions and deloopings to obtain $\omega_{T_\ell(p,q)}$ from $\omega_{T_\ell}$ naturally induce those for the path of null regular homotopies constructed as above. Further, since the extension from $\omega_{T_\ell(p,q)}$ to $\widehat{\omega}_{T_\ell(p,q)}$ was defined by the system of Brunnian null-isotopies, we obtain a natural extension of the family of null regular homotopies to that of $\widehat{\omega}_{T_\ell(p,q)}$ by Lemma~\ref{lem:omega-hat-brunnian}. The extension to that of $\overline{\omega}_{T_\ell(p,q)}$ is similar. 
\end{proof}


\section{Deformation to the rational homotopy group}\label{s:homotopy}

In this section, we prove Theorem~\ref{thm:pi}.
The combinatorial cycle $5\vec{\gamma}$ in $\overrightarrow{\mathcal{GC}}$ is a $\Z$-linear combination of directed graphs $(X,\alpha_X)$ and $(Y,\alpha_Y)$. To prove that $[\overline{\phi}_\gamma]$ is in the image of the rational Hurewicz map, we will see that $\overline{\phi}_{5\vec{\gamma}}$ can be represented by a map
\[ \kappa_{5\vec{\gamma}}\colon K_{5\vec{\gamma}}\to B\Diff_\partial(D^{2k}) \]
from some finite $(8k-10)$-dimensional complex $K_{5\vec{\gamma}}$, which is obtained by gluing finitely many compact oriented manifold pieces together along their boundaries (Lemma~\ref{lem:K}).
Then we will further see that the cycle $\kappa_{5\vec{\gamma}}$ is rationally homologous to a spherical one.

\subsection{The complex {$K_{5\vec{\gamma}}$}}

We define the poset $\calI_{5\vec{\gamma}}$ consisting of the following directed graphs of excess at most 2. 
\begin{itemize}
\item The excess 2 directed graphs in $\calI_{5\vec{\gamma}}$ are those corresponding to the terms in the $\Z$-linear combination $5\vec{\gamma}$. If a directed graph $(\Gamma,\alpha)$ has an integer coefficient $\lambda$ in $5\vec{\gamma}$, then we take $|\lambda|$ copies of $(\Gamma,\alpha)$ in $\calI_{5\vec{\gamma}}$. For example, if $5\vec{\gamma}$ has the term $2(X,\alpha_X)$, then the number of copies of $(X,\alpha_X)$ in $\calI_{5\vec{\gamma}}$ is 2, and we consider the two copies as distinct elements. 
\item The excess 1 directed graphs in $\calI_{5\vec{\gamma}}$ are those that can be obtained from $(X,\alpha_X)$ or $(Y,\alpha_Y)$ by expanding one vertex of valence $\geq 4$. We do not take copies in this case. 
\item The excess 0 directed graphs in $\calI_{5\vec{\gamma}}$ are those that can be obtained from those in the previous item. Again, we do not take copies in this case. 
\end{itemize}
A morphism $(\Gamma,\alpha)\to (\Gamma',\alpha')$ or an order $(\Gamma,\alpha)\leq (\Gamma',\alpha')$ in $\calI_{5\vec{\gamma}}$ is defined if $(\Gamma',\alpha')$ is obtained from $(\Gamma,\alpha)$ by contracting some edges.

\subsubsection{Gluing the pieces $\widehat{B}_{(\Gamma,\alpha)}$ together}\label{ss:gluing-B}

The complex $K_{5\vec{\gamma}}$ will be defined in Lemma~\ref{lem:K} as the homotopy colimit of some diagram $\calI_{5\vec{\gamma}}\to \mathrm{Top}$. For a directed graph $(\Gamma,\alpha)$, we define
\[ \widehat{B}_{(\Gamma,\alpha)}:=\prod_{v\in V(\Gamma)}\widehat{B}_{T(v)}, \]
where $T(v)=T_\ell(p,q)$ for $\ell,p,q$ determined by $\alpha$, and $\widehat{B}_{T(v)}$ was defined in \S\ref{ss:compress-codim1}. This space $\widehat{B}_{(\Gamma,\alpha)}$ has the homotopy type of a product of spheres. 
We define 
\[ \widehat{\omega}_{(\Gamma,\alpha)}=\mathrm{ext}\circ c\circ\prod_{v\in V(\Gamma)}\widehat{\omega}_{T(v)}\colon \widehat{B}_{(\Gamma,\alpha)}\to B\Diff_\partial(D^{2k}), \]
where $c\colon\prod_v\Emb^\fr_\partial((I^k)^{\cup p_v}\cup (I^{k-1})^{\cup q_v},I^{2k})\to \prod_v B\Diff_\partial(V_v)$ is the map defined in \S\ref{ss:v-surgery}, and $\mathrm{ext}\colon \prod_v B\Diff_\partial(V_v)\to B\Diff_\partial(D^{2k})$ is the extension by the identity on $D^{2k}-\mathrm{Int}(V_1\cup\cdots\cup V_{|V(\Gamma)|})$. 
We define the {\it multiplicity} $m_\Gamma$ of $\Gamma$ by
\[ m_\Gamma=\prod_{v\in V(\Gamma)}m_{\ell_v},\]
where $\ell_v$ is the valence of $v$. This is the multiplicity of excess 0 (3-valent) graphs in $\widehat{\omega}_{(\Gamma,\alpha)}$. Let
\[ \mu_{5\vec{\gamma}}=\mathrm{lcm}(\{m_\Gamma\mid (\Gamma,\alpha)\in\calI_{5\vec{\gamma}}\}). \]
Then the cycle $\overline{\phi}_{5\vec{\gamma}}$ can be represented by $\frac{1}{\mu_{5\vec{\gamma}}}$ times the map 
\[ \overline{\omega}_{5\vec{\gamma}}:=\hocolim{{(\Gamma,\alpha)}\in \calI_{5\vec{\gamma}}}{\frac{\mu_{5\vec{\gamma}}}{m_\Gamma}n_\Gamma\widehat{\omega}_{(\Gamma,\alpha)}}\colon \hocolim{{(\Gamma,\alpha)}\in \calI_{5\vec{\gamma}}}{\frac{\mu_{5\vec{\gamma}}}{m_\Gamma}n_\Gamma\widehat{B}_{(\Gamma,\alpha)}}=\colim{{(\Gamma,\alpha)}\in \calI_{5\vec{\gamma}}}{\frac{\mu_{5\vec{\gamma}}}{m_\Gamma}\overline{B}_{(\Gamma,\alpha)}}\to B\Diff_\partial(D^{2k}), \]
where $n_\Gamma=2^{\text{excess}(\Gamma)}$ and $\colim{{(\Gamma,\alpha)}\in \calI_{5\vec{\gamma}}}{\frac{\mu_{5\vec{\gamma}}}{m_\Gamma}\overline{B}_{(\Gamma,\alpha)}}$ includes a single $\frac{\mu_{5\vec{\gamma}}}{m_\Gamma}\overline{B}_{(\Gamma,\alpha)}$ for the multiple copies of $(\Gamma,\alpha)$ in $\calI_{5\vec{\gamma}}$.
Namely, the class $[\overline{\omega}_{5\vec{\gamma}}]\in H_{8k-10}(B\Diff_\partial(D^{2k});\Q)$ is represented by the $(8k-10)$-chain given by the sum of the restrictions to the (oriented) top cells of the pieces $\frac{\mu_{5\vec{\gamma}}}{m_\Gamma}\overline{\omega}_{(\Gamma,\alpha)}\colon \overline{B}_{(\Gamma,\alpha)}\to B\Diff_\partial(D^{2k})$. The possibility of extending the collection of the chains $\frac{\mu_{5\vec{\gamma}}}{m_\Gamma}n_\Gamma\widehat{\omega}_{(\Gamma,\alpha)}$ to $\hocolim{}{\frac{\mu_{5\vec{\gamma}}}{m_\Gamma}n_\Gamma\widehat{\omega}_{(\Gamma,\alpha)}}$ follows from the results in \S\ref{s:thicken}. 

\subsubsection{Gluing the pieces $\widehat{M}_{(\Gamma,\alpha)}$ together}

We modify this cycle to simplify the homology of the domains. Recall from \S\ref{ss:compress-codim1} that each $\widehat{B}_{T(v)}$ is homotopy equivalent to $I^{\ell(v)-3}\times S^{|\tvec{a}(v)|}$ for some $\ell(v)$ and $\vect{a}(v)$, and $\widehat{B}_{(\Gamma,\alpha)}$ is homotopy equivalent to $I^e\times \prod_v S^{|\tvec{a}(v)|}$, where $e$ is the excess of $\Gamma$. We may replace each $\widehat{B}_{T(v)}$ with $I^{\ell(v)-3}\times S^{|\tvec{a}(v)|}$ without changing the resulting homology class of the cycle $\overline{\omega}_{5\vec{\gamma}}$. So we let anew
\[  \widehat{B}_{T(v)}=I^{\ell(v)-3}\times S^{|\tvec{a}(v)|}, \]
and $\widehat{B}_{(\Gamma,\alpha)}=I^e\times \prod_v S^{|\tvec{a}(v)|}$.
\begin{Lem}\label{lem:E}
Let $R_{(\Gamma,\alpha)}$ be the fat wedge in the product $\prod_v S^{|\tvec{a}(v)|}\subset \widehat{B}_{(\Gamma,\alpha)}$ (see \S\ref{ss:merge}) and let $\partial_\sqcup I^e=(I\times \partial I^{e-1})\cup (\{0\}\times I^{e-1})$. We define the space $M_{(\Gamma,\alpha)}$ by the homotopy pushout of the following diagram.
\begin{equation}\label{eq:hpushout}
 \widehat{B}_{(\Gamma,\alpha)}\cup (I^e\times \widehat{C}R_{(\Gamma,\alpha)})\longleftarrow
\partial_\sqcup I^e\times (\textstyle\prod_v S^{|\tvec{a}(v)|}\cup \widehat{C}R_{(\Gamma,\alpha)})\longrightarrow
C(\partial_\sqcup I^e\times \textstyle\prod_v S^{|\tvec{a}(v)|}). 
\end{equation}
\begin{enumerate}
\item We have 
$ M_{(\Gamma,\alpha)}\simeq \left\{\begin{array}{ll}
S^{8k-12} & (e=0),\\
D^{8k-12+e} & (e\geq 1).
\end{array}\right.$
\item There is an extension of $\widehat{\omega}_{(\Gamma,\alpha)}\colon \widehat{B}_{(\Gamma,\alpha)}\to B\Diff_\partial(D^{2k})$ to a map $\ve_{(\Gamma,\alpha)}\colon M_{(\Gamma,\alpha)}\to B\Diff_\partial(D^{2k})$.
\end{enumerate}
\end{Lem}
The space $M_{(\Gamma,\alpha)}$ is obtained by attaching a mapping cylinder of the natural map $\partial_\sqcup I^e\times (\textstyle\prod_v S^{|\tvec{a}(v)|}\cup \widehat{C}R_{(\Gamma,\alpha)})\to C(\partial_\sqcup I^e\times \textstyle\prod_v S^{|\tvec{a}(v)|})$ to $\widehat{B}_{(\Gamma,\alpha)}\cup (I^e\times \widehat{C}R_{(\Gamma,\alpha)})$. The motivation for the definition of $M_{(\Gamma,\alpha)}$ would be clear from the assertion (1). The assertion (2) is equivalent to that 
\begin{itemize}
\item there is an extension $\widehat{B}_{(\Gamma,\alpha)}\cup (I^e\times \widehat{C}R_{(\Gamma,\alpha)})\to B\Diff_\partial(D^{2k})$ of $\widehat{\omega}_{(\Gamma,\alpha)}$,
\item whose restriction to the subspace $\partial_\sqcup I^e\times (\textstyle\prod_v S^{|\tvec{a}(v)|}\cup \widehat{C}R_{(\Gamma,\alpha)})$ is homotopic to the restriction of a null-homotopy of $\widehat{\omega}_{(\Gamma,\alpha)}|_{\partial_\sqcup I^e\times \prod_v S^{|\tvec{a}(v)|}}$.
\end{itemize}
\begin{proof}
(1) Since the right two terms in (\ref{eq:hpushout}) are empty when $e=0$, it suffices to consider only the case $e\geq 1$. Since $(\prod_v S^{|\tvec{a}(v)|})\cup \widehat{C}R_{(\Gamma,\alpha)}\simeq S^{8k-12}$, we have the following homotopy equivalences 
\[\begin{split}
 &\widehat{B}_{(\Gamma,\alpha)}\cup (I^e\times \widehat{C}R_{(\Gamma,\alpha)})
=(I^e\times\textstyle\prod_v S^{|\tvec{a}(v)|})\cup (I^e\times \widehat{C}R_{(\Gamma,\alpha)})\simeq I^e\times S^{8k-12},\\
 &\partial_\sqcup I^e\times (\textstyle\prod_v S^{|\tvec{a}(v)|}\cup \widehat{C}R_{(\Gamma,\alpha)})\simeq \partial_\sqcup I^e\times S^{8k-12},\qquad C(\partial_\sqcup I^e\times \textstyle\prod_v S^{|\tvec{a}(v)|})\simeq *,
\end{split} \]
and we have the following homotopy commutative diagram.
\[ \xymatrix{
  \widehat{B}_{(\Gamma,\alpha)}\cup (I^e\times \widehat{C}R_{(\Gamma,\alpha)}) \ar[d]_{\simeq} &
  \partial_\sqcup I^e\times (\textstyle\prod_v S^{|\tvec{a}(v)|}\cup \widehat{C}R_{(\Gamma,\alpha)}) \ar[d]_{\simeq} \ar[r] \ar[l]&
  C(\partial_\sqcup I^e\times \textstyle\prod_v S^{|\tvec{a}(v)|}) \ar[d]^{\simeq} \\
  I^e\times S^{8k-12} & \partial_\sqcup I^e\times S^{8k-12} \ar[r] \ar[l] & {*}\\
}\]
Taking the homotopy pushouts of the two rows, we get the homotopy equivalence \[M_{(\Gamma,\alpha)}\simeq C_*\Sigma^{e-1} S^{8k-12}\simeq D^{8k-12+e},\] where $C_*$ is the reduced cone.

(2) We prove this in two steps. The first step is to prove that $\widehat{\omega}_{(\Gamma,\alpha)}$ extends over $I^e\times \widehat{C}R_{(\Gamma,\alpha)}$. 
We choose a spanning tree $\Gamma_0$ of $\Gamma$. There is an increasing sequence $I_1\subset I_2\subset\cdots\subset I_r=\Gamma_0$, $r=|V(\Gamma)|$, of connected subtrees of $\Gamma_0$ such that 
\begin{itemize}
\item $I_1$ consists of one vertex.
\item For each $i\leq r-1$, $I_{i+1}$ is a connected tree obtained from $I_i$ by adding an edge. 
\end{itemize}
Let $\calV=\{V_1,V_2,\ldots,V_r\}$ be the set of the disjoint handlebodies in $D^{2k}$ used to define $\widehat{\omega}_{(\Gamma,\alpha)}$ (in \S\ref{ss:gluing-B}).
Let $J_i$ be the subset of $\calV$ corresponding to the set of vertices of $I_i$. Then the surgeries on the handlebodies in $J_i$ give a family of string links of several cubes in $I^{2k}$. Namely, the iterated surgery for $J_i$ can be merged into a surgery on a single handlebody, some leaf of which may be linked each other. The surgery on the bigger handlebody corresponds to a family of string link. We have a system of Brunnian null-isotopies for the family of string links for $J_r$ that is inductively constructed by iterating Lemma~\ref{lem:compos-brun} for the extensions $J_1\to J_2\to \cdots\to J_r$. The existence of a system of Brunnian null-isotopies for the family over $\widehat{B}_{(\Gamma,\alpha)}$ of string links for $J_r$ implies that the family extends over $I^e\times \widehat{C}R_{(\Gamma,\alpha)}$ as in Lemma~\ref{lem:brun-cone}. 
Thus we have an extension $\widehat{B}_{(\Gamma,\alpha)}\cup (I^e\times\widehat{C}R_{(\Gamma,\alpha)})\to B\Diff_\partial(D^{2k})$ of $\widehat{\omega}_{(\Gamma,\alpha)}$.

The next step is to extend the map obtained in the previous paragraph over the rest of $M_{(\Gamma,\alpha)}$, namely, over the mapping cylinder of $\partial_\sqcup I^e\times (\textstyle\prod_v S^{|\tvec{a}(v)|}\cup \widehat{C}R_{(\Gamma,\alpha)})\to C(\partial_\sqcup I^e\times \textstyle\prod_v S^{|\tvec{a}(v)|})$. The extension is obtained since we may assume that the restriction of the maps $\widehat{\omega}_{T(v)}\colon \widehat{B}_{T(v)}\to \Emb^\fr_\partial((I^k)^{\cup p_v}\cup (I^{k-1})^{\cup q_v},I^{2k})$ for the vertices $v$ in $\Gamma$ of valence $\geq 4$ to $\partial_\sqcup I^{\ell(v)-3}\times S^{|\tvec{a}|}$ is a constant map. Then by the Brunnian property of the brackets for $T(v)$, the map $\widehat{\omega}_{(\Gamma,\alpha)}$ can be deformed by a homotopy to that such that the restriction of $\widehat{\omega}_{(\Gamma,\alpha)}$ to $\partial_\sqcup I^e\times \prod_v S^{|\tvec{a}(v)|}$ is a constant map. 
\end{proof}

We would like to glue the maps $\ve_{(\Gamma,\alpha)}\colon M_{(\Gamma,\alpha)}\to B\Diff_\partial(D^{2k})$ together, instead of the maps $\widehat{\omega}_{(\Gamma,\alpha)}\colon \widehat{B}_{(\Gamma,\alpha)}\to B\Diff_\partial(D^{2k})$, to get a cycle in $B\Diff_\partial(D^{2k})$ homologous to $\overline{\omega}_{5\vec{\gamma}}$. To consider the resulting map $K_{5\vec{\gamma}}\to B\Diff_\partial(D^{2k})$ as a cycle, we will need a fundamental class on $K_{5\vec{\gamma}}$ defined by the sum of all the (oriented) $(8k-10)$-cells in $K_{5\vec{\gamma}}$ oriented by the rule in \S\ref{ss:ori-chain}. In order that the sum of the $(8k-10)$-cells gives a cycle on $K_{5\vec{\gamma}}$, we need to modify $M_{(\Gamma,\alpha)}$ further since there may be subsets of the boundary of each $(8k-10)$-dimensional piece $M_{(\Gamma,\alpha)}$ that is not cancelled by other pieces from different graphs. For convenience, we define $\partial M_{(\Gamma,\alpha)}$ by the homotopy pushout of the following diagram.
\[ \xymatrix{
  (\{1\}\times \partial I^{e-1})\times (\textstyle\prod_v S^{|\tvec{a}(v)|}\cup \widehat{C}R_{(\Gamma,\alpha)}) \ar[r] \ar[d] & 
  C((\{1\}\times \partial I^{e-1})\times\textstyle\prod_v S^{|\tvec{a}(v)|})\\
  (\{1\}\times I^{e-1})\times (\textstyle\prod_v S^{|\tvec{a}(v)|}\cup \widehat{C}R_{(\Gamma,\alpha)}) & 
}
\]
The restriction of $\ve_{(\Gamma,\alpha)}$ to $\partial M_{(\Gamma,\alpha)}$ represents the boundary of $\ve_{(\Gamma,\alpha)}$ as a chain.

The bare subset of the boundary of $M_{(\Gamma,\alpha)}$ for an excess 2 graph $(\Gamma,\alpha)$ is the complement of a regular neighborhood of the following subset. 
\begin{itemize}
\item For the excess 2 graph $(X,\alpha_X)$, the pieces of the lower excess graphs are glued to disks in a regular neighborhood of an embedded Lie-hedron $L_4$ in $\partial M_{(X,\alpha_X)}$. 
\item For the excess 2 graph $(Y,\alpha_Y)$, the pieces of the lower excess graphs are glued to disks in a regular neighborhood of an embedded complete bipartite graph $K_{3,3}$ in $\partial M_{(Y,\alpha_Y)}$ (as in Figure~\ref{fig:L5}.)
\end{itemize}
In any case, there is an embedded graph $G$ in $\partial M_{(\Gamma,\alpha)}$ such that the restriction of $\ve_{(\Gamma,\alpha)}$ to the complement $\partial M_{(\Gamma,\alpha)}-N_G$ of an open regular neighborhood $N_G$ of $G$ in $\partial M_{(\Gamma,\alpha)}$ is null-homotopic. Thus we get the following.
\begin{Lem}\label{lem:N_G}
Let $(\Gamma,\alpha)=(X,\alpha_X)$ or $(Y,\alpha_Y)$. Let $G$ be the embedded graph in $\partial M_{(\Gamma,\alpha)}$ as above.
The map $\ve_{(\Gamma,\alpha)}\colon M_{(\Gamma,\alpha)}\to B\Diff_\partial(D^{2k})$ has an extension $\widehat{\ve}_{(\Gamma,\alpha)}\colon \widehat{M}_{(\Gamma,\alpha)}\to B\Diff_\partial(D^{2k})$,
where 
\[ \widehat{M}_{(\Gamma,\alpha)}=M_{(\Gamma,\alpha)}\cup C(\partial M_{(\Gamma,\alpha)}-N_G)\simeq M_{(\Gamma,\alpha)}/(\partial M_{(\Gamma,\alpha)}-N_G). \]
\end{Lem}

For a morphism $(\Gamma,\alpha)\to (\Gamma',\alpha')$ in $\calI_{5\vec{\gamma}}$ that corresponds to contracting an edge $(v_1,v_2)$ of $\Gamma$ to a vertex $v$ of $\Gamma'$, a natural map $\widehat{M}_{(\Gamma,\alpha)}\to \widehat{M}_{(\Gamma',\alpha')}$ is induced from the collapsing maps $\widehat{B}_{T(v_1)}\times \widehat{B}_{T(v_2)}\to \widehat{B}_{T(v)}$ and their induced map $R_{(\Gamma,\alpha)}\to R_{(\Gamma',\alpha')}$. Note that the extension $\widehat{\ve}_{(\Gamma,\alpha)}$ depends on the choice of $\Gamma_0$ and the sequence $\{I_i\}$ in the proof of Lemma~\ref{lem:E}(2). We see that different choices give homotopically compatible extensions.

\begin{Lem}\label{lem:K}
The diagram
$\widehat{\ve}_\bullet\colon \widehat{M}_\bullet \to B\Diff_\partial(D^{2k})$
parametrized by $\calI_{5\vec{\gamma}}$ is homotopy coherent. Hence we have a natural extension
\[ \hocolim{(\Gamma,\alpha)\in\calI_{5\vec{\gamma}}}{\frac{\mu_{5\vec{\gamma}}}{m_\Gamma}n_\Gamma\widehat{\ve}_{(\Gamma,\alpha)}}\colon \hocolim{(\Gamma,\alpha)\in\calI_{5\vec{\gamma}}}{\frac{\mu_{5\vec{\gamma}}}{m_\Gamma}n_\Gamma\widehat{M}_{(\Gamma,\alpha)}}\to B\Diff_\partial(D^{2k}), \]
where $n_\Gamma=2^{\text{excess}(\Gamma)}$. 
We define $K_{5\vec{\gamma}}:=\hocolim{(\Gamma,\alpha)\in\calI_{5\vec{\gamma}}}{\frac{\mu_{5\vec{\gamma}}}{m_\Gamma}n_\Gamma\widehat{M}_{(\Gamma,\alpha)}}$ and $\kappa_{5\vec{\gamma}}:=\hocolim{(\Gamma,\alpha)\in\calI_{5\vec{\gamma}}}{\frac{\mu_{5\vec{\gamma}}}{m_\Gamma}n_\Gamma\widehat{\ve}_{(\Gamma,\alpha)}}$.
\end{Lem}
\begin{proof}
First, we consider the homotopy commutativity of the following diagram for a morphism $(\Gamma,\alpha)\to (\Gamma',\alpha')$ in $\calI_{5\vec{\gamma}}$.
\begin{equation}\label{eq:hcommutative-E}
\vcenter{
 \xymatrix{
  \widehat{M}_{(\Gamma,\alpha)} \ar[rrd]^-{\widehat{\ve}_{(\Gamma,\alpha)}} \ar[d] & &\\
  \widehat{M}_{(\Gamma',\alpha')} \ar[rr]_-{\widehat{\ve}_{(\Gamma',\alpha')}} & & B\Diff_\partial(D^{2k})
} 
}
\end{equation}
We know that there is a homotopy between the restrictions of the chains to $\widehat{B}_{(\Gamma,\alpha)}$ and $\widehat{B}_{(\Gamma',\alpha')}$ by Lemma~\ref{lem:codim1-coherence}. We also know that for a choice of the spanning tree $\Gamma_0$ of $\Gamma$ and of the extension $\widehat{\ve}_{(\Gamma,\alpha)}$ over the thickened cone $\widehat{C}R_{(\Gamma,\alpha)}$, the diagram (\ref{eq:hcommutative-E}) is homotopy commutative. Namely, if we choose the spanning trees $\Gamma_0$ and $\Gamma_0'$ of $\Gamma$ and $\Gamma'$, respectively, compatible with respect to the contraction of an edge in $\Gamma\to \Gamma'$, and if we choose the systems of Brunnian null-isotopies for $\Gamma_0$ and $\Gamma_0'$ compatibly, then the extensions $\widehat{\ve}_{(\Gamma,\alpha)}$ and $\widehat{\ve}_{(\Gamma',\alpha')}$ over the thickened cones $\widehat{C}R_{(\Gamma,\alpha)}$ and $\widehat{C}R_{(\Gamma',\alpha')}$ are homotopically compatible. For the general choices of the spanning trees $\Gamma_0$ and $\Gamma_0'$ of $\Gamma$ and $\Gamma'$, respectively, which may not be compatible with the edge contraction, we recall the following general fact: if $P$ is a product of spheres of dimensions $\geq 1$, and if $Q\subset P$ is the fat wedge of the spheres, the map $\Sigma Q\to \Sigma P$ induced by the inclusion splits with cofiber $\Sigma(P/Q)$ (\cite[p.1662]{BBCG}). It follows that the first map in the exact sequence $[\Sigma Q,Z]\to [P/Q,Z]\to [P,Z]$ is trivial. Hence an extension of a map $P\to Z$ to $P\cup CQ\to Z$ is unique up to homotopy\footnote{The homotopy between two choices of extensions could be constructed by using the Brunnian property. Namely, the null-homotopy over $Q$ along the spanning tree $\Gamma_0$ is constructed by the thickened cone which is obtained from the compositions of the systems of Brunnian null-isotopies of intermediate $\Psi$-graphs. Two spanning trees can be related by a sequence of local modifications in the graph $\Gamma$, each of which gives rise to a homotopy between thickened cones induced by the system of Brunnian null-isotopies.}. We apply this fact to $P=\widehat{B}_{(\Gamma,\alpha)}$ and $Q=R_{(\Gamma,\alpha)}$ to obtain the homotopy commutativity of the diagram (\ref{eq:hcommutative-E}). 

We next consider the homotopy coherence of the following diagram.
\[ \xymatrix{
  \widehat{M}_{(\Gamma,\alpha)} \ar[r] \ar[dd] \ar[rrd]|(.465)\hole  & \widehat{M}_{(\Gamma_1,\alpha_1)} \ar[dd] \ar[rd] & \\
  & & B\Diff_\partial(D^{2k})\\
  \widehat{M}_{(\Gamma_2,\alpha_2)} \ar[r] \ar[rru]|(.465)\hole& \widehat{M}_{(\Gamma_3,\alpha_3)} \ar[ru] & 
} \]
We know that at each arrow in this square the homotopy commutativity of (\ref{eq:hcommutative-E}) holds. We consider whether there is a homotopy between the two coherences along the two paths in the above square from $\widehat{M}_{(\Gamma,\alpha)}$ to $\widehat{M}_{(\Gamma_3,\alpha_3)}$. In this case we consider the fibration sequence
\[ \Map_*(\Sigma P,Z)\to \Map_*(\Sigma Q,Z)\to \Map_*(P/Q,Z)\to \Map_*(P,Z) \]
of pointed mapping spaces for the spaces $P,Q$ in the previous paragraph. Since the homomorphism $\pi_1(\Map_*(P/Q,Z))\to \pi_1(\Map_*(P,Z))$ is injective, the 1-parameter family of extensions $P/Q\simeq P\cup CQ\to Z$ of a single map $P\to Z$ is null-homotopic in the space $\Map_*(P/Q,Z)$. We apply this fact to $P=\widehat{B}_{(\Gamma,\alpha)}$ and $Q=R_{(\Gamma,\alpha)}$ to obtain the desired homotopy coherence. 
\end{proof}

\subsection{Homology of $\widehat{M}_{(\Gamma,\alpha)}$}

\begin{Lem}\label{lem:H(E)}
Let $(\Gamma,\alpha)=(X,\alpha_X)$ or $(Y,\alpha_Y)$. Let $G$ be as in Lemma~\ref{lem:N_G}. Then we have
\[ H_*(m\widehat{M}_{(\Gamma,\alpha)};\Q)\cong\left\{\begin{array}{ll}
\Q^{m-1} & (*=8k-10),\\
\Q^{b_1(G)} & (*=8k-12),\\
0 & (1\leq *\leq 8k-13,*=8k-11),\\
\Q & (*=0).
\end{array}\right. \] 
\end{Lem}
\begin{proof}
To compute the homology of $m\widehat{M}_{(\Gamma,\alpha)}\simeq mM_{(\Gamma,\alpha)}/(\partial M_{(\Gamma,\alpha)}-N_G)$, we consider the homology exact sequence for the triple $(mM_{(\Gamma,\alpha)},\partial M_{(\Gamma,\alpha)}, \partial M_{(\Gamma,\alpha)}-N_G)$:
\begin{equation}\label{eq:triple}
 \to H_*(\partial M_{(\Gamma,\alpha)},\partial M_{(\Gamma,\alpha)}-N_G;\Q)
\to H_*(mM_{(\Gamma,\alpha)},\partial M_{(\Gamma,\alpha)}-N_G;\Q)
\to H_*(mM_{(\Gamma,\alpha)},\partial M_{(\Gamma,\alpha)};\Q)
\to 
\end{equation}
By Lemma~\ref{lem:E}(1), we have $H_*(mM_{(\Gamma,\alpha)},\partial M_{(\Gamma,\alpha)};\Q)\cong H_*(D^{8k-10},\partial D^{8k-10};\Q)^{\oplus m}$. Also, by the homotopy equivalence $\partial M_{(\Gamma,\alpha)}\simeq S^{8k-11}$ and by Alexander duality, we have
\[ \widetilde{H}_*(\partial M_{(\Gamma,\alpha)}-N_G;\Q)\cong \widetilde{H}^{8k-12-*}(N_G;\Q)\cong \left\{\begin{array}{ll}
\Q^{b_1(G)} & (*=8k-13),\\
0 & (\text{otherwise}).
\end{array}\right. \]
Then it follows from the homology exact sequence for the pair $(\partial M_{(\Gamma,\alpha)}, \partial M_{(\Gamma,\alpha)}-N_G)$ that
\[ 
  H_*(\partial M_{(\Gamma,\alpha)}, \partial M_{(\Gamma,\alpha)}-N_G;\Q)\cong \left\{\begin{array}{ll}
  \Q & (*=8k-11),\\
  \Q^{b_1(G)} & (*=8k-12),\\
  0 & (\text{otherwise}).
  \end{array}\right.
\] 
Then the long exact sequence (\ref{eq:triple}) is as follows:
\[ \begin{array}{rclclcl}
  H_{8k-10}: & \qquad\longrightarrow & 0 &\longrightarrow & ? &\longrightarrow &\Q^m \\
  H_{8k-11}: & \qquad\longrightarrow &\Q &\longrightarrow & ? &\longrightarrow &0 \\
  H_{8k-12}: & \qquad\longrightarrow &\Q^{b_1(G)} &\stackrel{\cong}{\longrightarrow} & ? &\longrightarrow &0 \\
  H_{8k-13}: & \qquad\longrightarrow &0 &\stackrel{\cong}{\longrightarrow} & ? &\longrightarrow &0 \\
  \vdots\ \  & & \vdots & & \vdots & & \vdots\\
  H_0: & \qquad\longrightarrow &0 &\stackrel{\cong}{\longrightarrow} & ? &\longrightarrow &0 \\
\end{array}
\]
Since the connecting homomorphism 
\[ H_{8k-10}(mM_{(\Gamma,\alpha)},\partial M_{(\Gamma,\alpha)};\Q)\to H_{8k-11}(\partial M_{(\Gamma,\alpha)},\partial M_{(\Gamma,\alpha)}-N_G;\Q) \] is surjective, we have
\[ H_*(mM_{(\Gamma,\alpha)},\partial M_{(\Gamma,\alpha)}-N_G;\Q)\cong\left\{\begin{array}{ll}
\Q^{m-1} & (*=8k-10),\\
\Q^{b_1(G)} & (*=8k-12),\\
0 & (\text{otherwise}).
\end{array}\right.\]
\end{proof}

\begin{Lem}\label{lem:H(E)2}
Let $(\Gamma,\alpha)$ be an excess 1 graph obtained from $(X,\alpha_X)$ or $(Y,\alpha_Y)$ by one expansion. Let $G$ be a set of three distinct points in $\partial M_{(\Gamma,\alpha)}$. Then we have
\[ H_*(m\widehat{M}_{(\Gamma,\alpha)};\Q)\cong\left\{\begin{array}{ll}
\Q^{m-1} & (*=8k-11),\\
\Q^2 & (*=8k-12),\\
0 & (1\leq *\leq 8k-13,*=8k-10),\\
\Q & (*=0).
\end{array}\right. \]
\end{Lem}
\begin{proof}
The proof is parallel to Lemma~\ref{lem:H(E)}. Namely, $N_G$ is homotopy equivalent to a three point set, and we have $H_*(mM_{(\Gamma,\alpha)},\partial M_{(\Gamma,\alpha)};\Q) \cong H_*(D^{8k-11},\partial D^{8k-11};\Q)^{\oplus m}$, and
\[ \begin{split}
  & \widetilde{H}_*(\partial M_{(\Gamma,\alpha)}-N_G;\Q)\cong \widetilde{H}^{8k-13-*}(N_G;\Q)\cong\left\{\begin{array}{ll}
  \Q^{b_0(G)-1} & (*=8k-13),\\
  0 & (\text{otherwise}),
  \end{array}\right.\\
  & H_*(\partial M_{(\Gamma,\alpha)},\partial M_{(\Gamma,\alpha)}-N_G;\Q)\cong\left\{\begin{array}{ll}
  \Q^{b_0(G)} & (*=8k-12),\\
  0 & (\text{otherwise}).
  \end{array}\right.
\end{split} \]
Then the long exact sequence (\ref{eq:triple}) for the triple $(mM_{(\Gamma,\alpha)},\partial M_{(\Gamma,\alpha)}, \partial M_{(\Gamma,\alpha)}-N_G)$ in this setting is as follows:
\[ \begin{array}{rclclcl}
  H_{8k-11}: & \qquad\longrightarrow & 0 &\longrightarrow & ? &\longrightarrow &\Q^m \\
  H_{8k-12}: & \qquad\longrightarrow &\Q^{b_0(G)} &\longrightarrow & ? &\longrightarrow &0 \\
  H_{8k-13}: & \qquad\longrightarrow &0 &\stackrel{\cong}{\longrightarrow} & ? &\longrightarrow &0 \\
  \vdots\ \  & & \vdots & & \vdots & & \vdots\\
  H_0: & \qquad\longrightarrow &0 &\stackrel{\cong}{\longrightarrow} & ? &\longrightarrow &0 \\
\end{array}
\]
Since the connecting homomorphism $H_{8k-11}(mM_{(\Gamma,\alpha)},\partial M_{(\Gamma,\alpha)};\Q)\to H_{8k-12}(\partial M_{(\Gamma,\alpha)},\partial M_{(\Gamma,\alpha)}-N_G;\Q)$ takes $\Q^m$ onto a 1-dimensional subspace of $\Q^{b_0(G)}$, we have
\[ H_*(mM_{(\Gamma,\alpha)},\partial M_{(\Gamma,\alpha)}-N_G;\Q)\cong\left\{\begin{array}{ll}
\Q^{m-1} & (*=8k-11),\\
\Q^{b_0(G)-1} & (*=8k-12),\\
0 & (\text{otherwise}).
\end{array}\right.\]
\end{proof}

\subsection{Homology of the homotopy colimit}

We study the rational homology of $K_{5\vec{\gamma}}$ by using the spectral sequence for the homology of homotopy colimit (see e.g. \cite[Proposition~18.17]{Du}).
\begin{Thm}[{e.g. \cite[Theorem~18.3(a)]{Du}}]
Let $D\colon \calI\to \mathrm{Top}$ be a diagram of spaces.
There is a spectral sequence 
\[ E_{p,q}^2=H_p(\calI;H_q(D;\Q))\Rightarrow H_{p+q}(\hocolim{}{D};\Q).\]
The differentials have the form $d^r\colon E_{p,q}^r\to E_{p-r,q+r-1}^r$.
\end{Thm}
We apply this to $K_{5\vec{\gamma}}=\hocolim{(\Gamma,\alpha)\in\calI_{5\vec{\gamma}}}{\frac{\mu_{5\vec{\gamma}}}{m_\Gamma}n_\Gamma\widehat{M}_{(\Gamma,\alpha)}}$ of Lemma~\ref{lem:K}. Since $H_*(\frac{\mu_{5\vec{\gamma}}}{m_\Gamma}n_\Gamma\widehat{M}_{(\Gamma,\alpha)};\Q)=0$ for $1\leq *\leq 8k-13$ by Lemmas~\ref{lem:H(E)} and \ref{lem:H(E)2}, and the geometric realization of $\calI_{5\vec{\gamma}}$ is 2-dimensional, we have $E_{p,q}^2=0$ for $3\leq p+q\leq 8k-13$. Hence we have the following.  
\begin{Lem}\label{lem:vanish-H(K)}
\begin{enumerate}
\item $H_i(K_{5\vec{\gamma}};\Q)=0$ for $3\leq i\leq 8k-13$. 
\item $H_{8k-10}(K_{5\vec{\gamma}};\Q)$ has a nontrivial fundamental class.
\end{enumerate}
\end{Lem}
We also use the following result (see also Theorem~\ref{thm:KRW}).
\begin{Thm}[Kupers--Randal-Williams (\cite{KRW})]\label{thm:BDiff-connectivity}
$B\Diff_\partial(D^{2k})$ is rationally $(2k-2)$-connected.
\end{Thm}
\begin{Thm}\label{thm:Q-Hurewicz}
The image of the map $\kappa_{5\vec{\gamma}*}\colon H_{8k-10}(K_{5\vec{\gamma}};\Q)\to H_{8k-10}(B\Diff_\partial(D^{2k});\Q)$ induced by\\ $\kappa_{5\vec{\gamma}}\colon K_{5\vec{\gamma}}\to B\Diff_\partial(D^{2k})$ is included in the image of the rational Hurewicz map 
\[ \pi_{8k-10}(B\Diff_\partial(D^{2k}))\otimes\Q\to H_{8k-10}(B\Diff_\partial(D^{2k});\Q). \]
\end{Thm}
\begin{proof}
Since it is easy to see that $B\Diff_\partial(D^{2k})$ is a path-connected homotopy associative $H$-space with trivial Whitehead products $[\alpha,\beta]=[\alpha,0]+[0,\beta]=0$, we have that $H_*(B\Diff_\partial(D^{2k});\Q)$ is a graded polynomial algebra generated by its rational homotopy groups (\cite[Appendix]{MM}). 
By Theorem~\ref{thm:BDiff-connectivity}, we have that $H^{8k-10}(B\Diff_\partial(D^{2k});\Q)$ is generated by indecomposable classes and classes of the form $p\cup q$, where $p\in H^i$, $q\in H^j$, $i+j=8k-10$, and $2k-1\leq i,j\leq 6k-9$. It follows from Lemma~\ref{lem:vanish-H(K)} and $3\leq 2k-1$, $6k-9\leq 8k-13$ that $\kappa_{5\vec{\gamma}}^*(p\cup q)=\kappa_{5\vec{\gamma}}^*p\cup \kappa_{5\vec{\gamma}}^*q=0$. Hence the image of $\kappa_{5\vec{\gamma}*}$ in $H_{8k-10}(B\Diff_\partial(D^{2k});\Q)$ is detected only by indecomposable classes. This implies that the image of $\kappa_{5\vec{\gamma}*}$ is included in the image of the rational Hurewicz map. 
\end{proof}

\section{Leaf forms and orientations}\label{s:ind-ori}

We shall introduce leaf forms in a vertex surgery to prove a nondegeneracy property of the surgery and to fix an orientation convention for it. This uncovers a hidden structure in a vertex surgery that fits the graph complex.

\subsection{Half-edge orientations of the chains}\label{ss:ori-chain}

We split each edge $e$ of a directed graph into a pair of half-edges $e_{\pm}$, which are arranged so that $e$ is directed from $e_-$ to $e_+$. We consider that $\deg\,{e_+}=k-1$ and $\deg\,{e_-}=k$. Then a labelling $E(\Gamma)\to \{1,2,\ldots,|E(\Gamma)|\}$ determines the orientation
\begin{equation}\label{eq:ori-o(v)}
 \bigwedge_{e\in E(\Gamma)}(e_+\wedge e_-)=\bigwedge_{v\in V(\Gamma)}o(v) 
\end{equation}
of the $\R$-vector space spanned by the set of half-edges, 
where $o(v)$ is a wedge product of the half-edges incident to the vertex $v$ possibly with a sign $-1$ multiplied. We call $o(v)$ a {\it vertex orientation} of $v$. This gives a choice of orientations of vertex surgeries.

\begin{Exa}
Suppose $d=2k\geq 6$ and $k$ is odd. Then $k-1$ and $k$ are even and odd, respectively. We consider the decomposition of the edges into half-edges as in Figure~\ref{fig:graph_X-half-edges} for the directed graph $(X,\alpha_X)$ of (\ref{eq:graph_X_Y}).
\begin{figure}[h]
\[\fig{-15mm}{30mm}{graph_X2.pdf} \quad \longrightarrow\quad
 \fig{-15mm}{30mm}{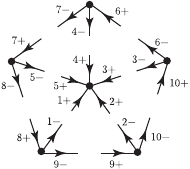}
 \]
\caption{Half-edge decomposition of $(X,\alpha_X)$.}\label{fig:graph_X-half-edges}
\end{figure}

The edge labelling from 1 to 10 gives the orientation
\[
\begin{split}
 \bigwedge_{i=1}^{10}(e_{i+}\wedge e_{i-})=&\,(e_{1+} e_{2+} e_{3+} e_{4+} e_{5+})\wedge (e_{6+}e_{4-}e_{7-})\wedge (-e_{7+}e_{5-}e_{8-})\\
& \wedge (-e_{8+}e_{1-}e_{9-})\wedge (-e_{9+}e_{2-}e_{10-})\wedge (-e_{10+}e_{3-}e_{6-}),
\end{split}
\]
where $\deg\,e_{i+}=k-1$ (even), $\deg\,{e_{i-}}=k$ (odd), and some $\wedge$ are omitted. We arrange each factor for a vertex so that the $+$ (incoming) elements are arranged before the $-$ (outgoing) elements.
\end{Exa}

\subsubsection{Normal Thom class ($\eta$-form)}\label{ss:thom}
For a topologically closed oriented smooth submanifold $A$ of an oriented manifold $N$, we denote by $\eta_A$ a closed form representative of the Thom class of the normal bundle $\nu_A$ of $A$. We identify the total space of $\nu_A$ with a small tubular neighborhood $N_A$ of $A\subset N$ and assume that $\eta_A$ has support in $N_A$. It has the useful property that the cohomology class $[\eta_A]$ is the Poincar\'{e}--Lefschetz dual of the fundamental class $[A]\in H_*(N,\partial N)$, when both $N$ and $A$ are compact. For another oriented submanifold $B$ of $N$ with $\dim\,B=\mathrm{codim}\,A$, the integral $\int_N\eta_A\wedge\eta_B=\int_B\eta_A$ gives the intersection pairing $A\cdot B$ in $N$.

\subsubsection{Orientation of the $V_i$-bundle.}\label{ss:ori-V-bundle}
A vertex orientation $o(v)$ defines an orientation of the $V_v$-bundle $\pi_{T_\ell(p,q)}\colon \widehat{E}^{T_\ell(p,q)}\to \widehat{B}_{T_\ell(p,q)}$ associated to the classifying map $\widehat{\omega}_{T_\ell(p,q)}\colon \widehat{B}_{T_\ell(p,q)}\to B\Diff_\partial^\fr(V_v)$ as follows.

We define the orientation $o(\widehat{E}^{T_\ell(p,q)})$ of $\widehat{E}^{T_\ell(p,q)}$ by the product of leaf forms defined as follows.
\begin{Def}[Leaf form]\label{def:leaf-form}
Let $b_1,\ldots,b_\ell\subset \mathrm{Int}\,V_v$, where $\ell$ is the valence of the vertex $v$, be the spheres given by the oriented leaves of the $\Psi_\ell$-graph (see Definition \ref{def:psi-graph}). Let $a_1,\ldots,a_\ell\subset \partial V_v$ be the meridian spheres of the leaves of the $\Psi_\ell$-graph oriented by the condition $\Lk(b_i,a_j)=\delta_{ij}$.
Let $\widetilde{a}_j=\widehat{B}_{T_\ell(p,q)}\times a_j$. Let $\lambda_j=\dim{b_j}$ and we consider $V_v$ as a fiber of $\widehat{E}^{T_\ell(p,q)}$. 
A {\it leaf form} $\theta_{a_j}$ for $a_j$ is a closed $\lambda_j$-form on $\widehat{E}^{T_\ell(p,q)}$ such that
\begin{enumerate}
\item $\int_{b_j}\theta_{a_m}=(-1)^{\dim{a_m}-1}\delta_{jm}$ for $(j,m)$ such that $\dim{b_j}+\dim{a_m}=2k-1$. 
\item The restriction of $\theta_{a_j}$ to $\widehat{B}_{T_\ell(p,q)}\times \partial V_v$ is $\pm\eta_{\widetilde{a}_j}$ for a choice of coorientation of $\widetilde{a}_j$.
\end{enumerate}
\end{Def}
The reason for the sign $(-1)^{\dim{a_m}-1}$ is that we have $\int_{b_j}\eta_{S(a_m)}=(-1)^{\dim{a_m}-1}\delta_{jm}$ (\cite[Lemma~4.1]{Wa21}).
It would help to recall that the Borromean string link to define a $\Psi_3$-surgery has spanning manifolds $S(a_1),S(a_2),S(a_3)\subset V_v$ that intersect in a general position with the triple intersection at one point. So we have $\int_{V_v}\eta_{S(a_1)}\wedge \eta_{S(a_2)}\wedge \eta_{S(a_3)}=\pm 1$. We will choose leaf forms that generalize the $\eta$-forms in this situation.

We choose leaf forms in $\widehat{E}^{T_\ell(p,q)}$ for all possible $\ell,p,q$ with $p+q=\ell$ in a compatible way.
\begin{Lem}\label{lem:leaf-form}
There exist leaf forms $\theta_{a_j}$ in $\widehat{E}^{T_\ell(p,q)}$ that satisfy the following conditions.
\begin{enumerate}
\item {\rm (Compatibility)} The system $\{\theta_{a_j}\}_j$ is compatible with those for the lower excess trees in the following sense. Let $(\sigma,\alpha)$ be a codimension one face tree of $T_\ell(p,q)$ that is decomposed into two trees $T'=T_{\ell_1}(p_1,q_1)$ and $T''=T_{\ell_2}(p_2,q_2)$, by cutting the internal edge at the middle point. Let $V_{T'}\tcoprod V_{T''}\subset V_v$ be the union of the handlebodies for the $\Psi$-graphs for $T'$ and $T''$, respectively. Let $\theta_{a_1}',\ldots,\theta_{a_{\ell_1}}'$ (resp. $\theta_{a_{\ell_1+1}}'',\ldots,\theta_{a_{\ell_1+\ell_2}}''$) be the leaf forms over $\widehat{B}_{T_\ell(p,q)}(\sigma,\alpha)=\widehat{B}_{T'}\times \widehat{B}_{T''}$ for the leaves of $T'$ (resp. $T''$) at the external half-edges corresponding to those of $T_\ell(p,q)$. Then $\theta_{a_j}'$ (resp. $\theta_{a_j}''$) is induced from $\theta_{a_j}$ by the natural bundle map over the map $\widehat{B}_{T_\ell(p,q)}((\sigma,\alpha)\to T_\ell(p,q))\colon \widehat{B}_{T_\ell(p,q)}(\sigma,\alpha)\to \widehat{B}_{T_\ell(p,q)}/{A_{T_\ell(p,q)}}$, which embeds a fiber $V_{T'}\tcoprod V_{T''}$ into $V_v$.
\item {\rm (Nondegeneracy)} The product $\theta_{a_1}\wedge\cdots\wedge \theta_{a_\ell}$ is a compact support volume form on $\mathrm{Int}\,\widehat{E}^{T_\ell(p,q)}$ giving nontrivial cohomology class in $H^{p(k-1)+qk}(\widehat{E}^{T_\ell(p,q)},\partial \widehat{E}^{T_\ell(p,q)};\R)$ such that
\[ \int_{\widehat{E}^{T_\ell(p,q)}}\theta_{a_1}\wedge\cdots\wedge \theta_{a_\ell}=\pm m_\ell. \]
Note that $p(k-1)+qk=\ell k-p=2k+\dim{\widehat{B}_{T_\ell(p,q)}}$.
\end{enumerate}
\end{Lem}
We postpone the proof of Lemma~\ref{lem:leaf-form}, which is technical and lengthy, to \S\ref{ss:comp-leaf-forms} and assume it now. 

\begin{Rem}\label{rem:crucial-nontriviality}
This is a crucial property of the $\Psi$-surgery. In fact, the nondegeneracy property of leaf forms is the only point in the computation of the integral invariant in \cite{Wa21} where the properties of the Borromean link is used and due to that the trivalence was essential. 
In Part I, this is used to handle the orientations of chains. Although it is used in Part I just to determine the orientation, to do so is a subtle task like distinguishing left-handed and right-handed Borromean links. 
In Part II, this is used to compute the relevant invariant, which can be expressed as a sum of products of the integrals of the forms as in Lemma~\ref{lem:int-wedge-theta}. In particular, the nondegeneracy property (2) is an analogue of the nontriviality of the Borromean string link, and is a crucial ingredient for the nontriviality of our extension of the construction to higher vertices.
\end{Rem}
 
\begin{Def}[Orientation of $\Psi$-surgery]\label{def:ori}
We define $o(\widehat{B}_{T_\ell(p,q)})$ so that the orientation $o(\widehat{E}^{T_\ell(p,q)})=o(\widehat{B}_{T_\ell(p,q)})\wedge o(D^{2k})$ agrees with that given by the nontrivial de Rham cohomology class 
\[ [\theta_{a_1}\wedge\cdots\wedge \theta_{a_\ell}]\in H^{p(k-1)+qk}(\widehat{E}^{T_\ell(p,q)},\partial \widehat{E}^{T_\ell(p,q)};\R). \]
\end{Def}
The following lemma is obvious from the above convention.
\begin{Lem}\label{lem:int-wedge-theta}
Under the orientation convention for $\widehat{E}^{T_\ell(p
,q)}$ above with respect to the leaf forms as in Lemma~\ref{lem:leaf-form}, we have
\[ \int_{\widehat{E}^{T_\ell(p,q)}}\theta_{a_1}\wedge\cdots\wedge \theta_{a_\ell}=m_\ell. \]
\end{Lem}

\begin{Rem}
Definition~\ref{def:leaf-form} of leaf forms can be generalized naturally to any $(V_v,\partial)$-bundles. In particular, they can be defined for the $(V_v,\partial)$-bundle $\overline{\pi}^{T_\ell(p,q)}\colon \overline{E}^{T_\ell(p,q)}\to \overline{B}_{T_\ell(p,q)}$, where $\overline{B}_{T_\ell(p,q)}=\underset{(\sigma,\alpha)\in\overrightarrow{\calP}_{T_\ell(p,q)}}{\mathrm{hocolim}}\, n_\sigma\widehat{B}_{T_\ell(p,q)}(\sigma,\alpha)$ (see Definition~\ref{def:barB}).
\begin{Lem}\label{lem:leaf-form-barE}
There exist leaf forms $\theta_{a_j}$ in $\overline{E}^{T_\ell(p,q)}$ that satisfy the following conditions.
\begin{enumerate}
\item The form $\theta_{a_j}$ extends that on $\widehat{E}^{T_\ell(p,q)}$.
\item The system $\{\theta_{a_j}\}_j$ is compatible with those on $\overline{E}^{\sigma}$ for the lower excess trees $\sigma$ in $\overrightarrow{\calP}_{T_\ell(p,q)}$.
\item The product $\theta_{a_1}\wedge\cdots\wedge \theta_{a_\ell}$ satisfies
\[ \int_{\overline{E}^{T_\ell(p,q)}-(2^{\ell-3}\widehat{E}^{T_\ell(p,q)})}\theta_{a_1}\wedge\cdots\wedge \theta_{a_\ell}=0. \]
In other words, the thickening of $2^{\ell-3}\widehat{E}^{T_\ell(p,q)}$ to $\overline{E}^{T_\ell(p,q)}$ does not change the integral of $\theta_{a_1}\wedge\cdots\wedge \theta_{a_\ell}$.
\end{enumerate}
\end{Lem}
Lemma~\ref{lem:leaf-form-barE} will be proved at the end of \S\ref{ss:comp-leaf-forms}.
\end{Rem}

\subsubsection{Orientation of the $(D^{2k},\partial)$-bundle.}
The product (\ref{eq:ori-o(v)}) of $o(v)$ defines an orientation of the $(D^{2k},\partial)$-bundle 
\[ \pi^{(\Gamma,\alpha)}\colon \widehat{E}^{(\Gamma,\alpha)}\to \widehat{B}_{(\Gamma,\alpha)}:=\prod_v \widehat{B}_{T(v)} \]
corresponding to the classifying map $\omega_{(\Gamma,\alpha)}=\prod_v\widehat{\omega}_{T(v)}\colon \widehat{B}_{(\Gamma,\alpha)}\to B\Diff_\partial^\fr(D^{2k})$, as follows.
\begin{enumerate}
\item[(i)] The order of the product of the half-edges at $v$ in $o(v)$ determines an identification of $T(v)$ with $T_\ell(p,q)$, and hence an identification of $V_v$ with the complement of the thickened string link $(N_{I^{k}})^{\cup p}\cup (N_{I^{k-1}})^{\cup q}\to I^{2k}$. If $o(v)$ has $-$ sign, then we define $o(\widehat{B}_{T(v)})$ to be the reverse of the orientation on $\widehat{B}_{T(v)}$ induced from $o(\widehat{B}_{T_\ell(p,q)})$.
\item[(ii)] The order of the factors $o(v)$ in (\ref{eq:ori-o(v)}) defines an orientation $o(\widehat{B}_{(\Gamma,\alpha)})=\bigwedge_v o(\widehat{B}_{T(v)})$ of $\widehat{B}_{(\Gamma,\alpha)}$, where we assume that the order of the product $\bigwedge_v o(\widehat{B}_{T(v)})$ agrees with that of the RHS of (\ref{eq:ori-o(v)}).
Then the total space $\widehat{E}^{(\Gamma,\alpha)}$ of the $(D^{2k},\partial)$-bundle is oriented by $\bigwedge_v o(\widehat{B}_{T(v)})\wedge o(D^{2k})$. We call the resulting orientation of $\widehat{E}^{(\Gamma,\alpha)}$ the {\it half-edge orientation}.
\end{enumerate}
The orientation of $\widehat{B}_{(\Gamma,\alpha)}$ induces that of the manifold pieces in $\overline{B}_{(\Gamma,\alpha)}$.

\begin{Thm}\label{thm:chainmap}
Suppose that $k$ is sufficiently large. The linear map $\overline{\phi}\colon \mathcal{GC}^{(\leq \ell)}\to S_*(B\Diff_\partial^\fr(D^{2k});\Q)$ satisfies the following identity (modulo degenerate chains)
\[ \partial \overline{\phi}_{(\Gamma,\alpha)}=\overline{\phi}_{\partial(\Gamma,\alpha)}. \]
\end{Thm}
It follows from Theorems~\ref{thm:l-valent-general} and \ref{thm:graph-surgery} that the identity holds up to the signs of the terms of $\partial(\Gamma,\alpha)$. Thus we need only to check the induced orientation on the boundary of $\overline{B}_{(\Gamma,\alpha)}$.

\subsection{Compatible leaf forms: Proof of Lemma~\ref{lem:leaf-form}}\label{ss:comp-leaf-forms}

We prove Lemma~\ref{lem:leaf-form} by induction on $\ell$. The proof is an analogue of that for intersection among spanning submanifolds in 4-valent vertex surgery in Appendix~\ref{s:leaf-form-mfd}. It is similar to Hain's construction of a formal connection (or a Maurer--Cartan element) on the complement of the Borromean rings in $S^3$ (\cite[Example~3]{Hai}, see also Remark~\ref{rem:MC-eq}). In the following, we consider the natural map 
\[ F_{(\sigma,\alpha)}:=\widehat{B}_{T_\ell(p,q)}((\sigma,\alpha)\to T_\ell(p,q))\colon\widehat{B}_{T_\ell(p,q)}(\sigma,\alpha)\to \widehat{B}_{T_\ell(p,q)}/{A_{T_\ell(p,q)}}.
\]
 We say that several smooth submanifolds $S_1,\ldots,S_r$ in a smooth manifold $M$ intersect transversally if at each intersection point $p\in \bigcap_i S_i$ the sum of the orthogonal complements $(T_pS_i)^{\perp}$ (as a subspace of $T_pM$) is the direct sum with respect to any Riemannian metric on $M$.

For $\ell=3$, the following lemma is sufficient.
\begin{Lem}
For $\ell=3$, there exist leaf forms in the $\Psi_3$-surgeries for $T_3(p,q)$ that satisfies the nondegeneracy condition of Lemma~\ref{lem:leaf-form}.
\end{Lem}
\begin{proof}
The $\Psi_3$-surgery for $T_3(p,q)$ is obtained from the Borromean string link of type $(2n-1,2n-1,2n-1;3n)$ by iterating suspensions and deloopings as in Example~\ref{ex:3-valent-family}. Since the graph of a delooping is isotopic to a suspension by Lemma~\ref{lem:suspension-delooping}, the graph of the family $\widehat{\omega}_{T_3(p,q)}\colon \widehat{B}_{T_3(p,q)}\to \Emb_\partial^\fr((I^k)^{\cup p}\cup (I^{k-1})^{\cup q},I^{2k})$ is isotopic to that obtained from $(2n-1,2n-1,2n-1;3n)$ by iterated suspensions. It follows from \cite[Definition~3.4 and Lemma~3.7]{Wa21} that the graph of $\widehat{\omega}_{T_3(p,q)}$ is a Borromean string link. Thus its components have framed spanning submanifolds that intersect transversally at one point. The $\eta$ forms for the spanning submanifolds satisfy the desired properties.
\end{proof}
\begin{Lem}\label{lem:leaf-form-extension}
Suppose $b\geq 1$ and that we have leaf forms on the restriction to 
$\partial(I^b\times V_v)$ of the trivial $V_v$-bundle $I^b\times V_v\to I^b$, where $V_v\simeq (S^{k-1})^{\vee p}\vee (S^k)^{\vee q}$. Then they can be extended over $I^b\times V_v$.
\end{Lem}
\begin{proof}
We consider the following exact sequence.
\begin{equation}\label{eq:extend-leaf}
\begin{split}
 & H^{k-1}(I^b\times V_v;\R)\to
  H^{k-1}(\partial(I^b\times V_v);\R)\to
  H^k(I^b\times V_v,\partial(I^b\times V_v);\R)\to \\
 & H^k(I^b\times V_v;\R)\to
  H^k(\partial(I^b\times V_v);\R)\to
  H^{k+1}(I^b\times V_v,\partial(I^b\times V_v);\R) \to \phantom{\displaystyle\sum}
\end{split}
\end{equation}
By Poincar\'{e}--Lefschetz duality, we have
\[ \left\{\begin{array}{l}
H^{k+1}(I^b\times V_v,\partial(I^b\times V_v);\R)\cong H^{k+b}(I^b\times V_v;\R),\\
  H^k(I^b\times V_v,\partial(I^b\times V_v);\R)\cong H^{k+1+b}(I^b\times V_v;\R). \phantom{\displaystyle\sum}
\end{array}\right. \]
Since $b\geq 1$, and $I^b\times V_v\simeq (S^{k-1})^{\vee p}\vee (S^k)^{\vee q}$, we have that the right hand sides are zero. Hence the left maps of (\ref{eq:extend-leaf}) are surjective and we obtain an extension over $I^b$. 
\end{proof}

\begin{Lem}[Compatibility]\label{lem:leaf-form-compatibility}
For $\ell>3$, suppose that Lemma~\ref{lem:leaf-form} holds for $\widehat{\omega}_{T_{\ell'}(p',q')}$ whenever $3\leq\ell'<\ell$ and $p'+q'=\ell'$. There exist leaf forms in the associated $V_v$-bundle for $\widehat{\omega}_{T_\ell(p,q)}$ that satisfies the compatibility condition (1) of Lemma~\ref{lem:leaf-form}.
\end{Lem}
\begin{proof}
First, we see that the families $\widehat{\omega}_{T_\ell(p,q)}(\sigma,\alpha)$ for face trees $(\sigma,\alpha)$ induce leaf forms on the restriction of $\widehat{\omega}_{T_\ell(p,q)}$ to $\partial\widehat{B}_{T_\ell(p,q)}$. Recall that $\widehat{B}_{T_\ell(p,q)}=I^{\ell-3}\times S^\lambda$ for some $\lambda>0$ and that $\widehat{B}_{T_\ell(p,q)}(\sigma,\alpha)=\widehat{B}_{T'}\times\widehat{B}_{T''}=(I^{\ell_1-3}\times S^{\lambda_1})\times (I^{\ell_2-3}\times S^{\lambda_2})=I^{\ell-4}\times S^{\lambda_1}\times S^{\lambda_2}$ for some $\ell_1,\ell_2\geq 3$, $\lambda_1,\lambda_2>0$, $\lambda_1+\lambda_2=\lambda$. By Lemma~\ref{lem:codim1-coherence}, we may deform the family 
\[\widehat{\omega}_{T_\ell(p,q)}(\sigma,\alpha)\colon \widehat{B}_{T_\ell(p,q)}(\sigma,\alpha)\to \Emb_\partial^\fr((I^k)^{\cup p}\cup (I^{k-1})^{\cup q},I^{2k})\]
so that its restriction to a small neighborhood of $(I^{\ell_1-3}\times S^{\lambda_1})\vee (I^{\ell_2-3}\times S^{\lambda_2})$ is constant. Also, the leaf forms over $\widehat{B}_{T_\ell(p,q)}(\sigma,\alpha)$ can be deformed near $(I^{\ell_1-3}\times S^{\lambda_1})\vee (I^{\ell_2-3}\times S^{\lambda_2})$ to the $\eta$-forms of the standard spanning disks. By Lemma~\ref{lem:leaf-form-extension}, we obtain leaf forms in the associated $V_v$-bundle over the image of $F_{(\sigma,\alpha)}$ after the deformation. Doing this for all codimension one face graphs $(\sigma,\alpha)$ of $T_\ell(p,q)$, we obtain leaf forms in the family $\widehat{\omega}_{T_\ell(p,q)}|_{\partial\widehat{B}_{T_\ell(p,q)}}$ over $\partial\widehat{B}_{T_\ell(p,q)}=\partial I^{\ell-3}\times S^\lambda$ induced from those in the families $\widehat{\omega}_{T_\ell(p,q)}(\sigma,\alpha)$. Note that the induced leaf forms on the codimension one faces of $\partial \widehat{B}_{T_\ell(p,q)}$ are compatible along the faces of higher codimensions by extending the homotopies over the codimension one faces by the families of homotopies as in Lemma~\ref{lem:hcoherence}.

Next, we prove that the leaf forms obtained over $\partial\widehat{B}_{T_\ell(p,q)}$ can be extended over $\widehat{B}_{T_\ell(p,q)}$.
For simplicity, we consider the pullback $\widehat{\omega}_{T_\ell(p,q)}^\square\colon I^{\ell-3}\times I^\lambda\to \Emb_\partial^\fr((I^k)^{\cup p}\cup (I^{k-1})^{\cup q},I^{2k})$ of $\widehat{\omega}_{T_\ell(p,q)}$. We suppose that we have leaf forms in the restriction of $\widehat{\omega}_{T_\ell(p,q)}^\square$ to $\partial (I^{\ell-3}\times I^\lambda)=\partial I^{\ell-3+\lambda}$. Then by Lemma~\ref{lem:leaf-form-extension}, we obtain an extension over $I^{\ell-3+\lambda}\times V_v$. 
\end{proof}

\begin{Lem}[Nondegeneracy]\label{lem:leaf-form-nondeg}
For $\ell>3$, suppose that Lemma~\ref{lem:leaf-form} holds for $\widehat{\omega}_{T_{\ell'}(p',q')}$ whenever $3\leq\ell'<\ell$ and $p'+q'=\ell'$. The leaf forms in the family $\widehat{\omega}_{T_\ell(p,q)}$ obtained in Lemma~\ref{lem:leaf-form-compatibility} satisfies the nondegeneracy condition (2) of Lemma~\ref{lem:leaf-form}.
\end{Lem}
We prove Lemma~\ref{lem:leaf-form-nondeg} in the rest of this subsection. A motivating example for $\ell=4$ which is more explicit is given in Appendix~\ref{s:leaf-form-mfd}. The proof given here is a cohomological analogue of that of Appendix~\ref{s:leaf-form-mfd}. 
The outline of the proof of Lemma~\ref{lem:leaf-form-nondeg} is as follows.
\begin{enumerate}
\item[(a)] There exist leaf forms $\theta_{a_1}',\ldots,\theta_{a_{\ell_1}}',\theta_{a_{\ell_1+1}}'',\ldots,\theta_{a_{\ell_1+\ell_2}}''$ on the $V_{T'}\tcoprod V_{T''}$-bundle over the base $\widehat{B}_{T_\ell(p,q)}(\sigma,\alpha)=\widehat{B}_{T'}\times \widehat{B}_{T''}$, where $(\sigma,\alpha)$ is a codimension one face tree of $T_\ell(p,q)$.
\item[(b)] The extensions of the leaf forms $\theta_{a_1}',\ldots,\theta_{a_{\ell_1}}',\theta_{a_{\ell_1+1}}'',\ldots,\theta_{a_{\ell_1+\ell_2}}''$ to those on the $V_v$-bundle over $\widehat{B}_{T_\ell(p,q)}(\sigma,\alpha)$ (recall that $V_{T'}\tcoprod V_{T''}\subset V_v$) satisfy some homological condition.
\item[(c)] The leaf forms $\overline\theta_{a_1}',\ldots,\overline\theta_{a_{\ell_1}}',\overline\theta_{a_{\ell_1+1}}'',\ldots,\overline\theta_{a_{\ell_1+\ell_2}}''$ on $\pi_{T_\ell(p,q)}^{-1}(\mathrm{Im}\,F_{(\sigma,\alpha)})$, where $\pi_{T_\ell(p,q)}\colon \widehat{E}^{T_\ell(p,q)}\to \widehat{B}_{T_\ell(p,q)}$ is the associated $V_v$-bundle to the map $\widehat{\omega}_{T_\ell(p,q)}\colon\widehat{B}_{T_\ell(p,q)}\to B\Diff_\partial^\fr(V_v)$ (\S\ref{ss:ori-V-bundle}), induced from $\theta_{a_1}',\ldots,\theta_{a_{\ell_1}}'$, $\theta_{a_{\ell_1+1}}'',\ldots,\theta_{a_{\ell_1+\ell_2}}''$ on the $V_v$-bundle over $\widehat{B}_{T_\ell(p,q)}(\sigma,\alpha)$ satisfy some conditions on their supports.
\item[(d)] The induced leaf forms $\overline\theta_{a_1}',\ldots,\overline\theta_{a_{\ell_1}}',\overline\theta_{a_{\ell_1+1}}'',\ldots,\overline\theta_{a_{\ell_1+\ell_2}}''$ on $\pi_{T_\ell(p,q)}^{-1}(\mathrm{Im}\,F_{(\sigma,\alpha)})$ can be glued together and extended to those on $\widehat{E}^{T_\ell(p,q)}$.
\item[(e)] The extended leaf forms on $\widehat{E}^{T_\ell(p,q)}$ satisfy some conditions on their supports.
\item[(f)] Computation of the integral of the wedge product of the extended leaf forms on $\widehat{E}^{T_\ell(p,q)}$.
\end{enumerate}

\subsubsection{{\rm (a)} The leaf forms on the $V_{T'}\tcoprod V_{T''}$-bundle over $\widehat{B}_{T_\ell(p,q)}(\sigma,\alpha)$.}

Let 
\[ \pi_{T_\ell(p,q)}(\sigma,\alpha)\colon \widehat{E}^{T_\ell(p,q)}(\sigma,\alpha)\to \widehat{B}_{T_\ell(p,q)}(\sigma,\alpha) \]
be the $V_{T'}\tcoprod V_{T''}$-bundle associated to $\widehat{\omega}_{T_\ell(p,q)}(\sigma,\alpha)\colon \widehat{B}_{T_\ell(p,q)}(\sigma,\alpha)\to \Emb_\partial^\fr((I^k)^{\cup p}\cup (I^{k-1})^{\cup q},I^{2k})$.
Let $\theta_{a_1}',\ldots,\theta_{a_{\ell_1}}'$ (resp. $\theta_{a_{\ell_1+1}}'',\ldots,\theta_{a_{\ell_1+\ell_2}}''$) be the leaf forms on the $V_{T'}\tcoprod V_{T''}$-bundle $\widehat{E}^{T_\ell(p,q)}(\sigma,\alpha)$ for the leaves of $T'=T_{\ell_1}(p_1,q_1)$ (resp. $T''=T_{\ell_2}(p_2,q_2)$) that correspond to the leaves of $T_\ell(p,q)$. Let $\xi'$ (resp. $\xi''$) be the leaf form on $\widehat{E}^{T_\ell(p,q)}(\sigma,\alpha)$ for the leaf of $T'$ (resp. $T''$) of the internal edge of $(\sigma,\alpha)$ that does not correspond to leaves of $T_\ell(p,q)$. All of these leaf forms are of degrees $k$ or $k-1$, and satisfy the conditions of Lemma~\ref{lem:leaf-form} along the factor $\widehat{B}_{T'}$ (resp. $\widehat{B}_{T''}$) by the assumption of Lemma~\ref{lem:leaf-form-nondeg}.

\subsubsection{{\rm (b)} Homological properties of the extensions of the leaf forms to the $V_v$-bundle over $\widehat{B}_{T_\ell(p,q)}(\sigma,\alpha)$.}

Suppose that the edge direction $\alpha$ on the internal edge of $\sigma$ is from the internal vertex of $T'$ to that of $T''$. Then the products of these forms give cohomology classes
\begin{equation}\label{eq:wedge-theta}
 \begin{split}
  & [\theta_{a_1}'\wedge\cdots\wedge \theta_{a_{\ell_1}}']\in H^{p_1(k-1)+(q_1-1)k}(\widehat{E}^{T'},\partial \widehat{E}^{T'};\R)\cong H_k(\widehat{E}^{T'};\R),\\
  & [\theta_{a_{\ell_1+1}}''\wedge\cdots\wedge\theta_{a_{\ell_1+\ell_2}}'']\in H^{(p_2-1)(k-1)+q_2k}(\widehat{E}^{T''},\partial \widehat{E}^{T''};\R)\cong H_{k-1}(\widehat{E}^{T''};\R).
\end{split} 
\end{equation}
Note that by the assumption of Lemma~\ref{lem:leaf-form-nondeg}, we have that the classes $[\theta_{a_1}'\wedge\cdots\wedge \theta_{a_{\ell_1}}'\wedge\xi']$ and $[\theta_{a_{\ell_1+1}}''\wedge\cdots\wedge\theta_{a_{\ell_1+\ell_2}}''\wedge\xi'']$ are nontrivial top degree cohomology classes in $H^{p_1(k-1)+q_1k}(\widehat{E}^{T'},\partial \widehat{E}^{T'};\R)$ and $H^{p_2(k-1)+q_2k}(\widehat{E}^{T''},\partial \widehat{E}^{T''};\R)$, respectively. 

Let us consider the precise values of the integrals of these volume forms. Recall that we took finitely many copies of the terms of the forms $\omega_{T'}'\circ\omega_{T''}'$ in the definition of the $(2^{\ell-1}-\ell-1)$T-link, where the coefficient of $\omega_{T'}'\circ\omega_{T''}'$ is the integer $\frac{\mu_\ell^\partial}{m_{\ell_1}m_{\ell_2}}$. We took more copies of the $(2^{\ell-1}-\ell-1)$T-link to define $\omega_{T_\ell}$ (Definition~\ref{def:beta_ell}), and the final coefficient of $\omega_{T'}'\circ\omega_{T''}'$ in $\partial\omega_{T_\ell}$ is the integer $\frac{m_\ell}{m_{\ell_1}m_{\ell_2}}$. 
So $\widehat{\omega}_{T_\ell(p,q)}(\sigma,\alpha)$ can be given by the composition $(\frac{m_\ell}{m_{\ell_1}m_{\ell_2}}\widehat{\omega}_{T'})\circ \widehat{\omega}_{T''}$, which gives a $V_v$-bundle, or the product of $\frac{m_\ell}{m_{\ell_1}m_{\ell_2}}\widehat{\omega}_{T'}$ and $\widehat{\omega}_{T''}$, which gives a $V_{T'}\tcoprod V_{T''}$-bundle. Note that the product of $\widehat{\omega}_{T'}$ and $\frac{m_\ell}{m_{\ell_1}m_{\ell_2}}\widehat{\omega}_{T''}$ gives the same result. Now assuming the former product, we have
\begin{equation}\label{eq:volumes}
\begin{split}
&\int_{\widehat{E}^{T'}}\theta_{a_1}'\wedge\cdots\wedge \theta_{a_{\ell_1}}'\wedge\xi'=\frac{m_\ell}{m_{\ell_1}m_{\ell_2}}\cdot m_{\ell_1}=\frac{m_\ell}{m_{\ell_2}},\\
&\int_{\widehat{E}^{T''}}\theta_{a_{\ell_1+1}}''\wedge\cdots\wedge\theta_{a_{\ell_1+\ell_2}}''\wedge\xi''=m_{\ell_2}
\end{split}
\end{equation}
according to the induction hypothesis of Lemma~\ref{lem:leaf-form-nondeg} and Lemma~\ref{lem:int-wedge-theta}.

We now see how the leaf forms in $\widehat{E}^{T_\ell(p,q)}(\sigma,\alpha)$ induce those on the corresponding face of $\widehat{E}^{T_\ell(p,q)}$, which is a $V_v$-bundle. Recall that $V_{T'}\tcoprod V_{T''}\subset V_v$. For $T=T'$ or $T''$, let $\widehat{E}^T(V_v)\to \widehat{B}_T$ be the $(V_v,\partial)$-bundle obtained from $\pi_T\colon\widehat{E}^T\to \widehat{B}_T$ by extending the fiber $V_T$ to $V_v$ by the trivial bundle. The leaf forms on the $V_{T'}\tcoprod V_{T''}$-bundle $\widehat{E}^{T_\ell(p,q)}(\sigma,\alpha)$ obtained by pullbacks from $\widehat{E}^{T'}$ and $\widehat{E}^{T''}$ can be extended to those on the $V_v$-bundle by extending in $V_v-\mathrm{Int}\,(V_{T'}\tcoprod V_{T''})$ by $\eta$-forms of (the meridian spheres of the leaves)$\times I$.
\begin{Lem}\label{lem:theta_T}
The classes (\ref{eq:wedge-theta}) are in the kernels of the natural maps 
\[ H^*(\widehat{E}^T,\partial \widehat{E}^T;\R)\to H^*(\widehat{E}^T(V_v),\partial \widehat{E}^T(V_v);\R),\quad T=T',T'',  \]
which extend the forms by zero. Hence there is a form $\theta_T$ on $\widehat{E}^T(V_v)$ with compact support in $\mathrm{Int}\,\widehat{E}^T(V_v)$ such that
\[ 
   d\theta_{T'}=\theta_{a_1}'\wedge\cdots\wedge \theta_{a_{\ell_1}}'\text{ and }\,\,\, d\theta_{T''}=\theta_{a_{\ell_1+1}}''\wedge\cdots\wedge\theta_{a_{\ell_1+\ell_2}}''.
 \]
\end{Lem}
\begin{proof}
This is because the $\Psi$-surgeries have the Brunnian property with respect to the components corresponding to the leaves at the internal edge, where the extension of the fiber from $V_T$ to $V_v$ corresponds to cancelling the leaf at the internal edge. The leaf forms of the external leaves of $T$ can be deformed in $\widehat{E}^T(V_v)$ without changing its restriction to a neighborhood of $\partial \widehat{E}^T(V_v)$ into the $\eta$-forms of the standard spanning disks. 
\end{proof}

\begin{Obs} Let $U_1\subset \mathrm{Int}\,\widehat{B}_{T'}$ (resp. $U_2\subset \mathrm{Int}\,\widehat{B}_{T''}$) be a small open ball such that the support of $\theta_{a_1}'\wedge\cdots\wedge \theta_{a_{\ell_1}}'$ (resp. $\theta_{a_{\ell_1+1}}''\wedge\cdots\wedge\theta_{a_{\ell_1+\ell_2}}''$) is included in $\pi_{T'}^{-1}(U_1)$ (resp. $\pi_{T''}^{-1}(U_2)$). We may assume that the support of $\theta_{a_1}'\wedge\cdots\wedge \theta_{a_{\ell_1}}'$ (resp. $\theta_{a_{\ell_1+1}}''\wedge\cdots\wedge\theta_{a_{\ell_1+\ell_2}}''$) in $\widehat{E}^{T_\ell(p,q)}(\sigma,\alpha)$ is included in $\pi_{T_\ell(p,q)}(\sigma,\alpha)^{-1}(U_1\times\widehat{B}_{T''})$ (resp. in $\pi_{T_\ell(p,q)}(\sigma,\alpha)^{-1}(\widehat{B}_{T'}\times U_2)$) (see Figure~\ref{fig:supp-theta}, left). 
\end{Obs}

\subsubsection{{\rm (c)} On the supports of the leaf forms on $\pi_{T_\ell(p,q)}^{-1}(\mathrm{Im}\,F_{(\sigma,\alpha)})$ induced from those on the $V_v$-bundle over $\widehat{B}_{T_\ell(p,q)}(\sigma,\alpha)$.}

Let $\overline\theta_{a_1}',\ldots,\overline\theta_{a_{\ell_1}}'$ (resp. $\overline\theta_{a_{\ell_1+1}}'',\ldots,\overline\theta_{a_{\ell_1+\ell_2}}''$) be the leaf forms over the image of $F_{(\sigma,\alpha)}\colon \widehat{B}_{T_\ell(p,q)}(\sigma,\alpha)\to \widehat{B}_{T_\ell(p,q)}/{A_{T_\ell(p,q)}}$ for the leaves of $T'$ (resp. $T''$) that are induced from $\theta_{a_1}',\ldots,\theta_{a_{\ell_1}}'$ (resp. $\theta_{a_{\ell_1+1}}'',\ldots,\theta_{a_{\ell_1+\ell_2}}''$) on $\widehat{E}^{T_\ell(p,q)}(\sigma,\alpha)$, as obtained in Lemma~\ref{lem:leaf-form-compatibility}. 

To understand the supports of $\overline{\theta}_{a_1}'\wedge\cdots\wedge \overline{\theta}_{a_{\ell_1}}'$ and $\overline{\theta}_{a_{\ell_1+1}}''\wedge\cdots\wedge\overline{\theta}_{a_{\ell_1+\ell_2}}''$, we consider for simplicity the pullback 
$\widehat{\omega}_{T_\ell(p,q)}^\square\colon I^{\ell-3+\lambda}\to \Emb_\partial^\fr((I^k)^{\cup p}\cup (I^{k-1})^{\cup q},I^{2k})$ of $\widehat{\omega}_{T_\ell(p,q)}$, as in the proof of Lemma~\ref{lem:leaf-form-compatibility}. Let $K^\square(\sigma,\alpha)\subset \partial I^{\ell-3+\lambda}$ be the subset that is projected onto the image of $F_{(\sigma,\alpha)}$. 
Abusing notation, we also denote by $\overline\theta_{a_1}',\ldots,\overline\theta_{a_{\ell_1}}'$ (resp. $\overline\theta_{a_{\ell_1+1}}'',\ldots,\overline\theta_{a_{\ell_1+\ell_2}}''$) the pullback leaf forms on $K^\square(\sigma,\alpha)\times V_v$. Let $K^\circ(\sigma,\alpha)=\mathrm{Im}\,F_{(\sigma,\alpha)}-A_{T_\ell(p,q)}$, which is considered as a subset of $\widehat{B}_{T_\ell(p,q)}$, and we identify it with its diffeomorphic preimage under the natural map $K^\square(\sigma,\alpha)\to K(\sigma,\alpha)/{A_{T_\ell(p,q)}}$.

\begin{Assum}\label{assum:theta-U}
The image of the support of $\overline{\theta}_{a_1}'\wedge\cdots\wedge \overline{\theta}_{a_{\ell_1}}'$ (resp. $\overline{\theta}_{a_{\ell_1+1}}''\wedge\cdots\wedge\overline{\theta}_{a_{\ell_1+\ell_2}}''$) in $K^\circ(\sigma,\alpha)\times V_v$ under the projection $K^\circ(\sigma,\alpha)\times V_v\to K^\circ(\sigma,\alpha)$ is included in an open set $\overline{U}_1$ (resp. $\overline{U}_2$), which is either an open ball in $\mathrm{Int}\,K^\circ(\sigma,\alpha)$ or a small open neighborhood of an arc in $K^\circ(\sigma,\alpha)$ from the boundary to the interior. Moreover, there is an open neighborhood $O(\sigma,\alpha)$ of the closure of $U_1\times U_2$ in $\widehat{B}_{T_\ell(p,q)}(\sigma,\alpha)$ such that 
\begin{itemize}
\item the restriction of $F_{(\sigma,\alpha)}$ to $O(\sigma,\alpha)$ is an embedding, and
\item the preimages $W_i=F_{(\sigma,\alpha)}^{-1}(\overline{U}_i)$ of $\overline{U}_i$ in $\widehat{B}_{T_\ell(p,q)}(\sigma,\alpha)$ satisfy
\[ W_1\cap O(\sigma,\alpha)=(U_1\times\widehat{B}_{T''})\cap O(\sigma,\alpha),\quad\text{and}\quad W_2\cap O(\sigma,\alpha)=(\widehat{B}_{T'}\times U_2)\cap O(\sigma,\alpha). \]
\end{itemize}
See Figure~\ref{fig:supp-theta}.
\end{Assum}
\begin{figure}[h]
\includegraphics[height=60mm]{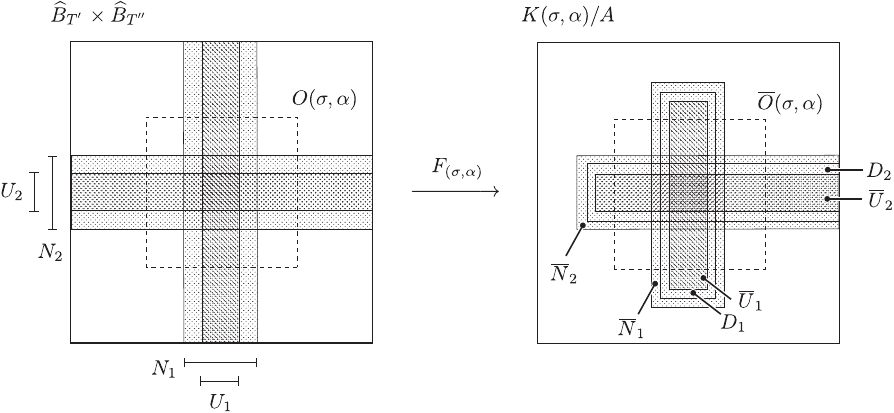} 
\caption{Left: where $\theta_{a_1}',\ldots, \theta_{a_{\ell_1}}'$ or $\theta_{a_{\ell_1+1}}'',\ldots, \theta_{a_{\ell_1+\ell_2}}''$ may be nonzero.\\ Right: where $\overline\theta_{a_1}',\ldots, \overline\theta_{a_{\ell_1}}'$ or $\overline\theta_{a_{\ell_1+1}}'',\ldots, \overline\theta_{a_{\ell_1+\ell_2}}''$ may be nonzero.}\label{fig:supp-theta}
\end{figure}

We may make the supports of the leaf forms more precise, and may assume more direct coincidences between the leaf forms when restricted to the subset $O(\sigma,\alpha)$.
\begin{Lem}\label{lem:theta2=0}
Let $\overline{O}(\sigma,\alpha)=F_{(\sigma,\alpha)}(O(\sigma,\alpha))$. We may assume that there are open subsets $D_1$ and $D_2$ of $K^\circ(\sigma,\alpha)$ such that
\begin{enumerate}
\item[(i)] $\overline{U}_i\subset D_i$ for $i=1,2$. $D_1\cap D_2\subset \overline{O}(\sigma,\alpha)$.
\item[(ii)] The leaf forms $\overline\theta_{a_1}',\ldots, \overline\theta_{a_{\ell_1}}'$ agrees with the $\eta$-forms of the standard spanning disks of $\widetilde{a}_1,\ldots,\widetilde{a}_{\ell_1}$ (see Definition~\ref{def:leaf-form}) over $K^\circ(\sigma,\alpha)-D_1$ up to signs. Moreover, we have $\overline\theta_{a_i}'\wedge \overline\theta_{a_{\ell_1+j}}''=0$ over $K^\circ(\sigma,\alpha)-D_1$ for all $1\leq i\leq \ell_1$, $1\leq j\leq \ell_2$.
\item[(iii)] The leaf forms $\overline\theta_{a_{\ell_1+1}}'',\ldots, \overline\theta_{a_{\ell_1+\ell_2}}''$ agrees with the $\eta$-forms of the standard spanning disks of $\widetilde{a}_{\ell_1+1},\ldots,\widetilde{a}_{\ell_1+\ell_2}$ over $K^\circ(\sigma,\alpha)-D_2$ up to signs. Moreover, we have $\overline\theta_{a_i}'\wedge \overline\theta_{a_{\ell_1+j}}''=0$ over $K^\circ(\sigma,\alpha)-D_2$ for all $1\leq i\leq \ell_1$, $1\leq j\leq \ell_2$.
\item[(iv)] The pullbacks of the restrictions of $\overline\theta_{a_1}',\ldots, \overline\theta_{a_{\ell_1}}'$ (resp. $\overline\theta_{a_{\ell_1+1}}'',\ldots, \overline\theta_{a_{\ell_1+\ell_2}}''$) to $\overline{O}(\sigma,\alpha)$ under the map $F_{(\sigma,\alpha)}$ agree with the restrictions of $\theta_{a_1}',\ldots, \theta_{a_{\ell_1}}'$ (resp. $\theta_{a_{\ell_1+1}}'',\ldots, \theta_{a_{\ell_1+\ell_2}}''$) to $O(\sigma,\alpha)$. 
\end{enumerate}
\end{Lem}
For a form $\alpha$ on the total space of a fiber bundle, let $\supp_b(\alpha)$ denote the image of the support of $\alpha$ under the bundle projection.
\begin{Lem}\label{lem:wedge-theta-support}
Let $\{1,\ldots,\ell_1\}\tcoprod \{\ell_1+1,\ldots,\ell_1+\ell_2\}$ be the decomposition of the set of the external vertices (leaves) of $T_\ell(p,q)$ associated to the codimension one face tree $\sigma$. Then 
\[ \supp_b(\overline\theta_{a_1}'\wedge\cdots\wedge \overline\theta_{a_{\ell_1}}')\cap \supp_b(\overline\theta_{a_{\ell_1+1}}''\wedge\cdots\wedge\overline\theta_{a_{\ell_1+\ell_2}}''), \]
which is a subset of $\partial I^{\ell-3+\lambda}$, 
is included in the interior of $K^\circ(\sigma,\alpha)$ (and disjoint from $K^\circ(\sigma',\alpha')$ for $\sigma'\neq \sigma$, $\sigma'\in\overrightarrow{\calP}_{T_\ell(p,q)}\setminus\{T_\ell(p,q)\}$).
\end{Lem}
\begin{proof}
Let $\sigma'\neq \sigma$ be another codimension one face tree of $T_\ell(p,q)$ that is decomposed into two trees $\widetilde{T}',\widetilde{T}''$ by cutting the internal edge of $\sigma'$. Let $J_1\tcoprod J_2$ be the decomposition of the set of the external vertices of $T_\ell(p,q)$ associated to the pair $\widetilde{T}', \widetilde{T}''$. Then either $\{1,\ldots,\ell_1\}$ or $\{\ell_1+1,\ldots,\ell_1+\ell_2\}$ is not included in $J_1$ nor $J_2$. Suppose without loss of generality that $\{1,\ldots,\ell_1\}\not\subset J_1$ and $\{1,\ldots,\ell_1\}\not\subset J_2$.
Then there are indices $j,j'\in\{1,\ldots,\ell_1\}$ such that $j\in J_1$ and $j'\in J_2$. By Lemma~\ref{lem:theta2=0} for $(\sigma',\alpha')$, the leaf form $\overline{\theta}_{a_j}'$ over $\overline{O}(\sigma',\alpha')$ for the $j$-th leaf is supported in $\overline{O}(\sigma',\alpha')\times V_{\widetilde{T}'}$, and $\overline{\theta}_{a_{j'}}'$ for the $j'$-th leaf is supported in $\overline{O}(\sigma',\alpha')\times V_{\widetilde{T}''}$. Hence we have $\overline{\theta}_{a_j}'\wedge \overline{\theta}_{a_{j'}}'=0$ in $\overline{O}(\sigma',\alpha')\times V_v$ since $V_{\widetilde{T}'}\cap V_{\widetilde{T}''}=\emptyset$. Moreover, $\overline\theta_{a_j}'\wedge \overline{\theta}_{a_{j'}}'=0$ over $(K^\circ(\sigma',\alpha')-D_1)\cup (K^\circ(\sigma',\alpha')-D_1)=K^\circ(\sigma',\alpha')-(D_1\cap D_2)$ by Lemma~\ref{lem:theta2=0} (ii), (iii), where we consider that $D_1,D_2\subset K^\circ(\sigma',\alpha')$ for the decomposition $J_1\tcoprod J_2$. This shows that $\overline\theta_{a_j}'\wedge \overline{\theta}_{a_{j'}}'=0$ over $K^\circ(\sigma',\alpha')$.
\end{proof}

\subsubsection{{\rm (d)} The induced leaf forms on $\pi_{T_\ell(p,q)}^{-1}(\mathrm{Im}\,F_{(\sigma,\alpha)})$ can be glued together and extended to those on $\widehat{E}^{T_\ell(p,q)}$.}

The leaf forms on $K^\circ(\sigma,\alpha)\times V_v$ for the codimension one face graphs $(\sigma,\alpha)$ of $T_\ell(p,q)$ can be glued together and extended to those on $I^{\ell-3+\lambda}\times V_v$. In particular, for the decomposition $\{1,\ldots,\ell_1\}\tcoprod\{\ell_1+1,\ldots,\ell_1+\ell_2\}$ of the set of external leaves of $\sigma$ for a fixed codimension one face tree $(\sigma,\alpha)$, it follows from Lemma~\ref{lem:leaf-form-extension} that the products $\overline\theta_{a_1}'\wedge\cdots\wedge \overline\theta_{a_{\ell_1}}'$ and $\overline\theta_{a_{\ell_1+1}}''\wedge\cdots\wedge\overline\theta_{a_{\ell_1+\ell_2}}''$ on $\partial I^{\ell-3+\lambda}\times V_v$ can be extended to closed forms $P_{T'}$ and $P_{T''}$ of degrees $p_1(k-1)+(q_1-1)k$ and $(p_2-1)(k-1)+q_2k$, respectively, on $I^{\ell-3+\lambda}\times V_v\simeq V_v$ that are the wedge products of the extended leaf forms and are exact by the dimensional reason. Hence, there exist forms $Q_{T'}$ and $Q_{T''}$ on $I^{\ell-3+\lambda}\times V_v$ such that for $T=T'$ or $T''$,
\begin{equation}\label{eq:Theta-exact}
 P_T=d Q_T. 
\end{equation}

To prove Lemma~\ref{lem:leaf-form-nondeg}, it suffices to prove the following.
\begin{Lem}\label{lem:int-Theta}
\[ \int_{I^{\ell-3+\lambda}\times V_v}P_{T'}\wedge P_{T''}=\pm m_\ell. \]
\end{Lem}

\subsubsection{{\rm (e)} On the supports of the extended leaf forms on $\widehat{E}^{T_\ell(p,q)}$.}

\begin{Lem}\label{lem:supp-H}
We may assume that 
\[ \supp_b(Q_{T'})\cap \supp_b(\overline\theta_{a_{\ell_1+1}}''\wedge\cdots\wedge\overline\theta_{a_{\ell_1+\ell_2}}'') \]
in $\partial I^{\ell-3+\lambda}$ is included in the interior of $K^\circ(\sigma,\alpha)$. Moreover, $Q_{T'}$ and $\theta_{T'}$ (of Lemma~\ref{lem:theta_T}) can be chosen so that there are open sets $N_1\subset \widehat{B}_{T'}$ and $\overline{N}_1\subset K^\circ(\sigma,\alpha)$ satisfying the following (see Figure~\ref{fig:supp-theta}).
\begin{enumerate}
\item[(i)] $U_1\subset N_1$ and $\overline{U}_1\subset D_1\subset \overline{N}_1$.
\item[(ii)] $\supp_b(Q_{T'})\cap K^\circ(\sigma,\alpha)$ is included in $\overline{N}_1$.
\item[(iii)] $\supp_b(\theta_{T'})\subset \widehat{B}_{T'}$ is included in $N_1$.
\item[(iv)] $\overline{N}_1\cap \overline{U}_2\subset \overline{O}(\sigma,\alpha)$, and $(N_1\times\widehat{B}_{T''})\cap U_2\subset O(\sigma,\alpha)$.
\item[(v)] The preimage $Y_1=F_{(\sigma,\alpha)}^{-1}(\overline{N}_1)$ of $\overline{N}_1$ in $\widehat{B}_{T_\ell(p,q)}(\sigma,\alpha)$ satisfies
\[ Y_1\cap O(\sigma,\alpha)=(N_1\times \widehat{B}_{T''})\cap O(\sigma,\alpha). \]
\item[(vi)] The restriction of the pullback $F_{(\sigma,\alpha)}^* Q_{T'}$ to $O(\sigma,\alpha)$ agrees with the pullback of $\theta_{T'}$ under the projection $\widehat{B}_{T_\ell(p,q)}(\sigma,\alpha)\to \widehat{B}_{T'}$.
\end{enumerate}
\end{Lem}
\begin{proof}
The former assertion holds because we can assume that $\supp_b(Q_{T'})$ is included in a small open neighborhood of $\supp_b(\overline\theta_{a_1}'\wedge\cdots\wedge\overline\theta_{a_{\ell_1}}')$ in $\partial I^{\ell-3+\lambda}$ by inductively constructing $Q_{T'}$ over the restrictions to contractible neighborhoods of $\supp_b(\overline\theta_{a_1}'\wedge\cdots\wedge\overline\theta_{a_{\ell_1}}')$ from lower excess faces. Then we see that $\supp_b(Q_{T'})\cap \supp_b(\overline\theta_{a_{\ell_1+1}}''\wedge\cdots\wedge\overline\theta_{a_{\ell_1+\ell_2}}'')$ is included in $\overline{O}(\sigma,\alpha)$ by Assumption~\ref{assum:theta-U} and Lemma~\ref{lem:wedge-theta-support}.

The latter assertion about the choices of $Q_{T'}$ and $\theta_{T'}$ can be proved by a similar argument as Lemma~\ref{lem:leaf-form-compatibility} about the deformation of the forms near $(I^{\ell_1-3}\times S^{\lambda_1})\vee (I^{\ell_2-3}\times S^{\lambda_2})$. Namely, the existence of $N_1$ satisfying (iii) and $U_1\subset N_1$ in (i) follows since the support of $\theta_{T'}$ must include that of $\theta_{a_1}'\wedge\cdots\wedge\theta_{a_{\ell_1}}'$. We may also assume that $N_1$ is a small neighborhood of $U_1$.
The deformations of (the pullback of) $\theta_{a_1}'\wedge\cdots\wedge\theta_{a_{\ell_1}}'$ near $(I^{\ell_1-3}\times S^{\lambda_1})\vee (I^{\ell_2-3}\times S^{\lambda_2})$ in Lemma~\ref{lem:leaf-form-compatibility} is accompanied by that of the pullback of $\theta_{T'}$, and we may define $Q_{T'}$ over $\partial I^{\ell-3+\lambda}$ by its result and may then extend over $I^{\ell-3+\lambda}$, hence we have (vi). 
Since the deformation of (the pullback of) $\theta_{a_1}'\wedge\cdots\wedge\theta_{a_{\ell_1}}'$ can be done within a small neighborhood of $(I^{\ell_1-3}\times S^{\lambda_1})\vee (I^{\ell_2-3}\times S^{\lambda_2})$, we may assume that the restriction of the relevant forms over $O(\sigma,\alpha)$ does not change, and we may choose $\overline{N}_1$ so that (v) holds. Hence we also have (iv). Since the support of $Q_{T'}$ in $K^\circ(\sigma,\alpha)$ must include that of $\overline\theta_{a_1}'\wedge\cdots\wedge\overline\theta_{a_{\ell_1}}'$, we have $\overline{U}_1\subset \overline{N}_1$, which gives (i). The conditions (ii) and (iii) are straightforward from the constructions of $\theta_{T'}$ and $Q_{T'}$.
\end{proof}

\subsubsection{{\rm (f)} Computation of the integral of the wedge product of the extended leaf forms on $\widehat{E}^{T_\ell(p,q)}$.}

The following lemma is an analogue of Lemma~\ref{lem:local-model-4-fold}.
\begin{Lem}\label{lem:int-wedge-Theta}
For any codimension one face graph $(\sigma,\alpha)$ of $T_\ell(p,q)$, we have 
\[\int_{I^{\ell-3+\lambda}\times V_v}P_{T'}\wedge P_{T''}
=\int_{\widehat{E}^{T_\ell(p,q)}(\sigma,\alpha)}\theta_{T'}\wedge (\theta_{a_{\ell_1+1}}''\wedge\cdots\wedge \theta_{a_{\ell_1+\ell_2}}''). \]
\end{Lem}
\begin{proof}
By (\ref{eq:Theta-exact}) and the Stokes theorem, we have
\[
\int_{I^{\ell-3+\lambda}\times V_v}P_{T'}\wedge P_{T''}
=\int_{I^{\ell-3+\lambda}\times V_v}dQ_{T'}\wedge P_{T''}
=\int_{\partial(I^{\ell-3+\lambda}\times V_v)}Q_{T'}\wedge (\overline\theta_{a_{\ell_1+1}}''\wedge\cdots\wedge\overline\theta_{a_{\ell_1+\ell_2}}'').
\]
Since the support of $Q_{T'}\wedge(\overline\theta_{a_{\ell_1+1}}''\wedge\cdots\wedge\overline\theta_{a_{\ell_1+\ell_2}}'')$ is included in $K^\circ(\sigma,\alpha)\times V_v\subset \partial I^{\ell-3+\lambda}\times V_v$ by Lemma~\ref{lem:supp-H} (first half), the last integral can be rewritten as
\begin{equation}\label{eq:integral-piece}
\begin{split}
\int_{K^\circ(\sigma,\alpha)\times V_v}
Q_{T'}\wedge (\overline\theta_{a_{\ell_1+1}}''\wedge\cdots\wedge\overline\theta_{a_{\ell_1+\ell_2}}'')
&=\int_{(\overline{N}_1\cap \overline{U}_2)\times V_v}Q_{T'}\wedge (\overline\theta_{a_{\ell_1+1}}''\wedge\cdots\wedge\overline\theta_{a_{\ell_1+\ell_2}}'')\\
&=\int_{\overline{O}(\sigma,\alpha)\times V_v}
Q_{T'}\wedge (\overline\theta_{a_{\ell_1+1}}''\wedge\cdots\wedge\overline\theta_{a_{\ell_1+\ell_2}}''),
\end{split}
\end{equation}
where the first equality follows by Assumption~\ref{assum:theta-U} and Lemma~\ref{lem:supp-H} (i), (ii). The second equality follows by Lemma~\ref{lem:supp-H} (iv).
We evaluate this integral by pulling back the forms to the $(V_v,\partial)$-bundle $\widehat{E}^{T_\ell(p,q)}(\sigma,\alpha)\to \widehat{B}_{T_\ell(p,q)}(\sigma,\alpha)$. Namely, $F_{(\sigma,\alpha)}^* Q_{T'}$ is a form on $\widehat{E}^{T_\ell(p,q)}(\sigma,\alpha)$. By Lemmas~\ref{lem:theta2=0} (iv) and \ref{lem:supp-H}, the integral (\ref{eq:integral-piece}) can be rewritten as
\begin{equation}\label{eq:integral-pullback}
\begin{split}
\int_{O(\sigma,\alpha)\times V_v}F_{(\sigma,\alpha)}^* Q_{T'}\wedge (\theta_{a_{\ell_1+1}}''\wedge\cdots\wedge \theta_{a_{\ell_1+\ell_2}}'')=&\int_{O(\sigma,\alpha)\times V_v}\theta_{T'}\wedge (\theta_{a_{\ell_1+1}}''\wedge\cdots\wedge \theta_{a_{\ell_1+\ell_2}}'')\\
  =&\int_{((N_1\times \widehat{B}_{T''})\cap O(\sigma,\alpha))\times V_v}\theta_{T'}\wedge (\theta_{a_{\ell_1+1}}''\wedge\cdots\wedge \theta_{a_{\ell_1+\ell_2}}'')\\
  =&\int_{\widehat{E}^{T_\ell(p,q)}(\sigma,\alpha)}\theta_{T'}\wedge (\theta_{a_{\ell_1+1}}''\wedge\cdots\wedge \theta_{a_{\ell_1+\ell_2}}''),
\end{split} 
\end{equation}
where the first equality follows by Lemma~\ref{lem:supp-H} (vi), the second equality follows by Lemma~\ref{lem:supp-H} (iii), (v), and the third equality follows by Lemma~\ref{lem:supp-H} (iv).
\end{proof}
The support of the integrand $\theta_{T'}\wedge (\theta_{a_{\ell_1+1}}''\wedge\cdots\wedge \theta_{a_{\ell_1+\ell_2}}'')$ of (\ref{eq:integral-pullback}) is included in the subbundle of $\widehat{E}^{T_\ell(p,q)}(\sigma,\alpha)\to \widehat{B}_{T_\ell(p,q)}(\sigma,\alpha)=\widehat{B}_{T'}\times\widehat{B}_{T''}$ with fiber $V_{T''}$, which is trivial along the factor $\widehat{B}_{T'}$ of $\widehat{B}_{T_\ell(p,q)}(\sigma,\alpha)$. Hence it suffices to integrate $\theta_{T'}\wedge (\theta_{a_{\ell_1+1}}''\wedge\cdots\wedge \theta_{a_{\ell_1+\ell_2}}'')$ over the subset $\widehat{E}^{T''}\times \widehat{B}_{T'}$, which is the trivial sub $V_{T''}$-bundle of $\widehat{E}^{T_\ell(p,q)}(\sigma,\alpha)\to \widehat{B}_{T_\ell(p,q)}(\sigma,\alpha)$. 
\[ 
\xymatrix{
 \widehat{E}^{T''}(V_v)\times\widehat{B}_{T'} \ar[d]_{\pr} &  \widehat{E}^{T''}\times\widehat{B}_{T'} \ar[l] \ar[d]_{\pr} \ar[r]^-{\pi_{T''}\times\mathrm{id}} & \widehat{B}_{T''}\times\widehat{B}_{T'} \ar[d]^{\pr}\\
  \widehat{E}^{T''}(V_v) & \widehat{E}^{T''} \ar[l] \ar[r]^-{\pi_{T''}}& \widehat{B}_{T''}
} \]

To do so, we observe that the form $\theta_{a_{\ell_1+1}}''\wedge\cdots\wedge \theta_{a_{\ell_1+\ell_2}}''$ is the pullback of the corresponding form under the projection $\widehat{E}^{T''}\times \widehat{B}_{T'}\to \widehat{E}^{T''}$. Recall that $\theta_{T'}$ of Lemma~\ref{lem:theta_T} is defined on $\widehat{E}^{T'}(V_v)$ and it can be pulled back to $\widehat{E}^{T''}(V_v)\times\widehat{B}_{T'}$. We denote the pullback also by $\theta_{T'}$. We consider the fiber integral of $\theta_{T'}|_{\widehat{E}^{T''}\times \widehat{B}_{T'}}$ along the fiber $\widehat{B}_{T'}$. 

\begin{Lem}\label{lem:fiber-int}
The restriction of the fiber integral $\int_{\widehat{B}_{T'}}\theta_{T'}$ to $\widehat{E}^{T''}$ is cohomologous to $\pm \frac{m_\ell}{m_{\ell_2}}\xi''$. 
\end{Lem}
\begin{proof}
Let $c'\subset \partial V_{T'}$ (resp. $c''\subset \partial V_{T''}$) be the meridian sphere of the leaf of $T'$ (resp. $T''$) of the internal edge of $(\sigma,\alpha)$. So the linking number of $c'$ and the core sphere of the corresponding leaf is $\pm 1$.
Let $\widetilde{c}'=\widehat{B}_{T'}\times c'\subset \widehat{E}^{T'}$. Then we have
\[ \begin{split}
  \int_{\widetilde{c}'}\theta_{T'}&=\int_{\partial\widehat{E}^{T'}}\theta_{T'}\wedge \eta_{\widetilde{c}'}
  =\int_{\partial\widehat{E}^{T'}}\theta_{T'}\wedge \xi'=\int_{\widehat{E}^{T'}}d\theta_{T'}\wedge \xi'=\int_{\widehat{E}^{T'}}(\theta_{a_1}'\wedge\cdots\wedge\theta_{a_{\ell_1}}')\wedge \xi'=\frac{m_\ell}{m_2}
\end{split} \]
by the Stokes theorem, the assumption for the leaf forms, and (\ref{eq:volumes}). Since the core sphere $b''$ of the leaf of $T''$ of the internal edge of $(\sigma,\alpha)$ is homologous in $V_v$ to $\pm c'$ and $\int_{\widetilde{c}'}=\int_{c'}\int_{\widehat{B}_{T'}}$, we have 
\[ \int_{b''}\int_{\widehat{B}_{T'}} \theta_{T'}=\pm\int_{c'}\int_{\widehat{B}_{T'}} \theta_{T'}
=\pm \frac{m_\ell}{m_2}=\pm \int_{b''}\frac{m_\ell}{m_2}\xi''. \]
For the core spheres $b'''$ of other leaves of $T''$, we have $\int_{b'''}\int_{\widehat{B}_{T'}} \theta_{T'}=0$ because the support of $\theta_{T'}$ does not intersect the external leaves in $V_{T''}$. This characterizes the cohomology class of $\int_{\widehat{B}_{T'}} \theta_{T'}$ in $\widehat{E}^{T''}$.
\end{proof}
\begin{proof}[Proofs of Lemmas~\ref{lem:int-Theta} and \ref{lem:leaf-form-nondeg}]
According to Lemma~\ref{lem:int-wedge-Theta}, it suffices to evaluate the integral of the RHS of the identity in Lemma~\ref{lem:int-wedge-Theta}. We integrate the integrand of (\ref{eq:integral-pullback}) along the fiber $\widehat{B}_{T'}$ of $\widehat{E}^{T''}\times \widehat{B}_{T'}$ first using Lemma~\ref{lem:fiber-int}
\[ \int_{\widehat{B}_{T'}}\theta_{T'}\wedge (\theta_{a_{\ell_1+1}}''\wedge\cdots\wedge \theta_{a_{\ell_1+\ell_2}}'')
\sim \pm \frac{m_\ell}{m_{\ell_2}}\xi''\wedge (\theta_{a_{\ell_1+1}}''\wedge\cdots\wedge \theta_{a_{\ell_1+\ell_2}}''). \]
Then the integral of this form over $\widehat{E}^{T''}$, which gives (\ref{eq:integral-pullback}), yields the value $\pm m_\ell$ by (\ref{eq:volumes}) for $\widehat{E}^{T''}$. This completes the proof of Lemma~\ref{lem:int-Theta} and hence of Lemma~\ref{lem:leaf-form-nondeg}.
\end{proof}

\begin{proof}[Proof of Lemma~\ref{lem:leaf-form-barE}]
The assertions (1) and (2) of Lemma~\ref{lem:leaf-form-barE} about the existence and compatibility of leaf forms over $\overline{E}^{T_\ell(p,q)}$ follows immediately from the compatibility of Lemma~\ref{lem:leaf-form-compatibility}. The assertion (3) of Lemma~\ref{lem:leaf-form-barE} is Lemma~\ref{lem:integral-thicken} below. 
\end{proof}
\begin{Lem}\label{lem:integral-thicken}
\[ \int_{\overline{E}^{T_\ell(p,q)}-(\widehat{E}^{T_\ell(p,q)})^{(2^{\ell-3})}}\theta_{a_1}\wedge\cdots\wedge \theta_{a_\ell}=0. \]
\end{Lem}
\begin{proof}
By Assumption~\ref{assum:theta-U}, Lemmas~\ref{lem:theta2=0} and \ref{lem:wedge-theta-support}, the intersection of the support of $\theta_{a_1}\wedge\cdots\wedge \theta_{a_\ell}$ with $\overline{E}^{T_\ell(p,q)}-(\widehat{E}^{T_\ell(p,q)})^{(2^{\ell-3})}$ is the $(V_v,\partial)$-bundle over the disjoint union of the mapping cylinders $\mathrm{Cyl}(F_{(\sigma,\alpha)})$ of the maps $F_{(\sigma,\alpha)}\colon O(\sigma,\alpha)\to \overline{O}(\sigma,\alpha)$. Since the supports of $\theta_{a_1}\wedge\cdots\wedge\theta_{a_{\ell_1}}$ and $\theta_{a_{\ell_1+1}}\wedge\cdots\wedge\theta_{a_{\ell_1+\ell_2}}$ are disjoint over $O(\sigma,\alpha)$, we have $\theta_{a_1}\wedge\cdots\wedge \theta_{a_\ell}=0$ there, and so the pullback over $\mathrm{Cyl}(F_{(\sigma,\alpha)})$ is.
\end{proof}

\subsection{The induced orientations: Proof of Theorem~\ref{thm:chainmap}}

Let us first consider locally and describe how the orientations of $\widehat{B}_{T_\ell(p,q)}(T_\ell(p,q))$ and $\widehat{B}_{T_\ell(p,q)}(\sigma)$ for a face $\sigma\in\overrightarrow{\calP}_{T_\ell(p,q)}$ of codimension one are related to each other. We put $T=T_\ell(p,q)$ and suppose that $\sigma$ consists of two internal vertices $v_1$ and $v_2$ from $T'=T_{\ell_1}(p_1,q_1)$ and $T''=T_{\ell_2}(p_2,q_2)$, respectively. Recall that there is a natural map $\widehat{B}_T(\sigma)\to K(\sigma)/{A_{T_\ell(p,q)}}\subset \partial \widehat{B}_T/{A_{T_\ell(p,q)}}$. An orientation of $\widehat{B}_T(\sigma)$ is induced from that of $\widehat{B}_T/{A_{T_\ell(p,q)}}$ by inducing on $K(\sigma)/{A_{T_\ell(p,q)}}$ and then by pulling back to $\widehat{B}_T(\sigma)$. We compare the following orientations of $\widehat{B}_T(\sigma)=\widehat{B}_{T'}\times \widehat{B}_{T''}$:
\begin{enumerate}
\item[(a)] that induced from $o(\widehat{B}_T)$, and 
\item[(b)] $o(\widehat{B}_{T'})\wedge o(\widehat{B}_{T''})$ determined by the orientation convention from the combinatorial data of the tree $\sigma$.
\end{enumerate}

We use the following convention to induce an orientation on the boundary. For a compact connected orientable $m$-manifold $M$ with nonempty boundary, we represent an orientation of $M$ by a nontrivial de Rham cohomology class $[\upsilon_M]\in H^m(M,\partial M;\R)$ for a closed $m$-form $\upsilon_M$ with compact support in $\mathrm{Int}\,M$. Since the connecting homomorphism $\delta\colon H^{m-1}(\partial M;\R)\to H^m(M,\partial M;\R)$ in the exact sequence for the pair $(M,\partial M)$ is an isomorphism, there is a canonical nontrivial element $\delta^{-1}([\upsilon_M])\in H^{m-1}(\partial M;\R)$. More explicitly, since $\upsilon_M$ is exact in $M$, there is an $(m-1)$-form $\zeta$ on $M$ such that $\upsilon=d\zeta$. By the Stokes theorem, we have $\int_M \upsilon_M=\pm\int_{\partial M}\zeta$, where the sign depends on the choice of orientation on $\partial M$. This shows that the restriction of $\zeta$ to $\partial M$ represents a nontrivial class in $H^{m-1}(\partial M;\R)$, and we have $\delta([\zeta|_{\partial M}])=[\upsilon_M]$.

The following is a key lemma to prove Theorem~\ref{thm:chainmap}.

\begin{Lem}\label{lem:o(BB)}
The orientation (a) of $\widehat{B}_T(\sigma)=\widehat{B}_{T'}\times \widehat{B}_{T''}$ is $o(\widehat{B}_{T'})\wedge o(\widehat{B}_{T''})$.
\end{Lem}
\begin{proof}
Let $v$ be the internal $\ell$-valent vertex of $T$. Let $e$ be the internal edge of $\sigma$, and $e_+,e_-$ be the half-edge decomposition of $e$. For a vertex orientation $o(v)$, we denote by $o(v)'$ the wedge product of all the external half-edges of $\sigma$ incident to $v_1$ and $v_2$ whose order of the product agrees with that of $o(v)$ when the internal edge of $\sigma$ is contracted. 
We assume that the vertex orientations $o(v),o(v_1),o(v_2)$ satisfies the following identity:
\begin{equation}\label{eq:o(v)}
 o(v_1)\wedge o(v_2)=(e_{+}\wedge e_{-})\wedge o(v)'.
\end{equation}
This is consistent with the orientations of the terms of $\text{split}(\Gamma,v)$ in (\ref{eq:dGamma}).

According to Definition~\ref{def:ori} in \S\ref{ss:ori-chain}, the parameter space $\widehat{B}_{T}$ is oriented so that
\begin{equation}\label{eq:o(B)}
 o(\widehat{B}_T)\wedge o(D^{2k})=\tbigwedge_{j=1}^\ell \theta_j, 
\end{equation}
where $\theta_j$ are the leaf forms in the $V$-bundle $\widehat{E}^T$ as in Lemma~\ref{lem:leaf-form}. The product (\ref{eq:o(B)}) gives an orientation of $\widehat{E}_T$.
By Definition~\ref{def:ori} and (\ref{eq:o(v)}), the parameter spaces $\widehat{B}_{T'}$ and $\widehat{B}_{T''}$ are oriented so that
\begin{equation}\label{eq:o(BB)}
 o(\widehat{B}_{T'})\wedge o(\widehat{B}_{T''})\wedge o(D^{2k})^{\wedge 2}=(\eta_1\wedge \eta_2)\wedge \tbigwedge_{j=1}^\ell \theta_j',
\end{equation}
where $\theta_j'$ are the pullbacks of the leaf forms in the $V_{T'}\times V_{T''}$-bundle over $K(\sigma)/{A_{T_\ell(p,q)}}$ as in Lemma~\ref{lem:leaf-form}, and $\eta_i$ ($i=1,2$) is the leaf form for the leaf corresponding to the internal half-edge incident to the $i$-th internal vertex.

To relate $o(\widehat{B}_T)$ and $o(\widehat{B}_{T'})\wedge o(\widehat{B}_{T''})$, we consider the induced orientation at the fiber $V\times V$ over $\partial \widehat{B}_T$ from $o(\widehat{B}_T)\wedge o(D^{2k})^{\wedge 2}$.  
Since the product $\tbigwedge_{j=1}^\ell \theta_j'$ in (\ref{eq:o(BB)}) converges to $\tbigwedge_{j=1}^\ell \theta_j$ in (\ref{eq:o(B)}) when a point in $V\times V$ come close to the diagonal $\Delta_V$, we can compare (\ref{eq:o(B)}) and (\ref{eq:o(BB)}) near $\Delta_V$.

We prove that $\rho=1$. We consider a bundle over $\widehat{B}_T$ with fiber $V\times V$ associated to $\pi_{T_\ell(p,q)}$, and consider the limit of (\ref{eq:o(BB)}) at a boundary point where the two points collide. 
If we identify $V$ with the diagonal $\Delta_V$ in $V\times V$, the orientation (\ref{eq:o(B)}) canonically induces an orientation 
\begin{equation}\label{eq:o(BDD)}
 o(\widehat{B}_T)\wedge o(D^{2k})^{\wedge 2}=(\tbigwedge_{j=1}^\ell \theta_j)\wedge o(N\Delta_V) 
\end{equation}
of the $V\times V$-bundle over $\widehat{B}_T$ at a point on $\Delta_V$, where $o(N\Delta_V)$ is the orientation of the normal bundle of $\Delta_V$ such that $o(D^{2k})^{\wedge 2}=o(\Delta_V)\wedge o(N\Delta_V)$. We may take $\eta_{\Delta_V}$ compatible with $o(N\Delta_V)$ to represent $o(N\Delta_V)$. Thus the following closed form gives the orientation (\ref{eq:o(BDD)}).
\begin{equation}\label{eq:vertical-d}
 \eta_{\Delta_V}\wedge \tbigwedge_{j=1}^\ell \theta_j. 
\end{equation}

Next we consider the orientation (\ref{eq:o(BB)}). 
Let $V_1,V_2\subset V$ be disjoint handlebodies obtained as small regular neighborhoods of the two $\Psi$-graphs for $\sigma$ (as $U$ in Figure~\ref{fig:U-4-valent}, left). The product (\ref{eq:o(BB)}) gives an orientation of the $V_1\times V_2$-bundle $\varpi_\sigma\colon \widehat{E}^{T'}\times \widehat{E}^{T''}\to \widehat{B}_{T'}\times \widehat{B}_{T''}$. Let $U$ be a small open ball in $\widehat{B}_{T'}\times \widehat{B}_{T''}$, and let $\widetilde{U}$ be a half open ball in $\widehat{B}_{T}$ such that $\partial \widetilde{U}$ is included in the face $K(\sigma)/{A_{T_\ell(p,q)}}$ and the support of the volume form $\bigwedge_{j=1}^\ell\theta_j$ is included in $\pi_{T_\ell(p,q)}^{-1}(\widetilde{U})$. We assume that the image of $U$ under the map $\widehat{B}_T(\sigma)=\widehat{B}_{T'}\times \widehat{B}_{T''}\to K(\sigma)/{A_{T_\ell(p,q)}}$ is included in $\partial\widetilde{U}$. We consider the associated $V\times V$-bundle $\widetilde{U}\times (V\times V)\to \widetilde{U}$ to $\pi_T=\pi_{T_\ell(p,q)}\colon \pi_T^{-1}(\widetilde{U})\to \widetilde{U}$ and its pullback $U\times (V\times V)\to U$ by the map $U\to \partial \widetilde{U}$. 

The closed form $\eta_1\wedge \eta_2$ in the RHS of (\ref{eq:o(BB)}) gives a cohomology class in $H^*(\varpi^{-1}_\sigma(U);\R)\cong H^*(V_1\times V_2;\R)$. Let $\varpi_T\colon \widetilde{U}\times (V\times V)\to \widetilde{U}$ be the $V\times V$-bundle associated to $\pi_T$. The following claim follows from Lemma~\ref{lem:ext-closed}.
\begin{Claim}\label{cla:eta-extend}
The closed form $\eta_1\wedge \eta_2$ on $\varpi^{-1}_\sigma(U)=U\times (V_1\times V_2)$ can be extended to a form $\widetilde{\eta}$ on $\widetilde{U}\times (V\times V)$ such that
\begin{itemize}
\item the restriction of $\widetilde{\eta}$ to a collar $([0,\ve)\times \partial U)\times (V_1\times V_2)\subset \widetilde{U}\times (V\times V)$ agrees with the pullback of $\eta_1\wedge\eta_2$,
\item the restriction of $\widetilde{\eta}$ to $\widetilde{U}\times (V\times V-N\Delta_V)$ is closed, and 
\item the de Rham cohomology class of $d\widetilde{\eta}\wedge\bigwedge_{j=1}^\ell\theta_j'$ in $H^{4k+\dim{\widehat{B}_{T}}}_c(\varpi_T^{-1}(\widetilde{U}),\partial \varpi_T^{-1}(\widetilde{U});\R)$ agrees with that of $\upsilon_T\wedge\bigwedge_{j=1}^\ell\theta_j'$, where $\upsilon_T$ is a closed $2k$-form supported on a small tubular neighborhood of $\widetilde{U}\times\Delta_V$ that is obtained by pulling back a standard volume form on $\R^{2k}$ with compact support by the map $\R^{2k}\times \R^{2k}\to \R^{2k}$; $(x,y)\mapsto y-x$. 
\end{itemize}
\end{Claim}
Since the supports of $\bigwedge_{j\in E(T')}\theta_j'$ and $\bigwedge_{j\in E(T'')}\theta_j'$ are included in $V_1$ and $V_2$, respectively, the support of $(\eta_1\wedge\eta_2)\wedge \bigwedge_j\theta_j'$ is included in $V_1\times V_2$.
The cohomology class of $\eta_1\wedge \eta_2$ in $\varpi_\sigma^{-1}(U)$ is the unique class which gives the linking number of two component links in $V_1\cup V_2$. 
The same formula $\upsilon_T\wedge\bigwedge_{j=1}^\ell\theta_j'$ in Claim~\ref{cla:eta-extend} above is obtained for other face graphs $\sigma$ of codimension 1.
This shows that the orientations on $\widehat{B}_T(\sigma)$ obtained by applying the above convention are all ``induced'' from the same form $\tbigwedge_{j=1}^\ell \theta_j$.

We obtain the following.
\begin{itemize}
\item We have $[\eta_{\Delta_V}\wedge \tbigwedge_{j=1}^\ell \theta_j]=[\upsilon_T\wedge\bigwedge_{j=1}^\ell\theta_j']$ in $H^{4k+\dim{\widehat{B}_T}}_c(\varpi_T^{-1}(\widetilde{U}),\partial \varpi_T^{-1}(\widetilde{U});\R)$.
\item The class of $d\widetilde{\eta}\wedge\bigwedge_{j=1}^\ell\theta_j'$ induces $o(\widehat{B}_{T'})\wedge o(\widehat{B}_{T''})\wedge o(D^{2k})^{\wedge 2}$ by $\delta^{-1}$, where $\delta^{-1}$ is the inverse of the connecting homomorphism 
\[ \delta\colon H^{4k+\dim{\widehat{B}_T}-1}_c(\partial\varpi_T^{-1}(\widetilde{U});\R)\to H^{4k+\dim{\widehat{B}_T}}_c(\varpi_T^{-1}(\widetilde{U}),\partial \varpi_T^{-1}(\widetilde{U});\R) \]
between the top degree cohomologies.
\end{itemize}
Since the class $[\eta_{\Delta_V}\wedge \tbigwedge_{j=1}^\ell \theta_j]$ represents (\ref{eq:o(BDD)}), it follows that the orientation on $\widehat{B}_{T'}\times \widehat{B}_{T''}$ induced from $o(\widehat{B}_T)$ is $o(\widehat{B}_{T'})\wedge o(\widehat{B}_{T''})$.
\end{proof}

\begin{proof}[Proof of Theorem~\ref{thm:chainmap}]
Now let us compute the orientation of a face of $\widehat{B}_{\text{split}(\Gamma,v_j)}$ induced from $o(\widehat{B}_{(\Gamma,\alpha)})$, where $v_j$ is an $\ell$-valent vertex of $\Gamma$. We fix a labelling on the external half-edges of $T(v_j)$ that is consistent with the vertex orientation $o(v_j)$. Let $\Gamma_\sigma$ be a graph of a term in the sum $\mathrm{split}(\Gamma,v_j)$ corresponding to $\sigma\in\overrightarrow{\calP}_{T(v_j)}$. 
Let $e$ be the edge of $\Gamma_\sigma$ obtained by splitting $v_j$. Assume that the orientation of $(\Gamma,\alpha)$ can be rewritten as
\[ \bigwedge_{v\in V(\Gamma)}o(v)=\mu\,o(v_j)\wedge\bigwedge_{{v\in V(\Gamma)}\atop{v\neq v_j}}o(v) \]
for some $\mu\in\{-1,1\}$. According to the orientation converntion for $\text{split}(\Gamma,v)$, this induces the following orientation of $\Gamma_\sigma$:
\[ 
  o(\Gamma_\sigma):=\mu\,(e_+\wedge e_-)\wedge o(v_j)'\wedge\bigwedge_{v\neq v_j}o(v)=\mu\,o(v_{j1})\wedge o(v_{j2})\wedge\bigwedge_{v\neq v_j}o(v), \]
where $v_{j1},v_{j2}$ are the boundary vertices of $e$, and the second equality follows by (\ref{eq:o(v)}). Recalling that $o(v_j)$ determines $o(\widehat{B}_{T(v_j)})$, $o(v_{j1})\wedge o(v_{j2})$ determines $o(\widehat{B}_{T'})\wedge o(\widehat{B}_{T''})$, and that $o(\widehat{B}_{T(v_j)})$ induces $o(\widehat{B}_{T'})\wedge o(\widehat{B}_{T''})$ by Lemma~\ref{lem:o(BB)}, we see that the orientation on $\widehat{B}_{\Gamma_\sigma}$ induced from $o(\widehat{B}_{(\Gamma,\alpha)})=\bigwedge_{v}o(\widehat{B}_{T(v)})=\mu\,o(\widehat{B}_{T(v_j)})\wedge \bigwedge_{v\neq v_j}o(\widehat{B}_{T(v)})$ is 
\[ \mu\,o(\widehat{B}_{T'})\wedge o(\widehat{B}_{T''})\wedge\bigwedge_{v\neq v_j}o(\widehat{B}_{T(v)}). \]
This agrees with that determined by $o(\Gamma_\sigma)$. 
\end{proof}

\section{Concluding remarks}\label{s:concluding}

The construction of the basic bracket operations for valence $\ell\geq 5$ in this paper is less explicit. 
\begin{Prob}\label{prob:explicit}
Construct the basic bracket operations for $\ell\geq 5$ from explicit families of string links.
\end{Prob}
To require the main properties of the basic bracket operations, we needed to impose the condition $n\geq 2\ell-3$, which resulted in the restriction $2k\geq 2\mu^2+8\mu+10$ in Theorem~\ref{thm:main}. This would be improved by a positive resolution of this problem\footnote{For only the improvement of the dimension restriction, there may be other strategies.}.
In particular, if our results can be extended to $2k=4$, the main results of \cite{LX} are strengthened. We know that there is a natural analogue of the constructon of Lemma~\ref{lem:null-iso} for $S^n\times S^n\times S^n\times S^n$, which partially looks like the Fulton--MacPherson compactification of the configuration space of 4 points on $D^n$. However, it was difficult for us to relate it to the decomposition structure of the 10T-link (in \S\ref{ss:10T-link}) in the boundary of a path, and we do not know whether such a model works. It would be interesting to compare our construction with the works of \cite{KKV20,KKV24,Koy,Kos24a,Kos24b}.

It is known that the rational homotopy types of $\overline{\Emb}_\partial (I^{m_1}\tcoprod\cdots\tcoprod I^{m_r}, I^n)_\iota$ for $n-m_i\geq 3$ can be expressed in terms of (hairy) graph complex (\cite{AT,BoaW,DH,FTW}) (see \S\ref{ss:v-surgery} for the notation $\overline{\Emb}_\partial (A, B)_{f_0}$).
A resolution of the following problem would be helpful to geometrically construct the rational homotopy type of the spaces of embeddings.
\begin{Prob}\label{prob:emb}
Give a construction similar to this paper for
\[ \begin{split}
  &H_*(\overline{\Emb}_\partial (I^m, I^n)_\iota; \Q)\quad(n-m\geq 3),\text{ or }\\
  &H_*(\overline{\Emb}_\partial (I^{m_1}\tcoprod\cdots\tcoprod I^{m_r}, I^n)_\iota; \Q)\quad(n-m_i\geq 3)
\end{split}\]
using hairy graphs, and relate it to the constructions by Longoni and K. Sakai, which involves 4-valent vertex (\cite{Lo,Sa}, see also \cite{PS}). 
\end{Prob}

\begin{appendix}

\section{Leaf forms via spanning manifolds}\label{s:leaf-form-mfd}

The leaf forms of Lemma~\ref{lem:leaf-form} could be given by the $\eta$-forms of spanning manifolds of string links. We see that the family of string links for the basic bracket for a 4-valent vertex admits spanning manifolds with the desired intersection property. Although the content of this section is not needed in the main result of this paper, it would help understanding the proofs of Lemma~\ref{lem:leaf-form} and Theorem~\ref{thm:chainmap} in a more explicit manner.

\subsection{Spanning manifolds for the basic brackets and their 4-fold intersection}

Let \[e_t\colon (I^{3n-2})^{\cup 4}\to I^{4n-1}, \ \ \  \mbox{(where $t\in [0,1]$)}\] be the null-isotopy of the IHX-link $e_1\colon (I^{3n-2})^{\cup 4}\to I^{4n-1}$. Thus $e_1$ is a sum of the three string links $e_1^1,e_1^2,e_1^3\colon (I^{3n-2})^{\cup 4}\to I^{4n-1}$ corresponding to the 3 terms in the IHX relation, respectively. For each $j\in\{1,2,3\}$ and for each component $a_\ell^j$ of $e_1^j$, there is a compact oriented $(3n-1)$-submanifold $S(a_\ell^j)$ of $I^{4n-1}$ such that
\begin{itemize}
\item $\partial S(a_\ell^j)$ is the closure of the string $a_\ell^j$ (taken along $\partial I^{4n-1}$),
\item $S(a_\ell^j)\cap a_m^j=\emptyset$ if $m\neq \ell$.
\end{itemize}
(See Figure~\ref{fig:double-link}.)
\begin{figure}[h]
\includegraphics[height=50mm]{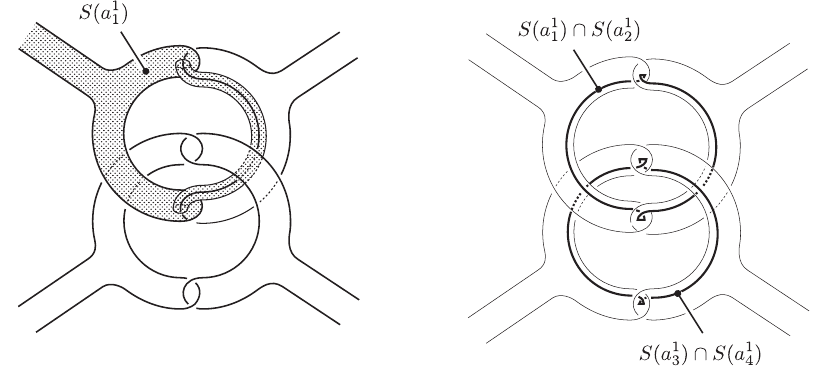}
\caption{Spanning manifolds $S(a_\ell^1)$ (for the first term). The loci of the double points form a link with linking number $\pm 1$.}\label{fig:double-link}
\end{figure}

The sum of the $(3n-1)$-submanifolds $S(a_\ell^j)$ can be closed in $I^{4n}=I^{4n-1}\times I$ as follows.
\begin{itemize}
\item Let $\widetilde{a}_\ell$ ($\ell=1,2,3,4$) be the locus of the $\ell$-th component of $e_t$ in $I^{4n}=I^{4n-1}\times I$. 
\item Let $a_\ell^0$ ($\ell=1,2,3,4$) be the components of the standard inclusion $e_0$. 
\item Let $S(a_\ell^0)$ be the spanning disks of $a_\ell^0$ in $I^{4n-1}\times \{0\}$.
\item  Let $S(a_\ell^1)\natural S(a_\ell^2)\natural S(a_\ell^3)$ denote the $(3n-1)$-submanifold of $I^{4n-1}\times \{1\}$ obtained by concatenating $S(a_\ell^j)$ and by rescaling, so that it bounds the closure of the $\ell$-th component of $e_1$. 
\end{itemize}
Let $S_\ell$ be the closed $(3n-1)$-manifold in $I^{4n}$ obtained by closing the $(3n-1)$-manifold 
\[ (S(a_\ell^1)\natural S(a_\ell^2)\natural S(a_\ell^3) )\cup \widetilde{a}_\ell \cup -S(a_\ell^0)\subset I^{4n} \]
by gluing $(\partial S(a_\ell^0)-\mathrm{Int}\,a_\ell^0)\times I\subset \partial{I^{4n-1}}\times I$ along the boundaries. (See Figure~\ref{fig:spanning-submfd}.) 
\begin{figure}[h]
\[ \includegraphics[height=40mm]{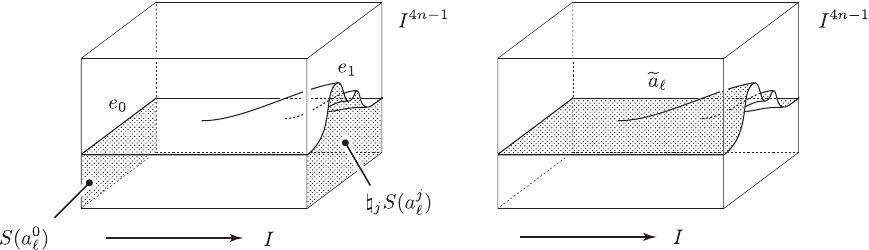} \]
\caption{The $(3n-1)$-manifolds $\natural_j S(a_\ell^j)$, $S(a_\ell^0)$, and $\widetilde{a}_\ell$ of $I^{4n-1}\times I$.}\label{fig:spanning-submfd}
\end{figure}

For an integer $\lambda>0$, let $\lambda S_\ell$ denote the closed $(3n-1)$-manifold in $I^{4n}$ defined similarly as $S_\ell$ but using the sum of $\lambda$ copies of $S(a_\ell^1)\natural S(a_\ell^2)\natural S(a_\ell^3)$.
The following proposition is a crucial property of the 4-valent vertex surgery to compute the configuration space integral. 

\begin{Prop}\label{prop:4-fold-int}
There exists an integer $\lambda>0$ such that for each $\ell$, the $(3n-1)$-manifold $\lambda S_\ell$ bounds a compact, framed, oriented $3n$-submanifold $C_\ell$ in $I^{4n}$ satisfying the following properties.
\begin{enumerate}
\item $C_\ell\cap \widetilde{a}_m=\emptyset$ if $m\neq \ell$.
\item The algebraic 4-fold intersection number 
$\langle C_1,C_2,C_3,C_4\rangle$ is nonzero.
\end{enumerate}
\end{Prop}

The existence of $\lambda$ and a submanifold $C_\ell$ can be verified by the Pontrjagin--Thom construction: by using the obstruction for the extension of a (smooth) map from the complement $E$ of an open tubular neighborhood of the strings $\widetilde{a}_1\cup \widetilde{a}_2\cup \widetilde{a}_3\cup \widetilde{a}_4\subset I^{4n}$ to $S^n$, which may exist in $H^{3n}(E,\partial E;\pi_{3n-1}(S^n))$ and $H^{4n}(E,\partial E;\pi_{4n-1}(S^n))$, and the finiteness of $\pi_{3n-1}(S^n)$ and $\pi_{4n-1}(S^n)$, which vanishes after possibly taking finite copies. Hence rather nontrivial part is the statement about the 4-fold intersection number.

\begin{Rem}
The manifold $S(a_\ell^1)\natural S(a_\ell^2)\natural S(a_\ell^3)$ should not be replaced by that induced from $S(a_\ell^0)$ by the isotopy extension along the null-isotopy, in which case the 4-fold intersection is zero. We must fix the behavior on the IHX side to normalize propagators. Note that $S(a_\ell^i)$ can be taken explicitly.
\end{Rem}

We thank Masamichi Takase for suggesting the following lemma.

\begin{Lem}\label{lem:local-model-4-fold}
Let $\lambda S_1,\lambda S_2,\lambda S_3,\lambda S_4\subset I^{4n}$ be as above. Suppose that
\begin{itemize}
\item $\lambda S_1,\lambda S_2,\lambda S_3,\lambda S_4$ intersects transversally in $\partial I^{4n}$ (see \S\ref{ss:comp-leaf-forms} for the transversality) with only double intersections (and without triple intersections),
\item for each pair $(i,j)$, the intersection $\lambda S_i\cap \lambda S_j$ is an embedded $(2n-1)$-sphere in $\partial I^{4n}$,
\item for each $i$, the $(3n-1)$-manifold $\lambda S_i$ bounds a $3n$-submanifold $C_i\subset I^{4n}$,
\item the intersections among $C_i$ are of \emph{generic} type.
\end{itemize}
Let $\langle C_1,C_2,C_3,C_4\rangle$ be the algebraic number of the 4-fold intersection among $C_1,C_2,C_3,C_4$ (counted with signs). For any partition 
$\{\{i,j\}, \{\ell,m\}\}$ of $\{1,2,3,4\}$, we have the following identity:
\[ \langle C_1,C_2,C_3,C_4\rangle=\pm \Lk(\lambda S_i\cap \lambda S_j, \lambda S_\ell\cap \lambda S_m). \]
\end{Lem}
\begin{figure}[h]
\includegraphics[height=45mm]{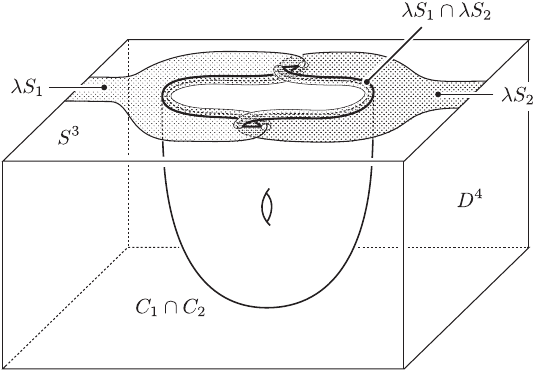}
\caption{Example for $n=1$. The knot $\lambda S_1\cap \lambda S_2$ bounds a surface $C_1\cap C_2$ in $I^4$.}\label{fig:C1-C2}
\end{figure}
\begin{proof}
For a partition $P=\{\{i,j\}, \{\ell,m\}\}$ of $\{1,2,3,4\}$, the \emph{submanifolds} $C_i\cap C_j$ and $C_\ell\cap C_m$ of $I^{4n}$ may intersect in finitely many points each of which gives a point of $C_1\cap C_2\cap C_3\cap C_4$ where $\Lk(\lambda S_i\cap \lambda S_j, \lambda S_\ell\cap \lambda S_m)$ changes. 
Conversely, any point of $C_1\cap C_2\cap C_3\cap C_4$ can be viewed as a bifurcation point for $\Lk(\lambda S_i\cap \lambda S_j, \lambda S_\ell\cap \lambda S_m)$. Note that the intersections among the double surfaces $C_i\cap C_j$ may not be generic in the sense that, for example, the intersection between $C_i\cap C_j$ and $C_\ell\cap C_m$ occurs at the same time as that between $C_i\cap C_\ell$ and $C_j\cap C_m$. 
It follows from the local model of a 4-fold intersection point (Figure~\ref{fig:link-4-fold-point}) that at each 4-fold interesction point, the number $\Lk(\lambda S_i\cap \lambda S_j, \lambda S_\ell\cap \lambda S_m)$ changes for all $P$. 
\end{proof}
\begin{figure}[h]
\[ \includegraphics[height=50mm]{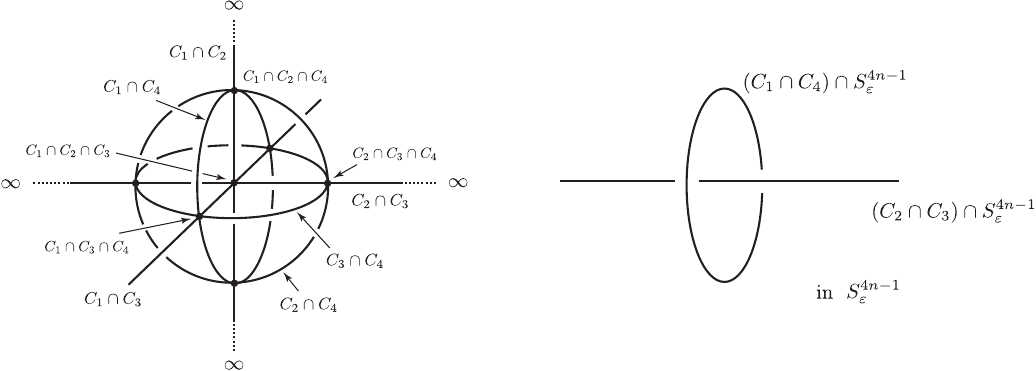} \]
\caption{The intersection $(C_1\cup C_2\cup C_3\cup C_3)\cap S^{4n-1}_\ve$ with a small $(4n-1)$-sphere $S^{4n-1}_\ve\subset \mathrm{Int}\,I^{4n}$ around a point of $C_1\cap C_2\cap C_3\cap C_3$. The linking number $\Lk((C_i\cap C_j)\cap S^{4n-1}_\ve,(C_\ell\cap C_m)\cap S^{4n-1}_\ve)$ is $\pm 1$ for any partition $\{i,j\}\cup\{\ell,m\}=\{1,2,3,4\}$.)}\label{fig:link-4-fold-point}
\end{figure}

\begin{Rem}
The assumption of Lemma~\ref{lem:local-model-4-fold} that the intersection among $\lambda S_i$ does not have triple point is in fact unnecessary. If there are triple points, then the double spheres $\lambda S_i\cap \lambda S_j$ are immersed, and $C_i\cap C_j$ are immersed, too. 
\end{Rem}

\begin{proof}[Proof of Proposition~\ref{prop:4-fold-int}]
As mentioned above, there are submanifolds $C_\ell$ in $I^{4n}$ such that $\partial C_\ell=\lambda S_\ell$. Suppose without loss of generality that the intersection among $C_1,C_2,C_3,C_4$ is of generic type and that the assumption of Lemma~\ref{lem:local-model-4-fold} is satisfied. 
Then by Lemma~\ref{lem:local-model-4-fold} and by $\Lk(\lambda S_i\cap \lambda S_j, \lambda S_\ell\cap \lambda S_m)=\pm \lambda$ (see Figure~\ref{fig:double-link}), we obtain that the number of points of $C_1\cap C_2\cap C_3\cap C_4$ is nonzero.
\end{proof}

\begin{Rem}\label{rem:MC-eq}
Let $\theta_i=\eta_{C_i}$ ($i=1,2,3,4$) be the $\eta$-forms of the submanifold $C_i$ of $E$ of Proposition~\ref{prop:4-fold-int}. Let $\theta_{ij}$ ($1\leq i<j\leq 4$), $\theta_{ijk}$ ($1\leq i<j<k\leq 4$), and $\theta_{ij,k}$ ($1\leq i<j\leq 4$, $1\leq k\leq 4$) be forms on $E$ of degrees $2n-1$, $3n-1$, $3n-2$, respectively, such that
\[ 
  d\theta_{ij}=-\theta_i\wedge\theta_j,\qquad d\theta_{ijk}=-\theta_i\wedge\theta_j\wedge\theta_k,\qquad d\theta_{ij,k}=\theta_{ijk}-\theta_{ij}\wedge\theta_k,
\]
where we consider that $\theta_{\sigma(i)\sigma(j)\sigma(k)}=(-1)^{n^2r}\theta_{ijk}$ for a permutation $\sigma\colon\{i,j,k\}\to \{i,j,k\}$ when $i<j<k$ and when $\theta_{\sigma(i)}\wedge\theta_{\sigma(j)}\wedge\theta_{\sigma(k)}=(-1)^{n^2r}\theta_i\wedge\theta_j\wedge\theta_k$, and that $\theta_{ijk}=0$ if at least two of $i,j,k$ agree.
The form $\theta_{ij}$ can be represented by the $\eta$-form of spanning disks of the disjoint union of spheres $C_i\cap C_j$ (in $I^{4n}$, modified near the boundary $C_i\cap C_j$), the form $\theta_{ijk}$ can be represented by the $\eta$-form of a spanning disk of the sphere $C_i\cap C_j\cap C_k$, and the form $\theta_{ij,k}$ can be represented by a spanning disk between the two spanning disks of the sphere $C_i\cap C_j\cap C_k$ taken for $\theta_{ijk}$ and $\theta_{ij}\wedge \theta_k$, respectively.
Let $\L=\L(x_1,x_2,x_3,x_4)$ be the free $L_\infty$-algebra generated over $\R$ by degree $n$ variables $x_1,x_2,x_3,x_4$ with $\partial x_i=0$, $\partial[x_i,x_j]=0$, and $\partial[[x_i,x_j],x_k]=0$. Let $I^4\L$ be the subspace of $\L$ generated by brackets with at least 4 inputs.
Let $\theta\in\Omega^*(E)\otimes\L/I^4\L$ be defined by
\[ \begin{split}
\theta=\sum_{1\leq i\leq 4}\theta_ix_i+\sum_{1\leq i<j\leq 4}\theta_{ij}[x_i,x_j]
&+\sum_{{1\leq i<j\leq 4,}\atop{1\leq k\leq 4}}\theta_{ij,k}[[x_i,x_j],x_k]\\
&+\sum_{1\leq i<j<k\leq 4}\theta_{ijk}[x_i,x_j,x_k] \quad (\text{mod $I^4\L$}).
\end{split}\]
Then it satisfies a version of the Maurer--Cartan equation:
\[ d\theta-\partial\theta+\frac{1}{2!}[\theta,\theta]+\frac{1}{3!}[\theta,\theta,\theta]=0\quad (\text{mod $I^4\L$}).\]
This is analogous to the construction of \cite[Example~3]{Hai} for the Borromean rings in $S^3$.
\end{Rem}
\section{Extending closed forms}\label{s:ext-closed}

\begin{Lem}\label{lem:ext-closed}
Let $M$ be a compact manifold and let $D$ be a codimension 0 compact submanifold of $M$. Suppose that closed $i$-forms $\alpha,\beta$ on $D,M$, respectively, are given. If the following identity holds for any piecewise smooth $i$-cycle $c$ in $S_i(D;\R)$: 
\[ \int_c\alpha=\int_c\beta, \]
then there is a closed extension of $\alpha$ over $M$ that is cohomologous to $\beta$.
\end{Lem}
\begin{proof}
Let $j\colon D\to M$ be the inclusion. The assumption implies that $j^*[\beta]=[\alpha]$ in $H^i(D;\R)$, which is identified with $\mathrm{Hom}(H_i(D;\R),\R)$ by the integration $\omega\leftrightarrow(c\mapsto \int_c\omega)$. This shows that $\alpha$ has a desired extension.
\end{proof}

\section{The signs in the Jacobi identity}\label{s:sign-jacobi}

\begin{Lem}\label{lem:sign-jacobi}
The following identity holds for the triple bracket defined as in \S\ref{ss:explict-null}:
\[ \partial[a,b,c]
=[[a,b],c]+(-1)^{pq+pr}[[b,c],a]+(-1)^{pr+qr}[[c,a],b]. \]
\end{Lem}
\begin{proof}
We consider the unit disks $D^p,D^q,D^r,D^s$ of $\R^p,\R^q,\R^r,\R^s$, respectively.
We take points $t=(1,0,\ldots,0)\in D^p$, $u=(1,0,\ldots,0)\in D^q$, $v=(1,0,\ldots,0)\in D^r$. We consider the cell $D^p\times D^q\times D^r=D^p\times D^q\times D^r\times\{0\}$ in $D^p\times D^q\times D^r\times D^s$. We choose the following orientation of $D^p\times D^q\times D^r$:
\[ o=\partial x_1\wedge\cdots\wedge\partial x_p\wedge \partial y_1\wedge\cdots\wedge\partial y_q\wedge
\partial z_1\wedge\cdots\wedge \partial z_r, \]
where we abbreviate $\frac{\partial}{\partial x_1}$ etc. as $\partial x_1$ etc. 
The orientations on the boundary faces induced from $o$ are as follows:
\begin{enumerate}
\item $\partial D^p\times D^q\times D^r$ at $(t,0,0)$:
$\partial x_2\wedge\cdots\wedge\partial x_p\wedge \partial y_1\wedge\cdots\wedge\partial y_q\wedge
\partial z_1\wedge\cdots\wedge \partial z_r$
\item $D^p\times\partial D^q\times D^r$ at $(0,u,0)$:
$(-1)^p\partial x_1\wedge\cdots\wedge\partial x_p\wedge \partial y_2\wedge\cdots\wedge\partial y_q\wedge
\partial z_1\wedge\cdots\wedge \partial z_r$
\item $D^p\times D^q\times \partial D^r$ at $(0,0,v)$:
$(-1)^{p+q}\partial x_1\wedge\cdots\wedge\partial x_p\wedge \partial y_1\wedge\cdots\wedge\partial y_q\wedge
\partial z_2\wedge\cdots\wedge \partial z_r$
\end{enumerate}

The $(p+q+r-2)$-cycle $[[a,b],c]$ in $\partial(D^p\times D^q\times D^r)$ can be obtained from the $(p+q+r-2)$-manifold $\partial(D^p\times D^q)\times \partial D^r\cong S^{p+q-1}\times S^{r-1}$ by attaching an $r$-handle (having $D^r$ as the core) along $\mathrm{pt}\times S^{r-1}$. The orientation of $\partial(D^p\times D^q)\times \partial D^r$ at $(t,0,v)$ induced from (3) is
\[ (-1)^{p+q}\partial x_2\wedge\cdots\wedge\partial x_p\wedge \partial y_1\wedge\cdots\wedge\partial y_q\wedge
\partial z_2\wedge\cdots\wedge \partial z_r. \]
Since the null-isotopy disk $[a,b,c]\subset \partial(D^p\times D^q\times D^r)$ is complemental to $(D^p_{1-\ve}\times D^q_{1-\ve})\times \partial D^r$, it has the orientation 
\[(-1)^{p+q-1}\partial x_2\wedge\cdots\wedge\partial x_p\wedge \partial y_1\wedge\cdots\wedge\partial y_q\wedge
\partial z_2\wedge\cdots\wedge \partial z_r \]
at $(t,0,v)$. 
On the other hand, $[a,b]$, which is represented by the sphere $\partial(D^p\times D^q)$, has orientation $\partial x_2\wedge\cdots\wedge \partial x_p\wedge \partial y_1\wedge\cdots\wedge \partial y_q$ at $(t,0)\in\partial D^p\times D^q$. Let $U^{p+q-1}\subset \partial(D^p\times D^q)$ be the $(p+q-1)$-cell. Then the cycle $[[a,b],c]$, which is represented by the sphere $\partial(U^{p+q-1}\times D^r)$, has the orientation 
\[ (-1)^{p+q-1}\partial x_2\wedge\cdots\wedge \partial x_p\wedge \partial y_1\wedge\cdots\wedge \partial y_q\wedge \partial z_2\wedge\cdots\wedge \partial z_r \]
at $(t,0,v)$. This gives the coefficient 1 of $[[a,b],c]$.

The $(p+q+r-2)$-cycle $[[b,c],a]$ in $\partial(D^p\times D^q\times D^r)$ can be obtained from the $(p+q+r-2)$-manifold $\partial D^p\times \partial(D^q\times D^r)\cong S^{p-1}\times S^{q+r-1}$ by attaching a $p$-handle (having $D^p$ as the core) along $S^{p-1}\times\mathrm{pt}$. The orientation of $\partial D^p\times \partial(D^q\times D^r)$ at $(t,u,0)$ induced from (1) is
\[ (-1)^{p-1}\partial x_2\wedge\cdots\wedge\partial x_p\wedge \partial y_2\wedge\cdots\wedge\partial y_q\wedge
\partial z_2\wedge\cdots\wedge \partial z_r. \]
Since the null-isotopy disk $[a,b,c]\subset \partial(D^p\times D^q\times D^r)$ is complemental to $\partial D^p\times (D^q_{1-\ve}\times D^r_{1-\ve})$, it has the orientation 
\[ (-1)^{p}\partial x_2\wedge\cdots\wedge\partial x_p\wedge \partial y_2\wedge\cdots\wedge\partial y_q\wedge
\partial z_2\wedge\cdots\wedge \partial z_r \]
at $(t,u,0)$. 
On the other hand, $[b,c]$, which is represented by the sphere $\partial(D^q\times D^r)$, has orientation $\partial y_2\wedge\cdots\wedge \partial y_q\wedge \partial z_1\wedge\cdots\wedge \partial z_r$ at $(v,0)\in\partial D^q\times D^r$. Let $U^{q+r-1}\subset \partial(D^q\times D^r)$ be the $(q+r-1)$-cell. Then the cycle $[[b,c],a]$, which is represented by the sphere $\partial(U^{q+r-1}\times D^p)$, has the orientation 
\[ \begin{split}
&(-1)^{q+r-1}\partial y_2\wedge\cdots\wedge \partial y_q\wedge \partial z_1\wedge\cdots\wedge \partial z_r\wedge \partial x_2\wedge\cdots\wedge \partial x_p\\
&=(-1)^{p(q+r-1)} \partial x_2\wedge\cdots\wedge \partial x_p\wedge\partial y_2\wedge\cdots\wedge \partial y_q\wedge \partial z_1\wedge\cdots\wedge \partial z_r
\end{split} \]
at $(t,u,0)$. This gives the coefficient $(-1)^{p(q+r-1)}(-1)^{p}=(-1)^{pq+pr}$ of $[[b,c],a]$.

The $(p+q+r-2)$-cycle $[[c,a],b]$ in $\partial(D^p\times D^q\times D^r)$ can be obtained from the $(p+q+r-2)$-manifold $\partial(D^r\times D^p)\times \partial D^q\cong S^{r+p-1}\times S^{q-1}$ by attaching a $q$-handle (having $D^q$ as the core) along $\mathrm{pt}\times S^{q-1}$. The orientation of $\partial(D^r\times D^p)\times \partial D^q$ at $(0,t,u)$ induced from (2) is
\[ (-1)^p\partial x_2\wedge\cdots\wedge\partial x_p\wedge \partial y_2\wedge\cdots\wedge\partial y_q\wedge
\partial z_1\wedge\cdots\wedge \partial z_r. \]
Since the null-isotopy disk $[a,b,c]\subset \partial(D^p\times D^q\times D^r)$ is complemental to $(D^r_{1-\ve}\times D^p_{1-\ve})\times \partial D^q$, it has the orientation 
\[ (-1)^{p-1}\partial x_2\wedge\cdots\wedge\partial x_p\wedge \partial y_2\wedge\cdots\wedge\partial y_q\wedge
\partial z_1\wedge\cdots\wedge \partial z_r \]
at $(0,t,u)$. 
On the other hand, $[c,a]$, which is represented by the sphere $\partial(D^r\times D^p)$, has orientation $(-1)^{pr}\partial x_2\wedge\cdots\wedge \partial x_q\wedge \partial z_1\wedge\cdots\wedge \partial z_r$ at $(0,t)\in D^r\times \partial D^p$. Let $U^{p+r-1}\subset \partial(D^r\times D^p)$ be the $(p+r-1)$-cell. Then the cycle $[[c,a],b]$, which is represented by the sphere $\partial(U^{p+r-1}\times D^q)$, has the orientation 
\[ \begin{split}
&(-1)^{pr}(-1)^{p+r-1}\partial x_2\wedge\cdots\wedge \partial x_p\wedge \partial z_1\wedge\cdots\wedge \partial z_r\wedge \partial y_2\wedge\cdots\wedge \partial y_q\\
&=(-1)^{pr+qr+p-1} \partial x_2\wedge\cdots\wedge \partial x_p\wedge \partial z_1\wedge\cdots\wedge \partial z_r\wedge \partial y_2\wedge\cdots\wedge \partial y_q
\end{split} \]
at $(0,t,u)$. This gives the coefficient $(-1)^{pr+qr+p-1}(-1)^{p-1}=(-1)^{pr+qr}$ of $[[c,a],b]$.
\end{proof}

\begin{Lem}\label{lem:sym-relation}
\[ \begin{split}
& [b,a,c]=(-1)^{pq}[a,b,c],\quad [a,c,b]=(-1)^{qr}[a,b,c],\quad 
[c,b,a]=(-1)^{pq+pr+qr}[a,b,c],\\
& [c,a,b]=(-1)^{qr+pr}[a,b,c],\quad [b,c,a]=(-1)^{pq+pr}[a,b,c].
\end{split} \]
In other words, $[-,-,-]$ is symmetric in the graded sense.
\end{Lem}
\begin{proof}
It follows from the construction of the triple bracket that the null-isotopy disk to define $[a,b,c]$ is symmetric up to sign with respect to the permutation of the factors.
The relations of the lemma follow by taking the boundaries on both sides and then applying Lemma~\ref{lem:sign-jacobi}.
\end{proof}

\end{appendix}


\end{document}